\documentclass[11 pt]{article}
\usepackage{latexsym}
\usepackage{amssymb, amsmath}
\usepackage{amsthm}
\usepackage{enumerate}
\usepackage{geometry}
\usepackage{tikz}
\usepackage[colorlinks]{hyperref}
\usepackage{changebar}
\usepackage{comment}
\usepackage{esint, bbm}
\usepackage[normalem]{ulem}

\numberwithin{equation}{section}
\allowdisplaybreaks

\newtheorem{theorem}{Theorem}[section]
\newtheorem{lemma}[theorem]{Lemma}
\newtheorem{cor}[theorem]{Corollary}

\newtheorem{prop}[theorem]{Proposition}

\newtheorem{defin}[theorem]{Definition}

\newtheorem{remark}[theorem]{Remark}

\newcommand{\cK}{\mathcal{K}}
\newcommand{\Lp}{\mathcal{L}}
\newcommand{\cL}{\mathcal{L}}

\newcommand{\cC}{\mathcal{C}}
\newcommand{\cD}{\mathcal{D}}
\newcommand{\cI}{\mathcal{I}}

\newcommand{\cF}{\mathcal{F}}

\newcommand{\G}{\mathcal{G}}
\newcommand{\cH}{\mathcal{H}}
\newcommand{\cS}{\mathcal{S}}

\newcommand{\cO}{\mathcal{O}}
\newcommand{\cQ}{\mathcal{Q}}
\newcommand{\cG}{\mathcal{G}}
\newcommand{\cW}{\mathcal{W}}

\newcommand{\cA}{\mathcal{A}}

\newcommand{\pa}{\mathcal{P}}
\newcommand{\cR}{\mathcal{R}}

\newcommand{\bN}{\mathbb{N}}
\newcommand{\R}{\mathbb{R}}

\newcommand{\bG}{\mathbb{G}}
\newcommand{\bH}{\mathbb{H}}
\newcommand{\bP}{\mathbb{P}}
\newcommand{\bR}{\mathbb{R}}

\newcommand{\bT}{\mathbb{T}}
\newcommand{\bV}{\mathbb{V}}
\newcommand{\bZ}{\mathbb{Z}}

\newcommand{\bd}{\mathbbm d}

\newcommand{\musrb}{\mu_{\srb}}
\newcommand{\srb}{\mbox{\tiny \textup SRB}}

\newcommand{\bx}{\bar{x}}
\newcommand{\by}{\bar{y}}

\newcommand{\diam}{\mbox{diam}}

\newcommand{\vf}{\varphi}

\newcommand{\tpsi}{\tilde{\psi}}

\newcommand{\ve}{\varepsilon}
\newcommand{\tS}{\tilde{S}}
\newcommand{\tT}{\widetilde{T}}
\newcommand{\hT}{\widehat{T}}
\newcommand{\rT}{\mathring{T}}
\newcommand{\rL}{\mathring{\cL}}
\newcommand{\rM}{\mathring{M}}
\newcommand{\tM}{\widetilde{M}}

\newcommand{\tC}{\widetilde{\cC}}

\newcommand{\tH}{\widetilde{H}}
\newcommand{\tJ}{\widetilde{J}}
\newcommand{\tW}{\widetilde{W}}

\newcommand{\tf}{\tilde{f}}

\newcommand{\hL}{\widehat{L}}

\newcommand{\wcW}{{\overline{W}}}

\newcommand{\ind}{\mathbbm{1}}

\newcommand{\tri}{{|\:\!\!|\:\!\!|}}

\def\beq{\begin{equation}}
\def\eeq{\end{equation}}

\begin{document}

\title{Projective Cones for Sequential Dispersing Billiards}
\author{Mark F. Demers\thanks{Department of Mathematics, Fairfield University, Fairfield CT 06824, USA.  Email: mdemers@fairfield.edu} 
\and Carlangelo Liverani\thanks{Dipartimento di Matematica, Universit\`a di Roma II, Tor Vergata, 00133 Roma, Italy.  Member of the  GNFM-INDAM Rome, Italy. Email: liverani@mat.uniroma2.it. \hfill\break
\indent We thank Jochen Broecker, Tobias Kuna and especially Lea Oljaca for pointing out to us the possibility to prove Lemma \ref{lem:b property}. This work was supported by the PRIN Grant ``Regular and stochastic behaviour in dynamical systems" (PRIN 2017S35EHN). C.L. acknowledges the MIUR Excellence Department Project awarded to the Department of Mathematics, University of Rome Tor Vergata, CUP E83C18000100006.  M.D. was partially supported by NSF grants DMS 1800321 and DMS 2055070.}
}

\date{\today}

\maketitle

\begin{abstract}
We construct Birkhoff cones for dispersing billiards, which are contracted by the action of the transfer operator. This construction permits the study of statistical properties not only of regular dispersing billiards but also of sequential billiards (the billiard changes at each collision in a prescribed manner), open billiards (the dynamics exits some region or dies when hitting some obstacle) and many other examples.  In particular, we include applications to
chaotic scattering and the random Lorentz gas.
\end{abstract}


\section{Introduction}

Billiards are a ubiquitous source of models in physics, in particular in Statistical Mechanics. The study of the ergodic properties of billiards is of paramount importance for such applications and also a source of innovative ideas in Ergodic Theory. In particular, starting with at least \cite{Kry}, it has become clear that a quantitive estimate of the speed of convergence to equilibrium is pivotal for this research program. The first strong result of this type dates back to Bunimovich, Sinai and Chernov \cite{bsc} in 1990, but it relies on a Markov-partition-like technology that is not very well suited to producing optimal results.  The next breakthrough is due to Lai-Sang Young \cite{You98, You99} who put forward two techniques, towers and coupling, well suited to study the decay of correlations for a large class of systems,  billiards included. The idea of coupling was subsequently refined by Dolgopyat \cite{Do04a,Do04b, Do05} who introduced the notion of standard pairs, which have proved a formidable tool to study the statistical properties of dynamical systems in general and billiards in particular \cite{Ch06, Ch08, CD, chernov zhang}.  See \cite[Chapter~7]{chernov book} for a detailed exposition of these ideas and related references. 

In the meantime another powerful idea has appeared, following the seminal work of Ruelle \cite{RS75, Ruelle76} and Lasota-Yorke \cite{LY73}, to study the spectral properties of the associated transfer operator acting on
spaces of functions adapted to the dynamics.  After some preliminary attempts \cite{Fr86, Rugh, Kitaev}, the functional approach for hyperbolic systems was launched by the seminal paper \cite{BKL}, which was quickly followed and
refined by a series of authors, including \cite{Ba1, gouezel liverani,  BaT, gouezel liverani2}. 
Such an approach, when applicable, has provided the strongest results so far, see \cite{Babook} for a recent review. In particular, building on a preliminary result by Demers and Liverani \cite{DL08}, it has been applied to billiards by Demers and collaborators \cite{demzhang11, demzhang13, demzhang14, dem bill, max, thermo}. 
This has led to manifold results, notably the proof of exponential decay of correlations for certain billiard flows \cite{BDL}.

Yet, lately there has been a growing interest in non-stationary systems, when the dynamical system changes with time. Since most systems of interest are not isolated, not even in first approximation, the possibility of a change to 
the system due to external factors clearly has physical relevance.  Another important scenario in which non-stationarity appears is in dynamical systems in random media, e.g. \cite{AL}. The functional approach as such seems not to be well suited to treat these situations since it is based on the study of an operator via spectral theory. 
In the non-stationary case a single operator is substituted by a product of different operators and spectral theory 
does not apply. 

There exist several approaches that can be used to overcome this problem, notably:
\begin{enumerate}
\item consider random systems; in this case, especially in the annealed case, it is possible to recover an averaged transfer operator to which the theory applies. More recently, the idea has emerged to study quenched systems via infinite dimensional Oseledets theory, see e.g. \cite{DFGV1, DFGV2} and references therein; 
\item consider only slowly changing systems that can be treated using the perturbation theory in \cite{KL99, gouezel liverani}. For example, see \cite{DS}, and references therein, for some recent work in this direction;
\item use the technology of standard pairs, which has the advantage of being very flexible and applicable to the non-stationary case \cite{young zhang}.  Note that the standard pair technology and the above perturbation ideas can be profitably combined together, see \cite{DS16, DS18, DLPV};
\item use the cone and Hilbert metric technology introduced in \cite{liv95, liv95b, LM},  which has also been extended
to the random setting \cite{AL, atnip}.
\end{enumerate}
The first two approaches, although effective, impose severe limitations on the class of nonstationary systems that can be studied. The second two approaches are more general and seem more or less equivalent. However, coupling arguments are often cumbersome to write in detail and usually provide weaker quantitative estimates compared to the cone method. 

Therefore, in the present article we develop the cone method and demonstrate that it can be successfully applied to billiards. Indeed, we introduce a relatively simple cone that is contracted by a large class of billiards. This implies that one can easily prove a loss of memory result for sequences of billiard maps.  To show that the previous results have concrete applications we devote one third of this paper to developing applications to several  physically relevant classes of models.

We emphasize that the present paper does not exhaust the possible applications of the present ideas. To have a more complete theory one should consider, to mention just a few, billiards with corner points, billiards with electric or magnetic fields, billiards with more general reflection laws, measures different from the SRB measure (that is transfer operators with generalized potentials as in \cite{max, thermo}), etc. We believe that most of these cases can be treated by small modifications of the present theory; however, the precise implementation does require a 
non-negligible amount of work and hence exceeds the scope of this presentation, which aims only at introducing the basic ideas and producing a viable cone for dispersing billiards.

The plan of the paper is as follows.
In Section \ref{setting} we introduce the class of billiards from which we will draw our sequential dynamics and summarize our main  analytical results regarding cone contraction. 
In Section~\ref{sec:hyp} we present the uniform properties of hyperbolicity and singularity sets enjoyed 
by our class of maps, listed as {\bf (H1)}-{\bf (H5)}; we also prove a Growth Lemma for our sequences of maps  and introduce one of our main characters, the transfer operator.
In Section \ref{sec:conedef} we introduce our protagonist, the cone (see Section \ref{sec:cone_dist}).
 Section \ref{sec:cone} is devoted to showing that the cone so defined is invariant under the action of the transfer operators of the billiards in question.
In Section \ref{sec:L contract} we show that in fact the cone is eventually strictly invariant (the image has finite diameter in the associated Hilbert metric) thanks to some mixing properties of the dynamics on a finite scale. 
The strict cone contraction implies exponential mixing for a very large class of observables and densities as is explained in Section \ref{sec:exp-mix}.
 Finally, Section \ref{sec:appl} contains the announced applications, first to sequential systems with holes (open systems), then to chaotic scattering and finally to the random Lorentz gas.

 
\section{Setting and Summary of Main Results }
\label{setting}
 
Since we are interested in studying sequential billiards, below we define a set of billiard tables that will have
uniform hyperbolicity constants, following \cite{demzhang13}.  Other classes of billiards are also studied in
\cite{demzhang13}, such as infinite horizon billiards, billiards under small external forces and some types of
nonelastic reflections.  While such classes of billiards are amenable to the present technique, we do not
treat the most general case here since the greater number of technicalities would obscure the main ideas we are 
trying to present.


\subsection{Families of billiard tables with uniform properties}
\label{sec:bill family}

We first choose $K \in \mathbb{N}$ 
and numbers $\ell_i >0$, $i = 1, \ldots K$.  Let
$M = \cup_{i=1}^K I_i \times [- \frac{\pi}{2}, \frac{\pi}{2}]$, where for each $i$, $I_i=[0, \ell_i]/\sim$ is an interval of length $\ell_i$ with endpoints identified. 
$M$ will be the phase space common to our collection of billiard maps.

Given $K$ and $\{ \ell_i \}_{i=1}^K$, 
we use the notation $Q = Q(\{ B_i \}_{i=1}^K)$ to denote the 
billiard table $\mathbb{T}^2 \setminus (\cup_{i=1}^K B_i)$, where each $B_i$ is a closed, convex set whose boundary has 
arclength $\ell_i$.  We assume that the scatterers $B_i$ are pairwise disjoint and that each
$\partial B_i$ is a $C^3$ curve with strictly positive curvature.

The billiard flow is defined by the motion of a point particle traveling at unit speed in
$Q := \mathbb{T}^2 \setminus (\cup_i B_i)$ and reflecting elastically at collisions.
The associated billiard map $T$ is the discrete-time collision map which maps a point on
$\partial Q$ to its next collision.  Parameterizing $\partial Q$ according to an arclength parameter
$r$ (oriented clockwise on each obstacle $B_i$) and denoting by $\varphi$ the angle made
by the post-collision velocity vector and the outward pointing normal to the boundary yields
the canonical coordinates for the phase space $M$ of the billiard map.  In these coordinates,
$M = \cup_i I_i \times [-\frac \pi2, \frac \pi2]$, as defined previously.

For $x = (r,\vf) \in M$, let $\tau(x)$ denote the time until the next collision for $x$ under the
flow.  We assume that $\tau$ is bounded on $M$, i.e. the billiard has finite horizon.
Thus since the scatterers are disjoint, there exist constants $\tau_{\min}(Q), \tau_{\max}(Q) > 0$ depending on
the configuration $Q$ such that $\tau_{\min}(Q) \le \tau(x) \le \tau_{\max}(Q) < \infty$
for all $x \in M$.  
Moreover, by assumption there exists $\cK_{\min}(Q), \cK_{\max}(Q) >0$ such that if $\cK(r)$ denotes the curvature
of the boundary at coordinate $r$, then $\cK_{\min}(Q) \le \cK(r) \le \cK_{\max}(Q)$.  Finally, let
$E_{\max}(Q)$ denote the maximum value of the $C^3$ norm of the curves comprising $\partial Q$
when parametrized according to arclength.

Now fix $\tau_*, \cK_*, E_* \in \mathbb{R}^+$, and let $\cQ(\tau_*, \cK_*, E_*)$ denote the collection
of all billiard tables $Q( \{ B_i \}_{i=1}^K )$ such that 
\[ 
\tau_* \le \tau_{\min}(Q) \le \tau_{\max}(Q) \le \tau_*^{-1}, \; \;
\cK_* \le \cK_{\min}(Q) \le \cK_{\max}(Q) \le \cK_*^{-1}, \mbox{ and } \cK_* \le E_{\max}(Q) \le E_*.
\]  
To each table in $Q \in \cQ(\tau_*, \cK_*, E_*)$
corresponds a billiard flow and hence a billiard map $T = T(Q)$ and associated collision times.
Let $\cF(\tau_*, \cK_*, E_*)$ denote the collection of billiard maps induced by configurations
in $\cQ(\tau_*, \cK_*, E_*)$, i.e.,
\[
\cF(\tau_*, \cK_*, E_*) = \{ T = T(Q) : Q \in \cQ(\tau_*, \cK_*, E_*) \}, .
\]
Thus each $T \in \cF(\tau_*, \cK_*, E_*)$ is identified with\footnote{ We do not claim
that each such $T$ is unique.  It may be that $T(Q) = T(Q')$ pointwise
(consider a $90^\circ$ rotation of a given configuration $Q$), yet for our purposes they will be considered distinct elements of $\cF(\tau_*, \cK_*, E_*)$.} 
a table $Q \in \cQ(\tau_*, \cK_*, E_*)$, which we denote by $Q(T)$.
Note that all $T \in \cF(\tau_*, \cK_*, E_*)$ have the same phase space $M$
since we have fixed $K$ and the arclengths $\{ \ell_i \}_{i=1}^K$.

It is a standard fact that all $T \in \cF(\tau_*, \cK_*, E_*)$ preserve the same smooth invariant probability 
measure, $d\musrb = c \cos \vf \, dr \, d\vf$, where
$c = \frac{1}{2|\partial Q|} = \frac{1}{2 \sum_{i=1}^K \ell_i}$ is the normalizing constant \cite{chernov book}.
In addition, all $T \in \cF(\tau_*, \cK_*, E_*)$ are mixing with respect to $\musrb$ and so are
topologically mixing \cite{Sinai} (see also, \cite[Section~6.7]{chernov book}). 

It is proved\footnote{The abstract set-up in \cite{demzhang13} also allows billiard tables with infinite
horizon and those subjected to external forces, but we are not concerned with the most general case here.} in \cite[Theorem~2.7]{demzhang13} that all $T \in \cF(\tau_*, \cK_*, E_*)$ satisfy properties
(H1)-(H5) of that paper with uniform constants depending only on $\tau_*, \cK_*$ and $E_*$.  We recall
the relevant properties in Section~\ref{sec:hyp} that we shall use throughout the paper and label 
them {\bf (H1)}-{\bf (H5)}.

\begin{remark}
The assumption that all scatterers have the same arclength is made for convenience so
that there
is a single cone $\cC$ on which all our operators $\cL_T$, $T \in \cF$, act.  This can be relaxed slightly 
once the hyperbolicity constant $\Lambda := 1 + 2\cK_* \tau_*$ has been introduced in {\bf (H1)} by
allowing the arclength of the boundary of each scatterer to change by no more than $\ve_1$, where
$\ve_1< \frac{\Lambda-1}{\Lambda+1}$, since then rescaling the arclength parametrization of $\partial B_i$ 
to be again $[0, \ell_i]$ yields a map with similar properties {\bf (H1)}-{\bf (H5)}, but
with slightly weakened hyperbolicity constant $\tilde{\Lambda} = \Lambda \frac{1-\ve_1}{1+\ve_1}>1$ (and $\theta_0$ from {\bf (H3)} is weakened
accordingly.)

To change the arclengths drastically would force us to consider a sequence of cones $\cC_n$ on a
sequence of phase spaces $M_n$.  This would require further suitable assumptions on the maps
$T_n : M_n \to M_{n+1}$ in order to ensure hyperbolicity, and such assumptions could be tailored to 
specific applications.  We do not pursue this generality here,
but remark that for example, it would be possible to formulate such a generalization for the random Lorentz gas
with gates described in Section~\ref{sec:lorentz}, in which the central scatterer in each cell is allowed to change arclength
and the resulting billiard map between cells would still satisfy {\bf (H1)}-{\bf (H5)} (albeit the normalization in
{\bf (H5)} would vary).
\end{remark}

Next, we define a notion of distance in $\cQ(\tau_*, \cK_*, E_*)$ as follows. 
Each table $Q$ comprises $K$ obstacles $B_i$.  Each $\partial B_i$ can be parametrized according to arclength
by a function $u_i : I_i \to \mathbb{R}^2$ (unfolding $\mathbb{T}^2$).  Since two arclength parametrizations of $\partial B_i$ can differ only in their
starting point, the collection $u_{i,\theta}$, $\theta \in [0, \ell_i)$, denotes the set of parametrizations associated with 
$\partial B_i$.  
Similarly, for a configuration $\widetilde Q$, denote the parametrizations of obstacles
by $\tilde{u}_{i, \theta}$, $\theta \in [0, \tilde{\ell}_i)$.
Let $\Pi_K$ denote the set of permutations $\pi$ on $\{ 1, \ldots K \}$
which satisfy $\tilde{\ell}_{\pi(i)} = \ell_i$.
Then define
\begin{equation}
\label{eq:d def}
\bd(Q, \widetilde{Q}) = \min_{\pi \in \Pi_K} \min_{\theta \in [0, \ell_i ) }  \sum_{i=1}^K |u_{i,0} - \tilde{u}_{\pi(i), \theta}|_{C^2(I_i, \mathbb{R}^2)} \, .  
\end{equation}

Fix $Q_0 \in \cQ(\tau_*, \cK_*, E_*)$ and choose $\kappa \le \frac{1}{2} \min \{ \tau_*, \cK_* \}$.  
Let $\cQ(Q_0, E_*; \kappa)$ denote the set of 
billiard tables $Q$ with\footnote{Indeed, the distance $\bd$ allows configurations to move from finite to infinite horizon (see \cite[Section~6.2]{demzhang13}), but we will not need that here as we will restrict ourselves to finite horizon configurations.}
 $\bd(Q, Q_0) < \kappa$ and $E_{\max}(Q) \le E_*$, $\tau_{\max} \le 2/\tau_*$.
 Let $\cF(Q_0, E_*; \kappa)$
denote the corresponding set of billiard maps. The following result
is \cite[Theorem~2.8 and Section~6.2]{demzhang13}.

\begin{prop}
\label{prop:close maps}
Let $Q_0 \in \cQ(\tau_*, \cK_*, E_*)$.  For all $\kappa \le \frac 12 \min \{ \tau_*, \cK_* \}$,
we have $\cQ(Q_0, E_*; \kappa) \subset \cQ(\frac{\tau_*}{2}, \frac{\cK_*}{2}, E_*)$.  Moreover, there
exists $C>0$ such that for any $T_1, T_2 \in \cF(Q_0, E_*; \kappa)$,
\begin{itemize}
  \item[a)] $d_H(\cS_{-1}^{T_1}, \cS_{-1}^{T_1}) \le C \kappa^{1/2}$, where $d_H$ is the Hausdorff metric and
  $\cS_{-1}^T$ is the singularity set for $T^{-1}$ defined in {\bf (H1)};
  \item[b)] for $x \notin N_{C\kappa^{1/2}}(S_{-1}^{T_1}, \cS_{-1}^{T_2})$, $d(T_1^{-1}(x), T_2^{-1}(y)) \le C \kappa^{1/2}$,
  where $N_\ve(\cdot)$ denotes the $\ve$-neighborhood of a set in $M$ in the Euclidean metric. 
\end{itemize}
\end{prop}

We use an (uncountable) index set $\cI(\tau_*, \cK_*, E_*)$, identifying $\iota \in \cI(\tau_*, \cK_*, E_*)$ with a map $T_\iota \in \cF(\tau_*, \cK_*, E_*)$.
Choosing a sequence $(\iota_j)_{j \in \mathbb{N}} \subset \cI(\tau_*, \cK_*, E_*)$,
we will be interested in the dynamics of
\begin{equation}
\label{eq:comp}
T_n := T_{\iota_n} \circ \cdots \circ T_{\iota_2} \circ T_{\iota_1} \, , \quad n \in \mathbb{N} \, .
\end{equation}
If we choose $\iota_j = \iota$ for each $j$, then $T_n = T_\iota^n$, the iterates of a single map.
For convenience, denote $T_0 = $ Id.


\subsection{Main analytical results:  Cone contraction and loss of memory}
\label{sec:main}
 
As announced in the introduction, the main analytical tool developed in this paper is the construction of a convex
cone of functions $\cC_{c,A,L}(\delta)$, depending on parameters $\delta >0$, $c, A, L > 1$, as defined in
Section~\ref{sec:cone_dist}, that is contracted under the sequential action of the 
transfer operators $\cL f = f \circ T^{-1} $, defined in Section~\ref{sec:transfer} for $T \in \cF(\tau_*, \cK_*, E_*)$.   
For a sequence of maps $T_n$ as in \eqref{eq:comp}, define $\cL_n f = f \circ T_n^{-1}$.

In order to state our main result on cone contraction, 
we define open neighborhoods in $\cF(\tau_*, \cK_*, E_*)$ using the distance 
$\bd$ between tables defined in \eqref{eq:d def}.
Let $T \in \cF(\tau_*, \cK_*, E_*)$, and for 
$0 < \kappa < \frac 12 \min \{ \tau_*, \cK_* \}$, define
\begin{equation}
\label{eq:close d 1}
\cF(T, \kappa) = \{ \tT \in \cF(\tau_*, \cK_*, E_*) : \bd(Q(\tT), Q(T) ) < \kappa \} \, .
\end{equation}
Remark that since $\cF(T, \kappa) \subset \cF(Q(T), E_*; \kappa)$, the conclusions of Proposition~\ref{prop:close maps} apply as well to maps in $\cF(T, \kappa)$.
We will denote the index set corresponding to $\cF(T, \kappa)$ by $\cI(T, \kappa) \subset \cI(\tau_*, \cK_*, E_*)$.
Thus $\iota \in \cI(T, \kappa)$ if and only if $T_\iota \in \cF(T, \kappa)$.

\begin{theorem}
\label{thm:main}
Suppose $c, A$ and $L$ satisfy the conditions of Section~\ref{sec:conditions}, and that $\delta>0$ satisfies \eqref{eq:delta_0 ineq} and \eqref{eq:A-cond-s4}.
Let $N_{\cF} := N(\delta)^- + k_*n_*$ be from Theorem~\ref{thm:cone contract} and let $\kappa>0$ be from
Lemma~\ref{lem:proper cross}(b).
\begin{enumerate}[a)]
\item There exists $\chi \in(0,1)$ such that if $n \ge N_{\cF}$, $T \in \cF(\tau_*, \cK_*, E_*)$ 
and $\{ \iota_j \}_{j=1}^n \subset \cI(T, \kappa)$, 
then $\Lp_n \cC_{ c, A, L}(\delta) \subset \cC_{  \chi c, \chi A, \chi L}(\delta) $.
\item For any $\chi \in \left( \max \{ \frac 12, \frac 1L, \frac{1}{\sqrt{A-1}} \}, 1 \right)$, the cone 
$\cC_{\chi c, \chi A, \chi L}(\delta)$ has diameter at most 
\[
\log \left( \frac{(1+\chi)^2}{(1-\chi)^2} \chi L \right) < \infty
\]
in the Hilbert metric associated to $\cC_{c, A, L}(\delta)$ (see \eqref{eq:H def} for a precise definition), provided $\delta>0$ is chosen sufficiently small to satisfy \eqref{eq:delta cond}.
\end{enumerate}
\end{theorem}

The first statement of this theorem is proved in two steps:  first, Proposition~\ref{prop:almost} shows that the parameters
$c$ and $A$ contract due to the uniform hyperbolicity properties {\bf (H1)}-{\bf (H5)} of the maps in 
$\cF(\tau_*, \cK_*, E_*)$, subject to the constraints listed in
Section~\ref{sec:conditions} (all that is needed is $\{ \iota_j \}_{j=1}^n \subset \cI(\tau_*, \cK_*, E_*)$, and not
the stronger assumption $\{ \iota_j \}_{j=1}^n \subset \cI(T, \kappa)$);  second, Theorem~\ref{thm:cone contract} proves the contraction of $L$ using
the uniform mixing property of maps $\tT \in \cF(T, \kappa)$ as expressed by Lemma~\ref{lem:proper cross}. The second statement of Theorem~\ref{thm:main} is proved by
Proposition~\ref{prop:diameter}.

From this theorem follow our results on exponential loss of memory for sequential systems of billiard maps.
In the case that $T_{\iota_j} = T$ for each $j$, these results read as exponential decay of correlations and convergence
to equilibrium.  Since our maps $T \in \cF(\tau_*, \cK_*, E_*)$ all preserve the measure $\musrb$,
we also obtain a type of convergence to equilibrium in the sequential case (see Theorem~\ref{thm:equi state}).

In order to state our result for the sequential system, we define the notion of an {\em admissible sequence} of maps from
$\cF(\tau_*, \cK_*, E_*)$. 
As before, let $\kappa>0$ be from
Lemma~\ref{lem:proper cross}(b).

\begin{defin}
\label{def:admissible}
For $N \in \mathbb{N}$, we 
call a sequence $( \iota_j )_{j \ge 1}$, $\iota_j \in \cI(\tau_*, \cK_*, E_*)$, 
{\em $N$-admissible}
if there exist sequences $(T_k)_{k \ge 1} \subset \cF(\tau_*, \cK_*, E_*)$ and $(N_k)_{k \ge 1}$ 
with $N_k \ge N$, such that $T_{\iota_j} \in \cF(T_k, \kappa)$
for all $k \ge 1$ and $j \in [1 + \sum_{i=1}^{k-1} N_i, \sum_{i=1}^k N_i]$.
\end{defin}
  Thus an $N$-admissible sequence is a sequence which remains
in a $\kappa$ neighborhood of a fixed map $T_k$ for $N_k \ge N$ iterates at a time, but which may undergo a large
change between such blocks.

\begin{remark}
One can generalize the definition of $N$-admissible sequence to include short blocks where maps are not required to be close to a fixed map. As long as these short blocks can be grouped to
contain at least $n_0$ iterates, where $n_0$ is from Proposition~\ref{prop:almost}, and they are interspersed regularly with
long blocks of length at least $N = N_{\cF}$ then one can still set up a regular contraction using 
Theorem~\ref{thm:cone contract} on the long blocks.
\end{remark}

We first state our results regarding loss of memory, both with respect to $\musrb$ and leafwise: the difference of integrals along individual stable curves converge
to 0 exponentially fast along any $N_{\cF}$-admissible sequence.  Let $\cW^s(\delta)$ denote the set of  homogeneous  cone
stable curves $\cW^s$ defined in Section~\ref{sec:uni hyp}, having length between $\delta$ and $2\delta$.
We denote by $\musrb(f) = \int_M f \, d\musrb$ and by $|W|$ the (Euclidean) length of a stable curve $W$ in $M$.

Also, we denote the average value of
$\psi$ on $W$ by $\fint_W \psi\, dm_W=\frac{1}{|W|}\int_W \psi \, dm_W$,
where $m_W$ denotes the arclength measure on $W$ induced by the Euclidean metric in $M$.

In Lemma~\ref{lem:dominate}, we prove that our cone $\cC_{c,A,L}(\delta)$ contains translations 
of piecewise
H\"older continuous functions, as long as the discontinuities are transverse to the stable cone defined in {\bf (H1)}.
We make this precise as follows.  

\begin{defin}
\label{def:part}
We call a countable (mod 0) partition $\pa = \{ P_i \}_i$ of $M$ {\em regular} if each $P_i$
is an open, simply connected set, and there exist constants $K, C_\pa >0$ such that
for all $W \in \cW^s$, $W \setminus \partial \pa$ comprises at most $K$ connected components and for
any $\ve > 0$, $m_W(N_\ve(\partial \pa)) \le K C_\pa \ve$, where
$N_\ve(A)$ denotes the $\ve$-neighborhood of a set $A$ in $M$.
\end{defin}

For $t>0$, denote by $C^t(\pa)$ the set of functions on $M$ that are H\"older continuous on
each element of $\pa$ and such that
\[
|f|_{C^t(\pa)} = \sup_{P \in \pa} |f|_{C^t(P)} < \infty \, .
\]

We recall again $N_{\cF} = N(\delta)^- + k_*n_*$ from Theorem~\ref{thm:cone contract}.

\begin{theorem}
\label{thm:memory state}
Let $\pa$ be a regular partition of $M$ and let\footnote{The parameter
$\gamma \in (0,1)$ is from the cone condition \eqref{eq:cone 5}.} $t \ge \gamma$. 
There exist $C>0$ and $\vartheta<1$ such that for all $N_{\cF}$-admissible
sequences $(\iota_j)_j$, all $n \ge 0$, and all
$f, g \in C^t(\pa)$ with $\musrb(f) = \musrb(g)$:
\begin{itemize}
  \item[a)]  For all $W \in \cW^s(\delta)$ and all $\psi \in C^1(W)$, we have
\[
\left| \fint_{W} \Lp_n f  \,  \psi  \, dm_{W} - \fint_{W} \Lp_n g \,  \psi \, dm_{W} \right|
\le C \vartheta^n  \, |\psi|_{C^1(W)} \max \{ \| f \|_{C^t(\pa)},  \| g \|_{C^t(\pa)} \} \, ;
\]
\item[b)] For all $\psi  \in C^1(M)$,
\[
\left| \int_M \cL_n f \, \psi \, d\musrb - \int_M \cL_n g \, \psi \, d\musrb \right| 
\le C \vartheta^n |\psi|_{C^1(M)} \max \{ \| f \|_{C^t(\pa)},  \| g \|_{C^t(\pa)} \} \, .
\]
\end{itemize}
\end{theorem}

We remark that the regularity of $\psi \in C^1(M)$ can be relaxed to $\psi \in C^\varsigma(M)$ for any
$\varsigma>0$ by a standard approximation argument, but at the expense of obtaining a weaker rate
$\vartheta$. 

Since all our maps preserve the same invariant measure $\musrb$, we obtain additionally an 
equidistribution result for stable curves as well as convergence to equilibrium 
along admissible sequences.

\begin{theorem}
\label{thm:equi state}
Under the hypotheses of Theorem~\ref{thm:memory state},
there exists $C>0$ such that for all $N_{\cF}$-admissible
sequences $(\iota_j)_j \subset \cI(\tau_*, \cK_*, E_*)$,  
all $f, g \in C^t(\pa)$ with $\musrb(f) = \musrb(g)$, and all $n \ge 0$,
\begin{itemize}
  \item[a)]  For all $W_1, W_2 \in \cW^s(\delta)$  and all $\psi_i \in C^1(W_i)$ with $\fint_{W_1} \psi_1 = \fint_{W_2} \psi_2$, we have
\[
\left| \fint_{W_1}  \Lp_n f  \,  \psi_1  \, dm_{W_1} - \fint_{W_2} \Lp_n g \,  \psi_2 \, dm_{W_2} \right|
\le C \vartheta^n  \, (|\psi_1|_{C^1(W_1)} + |\psi_2|_{C^1(W_2)} ) \max \{ \| f \|_{C^t(\pa)},  \| g \|_{C^t(\pa)} \} \, ;
\]
in particular, for all $W \in \cW^s(\delta)$ and $\psi\in C^1(W)$,
\[
\left| \fint_W \Lp_n f \, \psi \, dm_W - \musrb(f) \fint_W \psi \, dm_W \right| 
\le C \vartheta^n \,  |\psi|_{C^1(W)} \max \{ \| f \|_{C^t(\pa)},  \| g \|_{C^t(\pa)} \}  \, ;
\]
  \item[b)]  for all $\psi \in C^1(M)$,
\[
\left| \int_M f \, \psi \circ T_n \, d\musrb - \int_M f \, d\musrb \int_M \psi \, d\musrb \right|
\le C \vartheta^n |\psi|_{C^1(M)}\max \{ \| f \|_{C^t(\pa)},  \| g \|_{C^t(\pa)} \} \, .
\]
\end{itemize}
\end{theorem}

Theorems~\ref{thm:memory state} and \ref{thm:equi state} are proved in Section~\ref{sec:exp-mix}, 
specifically in Theorems~\ref{thm:memory} and \ref{thm:equi} and Corollary~\ref{cor:extend}. 

\begin{remark}
\label{rmk:stenlund}
Theorem~\ref{thm:memory state} has some overlap with \cite{young zhang}, which also considers
sequential billiards in which scatterers shift slightly between collisions.  Note, however, that our
definition of admissible sequence allows abrupt and large changes in the configuration of scatterers
within the family $\cF(\tau_*, \cK_*, E_*)$ every $N_{\cF}$ iterates, compared to the slowly changing
requirement throughout \cite{young zhang}.  This may seem like merely a technical difference due to the
cone technique, yet it is precisely this ability to introduce occasional large changes in the dynamics
that allows us to apply our results to the chaotic scattering problem and random Lorentz gas described
in Section~\ref{sec:appl}.
\end{remark}

With these convergence results in hand, we are able to provide three applications to concrete
problems of physical interest:  sequential open systems in Section~\ref{sequential},
a chaotic scattering problem without assuming a no-eclipse condition in Section~\ref{sec:scattering},
and a variant of the random Lorentz gas in Section~\ref{sec:lorentz}.


\section{Uniform Hyperbolicity, Singularities and Transfer Operators}
\label{sec:hyp}

\subsection{Uniform properties for $T \in \cF(\tau_*, \cK_*, E_*)$} 
\label{sec:uni hyp}

Fixing $K$ and $\{ \ell_i \}_{i=1}^K$, we recall some fundamental properties of billiard maps 
$T \in \cF(\tau_*, \cK_*, E_*)$ that
depend only on the quantities $\tau_*$, $\cK_*$ and $E_*$. 
Although many of these properties are well known, a proof of their dependence on $\tau_*$, $\cK_*$, $E_*$ can be found, for example, in \cite[Section~6.1]{demzhang13}.

In order to better align with the abstract framework in 
\cite{demzhang13}, we also label our properties {\bf (H1)}-{\bf (H5)}, although our set-up here is simpler than in
\cite{demzhang13}.  We recall the corresponding index set $\cI(\tau_*, \cK_*, E_*)$ from Section~\ref{sec:bill family}
and the notation $T_n$ from \eqref{eq:comp}.

\medskip
\noindent
{\bf (H1) Hyperbolicity and Singularities.}
The (constant) family of cones 
\[
C^s(x) = \{ (dr, d\vf) \in \mathbb{R}^2 : -\cK_*^{-1} - \tau_*^{-1} \le d\vf/dr \le - \cK_* \},
\quad \mbox{for $x \in M$,}
\]
 is strictly invariant, $DT^{-1}C^s(x) \subset C^s(T^{-1}x)$, for all $T \in \cF(\tau_*, \cK_*, E_*)$.
 Moreover, 
$T^{-1}$ enjoys uniform expansion of vectors in the stable cone:  set  $\Lambda = 1 + 2\cK_* \tau_* > 1$; then
there exists $C_1 \in (0,1]$ such that,
\begin{equation}
\label{eq:hyp}
\| DT_n^{-1}(x) v \| \ge C_1 \Lambda^n \| v \|, \qquad \mbox{for all $v \in C^s(x)$},
\end{equation}
where $\| \cdot \|$ denotes the Euclidean norm given by $dr^2 + d\vf^2$. 
There is a family of unstable cones $C^u$ defined similarly, but with
$\cK_* \le d\vf/dr \le \cK_*^{-1} + \tau_*^{-1}$, which is strictly invariant under $DT$ for all $T \in \cF(\tau_*, \cK_*, E_*)$.

Due to the unbounded expansion of $DT$ near tangential collisions,  we define the standard
homogeneity strips, following \cite{bsc}.  For some $k_0 \in \mathbb{N}$, to be chosen later in  \eqref{eq:one step}, 
define
\begin{equation}
\label{eq:Hk}
\bH_{\pm k} = \{ (r, \vf) \in M : (k+1)^{-2} \le |\pm \tfrac{\pi}{2} - \vf| \le k^{-2} \} ,
\mbox{ for all } k \ge k_0.
\end{equation}
Set $\cS_0 = \{ (r, \vf) \in M : \vf = \pm \frac{\pi}{2} \}$.  
For $n \ge 1$, the singularity set for $T_n$
is denoted by $\cS_n^{T_n} = \cup_{i=0}^{n} T_i^{-1} (\cS_0)$, while
the singularity set for $T_n^{-1}$ is denoted by $\cS_{-n}^{T_n} = \cup_{i=0}^n T_i ( \cS_0)$.
On $M \setminus \cS_n^{T_n}$, $T_n$ is a $C^2$ diffeomorphism onto its image.

There exists a constant, which we still call $C_1>0$, such that
\[
\frac{C_1}{\cos \vf(Tx)} \le \frac{\| DT(x) v \|}{\| v \|} \le \frac{1}{C_1 \cos \vf(Tx)} \, , \quad \mbox{for $x \notin \cS_0$.}
\]

In order to achieve bounded distortion, we will consider the boundaries of the homogeneity strips
as an extended singularity set for $T$.  To this end, define $\cS_0^{\bH} = \cS_0 \cup (\cup_{k \ge k_0}
(\partial \bH_k \cup \partial \bH_{-k}))$, and for $n \ge 1$,
\begin{equation}\label{eq:cSbH}
\cS_n^{\bH} =\cup_{i=0}^{n} T_i^{-1} (\cS_0^{\bH}), \quad \cS_{-n}^{\bH} = \cup_{i=0}^n T_i (\cS_0^{\bH}) \, .
\end{equation}

\medskip
\noindent
{\bf (H2) Families of Stable and Unstable Curves.}
We call a curve $W \subset M$ a stable curve if for each $x \in W$, the tangent vector to $W$ at $x$ belongs to
$C^s$.  A stable curve is called homogeneous if it lies in one homogeneity strip or outside their union.
Denote by $\cW^s$ the set of homogeneous stable curves with length
at most $\delta_0\in (0,1/2)$ (defined by \eqref{eq:one step}) and with curvature at most $\bar B$.  

By \cite[Proposition~4.29]{chernov book}, we may choose $\bar B$ sufficiently large that
$T^{-1} \cW^s \subset \cW^s$, up to subdividing the curves of length larger than $\delta_0$,
for all $T \in \cF(\tau_*, \cK_*, E_*)$.

Similarly, we define an analogous set of homogeneous unstable curves by $\cW^u$.

\medskip
\noindent
{\bf (H3) One-Step Expansion.}  Defining the adapted norm $\| v \|_*$, $v = (dr, d\vf)$ as in
\cite[Sect.~5.10]{chernov book}, we have
$\| DT^{-1}(x) v \|_* \ge \Lambda \| v \|_*$ for all $v \in C^s(x)$, wherever $DT^{-1}$ is defined.
For $W \in \cW^s$, let $V_i$ denote the maximal homogeneous components of $T^{-1}W$.
Then by \cite[Lemma~5.56]{chernov book}, there exists $\theta_0 \in (\Lambda, 1)$,
a choice of $k_0$ for the homogeneity strips
and $\delta_0\in (0,1/2)$ such that,
\begin{equation}
\label{eq:one step}
\sup_{T \in \cF(\tau_*, \cK_*, E_*)} \sup_{W \in \cW^s }
\sum_i |J_{V_i}T|_* \le \theta_0 \, ,
\end{equation}
where $| J_{V_i}T |_*$ denotes the supremum of the Jacobian of $T$ along $V_i$ in the adapted metric.

Since the stable/unstable cones are global and bounded away from one another, the adapted metric can be
extended so that it is uniformly equivalent to the Euclidean metric:  There exists
$C_0 \ge 1$ such that $C_0^{-1} \| v \| \le \| v \|_* \le C_0 \| v \|$ for all $v \in \mathbb{R}^2$.

\medskip
\noindent
{\bf (H4) Distortion Bounds.}
Suppose $W \in \cW^s$ and for $n \ge 1$,
$\{ \iota_j \}_{j=1}^n \subset \cI(\tau_*, \cK_*, E_*)$ are such that $T_jW \in \cW^s$ for $j=0, \ldots n$.  
There exists $C_d>0$,
independent of $W$, $n$ and $\{ \iota_j \}_{j=1}^n$,
such that for all $x, y \in W$,
\begin{equation}
\label{eq:distortion}
| \log J_WT_n(x) - \log J_WT_n(y)| \le C_d d(x,y)^{1/3},
\end{equation}
where $J_WT_n$ is the (stable) Jacobian of $T_n$ along $W$ and $d(\cdot, \cdot)$ denotes
arclength on $W$ with respect to the metric $dr^2+d\vf^2$.

Similar bounds hold for stable Jacobians lying on the same unstable curve.
Suppose, for $n \ge 1$, that $V_1$, $V_2 \in \cW^s$ are such that 
$T_jV_1$, $T_jV_2 \in \cW^s$ for 
$0 \le j \le n$,  in  particular they are not cut by any singularity, and there exists a foliation of unstable curves $\{ \ell_x \}_{x \in V_1} \subset \cW^u$ creating a one-to-one correspondence between $V_1$ and $V_2$ and such that $\{ T_n( \ell_x) \}_{x \in V_1} \subset \cW^u$ creates a one-to-one correspondence between $T_nV_1$ and $T_nV_2$.
For $x \in V_1$, define $\bx = \ell_x \cap V_2$.  Then there exists $C_d>0$, independent
of $n$, $\{ \iota_j \}_{j=1}^n$, $V_1$, $V_2$, and $x$, such that,
\begin{equation}
\label{eq:u dist}
|\log J_{V_1}T_n(x) - \log J_{V_2}T_n(\bx)| \le C_d(d(x, \bx)^{1/3} + \phi(x, \bx)), 
\end{equation}
where $\phi(x, \bx)$ denotes the angle between the tangent vectors to $V_1$ and $V_2$
at $x$ and $\bx$, respectively.  For simplicity, we use the same symbol $C_d$ to represent the
distortion constants in \eqref{eq:distortion} and \eqref{eq:u dist}.
The proofs for these distortion bounds in this form for a single map can be found in \cite[Appendix A]{demzhang11}
(see also \cite[Section 5.8]{chernov book}).  The analogous bounds for sequences of maps in
$\cF(\tau_*, \cK_*, E_*)$
are proved in \cite[Lemma~3.3]{demzhang13}.  The constant $C_d$ depends only on the choice of
$k_0$ from {\bf (H3)} and the hyperbolicity constants $C_1$ and $\Lambda$ from {\bf (H1)}. 

\medskip
\noindent
{\bf (H5) Invariant measure.}  
All $T \in \cF(\tau_*, \cK_*, E_*)$ preserve the same invariant measure,
$d\musrb = c \cos \vf \, dr \, d\vf$, where
$c = \frac{1}{2|\partial Q|} = \frac{1}{2 \sum_{i=1}^K \ell_i}$ is the normalizing constant \cite{chernov book}.

\begin{remark}
Property {\bf (H5)} is enjoyed by the class of maps we have chosen, but it is not necessary for this technique to work.
Indeed, \cite{demzhang13} replaces this condition by:  There exists $\eta>0$ so that $1+\eta$ is sufficiently small compared to the hyperbolicity constant $\Lambda$ from {\bf (H1)}, such that $(J_{\musrb}T)^{-1} \le 1+\eta$, where
$J_{\musrb}T$ is the Jacobian of $\musrb$ with respect to $T$.  

Thus $T$ does not have to preserve $\musrb$, but in this sense must be close to a map that does.  This permits the
application of the current technique to billiards under small external forces and nonelastic reflections, as described
in \cite[Section 2.4]{demzhang13}.  See also  \cite{Ch08, zhang}.   Note however, that while 
Theorem~\ref{thm:memory state} will continue to hold in this generalized context, 
Theorem~\ref{thm:equi state} will not hold
once there is no common invariant measure.   
\end{remark}


\subsection{Growth Lemma}

Although all maps in $\cF(\tau_*, \cK_*, E_*)$ enjoy the uniform properties {\bf (H1)}-{\bf (H5)}, 
in Section~\ref{sec:scale}, we will find it convenient to increase the 
contraction provided in \eqref{eq:one step} by
replacing $T$ with a higher iterate $T_n$ and choosing $\delta_0$
sufficiently small so that \eqref{eq:one step} holds for $T _*:= T_n$ with constant $\theta_0^n$.  
This is possible since if $W$ is a stable curve, then there exists $C>0$, depending only on the
family $\cF(\tau_*, \cK_*, E_*)$, such that, for each $T\in \cF(\tau_*, \cK_*, E_*)$,
$|T^{-1}W| \le C|W|^{1/2}$ \cite[Exercise~4.50]{chernov book}.  
Thus we may redefine $\delta_0$ so small
that no connected component of $T_k^{-1}(W)$ is longer than $\delta_0$, from hypothesis {\bf (H1)}, for $k=0, \ldots, n$.
Since no artificial subdivisions are necessary, we apply \eqref{eq:one step} inductively in $k$
to obtain the desired contraction.

Choose $\bar n$ such that $\theta_1 := \theta_0^{\bar n}$  satisfies 
\begin{equation}
\label{eq:theta1 def}
3C_0 \frac{\theta_1}{1 - \theta_1} \leq \frac 14,
\end{equation}
where  $C_0 \ge 1$ is from {\bf (H3)}, and then fix $\delta_0$, as explained above, such that
\begin{equation}
\label{eq:theta_1}
\sup_{\substack{W \in \cW^s \\ |W| \le \delta_0}} \sum_{V_i} |J_{V_i}T_{\bar n}|_* \leq \theta_1,
\end{equation}
where $V_i$ are the homogeneous components of $T_{\bar n}^{-1}W$.  
Note that if we shrink $\delta_0$ further, then \eqref{eq:theta_1} will continue to hold for the same
value of $\bar n$.

We shall work with the map $T_*:= T_{\bar n}$ throughout the following. To simplify notation we will call $T_*$ again $T$ as no confusion can arise.
Note, however, that the definition of $N$-admissible sequence must be modified since the length
$N_{\cF}$ of the blocks comprising the sequence, for example in Theorem~\ref{thm:memory state},
 is computed for the map $T_*$.  Thus 
a single block in an $N$-admissible sequence should comprise
at least $\bar n N$ billiard maps that are close in the sense of Definition~\ref{def:admissible}.

\begin{defin}\label{def:Gn}
For $W \in \cW^s$,  for $T_n$ as in \eqref{eq:comp}  we denote by $\cG_n(W)$ the homogeneous components of $T_n^{-1}W$,
where we have subdivided the elements of $T_n^{-1}W$ longer than $\delta_0$ into elements with length between
$\delta_0/2$ and $\delta_0$ so that $\cG_n(W)\subset \cW^s$.  We call $\cG_n(W)$ the $n$th generation
of $W$. 

Let $\cI_n(W)$ denote the set of curves $W_i \in \cG_n(W)$ such that $T_j(W_i)$ is not contained
in an element of $\cG_{n-j}(W)$ having length at least $\delta_0/2$ for all $j = 0, \ldots, n$.
\end{defin}

The following growth lemma is contained in
\cite[Lemma~5.5]{demzhang13}, but we include the proof of item (b) 
here for convenience and to draw out the
explicit dependence on the constants.

\begin{lemma}
\label{lem:full growth}
There exists $\bar{C}_0 > 0$ such that for all $W \in \cW^s$ and $n \ge 0$ and $\{ \iota_j \}_{j=1}^n$, 
\begin{itemize}
  \item[a)] $\displaystyle
  \sum_{W_i \in \cI_n(W)} |J_{W_i}T_n|_{C^0(W_i)} \le C_0 \theta_1^n$;
  \item[b)] $\displaystyle
\sum_{W_i \in \cG_n(W)} |J_{W_i}T_n|_{C^0(W_i)} \leq \bar{C}_0 \delta_0^{-1} |W| + C_0 \theta_1^n$.
\end{itemize}
\end{lemma}

\begin{proof}
Item (a) follows by induction on $n$ from \eqref{eq:theta_1} and the constant
$C_0$ from {\bf (H3)} comes from translating from 
the adapted metric to the Euclidean metric at the last step.  We focus on proving item (b).

For $W \in \cW^s$, let $L_k(W) \subset \cG_k(W)$ denote those elements of $\cG_k(W)$ having
length at least $\delta_0/2$. For $k \le n$ and $W_i \in \cG_n(W)$, 
we say that $V_j \in L_k(W)$ is the most recent long ancestor of $W_i$ if $k \le n$ is the largest time 
that $T_{n-k}W_i$ is contained in an element of $L_k(W)$.  Then by definition, $W_i \in \cI_{n-k}(V_j)$.
Note that if $W_i \in L_n(W)$, then $k=n$ and $W_i = V_j$.  Now we estimate,
\[
\begin{split}
\sum_{W_i \in \cG_n(W)} |J_{W_i}T_n|_{C^0(W_i)} 
& \le \sum_{k=1}^n \sum_{V_j \in L_k(W)} \sum_{W_i \in \cI_{n-k}(V_j)} |J_{W_i}T_{n-k}|_{C^0(W_i)}
|J_{V_j}T_k|_{C^0(V_j)} \\
& \qquad + \sum_{W_i \in \cI_n(W)} |J_{W_i}T_n|_{C^0(W_i)} \\
& \le \sum_{k=1}^n \sum_{V_j \in L_k(W)} C_0 \theta_1^{n-k} e^{C_d \delta_0^{1/3}} \frac{ |T_kV_j|}{|V_j|}
+ C_0\theta_1^n \, ,
\end{split}
\]
where we have used item (a) of the lemma to sum over $W_i \in \cI_{n-k}(W)$ and \eqref{eq:distortion}
to replace $|J_{V_j}T_k|_{C^0(V_j)}$ with $\frac{|T_kV_j|}{|V_j|}$.  Now since 
$\cup_{V_j \in L_k(W)} T_kV_j \subset W$, and $|V_j| \ge \delta_0/2$, we have
\[
\sum_{W_i \in \cG_n(W)} |J_{W_i}T_n|_{C^0(W_i)} 
\le \sum_{k=1}^n C_0 \theta_1^{n-k} 2 \delta_0^{-1} |W| e^{C_d \delta_0^{1/3}} + C_0 \theta_1^n \, ,
\]
which proves the lemma with $\bar C_0 := \frac{ 2  C_0}{1-\theta_1} e^{C_d \delta_0^{1/3}}$.
\end{proof}

\begin{remark}
\label{rem:improve}
It is not necessary to work with $T = T_{\bar n}$ in Lemma~\ref{lem:full growth}.  
It follows equally well from 
\eqref{eq:one step} with $\theta_1$ replaced by $\theta_0$.  
However, the stronger contraction provided by \eqref{eq:theta_1} is useful for
Lemma~\ref{lem:0 scale} and the arguments following it. 

Observe that if $|W| \ge \delta_0/2$, then all pieces $W_i \in \cG_n(W)$ have a long ancestor and can be included in the sum over $k$; in this case, the second term on the right side of item (b) is not needed, and the value of $\bar C_0$ remains unchanged.
\end{remark}


\subsection{Transfer operator}
\label{sec:transfer}

We define the transfer operator $\Lp$ associated with $T$ acting on scales of
spaces of distributions as in \cite{demzhang11}.  For $\{ \iota_j \}_{j=1}^n \subset \cI(\tau_*, \cK_*, E_*)$, 
we denote by $T_n^{-1}\cW^s$ the set of
curves  $W \in \cW^s$ such that $T_jW \in \cW^s$ for all $j = 0, \ldots n$. 
For $\alpha \le 1/3$, let $C^\alpha(T_n^{-1}\cW^s)$ denote
the set of complex valued functions on $M$ that are H\"older continuous on elements of
$T_n^{-1}\cW^s$.  Then for $\psi \in C^\alpha(\cW^s)$, we have 
$\psi \circ T_n \in C^\alpha(T_n^{-1}\cW^s)$ (see Lemma~\ref{lem:test contract}(a)).
Define
\[
 \Lp_n \mu(\psi) = \mu(\psi \circ T_n) , \mbox{ for $\mu \in (C^\alpha(T_n^{-1}\cW^s ))^*$ .} 
\]
This defines $\Lp_{T_{\iota_n}} : (C^\alpha(T_n^{-1}\cW^s))^* \to (C^\alpha(T_{n-1}^{-1}\cW^s))^*$ for any
$n \ge 1$.  See \cite{demzhang11} for details.

Recall that by {\bf (H5)}, all our maps $T$ preserve the smooth invariant measure $d\musrb = c \cos \vf dr d\vf$,
where $c$ is the normalizing constant.  
When $d\mu = f d\musrb$ is a measure absolutely continuous with respect to $\musrb$,
we identify $\mu$ with its density $f$.  With this identification, the transfer operator acting on densities has
the following familiar expression,
\[
\Lp_T f = f \circ T^{-1} ,
\]
and so $\Lp_n f = \Lp_{T_{\iota_n}} \cdots \Lp_{T_{\iota_1}} f$, pointwise.
We choose this identification of functions in order to simplify our later work:  using the
reference measure $\musrb$,
the Jacobian of the transformation is 1, making $\Lp$ simpler to work with.

\section{Cones and Projective Metrics}\label{sec:conedef}

Given a closed,\footnote{\label{foo:close} Closed here means that for all $f,g \in\cC$ and sequence $\{\alpha_n\}\subset \R$ such that $\lim_{n\to\infty}\alpha_n=\alpha$ and $g+\alpha_n f\in\cC$ for all $n\in\bN$ we have $g+\alpha f\in\cC\cup\{0\}$. } convex cone $\cC$ satisfying $\cC \cap - \cC = \emptyset$, we define an
order relation by $f \preceq g$ if and only if $g - f \in \cC \cup \{0\}$.  We can then define
a projective metric by
\begin{equation}
\label{eq:H def}
\begin{split}
\bar\alpha(f,g) & = \sup \{ \lambda \in \mathbb{R}^+ : \lambda f \preceq g \} \\
\bar\beta(f,g) & = \inf \{ \mu \in \mathbb{R}^+ : g \preceq \mu f \} \\
\rho(f,g) & = \log \left( \frac{\bar\beta(f,g)}{\bar\alpha(f,g)} \right) .
\end{split}
\end{equation}

\subsection{ A cone of test functions}
For $W \in \cW^s$, $\alpha \in (0, 1]$ and $a \geq 1$, define a cone of test functions by
\[
\cD_{a, \alpha}(W) = \left\{ \psi \in C^0(W) : \psi > 0 , \frac{\psi(x)}{\psi(y)} \le e^{a d(x,y)^\alpha} \ \ \forall x,y\in W\right\},
\]
where $d(\cdot, \cdot)$ is the arclength distance along $W$.

The Hilbert metric associated with this cone and defined by \eqref{eq:H def} 
depends on the constant $a$ and the exponent $\alpha$ determining the regularity of the functions.
For each such choice, the Hilbert metric has the following
convenient representation.

\begin{lemma}[{\cite[Lemma 2.2]{liv95}}]
\label{lem:H metric}
Choose $\alpha \in (0,1]$.  For $\psi_1, \psi_2 \in \cD_{a,\alpha}(W)$, the corresponding metric
$\rho_{W, a, \alpha}(\cdot, \cdot)$ is given by
\[
\rho_{W, a, \alpha}(\psi_1, \psi_2) = \log \left[ \sup_{x,y,u,v \in W} 
\frac{e^{ad(x,y)^\alpha} \psi_1(x) - \psi_1(y)}{e^{a d(x,y)^\alpha} \psi_2(x) - \psi_2(y)}
\cdot \frac{e^{ad(u,v)^\alpha} \psi_2(u) - \psi_2(v)}{e^{a d(u,v)^\alpha} \psi_1(u) - \psi_1(v)} \right] .
\]
\end{lemma}

A corollary of this lemma is that $\cD_{a,\alpha}(W)$ has finite diameter in $\cD_{a,\beta}(W)$
if $\beta < \alpha$ and $|W| < 1$ (the proof is similar to {\cite[Lemma 2.3]{liv95}} noting that $d(x,y)^\alpha\leq |W|^{\alpha-\beta}d(x,y)^\beta$).  

The next two lemmas are simple consequences of the regularity of functions in $\cD_{a, \alpha}(W)$
for $W \in \cW^s$.  We denote by $m_W$ the measure induced by arclength along $W$. 

\begin{lemma}
\label{lem:avg}
For any $\alpha \in (0,1]$ and $W \in \cW^s$ with $|W| \in [\delta, 2 \delta]$, any 
$\psi \in \cD_{a, \alpha}(W)$ and
$x \in W$, we have
\[
\frac{\delta \psi(x)}{\int_W \psi \, dm_W} \le \frac{|W| \psi(x)}{\int_W \psi dm_W} \le e^{a|W|^\alpha} .
\]
\end{lemma}
\begin{proof}
The estimate is immediate since $\inf_{y \in W} \psi(y)\geq \psi(x)e^{-a|W|^\alpha}$.
\end{proof}

\begin{lemma}
\label{lem:pos}
Given $\alpha \in (0,1]$, $W \in \cW^s$, $\psi_1, \psi_2 \in \cD_{a, \alpha(W)}$ and $x, y \in W$,
\[
e^{- \rho_{W, a, \alpha}(\psi_1, \psi_2)} \le
\frac{\psi_1(x)\psi_2(y)}{\psi_2(x)\psi_1(y)} \le e^{\rho_{W, a, \alpha}(\psi_1, \psi_2)}
\]
\end{lemma}
\begin{proof}
According to \eqref{eq:H def},  we must have,
\[
\psi_2(x) - \bar\alpha \psi_1(x) \ge 0 \quad \forall x \in W \qquad \mbox{and} \qquad
\psi_2(y) - \bar\beta \psi_1(y) \le 0 \quad \forall y \in W .
\]
This in turn implies that
\[
\rho_{W, a, \alpha}(\psi_1, \psi_2) = \log \frac{\bar\beta(\psi_1, \psi_2)}{\bar\alpha(\psi_1, \psi_2)} \
\ge \log \left[ \frac{\psi_1(x)\psi_2(y)}{\psi_2(x)\psi_1(y)} \right]
\qquad \forall x, y \in W .
\]
\end{proof}


\subsection{Distances between curves and functions}
\label{sec:distances}

Due to the global stable cones $C^s$ defined in {\bf (H1)}, we may consider stable curves $W \in \cW^s$
as graphs of $C^2$ functions over an interval $I_W$ in the $r$-coordinate:
\[
W = \{ G_W(r) = (r, \vf_W(r)) :  r \in I_W \} .
\]
Using this representation, we define a notion of distance between $W^1, W^2 \in \cW^s$ by
\begin{equation}\label{eq:W-distance}
d_{\cW^s}(W^1, W^2) = |\vf_{W^1} - \vf_{W^2}|_{C^1(I_{W^1} \cap I_{W^2})} +
|I_{W^1} \bigtriangleup I_{W^2}|,
\end{equation}
if $W^1$ and $W^2$ lie in the same homogeneity strip and $|I_{W^1} \cap I_{W^2}| > 0$;
otherwise, we set $d_{\cW^s}(W^1, W^2) = \infty$.  Note that $d_{\cW^s}$ is not a metric,
but this is irrelevant for our purposes.

We will also find it necessary to compare between test functions on two different 
stable curves.  Given $W^1, W^2 \in \cW^s$ with $d_{\cW^s}(W^1, W^2) < \infty$,
and $\psi_i \in \cD_{a, \beta}(W^i)$, define 
\begin{equation}
\label{eq:psi-distance}
d_*(\psi_1, \psi_2) = | \psi_1 \circ G_{W^1} \|G_{W^1}'\|- \psi_2 \circ G_{W^2} \|G_{W^2}'\|\, |_{C^\beta(I_{W^1} \cap I_{W^2})},
\end{equation}
to be the (H\"older) distance between $\psi_1$ and $\psi_2$,
where $\| G_W' \| = \sqrt{1+ (d\vf_W/dr)^2}$. Note that $d_*$ depends on $\beta$.

Also, by the bound $\bar B$ on the curvature of elements of $\cW^s$, there exists
$B_*>0$ such that 
\begin{equation}
\label{eq:2deriv}
B_* = \sup_{W \in \cW^s} |\vf_W''|_{C^0(W)} < \infty \, . 
\end{equation}

\begin{remark}\label{rem:change-int} 
Note that if $d_*(\psi_1,\psi_2)=0$, then
\[
\int_{\overline{W}^1}\psi_1  \, dm_{W^1} =\int_{\overline{W}^2}\psi_2 \, dm_{W^2} \, ,
\]
where $\overline{W}^k = G_{W^k}(I_{W^1} \cap I_{W^2})$, $k=1,2$.
\end{remark}


\subsection{Definition of the cone}\label{sec:cone_dist}

In order to define a cone of functions adapted to our dynamics, we will fix the following
exponents, $\alpha, \beta, \gamma, q > 0$ and constant $a>1$ large enough.  Choose $q \in (0, 1/2)$, $\beta < \alpha \le 1/3$ 
and finally $\gamma \le \min  \{ \alpha - \beta, q  \}$.

For a length scale $\delta \le \delta_0/3$, define 
\[
\cW^s_-(\delta) = \{ W \in \cW^s : |W| \le 2 \delta \} \quad \mbox{and} \quad
\cW^s(\delta) = \{ W \in \cW^s : |W| \in [\delta, 2\delta] \} \, .
\]
 
Let $\cA_*$ denote the set of functions on $M$ whose restriction to each $W \in \cW^s$
is integrable with respect to the arclength measure $dm_W$.  For $f \in \cA_*$ define,
\[
\tri f \tri_+^{\sim} = \sup_{\stackrel{\scriptstyle W \in \cW^s_-(\delta)}{\psi \in \cD_{a,\beta}(W)}} \frac{\left|\int_W f \psi \, dm_W\right|}{\int_W \psi \, dm_W} .
\]
Setting $\cA_0=\{f\in\cA_*\;:\; \tri f \tri_+^{\sim} < \infty\}$, we have that $\tri \cdot \tri_+^{\sim}$ is a seminorm on the vector space $\cA_0$.
\begin{defin}\label{def:cA}
As usual, we consider the vector space of the classes of equivalence determined by the semi-norm ($f\sim g$ iff $\tri f-g \tri_+^{\sim}=0$) and call $\cA$ the resulting normed vector space. Remark that if $f\sim g$, then $f$ and $g$ are equal almost everywhere both with respect to Lebesgue and to SRB measure $\musrb$.
In the following, we can then safely ignore the issue of equivalence classes and we will not mention it explicitly.
\end{defin}

We will find it convenient to measure the average of functions in our cone on long stable curves, i.e.
elements of $\cW^s(\delta)$.  To this end, define for $f \in \cA$,
\begin{equation}
\label{eq:tri def}
\tri f \tri_+ = \sup_{\stackrel{\scriptstyle W \in \cW^s(\delta)}{\psi \in \cD_{a,\beta}(W)}} \frac{\left|\int_W f \psi \, dm_W\right|}{\int_W \psi \, dm_W} ,
\qquad \qquad
\tri f \tri_- = \inf_{\stackrel{\scriptstyle W \in \cW^s(\delta)}{\psi \in \cD_{a,\beta}(W)}} \frac{\int_W f \psi \, dm_W}{\int_W \psi \, dm_W} .
\end{equation}

Recall that we denote the average value of 
$\psi$ on $W$ by $\fint_W \psi\, dm_W=\frac{1}{|W|}\int_W \psi \, dm_W$.  
Since all of our integrals in this section and the next
will be taken with respect to the arclength $dm_W$, to keep our notation concise, we will drop the measure from our 
integral notation in what follows.

Now for $a, c, A, L >1$, and  $\delta \in (0, \delta_0/3]$, define the cone
\begin{align}
\cC_{c,A, L}(\delta)  =  \Big\{&  f \in \cA\setminus \{0\} : 
\qquad \tri f \tri_+\leq L\tri f \tri_- ;
\label{eq:cone 2} \\
& \sup_{W \in \cW^s_-(\delta)} \sup_{\psi \in \cD_{a, \beta}(W)} |W|^{-q}\frac{|\int_W f \psi|}{\fint_W\psi}  \le  A \delta^{1-q} \tri f \tri_- ;
\label{eq:cone 3} \\
&\forall\, W^1, W^2 \in \cW^s_-(\delta):  d_{\cW^s}(W^1, W^2) \le \delta,  \forall \psi_i \in \cD_{a, \alpha}(W^i): d_*(\psi_1, \psi_2)=0, \nonumber\\
&\left|\frac{\int_{W^1} f \psi_1}{\fint_{W^1}\psi_1}  - \frac{\int_{W^2} f \psi_2}{\fint_{W^2}\psi_2} \right|\leq
d_{\cW^s}(W^1, W^2)^\gamma \, \delta^{1-\gamma}   c A \tri f \tri_-  \Big\} .
\label{eq:cone 5}
\end{align}
We write the constants $c, A, L$ explicitly as subscripts in our notation for the
cone since these will be the parameters which are contracted by the dynamics.  
By contrast, the exponents $\alpha, \beta, \gamma, q$ are fixed and will not be altered by the dynamics, while the constant $a$, which will be chosen in Lemma  \ref{lem:test contract}, will not appear directly in the contraction constant of the cone.

Intuitively,  \eqref{eq:cone 2} requires that the `mass' of functions in the cone be evenly distributed throughout the phase space, while \eqref{eq:cone 3} implies that, even though the functions in the cone are not necessarily bounded, their average on a short stable curve $W$ cannot be larger than some constant times $|W|^{q-1}$.  Condition \eqref{eq:cone 5} says that, once you integrate along stable curves, you get an object which is morally $\gamma$-H\"older on the space of curves with the `metric' $d_{\cW^s}$. That is, \eqref{eq:cone 5} implies some form of weak H\"older regularity for $f$
transverse to the stable cone. 

\begin{remark}
The above cone is a considerable simplification of the one introduced in {\cite[Section 4.1]{liv95}}. The parameter $\zeta$ in {\cite[Section 4.1]{liv95}} plays the role of the parameter $q$ here: 
it allows one to control the integral of an element of the cone on short stable curves.
By contrast, the introduction of the new H\"older exponents $\alpha,\beta,\gamma$ is necessary, as already evident in {\cite{demzhang11} and \cite{demzhang13}}, to allow for the wilder singularities present in billiard maps with respect to the ones treated in {\cite[Section 4]{liv95}}. 
In particular, the requirement $\alpha \le 1/3$ is forced by the distortion bound {\bf (H4)},
which in turn depends on the choice of homogeneity strips.
The relation between the above cone and the norms in  {\cite{demzhang11} and \cite{demzhang13}} is very close: the cone has a natural norm associated to it (see \cite[Appendix D.2 and, in particular, equation (D.2.1)]{dkl21}) which is very similar to the norms in {\cite{demzhang11} and \cite{demzhang13}}.
\end{remark}

For convenience, we will require that $\delta_0$ is sufficiently small that
\begin{equation}
\label{eq:delta_0 condition}
e^{2 a \delta_0^\beta} \le 2 \, . 
\end{equation}
This will imply similar bounds in terms of $\delta$ since $\delta \le \delta_0/3$.

\begin{remark}
Note that, by definition, $\tri\cdot\tri_+$ decreases when $\delta$ decreases, while $\tri\cdot\tri_-$ increases, this if \eqref{eq:cone 2} hold for some $\delta$ is will hold automatically for smaller $\delta.$ We will see that cone invariance has the same property. 
In fact, as will become clear from our estimates in Sections~\ref{sec:cone} and \ref{sec:L contract}, in order to prove that the
parameters contract, we will need to choose $A$ large compared to $L$,  and $c$ large compared to $A$.
This yields the compatible set of restrictions, $1< L < A < c$.

By contrast, the exponents are fixed by the regularity properties of the maps in question:  $\alpha \le 1/3$ due to \eqref{eq:distortion},
and $\beta < \alpha$ so that $\cD_{a, \beta}(W)$ has finite diameter in $\cD_{a, \alpha}(W)$, while $\gamma \le \alpha - \beta$
is convenient to obtain the required contraction in Lemma~\ref{lem:compare}.  
See Section \ref{sec:conditions} for all the conditions the constants must satisfy for Proposition~\ref{prop:almost}.   Several further conditions are specified in Theorem~\ref{thm:cone contract} to prove the strict contraction of the cone.
\end{remark}

\begin{remark}\label{rem:A-L}
Note that, since $0 \notin \cC_{c,A,L}(\delta)$, condition \eqref{eq:cone 2} implies $\tri f \tri_- > 0$, 
hence for all $W \in \cW^s(\delta), \psi \in \cD_{a, \beta}(W)$,
\begin{equation}\label{eq:posi}
\int_W f \psi \, dm_W  \ge  \tri f\tri_-\int_W \psi \, dm_W  >  0.
\end{equation}
In particular, this implies
\[
\tri f \tri_+ = \sup_{\stackrel{\scriptstyle W \in \cW^s(\delta)}{\psi \in \cD_{a,\beta}(W)}} \frac{\int_W f \psi \, dm_W}{\int_W \psi \, dm_W} \; \mbox{ for $f \in \cC_{c,A,L}(\delta)$.}
\]
In addition, condition \eqref{eq:cone 3} implies
\[
A \tri f \tri_-\geq \sup_{W \in \cW_-^s (\delta)} \sup_{\psi \in \cD_{a, \beta}(W)}\delta^{q-1}  |W|^{1-q}\frac{|\int_W f \psi|}{\int_W\psi}  \geq \tri f \tri_+.
\]
However condition \eqref{eq:cone 2} is not vacuous since we assume $A>L$.
\end{remark}  

\begin{remark} To have an idea of which functions can belong to the cone note that a function that is strictly negative on a ball of size $2\delta$ cannot satisfy \eqref{eq:cone 2} and hence does not belong to the cone. On the other hand each $f\in C^1$ such that $\inf f\geq L\|f\|_\infty$ and $\|f\|_{C^1}\leq c \inf f$ belongs to the cone.
See also Lemma~\ref{lem:dominate} for a more detailed description of functions that belong to the cone.
\end{remark}

We will need the following lemma in Section~\ref{sec:mix}.

\begin{lemma}
\label{lem:b property}
For all $f \in \cC_{c,A, L}(\delta)$, $W \in \cW^s(\delta)$ and all $\psi_1, \psi_2 \in \cD_{a, \beta}(W)$,
\[
\left|\frac{\int_W f \psi_1}{\fint_W\psi_1}  -  \frac{\int_W f \psi_2}{\fint_W\psi_2}\right|\leq
2 \delta L  \rho_{W, a, \beta}(\psi_1, \psi_2)  \tri f \tri_- \, .
 \]
\end{lemma}
\begin{proof}
Let $f \in \cC_{c, A, L}(\delta)$, $W \in \cW^s(\delta)$ and $\psi_1, \psi_2 \in \cD_{a, \beta}(W)$.  
For each $\lambda, \mu>0$ such that $\lambda  \psi_1\preceq \psi_2\preceq \mu\psi_1$, hence also $\lambda  \psi_1\leq \psi_2\leq \mu\psi_1$, we have
\[
\frac{\int_W f \psi_2}{\fint_W\psi_2} = \frac{\lambda \int_W f \psi_1 + \int_W f (\psi_2 - \lambda \psi_1)}{\fint_W \psi_2} \geq \frac{\lambda\int_W f  \psi_1}{\mu\fint_W \psi_1}, 
\]
where we have dropped the second term above due to \eqref{eq:posi} since $\psi_2 - \lambda \psi_1 \in \cD_{a, \beta}(W)$.
Taking the sup on $\lambda$ and the inf on $\mu$, and recalling \eqref{eq:H def}, yields
\[
\frac{\int_W f \psi_1}{\fint_W\psi_1}  -  \frac{\int_W f \psi_2}{\fint_W\psi_2} \le \frac{\int_W f \psi_1}{\fint_W\psi_1} ( 1 - e^{-\rho_{W, a, \beta}(\psi_1, \psi_2)} )
\le \rho_{W, a, \beta} (\psi_1, \psi_2) \frac{\int_W f \psi_1}{\fint_W\psi_1} \, .
\]
Then, since $|W| \ge \delta$, we use \eqref{eq:cone 2} to estimate,
\[
\frac{\int_W f \psi_1}{\fint_W\psi_1} \le |W| \tri f \tri_+ \le  2\delta L \tri f \tri_- \, .
\]
Reversing the roles of $\psi_1$ and $\psi_2$ completes the
proof of the lemma.
\end{proof}

\section{Cone Estimates: Contraction of $c$ and $A$ }
\label{sec:cone}

In this section, fixing $\cF(\tau_*, \cK_*, E_*)$,  we will prove the following proposition.  Let $n_0 \ge 1$ be such that $A C_0 \theta_1^{n_0} \le 1/16$. 
\begin{prop}
\label{prop:almost}
If the conditions on $\delta, n_0, a, c, A,L$ specified in Section~\ref{sec:conditions} are satisfied, then there exists 
$\chi < 1$, independent of the cone parameters,\footnote{ By independent of the cone parameters, we mean that
we may first fix $\chi < 1$, and then choose $c,A,L, \delta$ satisfying the conditions of Section~\ref{sec:conditions} 
so that the contraction by $\chi$ is obtained for all choices of $c' > c$, $A'>A$, $L'>L$ and $\delta' < \delta$
that satisfy those conditions. Note, however, that larger $A'>A$ requires $n_0$ to increase in size.}
such that  for all $n \ge n_0$ and $\{ \iota_j \}_{j=1}^n \subset \cI(\tau_*, \cK_*, E_*)$,
\[
\cL_n  \cC_{c,A,L}(\delta) \subseteq \cC_{ \chi c, \chi A, 3 L}(\delta) \,.
\]
\end{prop}

Note that the parameter $L$ is not contracted, although it cannot grow too much. To have a contraction of $L$ we need to use the global properties of the map (some kind of topological mixing, see Section \ref{sec:L contract} for details), while the proof of Proposition \ref{prop:almost} is based only on local arguments.

Before proving Proposition \ref{prop:almost} we need some facts concerning the behaviour of the test functions under the dynamics.


\subsection{Contraction of test functions}
\label{sec:test}

For $\{ \iota_j \}_{j=1}^n \subset \cI(\tau_*, \cK_*, E_*)$, $W \in \cW^s$, $\psi \in \cD_{a, \beta}(W)$, and $W_i \in \cG_n(W)$, define 
\[
\hT_{n , W_i} \psi = \hT_{n, i} \psi := \psi \circ T_n \cdot J_{W_i}T_n,
\]
where $J_{W_i}T_n$ denotes the Jacobian of $T_n$ along $W_i$ with respect to arclength.

The following lemma is a consequence of {\bf (H1)}.

\begin{lemma}
\label{lem:test contract}
Let $n\geq 0$ be such that $C_1^{-\beta} \Lambda^{-\beta n}<1$, where $C_1 \le 1$ is from \eqref{eq:hyp},  and fix $a > (1- C_1^{-\beta} \Lambda^{-\beta n})^{-1}C_d \delta_0^{1/3 - \beta}$. For each $\beta \in (0, 1/3]$, there exist $\sigma, \bar\xi < 1$ such that for all $W \in \cW^s$ and $W_i \in \cG_n(W)$,
\begin{itemize}
  \item[a)]  $\hT_{n,i} (\cD_{a, \beta}(W)) \subset \cD_{\sigma a, \beta}(W_i)$;
  \item[b)]  $\rho_{W_i,a, \beta} (\hT_{n,i} \psi_1, \hT_{n,i} \psi_2) \le \bar\xi \rho_{W, a , \beta}(\psi_1, \psi_2)$
  for all $\psi_1, \psi_2 \in \cD_{a, \beta}(W)$. 
\end{itemize}
\end{lemma}

\begin{proof}
(a)  We need to measure the log-H\"older norm of $\hT_{n,i} \psi$ for $\psi \in \cD_{a, \beta}(W)$.
For $x, y \in W_i$, recalling \eqref{eq:hyp}, we estimate,
\[
 \frac{\hT_{n,i} \psi(x)}{\hT_{n,i} \psi(y)} = \frac{\psi(T_nx) J_{W_i}T_n(x)}{\psi(T_ny) J_{W_i}T_n(y)}
\le e^{a d(T_nx, T_ny)^\beta + C_d d(x,y)^{1/3}} 
\le e^{(a C_1^{-\beta} \Lambda^{-\beta n} + C_d \delta_0^{1/3-\beta}) d(x,y)^{\beta}},
\]
where we have used \eqref{eq:hyp} and \eqref{eq:distortion} as well as the fact that $\beta \le 1/3$.
This proves the first statement of the lemma 
with $\sigma = C_1^{-\beta} \Lambda^{-\beta n} + a^{-1} C_d \delta_0^{1/3-\beta}$.

\smallskip
\noindent
(b)  Using Lemma~\ref{lem:H metric}, if $\psi_1, \psi_2 \in \cD_{\sigma a, \beta}(W_i)$, then,
\begin{equation}
\label{eq:finite diam}
\begin{split}
\rho_{W_i,a,\beta}(\psi_1, \psi_2) & = \log \left[ \sup_{x,y,u,v \in W_i} 
\frac{e^{ad(x,y)^\beta}\cdot \psi_1(x) - \psi_1(y)}{e^{a d(x,y)^\beta}\cdot \psi_2(x) - \psi_2(y)}
\cdot \frac{e^{ad(u,v)^\beta}\cdot \psi_2(u) - \psi_2(v)}{e^{a d(u,v)^\beta} \cdot\psi_1(u) - \psi_1(v)} \right] \\
& \le \log \left[ \sup_{x, y, u, v \in W} \frac{e^{(a+\sigma a)d(x,y)^\beta} - 1}{e^{(a-\sigma a)d(x,y)^\beta} -1}\cdot
\frac{e^{(a+\sigma a)d(u,v)^\beta} - 1}{e^{(a-\sigma a)d(u,v)^\beta} -1} \cdot\frac{\psi_1(y)\cdot \psi_2(v)}{\psi_2(y)\cdot \psi_1(u)} \right]  \\
& \le \log \left[ \frac{(a+\sigma a)^2}{(a - \sigma a)^2} \cdot e^{2 a(1+\sigma) \delta_0^\beta}\cdot e^{2 a \delta_0^\beta} \right] =: K .
\end{split}
\end{equation}
Thus the diameter of $\cD_{\sigma a, \beta}(W_i)$ is finite in $\cD_{a, \beta}(W_i)$.
Part (b) of the lemma then follows from \cite[Theorem 1.1]{liv95}, with
$\bar\xi = \tanh (K/4) < 1$.
\end{proof}

\begin{cor}
\label{cor:contract}
Let $n_1$ denote the least positive integer satisfying $C_1^{-\beta} \Lambda^{-\beta n_1}<1$ and $a C_1^{-\beta} \Lambda^{-\beta n_1} + C_d \delta_0^{1/3 - \beta} < a$.  Define $\xi = \bar\xi^{\frac{1}{2n_1}}<1$.
Then for $W \in \cW^s$, $n \ge n_1$ and $W_i \in \cG_n(W)$,
\[
\rho_{W_i, a, \beta} (\hT_{n,i} \psi_1, \hT_{n,i} \psi_2)  \le \xi^n \rho_{W, a, \beta}(\psi_1, \psi_2) \qquad \mbox{for all $\psi_1, \psi_2 \in \cD_{a, \beta}(W)$} .
\]
\end{cor}
\begin{proof}
The proof follows immediately from Lemma~\ref{lem:test contract} once we
decompose $n = kn_1 + r$, where $ r \in [0, n_1)$ and write
\[
\hT_{n, W_i} \psi = \hT_{n_1+r, W_i} \circ \hT_{n_1, T_{n_1+r}(W_i)}
\circ \hT_{n_1, T_{2n_1 + r}(W_i)} \circ \cdots \circ \hT_{n_1, T_{(k-1)n_1+r}(W_i)} \psi .  
\]
Each of the operators $\hT_{n_1, T_{jn_1+r}(W_i)}$ satisfies Lemma~\ref{lem:test contract}
with the same $\sigma$ and $\bar\xi$.  The corollary then follows
using the observation that $\bar\xi^{\lfloor n/n_1 \rfloor} \le \xi^n$, $\forall n \ge n_1$.
\end{proof}

It is important for what follows that the contractive factor
$\bar\xi<1$ is explicitly given in terms of the diameter $K$, which depends only on 
$a$, $\sigma$, $\delta_0$ and $\beta$, but not on $\delta$.  While $n_1$ depends on the parameter choice
$\beta$, it also is independent of $\delta$.

In what follows, we require $n_0 \ge n_1$ by definition, so that Lemma~\ref{lem:test contract} and Corollary~\ref{cor:contract} will hold for all $n \ge n_0$.


\subsection{Proof of Proposition~\ref{prop:almost}}
\label{sec:prop proof}

This section is devoted to the proof of Proposition~\ref{prop:almost}.


\subsubsection{Preliminary estimate on $L$}
For $W\in\cW^s$, recalling Defintion \ref{def:Gn}, we
denote by $Sh_n(W; \delta)$ the
elements of $\cG_n(W)$ of length less than $\delta$ and by $Lo_n(W; \delta)$ the elements of $\cG_n(W)$ of length at least 
$\delta$.

\begin{lemma}
\label{lem:first L}
Fix $\delta \in (0, \delta_0/3)$ so that $4A \delta \delta_0^{-1} \bar C_0 \le 1/4$, then, for all $f \in \cC_{c,A,L}(\delta)$ and
$n \ge n_0$,
\[
\tri \Lp_n f \tri_+ \le  \tfrac 32 \tri f \tri_+ \quad \mbox{and} \quad
\tri \Lp_n f \tri_-  \ge \tfrac 12 \tri f \tri_- .
\]
\end{lemma}
\begin{proof}
Let $W \in \cW^s(\delta)$, $\psi \in \cD_{a,\beta}(W)$.    Then,
\begin{equation}
\label{eq:first L}
\int_W \Lp_n f \, \psi = \sum_{W_i \in Lo_n(W; \delta)} \int_{W_i} f \, \psi \circ T_n \, J_{W_i}T_n
+  \sum_{W_i \in Sh_n(W; \delta)} \int_{W_i} f \, \psi \circ T_n \, J_{W_i}T_n .
\end{equation}
Now since $\psi \circ T_n J_{W_i}T_n \in \cD_{a, \beta}(W_i)$ by Lemma~\ref{lem:test contract}, we  
subdivide elements $W_i \in Lo_n(W; \delta)$ into curves 
$U^i_\ell$ having length between $\delta$ and $2\delta$ and use the definition of $\tri f \tri_+$ on
each such curve to estimate,
\[
\int_{W_i} f \, \psi \circ T_n \, J_{W_i}T_n \le \sum_\ell \tri f \tri_+ \int_{U^i_\ell} \psi \circ T_n \, J_{W_i}T_n
= \tri f \tri_+ \int_{T_nW_i} \psi \, .
\]
To estimate the short pieces, we apply \eqref{eq:cone 3}, 
change variables again and use $\psi \in \cD_{a, \beta}(W)$, and finally apply
Lemma~\ref{lem:full growth}(b) since $Sh_n(W;\delta) \subset \cG_n(W)$, to estimate 
\[
\begin{split}
 \sum_{W_i \in Sh_n(W; \delta)} \int_{W_i} f \, \psi \circ T_n \, J_{W_i}T_n
& \le  \sum_{W_i \in Sh_n(W; \delta)} \tri f \tri_- A |W_i|^q \delta^{1-q} \fint_{W_i} \psi \circ T_n \, J_{W_i}T_n \\
& \le \delta A \tri f \tri_- e^{a(2\delta)^\beta} \fint_W \psi \; (\bar C_0 \delta_0^{-1} |W| + C_0 \theta_1^n) .
\end{split}
\]
Putting these estimates together in \eqref{eq:first L} and since $|W| \ge \delta$ implies $\delta\fint_W \psi\leq \int_W \psi$, we obtain,
\[
\begin{split}
\int_W \Lp_n f \, \psi & \le  \sum_{W_i \in Lo_n(W; \delta)} \tri f \tri_+  \int_{T_nW_i} \psi + 
A  \tri f \tri_- e^{a(2\delta)^\beta} \int_W \psi \; (\bar C_0 \delta \delta_0^{-1} + C_0 \theta_1^n) \\
& \le \tri f \tri_+ \int_W \psi \; \left( 1 + A e^{a(2\delta)^\beta} (\delta \delta_0^{-1} \bar C_0 + C_0 \theta_1^n ) \right).
\end{split}
\]
Now \eqref{eq:delta_0 condition} implies $e^{a(2\delta)^\beta} \le 2$, and our choices of $n_0$ and $\delta$
imply $2A\max\{ \bar C_0 \delta \delta_0^{-1} , C_0 \theta_1^{n_0} \} \le 1/4$, which yields the
required estimate on $\tri \Lp_n f \tri_+$ for all $n \ge n_0$.

For the bound on $\tri \Lp_n f \tri_-$, we perform a similar estimate, except noting that
for $W_i \in Lo_n(W; \delta)$,
\[
\int_{W_i} f \, \psi \circ T_n \, J_{W_i}T_n \ge \tri f \tri_- \int_{T_nW_i} \psi,
\]
we follow \eqref{eq:first L} to estimate,
\[
\begin{split}
\int_W \Lp_n f \, \psi & \ge  \sum_{W_i \in Lo_n(W; \delta)} \tri f \tri_-  \int_{T_nW_i} \psi - 
A  \tri f \tri_- e^{a(2\delta)^\beta} \int_W \psi \; (\bar C_0 \delta \delta_0^{-1} + C_0 \theta_1^n)  \\
& \ge \tri f \tri_- \int_W \psi \; \left( 1 - 2A e^{a(2\delta)^\beta} (\delta \delta_0^{-1} \bar C_0 + C_0 \theta_1^n )\right) \, . 
\end{split}
\]
Again using our choice of $n_0$ and $\delta$, we have $4AC_0 \theta_1^n \le 1/4$
and $4A \delta \delta_0^{-1} \bar C_0 \le 1/4$, which  yields 
$\tri \Lp_n f \tri_- \ge \frac 12 \tri f \tri_-$.
\end{proof}
In particular the above implies the estimate: for all $n \ge n_0$,

\begin{equation}\label{eq:no-expansion-L}
\frac{\tri \Lp_n f \tri_+}{\tri \Lp_n f \tri_-}\leq 3 \frac{\tri f \tri_+}{\tri f \tri_-}\leq 3L.
\end{equation}


\subsubsection{Contraction of the parameter $A$}
\label{sec:contraction-A}
We prove that the parameter $A$ contracts in \eqref{eq:cone 3}.  Choose $f \in \cC_{c, A, L}(\delta)$.
Let $W \in \cW^s$ with $|W| \le 2\delta$,  $\psi \in \cD_{a,\beta}(W)$ and $x \in W$.  From now on,
we will refer to $Lo_n(W; \delta)$ and $Sh_n(W; \delta)$ as simply
$Lo_n(W)$ and $Sh_n(W)$.  We follow \eqref{eq:first L} to write
\[
\begin{split}
\left| \int_W \Lp_n f \, \psi \right| & \le \sum_{W_i \in Lo_n(W)} \int_{W_i} f \, \psi \circ T_n \, J_{W_i}T_n
+  \sum_{W_i \in Sh_n(W)} \left| \int_{W_i} f \, \psi \circ T_n \, J_{W_i}T_n \right| \\
& \le \sum_{W_i \in Lo_n(W)} \tri f \tri_+ \int_{W_i} \psi \circ T_n \, J_{W_i}T_n
+ \sum_{W_i \in Sh_n(W)} A \delta^{1-q} |W_i|^q \tri f \tri_- \fint_{W_i} \psi \circ T_n \,J_{W_i}T_n \\
& \le \sum_{W_i \in Lo_n(W)} \tri f \tri_- L \int_{T_nW_i} \psi + A \delta^{1-q} |W|^q \tri f \tri_- |\psi|_{C^0}
\sum_{W_i \in Sh_n(W)} \frac{|W_i|^q}{|W|^q}  \frac{|T_nW_i|}{|W_i|} ,
\end{split}
\]
where in the second line we have used \eqref{eq:cone 3} for the sum on short pieces.
Since $|W| \le 2\delta$, the first sum above is bounded by
\[
\tri f \tri_- L |W| \fint_W \psi \le \tri f \tri_- 2 L \delta^{1-q} |W|^q \fint_W \psi \, .
\]
For the sum on short pieces, we use Lemma~\ref{lem:full growth}(b)
and the H\"older inequality to estimate\footnote{Note that $\sum_i |T_nW_i| \le |W|$ and $ \frac{|T_nW_i|}{|W_i|}\leq|J_{W_i}T_n|_{C^0(W_i)}$.}
\[
\begin{split}
\sum_{W_i \in Sh_n(W)} \frac{|W_i|^q}{|W|^q} \frac{|T_nW_i|}{|W_i|}
& \le \left(\sum_{W_i \in Sh_n(W)} \frac{|T_nW_i|}{|W|} \right)^q 
\left( \sum_{W_i \in Sh_n(W)} |J_{W_i}T_n|_{C^0(W_i)} \right)^{1-q} \\
& \le ( \bar C_0 \delta_0^{-1} |W| + C_0 \theta_1^n )^{1-q} .
\end{split}
\]
Combining these two estimates with Lemma~\ref{lem:avg} yields,
\begin{equation}
\label{eq:A}
\frac{|\int_W \Lp_n f \, \psi|}{\fint_W \psi} \le A \delta^{1-q} |W|^q \tri f \tri_-
\left( 2 L A^{-1} + e^{a(2\delta)^\beta} ( \bar C_0 \delta_0^{-1} |W| + C_0 \theta_1^n )^{1-q}  \right) .  
\end{equation}
This contracts the parameter $A$ if $2L A^{-1} + e^{a(2\delta)^\beta} ( 2\bar C_0 \delta  \delta_0^{-1} + C_0 \theta_1^n )^{1-q} <1$,
which we can achieve if $e^{a(2\delta)^\beta} \le 2$,
\begin{equation}
\label{eq:AL}
A>4L, \quad \mbox{ and } \quad
( 2\bar C_0 \delta  \delta_0^{-1} + C_0 \theta_1^{n_0} )^{1-q} <1/4 \, .
\end{equation}
Remark that since $L \ge 1$, we have $A > 4$, and so according to the assumption
of Lemma~\ref{lem:first L}, $2 \bar C_0 \delta \delta_0^{-1} \le 1/32$.  Moreover,
$C_0 \theta_1^{n_0} \le 1/64$ by choice of $n_0$, and since $1-q \ge 1/2$, the
second condition in \eqref{eq:AL} is always satisfied under the assumption of
Lemma~\ref{lem:first L}.


\subsubsection{Contraction of the parameter $c$}
\label{subsec:contract c}

Finally, we verify the contraction of $c$ via \eqref{eq:cone 5}.  
Let $f \in \cC_{c,A,L}(\delta)$ and $W^1, W^2 \in \cW^s$ with $|W^k| \le 2 \delta$ and
$d_{\cW^s}(W^1, W^2) \le \delta$.  Take $\psi_k \in \cD_{a, \alpha}(W^k)$ with
$d_*(\psi_1, \psi_2) = 0$.

Without loss of generality we can assume $|W^2|\geq |W^1|$ and $\fint_{W_1}\psi_1=1$.
Next, note that cone condition \eqref{eq:cone 3} implies (see Lemma \ref{lem:first L})
\[
\left|\frac{\int_{W^1} \cL_n f  \, \psi_1}{\fint_{W^1}\psi_1}  - \frac{\int_{W^2}\cL_n f \, \psi_2}{\fint_{W^2}\psi_2} \right|\leq 4A\delta^{1-q}|W^2|^q\tri \cL_n f\tri_-
\]
It follows that the contraction of the parameter $c$ is trivial for 
$ |W^2|^q \leq \delta^{q-\gamma}  \frac{d_{\cW^s}(W^1,W^2)^{\gamma}c}8$.
Thus it suffices to consider the case 
\begin{equation}
\label{eq:W2-long}
 |W^2|^q \geq \delta^{q - \gamma}   \frac{d_{\cW^s}(W^1,W^2)^{\gamma}c}8 \, .
\end{equation}

Remark that by definition, $d_{\cW^s}(W^1, W^2) \le \delta$ implies 
$I_{W_1} \cap I_{W_2} \neq \emptyset$.  To proceed, define
$C_s := \sqrt{1 + (\cK_*^{-1} + \tau_*^{-1})^2  }$, which depends on the maximum absolute value of the slopes of curves in the 
stable cone defined in \eqref{eq:hyp}.  We assume,

\begin{equation}
\label{eq:q-gamma1}
 q \ge \gamma   \; , \mbox{and }\; c \ge 16 C_s^q  \, . 
\end{equation}

Next, for any two  manifolds $U^i\in\cW^s_-(\delta)$ defined on the intervals $I_i$ with $J=I_1\cap I_2$, by the distance definition \eqref{eq:W-distance} we have,
\begin{equation}\label{eq:W-difference}
\begin{split}
|\,|U^1|-|U^2|\,|&\leq \int_J|\,\|G'_1\|-\|G'_2\|\,| \,dr +\sum_{i=1}^2\int_{I_i\setminus J}\|G'_i\| dr  \\
&\leq \int_J \|G'_2-G'_1\| dr + C_s |I_1 \bigtriangleup I_2| \leq (|U^1|+C_s)d_{\cW^s}(U^1,U^2).
\end{split}
\end{equation}

Since $\fint_{W_1} \psi_1 = 1$, we have $|\psi_1|_\infty \le e^{a(2\delta)^\alpha}$.  On the other hand, since
$I_{W^1} \cap I_{W^2} \neq \emptyset$ and 
$d_*(\psi_1, \psi_2)=0$,
there must exist $r \in  I_{W^1} \cap I_{W^2}$ such that $\psi_1 \circ G_{W^1}(r) \| G_{W^1}'(r) \| = \psi_2 \circ G_{W^2}(r) \| G_{W^2}'(r) \|$.  Thus since,
\[
\begin{split}
\frac{\| G'_{W^1}(r) \| }{\| G'_{W^2}(r) \| } & = \sqrt{ \frac{1 + (\vf_{W^1}'(r))^2}{ 1 + (\vf_{W^2}'(r))^2 } }
= \sqrt{ 1 + \frac{(\vf_{W^1}'(r) - \vf_{W^2}'(r))( 2 \vf_{W^2}'(r) + (\vf_{W^1}'(r) - \vf_{W^2}'(r)) )}{1 + (\vf'_{W^2}(r))^2} } \\
& \le \sqrt{1 + d_{\cW^s}(W^1, W^2) ( 2 + d_{\cW^s}(W^1, W^2)) }
\le \sqrt{1 + 3 \delta} \le 2 \, ,
\end{split}
\]
where we use $\delta < 1$, we estimate,
\begin{equation}
\label{eq:psi2}
|\psi_2|_{\infty} \le 2e^{a(2\delta)^\alpha}  |\psi_1|_\infty \le 2e^{2a(2\delta)^\alpha}  \, .
\end{equation}
Then recalling Remark~\ref{rem:change-int} and \eqref{eq:delta_0 condition},  it follows that
\[
\begin{split}
\left|\int_{W^1}\psi_1  - \int_{W^2}\psi_2\right| \; &\leq \;  e^{a(2\delta)^\alpha} C_s |I_{W^1}\setminus I_{W^2}| 
+ e^{2a(2\delta)^\alpha} 2C_s |I_{W^2}\setminus I_{W^1}| \\
& \le 2C_s e^{2a(2\delta)^\alpha}d_{\cW^s}(W^1,W^2)\le 4.8 \;C_s d_{\cW^s}(W^1,W^2).
\end{split}
\]
Putting this together with \eqref{eq:W-difference} and using $\int_{W_1} \psi_1 = |W_1|$, we estimate,
\begin{equation}
\label{eq:new-difference}
\begin{split}
 \left| |W^2| - \int_{W^2} \psi_2 \right| \; & \le \; \left| |W^2| - |W^1| \right| 
 + \left| \int_{W^1} \psi_1 - \int_{W^2} \psi_2 \right| \\
& \le \; (|W^1| + 5.8\  C_s  ) d_{\cW^s}(W^1, W^2)  \le 6 C_s d_{\cW^s}(W^1, W^2) \, ,
\end{split}
\end{equation}
where we have used \eqref{eq:delta_0 condition}, $3\delta\leq \delta_0\leq 1/2\leq C_s/2$ and $\alpha>\beta$.\\
Hence, recalling Lemma~\ref{lem:first L} and \eqref{eq:A}, $d_{\cW^s}(W^1,W^2)\leq \delta$ and using 
\eqref{eq:W2-long},  \eqref{eq:q-gamma1} and \eqref{eq:new-difference},  we have
\begin{equation}
\label{eq:prepare-c}
\begin{split}
&\left|\frac{\int_{W^1} \cL_n f \psi_1}{\fint_{W^1}\psi_1}  - \frac{\int_{W^2}\cL_n f \psi_2}{\fint_{W^2}\psi_2} \right|\leq \left|\int_{W^1}  \cL_n f \psi_1  - \int_{W^2} \cL_n f \psi_2\right|
+\left|\int_{W^2} \cL_n f \psi_2\right|\left|\frac{|W^2|}{\int_{W^2}\psi_2}-1\right|\\
&\leq\left|\int_{W^1}  \cL_n f \psi_1  - \int_{W^2} \cL_n f \psi_2\right|
+A\left[\frac{\delta}{|W^2|}\right]^{1-q} \left| |W^2|-\int_{W^2}\psi_2\right|  2  \tri \cL_n f\tri_-\\
&\leq\left|\int_{W^1}  \cL_n f \psi_1  - \int_{W^2} \cL_n f \psi_2\right|
+ 2^{3-1/q}  3C_s^q A  \delta^{1-\gamma}   d_{\cW^s}(W^1,W^2)^{\gamma}\tri \cL^n f\tri_- \, .
\end{split}
\end{equation}

To conclude it suffices then to compare $\int_{W^1} \Lp_n f \, \psi_1$ and $\int_{W^2} \Lp_n f \, \psi_2$. 
To this end,  define $\cG_n^\delta(W^k)$ as the $n$th generation of pieces in $T_n^{-1}W^k$
as in Definition~\ref{def:Gn}, but with pieces subdivided between length $\delta$ and $2\delta$ rather than
$\delta_0/2$ and $\delta_0$.  
We create partitions of $\cG_n^\delta(W^k)$ into `matched' and `unmatched' pieces as follows.
For each curve $W^1_i \in \cG_n^\delta(W^1)$, we construct a foliation of vertical line segments $\{ \ell_x \}_{x \in W^1_i}$ centered at $x$ and having length at most $3C_1 \Lambda^{-n+1} d_{\cW^s}(W^1, W^2)$ such that their images under $T_n$ either end on a singularity curve in $\cS_{-n}^{\bH}$ or, if not cut by a singularity, have length $3 d_{\cW^s}(W^1, W^2)$, with length at least $d_{\cW^s}(W^1, W^2)$ on each side of $T_n(x)$.  

In the latter case, this implies that $\ell_x$ intersects a unique homogeneous element of $T_n^{-1}W^2$.
Let the subcurve $U^1_{i,+}\subset W^1_i$ be the union of the points $x$ for which this happens and let $U^2_{i,+}=\{\ell_x\cap T_n^{-1}W^2\}_{x\in  U_{i,+}}$ be the corresponding subcurve in $T_n^{-1}W^2$.\footnote{Note that, by \cite[Proposition 4.47]{chernov book}, given two maximal homogeneous subcurves of $T_n^{-1}W^k$ that are connected by a vertical segment disjoint from $\cS_{n}^{\bH}$, there must exist two piecewise smooth curves in $\cS_{n}^{\bH}$ that connect the boundaries of such two subcurves forming a {\em rectangle} that does not contain any element of $\cS_{n}^{\bH}$ in its interior. Thus $U^2_{i,+}$ must be a connected subcurve.} Since $U^k_{i,+}$ has length at most $2\delta$, then $U^2_{i,+}$ can intersect at most $3$ elements of $\cG_n^\delta(W^2)$, due to the possible different ways in which long pieces have been split in $\cG_n^\delta(W^1)$ and $\cG_n^\delta(W^2)$. We call $U^2_j$ the elements of $\{U^2_{i,+}\cap W^2_l\}_{W^2_l\in \cG_n^\delta(W^2)}$ and set $U^1_j=\{x\in U^1_{i,+}\;:\; \ell_x\cap U^2_j\neq \emptyset\}$.
We call the subcurves $U^k_j$, $k=1,2$ `matched', while we call the remaining subcurves $V^k_j$ `unmatched'.  Note that, by construction,  each $W_i^k\in \cG_n^\delta(W^k)$ can contain at most two unmatched elements
and at most 3 matched elements. In addition, for $x \in V^1_j$, either $T_n(\ell_x)$ intersects $\cS^{\bH}_{-n}$ or $T_n(x)$ is near an end point of $W^1$.  In either case, due to the uniform transversality of stable and unstable cones, $|T_n(V^1_j)|$ is short in a sense we will make precise below.
 
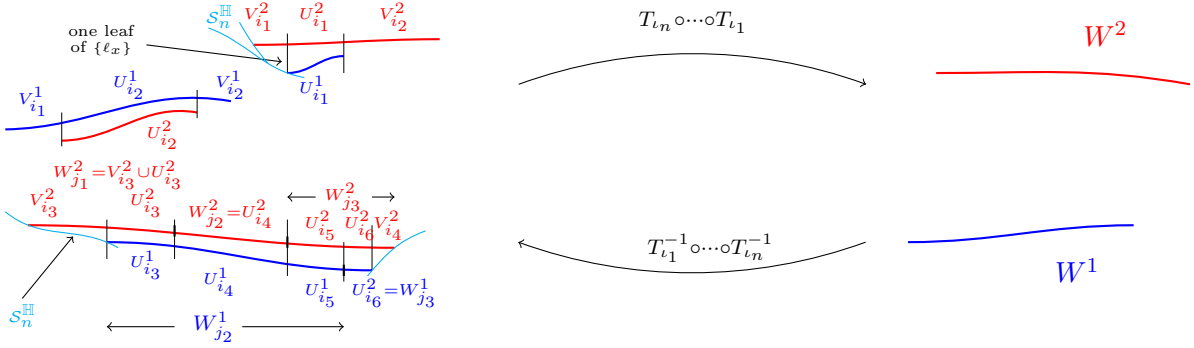
\begin{figure}
\begin{centering}
\begin{tikzpicture}[scale=0.75]
\draw[blue,thick] (-5,12) to[out=01, in=170] (-1,12.5);
\draw[blue,thick] (-3.2,10) to[out=0, in=180] (1.5,9.5);  
\draw[blue,thick] (0,13) to[out=0, in=180] (1,13.3); 
\draw[red,thick] (-4,11.8) to[out=01, in=170] (-1.6,12.3);
\draw[red,thick] (-4.6,10.3) to[out=0, in=180] (1.9,9.9);  
\draw[red,thick] (-.6,13.5) to[out=0, in=180] (2.7,13.6); 
\draw[blue, thick] (11,10) to[out=0, in=180] (15,10.3);
\draw[red, thick] (11.5,13) to[out=0, in=170] (16,12.8);  
\draw [very thin] (-3.2,9.7) to (-3.2,10.4); 
\draw [very thin] (-2,9.7) to (-2,10.4); 
\draw [very thin] (-0,9.4) to (-0,10.4);
\draw [thick] (-2,10.1) to (-2,10.3); 
\draw [thick] (-0,9.9) to (-0,10.1);
\draw [thick] (1,9.4) to (1,9.6);
\draw [very thin] (1,9.3) to (1,10);
\draw [very thin] (1.5,9.5) to (1.5,10.3);
\draw [very thin] (-4,11.7) to (-4,12.3);
\draw [very thin] (-1.6,12.2) to (-1.6,12.7);
\draw [very thin] (0,13) to (0,13.7);
\draw [very thin] (1,13) to (1,13.7);
\draw [->] (10.25,10) arc (-70:-110:9);
\draw [->] (4.1,12.8) arc (110:70:9);
\draw [->] (-2.5,13.5) -- (-0.1, 13.2);
\draw [->] (-4.7,9) -- (-3.8,10.1);
\draw [cyan] (0.3,12.92) to [out=170, in=-20] (-1.4,13.8); 
\draw [cyan] (-.35,13.177) to [out=130, in=-60] (-.85,13.9); 
\draw [cyan] (-3,9.9) to [out=150, in=-40] (-5,10.54); 
\draw [cyan] (1.42,9.4) to [out=50, in=200] (2.45,10.2); 
\draw [<-] (-3.2,8.5)--(-2,8.5);
\draw [->] (-.8,8.5)--(1,8.5);
\draw [<-] (0,10.8)--(0.5,10.8);
\draw [->] (1.5,10.8)--(1.9,10.8);
\node at (-1.3, 8.5) {$\scriptstyle\color{blue} W^1_{j_2}$};
\node at (14,9.5) {\color{blue}$W^1$};
\node at (14.5,13.7) {\color{red}$W^2$};
\node at (7.5, 9.9){$\scriptstyle T_{\iota_1}^{-1}\circ \cdots \circ T_{\iota_n}^{-1}$};
\node at (7.2, 13.9){$\scriptstyle T_{\iota_n}\circ \cdots \circ T_{\iota_1}$};
\node at (1.9,14) {$\scriptscriptstyle\color{red} V^2_{i_2}$};
\node at (0.5,14) {$\scriptscriptstyle\color{red} U^2_{i_1}$};
\node at (-0.5,14) {$\scriptscriptstyle\color{red} V^2_{i_1}$};
\node at (0.5,12.7) {$\scriptscriptstyle\color{blue} U^1_{i_1}$};
\node at (-4.5,12.5) {$\scriptscriptstyle\color{blue} V^1_{i_1}$};
\node at (-2.8,12.8) {$\scriptscriptstyle\color{blue} U^1_{i_2}$};
\node at (-1,12.8) {$\scriptscriptstyle\color{blue} V^1_{i_2}$};
\node at (-2.2,11.9) {$\scriptscriptstyle\color{red} U^2_{i_2}$};
\node at (-4.3,10.7) {$\scriptscriptstyle\color{red} V^2_{i_3}$};
\node at (-2.5,10.7) {$\scriptscriptstyle\color{red} U^2_{i_3}$};
\node at (-3,11.2) {$\scriptscriptstyle\color{red} W^2_{j_1}=V^2_{i_3}\cup U^2_{i_3}$};
\node at (-1,10.5) {$\scriptscriptstyle\color{red} W^2_{j_2}=U^2_{i_4}$};
\node at (1,10.8) {$\scriptscriptstyle\color{red} W^2_{j_3}$};
\node at (0.6,10.3) {$\scriptscriptstyle\color{red} U^2_{i_5}$};
\node at (1.3,10.3) {$\scriptscriptstyle\color{red} U^2_{i_6}$};
\node at (1.8,10.3) {$\scriptscriptstyle\color{red} V^2_{i_4}$};
\node at (-2.5,9.6) {$\scriptscriptstyle\color{blue} U^1_{i_3}$};
\node at (-1.2,9.3) {$\scriptscriptstyle\color{blue} U^1_{i_4}$};
\node at (0.6,9.1) {$\scriptscriptstyle\color{blue} U^1_{i_5}$};
\node at (1.9,9.1) {$\scriptscriptstyle\color{blue} U^2_{i_6}=W^1_{j_3}$};
\node at (-1.2, 14) {$\scriptscriptstyle\color{cyan}\cS_{n}^{\bH}$};
\node at (-4.7,8.7) {$\scriptscriptstyle\color{cyan}\cS_{n}^{\bH}$};
\node at (-3.3, 13.75){\tiny one leaf};
\node at (-3.3, 13.45){\tiny of $\scriptstyle \{\ell_x\}$};
\end{tikzpicture}
\caption{The Decomposition $U^k_j, V^k_j$. }\label{fig:construction_one}
\end{centering}
\end{figure}
 
Thus we have defined a decomposition of $\cG_n^\delta(W^k) = \cup_j U^k_j \cup \cup_j V^k_j$, such
that $U^1_j$ and $U^2_j$ are defined as the graphs of functions $G_{U^k_j}$ over the
same $r$-interval $I_j$ for each $j$. 
 
Using this decomposition, writing $\hT_{n, U^k_j} \psi_k = \psi_k \circ T_n J_{U^k_j}T_n$ 
and similarly for $\hT_{n, V^k_j}\psi_k$, 
we have
\begin{equation}
\label{eq:unstable split}
\begin{split}
\int_{W^k} \Lp_n f \, \psi_k = \sum_{j} \int_{U^k_j} f \, \hT_{n, U^k_j}\psi_k
+ \sum_j \int_{V^k_j} f \, \hT_{n, V^k_j} \psi_k.
\end{split}
\end{equation}
We estimate the contribution from unmatched pieces first.  To do so, we group the $V^k_j$
as follows.  We say $V^k_j$ is `created' at time $0 \le i \le n-1$ if $i$ is the smallest $t$ such
that either an endpoint of $T_{n-t}(V^k_j)$  intersects $T_{\iota_{t+1}}(\cS_0^\bH)$,
or $T_{n-t}(V^k_j)$ is
contained in a larger unmatched piece with this property (this second case can happen when both endpoints of
$V^k_j$ do not belong to $\cS_{T_n}^{\bH}$). 
Due to the uniform transversality of the stable cone with curves in $T_{\iota_{i+1}}(\cS_0^\bH)$ as well
as the uniform transversality of the stable and unstable cones, we have
$|T_{n-i}V^k_j| \le \bar C_3 \Lambda^{-i} d_{\cW^s}(W^1, W^2)$, for some constant $\bar C_3 > 0$. 
Define $P(i) = \{ j : V^1_j \mbox{ created at time } i \}$.

For ease of notation, when we change variables, we will adopt the following notation for $n \ge 1$,
$k \ge 0$,
\begin{equation}
\label{eq:n k notation}
T_{n, n-k} = T_{\iota_n} \circ \cdots \circ T_{\iota_{n-k+1}} \, .
\end{equation}
In this notation, $T_{n,0} = T_n$ and $T_n = T_{n, n-k} \circ T_{n-k}$.

Although we would like to change variables to estimate the contribution on the curves $T_{n-i}(V^1_j)$ 
for $j \in P(i)$, this is one time step before such cuts would be introduced according to our
definition of $\cG_n^\delta(W)$, so Lemma~\ref{lem:full growth} would not apply since there may be
many such $T_{n-i}(V^1_j)$ for each $W^1_\ell \in \cG_i^\delta(W^1)$.  However, there can be at most
two curves $T_{n-i-1}(V^1_j)$, $j \in P(i)$, per element of $W^1_\ell \in \cG_{i+1}^\delta(W^1)$, so we will change variables
to estimate the contribution from curves of the form $T_{n-i-1}(V^1_j)$ instead.
We have two cases.

\smallskip
\noindent
{\em Case 1.  The curve in  $T_{\iota_{i+1}}(\cS_0^\bH)$  that creates $V^1_j$ at time $i$
is the preimage of the boundary of a homogeneity strip.}  Then $T_{n-i-1}V^1_j$ still enjoys uniform transversality with the
boundary of the homogeneity strip and the unstable cone, and so
$|T_{n-i-1}V^1_j| \le \bar C_3 \Lambda^{-i-1} d_{\cW^s}(W^1, W^2)$ as before.

\smallskip
\noindent
{\em Case 2. The curve in $T_{\iota_{i+1}}(\cS_0^\bH)$ that creates $V^1_j$ at time $i$
is not the preimage of the boundary of a homogeneity strip.}  Then $V^1_j$ undergoes bounded expansion from time
$n-i$ to time $n-i-1$.  Thus $|T_{n-i-1}(V^1_j)| \le C \bar C_3 \Lambda^{-i} d_{\cW^s}(W^1, W^2)$, where
$C>0$ depends only on our choice of $k_0$, the minimum index of homogeneity strips.

\smallskip
In either case, we conclude that $|T_{n-i-1}(V^1_j)| \le  C_3 \Lambda^{-i} d_{\cW^s}(W^1, W^2)$, for a uniform
constant $C_3>0$.  Also, since $T_{n-i-1}(V^k_j)$ is contained in an element of $\cG_{n-i-1}^\delta(W)$, 
it follows that all such curves have length at most $2 \delta$, thus we may apply \eqref{eq:cone 3},
\[
\begin{split}
&\left| \sum_j \int_{V^1_j} f \, \hT_{n, V^1_j} \psi_1 \right|
 \le \sum_{i=0}^{n-1} \left| \sum_{j \in P(i)}  \int_{T_{n-i-1}(V^1_j)}  \Lp_{n-i-1} f \cdot 
 \psi_1 \circ T_{n, n-i-1} \, J_{T_{n-i-1}(V^1_j)} T_{n, n-i-1}
  \right| \\
& \le \sum_{i=0}^{n-1} \sum_{j \in P(i)} A \delta^{1-q} |T_{n-i-1}(V^1_j)|^q \tri \Lp_{n-i-1}f \tri_- 
|\psi_1|_{C^0(W^1)} |J_{T_{n-i-1}(V^1_j)}T_{n, n-i-1}|_{C^0(T_{n-i-1}(V^1_j))} \\
& \le \sum_{i=0}^{n-1} A \delta^{1-q} C_3^q \Lambda^{-iq} d_{\cW^s}(W^1, W^2)^q
 \tri \Lp_{n-i-1} f \tri_- \, (2 \bar C_0  + C_0 \theta_1^{i+1} ) | \psi_1|_{C^0(W^1)},
\end{split}
\]
where we have used Lemma~\ref{lem:full growth}-(b) for the sum over $j \in P(i)$ since there are at most two curves
$T_{n-i-1}(V^1_j)$ for each element $W^1_\ell \in \cG_{i+1}^\delta(W)$.\footnote{ Notice that
since we subdivide curves in $\cG_n^\delta(W)$ according to length $\delta$ and not
$\delta_0$, the estimate of Lemma~\ref{lem:full growth}(b) becomes
$\bar C_0 \delta^{-1} |W| + C_0 \theta_1^n \le 2 \bar C_0 + C_0 \theta_1^n$, since $|W| \le 2\delta$.  }

Since $n \ge 2n_0$, we have either that $i+1 \ge n_0$ or $n - (i+1) \ge n_0$.  In the former
case, $\tri \Lp_{n-i-1}f \tri_- \le 2 \tri \Lp_n f \tri_-$ by Lemma~\ref{lem:first L}.  In the latter
case, 
\begin{equation}
\label{eq:tri}
\tri \Lp_{n-i-1} f \tri_- \le \tri \Lp_{n-i-1} f \tri_+ \le \tfrac 32 \tri f \tri_+ \le \tfrac 32 \, L \tri f \tri_-
\le 3 L \tri \Lp_n f \tri_-,
\end{equation}
where we have used Lemma~\ref{lem:first L} twice, once on $\tri \Lp_{n-i-1}f \tri_+$ and
once on $\tri f \tri_-$.  Since the latter estimate \eqref{eq:tri} is the larger of the two, we may use it
for all $i$.  

Also, using the assumption that $d_{\cW^s}(W^1, W^2) \le  \delta  $ and \eqref{eq:q-gamma1} yields,
\[
\delta^{1-q} d_{\cW^s}(W^1, W^2)^{q} 
\le \delta^{1 - \gamma}  d_{\cW^s}(W^1, W^2)^\gamma.
\]
Collecting these estimates and summing over the exponential factors yields (since the estimate for $V^2_j$ is the same),

\begin{equation}
\label{eq:V}
\sum_{j,k} \left| \int_{V^k_j} f \, \hT_{n, V^k_j} \psi_k \right|
\le C_4 AL \delta^{1-\gamma}   d_{\cW^s}(W^1, W^2)^\gamma \tri \Lp_n f \tri_- ,
\end{equation}
for some uniform constant $C_4$ depending only on $\cF(\tau_*, \cK_*, E_*)$ and not on the parameters of the cone.

Next, we estimate the contribution on matched pieces $U^k_j$.  To do this, we will need to change 
test functions on the relevant curves.  Define the following functions on $U^1_j$,
\begin{equation}
\label{eq:switch}
\begin{split}
&\tpsi_2 = \psi_2 \circ T_n \circ G_{U^2_j} \circ G_{U^1_j}^{-1}\;;
\quad
\tJ_{U^2_j}T_n = J_{U^2_j}T_n \circ G_{U^2_j} \circ G_{U^1_j}^{-1}, \\
& \tT_{n, U^2_j} (\psi_2) = \tpsi_2 \cdot \tJ_{U^2_j}T_n \frac{\|G_{U^2_j}'\| \circ G_{U^1_j}^{-1}}{\|G_{U^1_j}'\|\circ G_{U^1_j}^{-1}}.
\end{split}\end{equation}
Note that $d_*(\hT_{n, U^2_j} (\psi_2), \tT_{n, U^2_j} (\psi_2))=0$ by construction.
Also we define 
\begin{equation}\label{eq:psi-diff}
\begin{split}
&\psi^-_j=\min \big\{\hT_{n, U^1_j}( \psi_1),\tT_{n, U^2_j}(\psi_2) \big\}\\
&\psi^\Delta_{1,j}=\hT_{n, U^1_j}( \psi_1)-\psi^-_j\;;\quad\psi^\Delta_{2,j}=\tT_{n, U^2_j}(\psi_2) -\psi^-_j.
\end{split}
\end{equation}

We will need the following lemma to proceed.

\begin{lemma}
\label{lem:compare}
If $c> 4(1+M_0)^q$, $M_0$ is defined in \eqref{eq:M0}, then there exists $C_5 \ge 1$, independent of $n$, $W^1$ and $W^2$ satisfying \eqref{eq:W2-long}, such that for each $j$,
\begin{itemize}
  \item[a)] $\displaystyle d_{\cW^s}(U^1_j, U^2_j) \le C_5 \Lambda^{-n} d_{\cW^s}(W^1, W^2) \,$ ;
  \item[b)]  $\displaystyle e^{-C_5 d_{\cW^s}(W^1, W^2)^\alpha}\leq \frac{\hT_{n, U^1_j} \psi_1(x)}{\tT_{n, U^2_j} \psi_2(x)} \le e^{C_5 d_{\cW^s}(W^1, W^2)^\alpha}\quad \forall x \in U^1_j \,$ ;
\item[c)]  setting $B= 8 \left[ C_5a^{-1}\right]^{\frac{\alpha-\beta}\alpha}d_{\cW^s}(W^1, W^2)^{\alpha-\beta}$
we have $\psi^\Delta_{i,j}+B\psi_j^-\in\cD_{a,\beta}(U^1_j)$, $i =1, 2$.
\end{itemize}
Moreover, $\tT_{n, U^2_j} \psi_2$ and $\psi_j^-$ belong to $\cD_{a,\alpha}(U^1_j)$.
\end{lemma}

We postpone the proof of the lemma and use it to conclude the estimates of this section.

 For future use note that Lemma \ref{lem:compare}(b) implies
\begin{equation}\label{eq:deltapsi}
0\leq \psi^\Delta_{k,j}(x)\leq 2 C_5 d_{\cW^s}(W^1, W^2)^\alpha \psi^-_j(x).
\end{equation}

Observe that since 
$\psi_{k,j}^\Delta + B \psi_j^-, \psi_j^- \in \cD_{a,\beta}(U^1_j)$, $k=1,2$, recalling
\eqref{eq:posi} we may estimate,
\begin{equation}
\label{eq:prelim}
\begin{split}
\left|\int_{U^1_j} f \, \psi^{\Delta}_{k,j}  \right| 
& = \left| \int_{U^1_j} f \, (\psi^\Delta_{k,j} + B \psi^-_j ) - \int_{U^1_j} f \, B \psi^-_j \right| \\
& \le A \delta^{1-q} |U^1_j|^q \tri f \tri_- \max \left\{ \fint_{U^j}  (\psi^\Delta_{k,j} + B \psi^-_j ),  
\fint_{U^1_j} B \psi^-_j \right\} \\
& \le  A \delta^{1-q} |U^1_j|^q \tri f \tri_-  \fint_{U^j}  (\psi^\Delta_{k,j} + B \psi^-_j ) \, . 
\end{split}
\end{equation}

Since $d_*(\hT_{n, U^2_j} \psi_2, \tT_{n, U^2_j} \psi_2) =0$ by construction, and recalling
Remark \ref{rem:change-int}, Lemma \ref{lem:compare}(c), condition \eqref{eq:cone 3}, and \eqref{eq:psi-diff}, \eqref{eq:deltapsi},
\begin{equation}
\label{eq:c-decomposition}
\begin{split}
&\left|\int_{U^1_j} f \, \hT_{n, U^1_j} \psi_1 -\int_{U^2_j} f \,  \hT_{n, U^2_j} \psi_2\right| \leq 
\left|\int_{U^1_j} f \, \hT_{n, U^1_j} \psi_1 -\int_{U^1_j} f \, \tT_{n, U^2_j} \psi_2\right|  \\
&+\left|\frac{\int_{U^1_j} f \, \tT_{n, U^2_j} \psi_2}{\fint_{U^1_j}  \tT_{n, U^2_j} \psi_2} -\frac{\int_{U^2_j} f \,  \hT_{n, U^2_j} \psi_2}{\fint_{U^2_j}  \hT_{n, U^2_j} \psi_2}\right|\fint_{U^2_j}  \hT_{n, U^2_j} \psi_2\\
&+\left|\frac{\int_{U^1_j} f \, \tT_{n, U^2_j} \psi_2}{\fint_{U^1_j} \tT_{n, U^2_j} \psi_2}\right|\,\left| \frac{|U^2_j|-|U^1_j|}{|U^1_j|}\right|\fint_{U^2_j}  \hT_{n, U^2_j} \psi_2\\
&\leq  A\delta^{1-q}|U^1_j|^q\frac{\left[\fint_{U^1_j}(\psi^\Delta_{1,j}+B\psi^-_j)+\fint_{U^1_j}(\psi^\Delta_{2,j}+B\psi^-_j)\right]}{ \fint_{U^2_j}  \hT_{n, U^2_j} \psi_2 } \fint_{U^2_j}  \hT_{n, U^2_j} \psi_2 \tri f \tri_- \\
&\quad +d_{\cW^s}(U^1_j, U^2_j)^\gamma  \delta^{1-\gamma}  c A \fint_{U^2_j}  \hT_{n, U^2_j} \psi_2 \tri f \tri_- \\
&\quad +A\delta^{1-q}|U^1_j|^q \left| \frac{|U^2_j|-|U^1_j|}{|U^1_j|}\right|  \fint_{U^2_j}  \hT_{n, U^2_j} \psi_2  \tri f \tri_- \, ,
\end{split}
\end{equation}
where for the first term, we have used that $|\hT_{n, U^1_j} \psi_1 - \tT_{n, U^2_j} \psi_2 | = \psi_{1, j}^\Delta + \psi_{2,j}^\Delta$ in order to apply \eqref{eq:prelim},
and
for the second and third terms we used that $\tT_{n, U^2_j}\psi_2 \in \cD_{a,\alpha}(U^1_j)$ by Lemma~\ref{lem:compare} 
to apply cone conditions \eqref{eq:cone 5} and \eqref{eq:cone 3}, respectively.
Then, recalling Lemma~\ref{lem:full growth}(b), \eqref{eq:psi2} and \eqref{eq:delta_0 condition}, 
and using that by construction, there are at most 3 curves $U^2_j$ in each element of $\cG_n^{\delta}(W^2)$,
we can estimate
\begin{equation}
\label{eq:summingU2}
\sum_j \fint_{U^2_j}  \hT_{n, U^2_j}\psi_2\leq \sum_j \fint_{U^2_j} |J_{U_2^j}T_n |_\infty \psi_2\circ T_n
\leq 3 (\bar C_0\delta^{-1} |W^2| + C_0\theta_1^n) 2 e^{2a(2\delta)^\alpha}\leq 36 \bar C_0 \, .
\end{equation}
Next, recalling \eqref{eq:W-difference}, we have\footnote{ Since the $U^k_j$ are vertically matched, the term on the right hand side of  \eqref{eq:W-difference} proportional to $C_s$ is absent here.}
\[
|U^2_j| \le |U^1_j|(1 + d_{\cW^s}(U^1_j, U^2_j)) \le 2|U^1_j|
\]
provided we impose
\begin{equation}\label{eq:c-cond}
C_5 \Lambda^{-n_0}  \delta  \le 1
\end{equation}
where $C_5$ is from Lemma~\ref{lem:compare}-(a) and $\Lambda$ is defined in \eqref{eq:hyp}.
Moreover, remembering the definition of $B$ in Lemma \ref{lem:compare}-(c) and equation \eqref{eq:deltapsi},
\begin{equation}\label{eq:Delta-psi-B}
\begin{split}
\fint_{U^1_j}(\psi^\Delta_{k,j}+B\psi^-_j)&\leq \fint_{U^1_j} 10 C_5 \tT_{n, U^2_j}( \psi_2) d_{\cW^s}(W^1,W^2)^\gamma\\
& \leq 10 C_5\frac{|U^2_j|}{|U^1_j|} \fint_{U^2_j}\hT_{n, U^2_j} (\psi_2) d_{\cW^s}(W^1,W^2)^\gamma\\
&\leq 20 C_5 \fint_{U^2_j}\hT_{n, U^2_j} (\psi_2) d_{\cW^s}(W^1,W^2)^\gamma,
\end{split}
\end{equation}
where we have used the assumptions $\alpha-\beta\geq \gamma$ and $a>1$.

Again using \eqref{eq:W-difference} and Lemma \ref{lem:compare}-(a) we have
\begin{equation}
\label{eq:length-ratio}
\begin{split}
\left| \frac{|U^2_j|-|U^1_j|}{|U^1_j|^{1-q}}\right|&\leq  d_{\cW^s}(U^2_j, U^1_j)|U^1_j|^{q}\leq (2 \delta)^{q} C_5 \Lambda^{-n}d_{\cW^s}(W^2, W^1).
\end{split}
\end{equation}
Inserting \eqref{eq:summingU2}, \eqref{eq:Delta-psi-B} and \eqref{eq:length-ratio} in \eqref{eq:c-decomposition} and 
recalling Lemmas~\ref{lem:first L} and \ref{lem:compare}-(a) yields,
\begin{equation}\label{eq:use-in-4}
\begin{split}
& \sum_j \left|\int_{U^1_j} f \, \hT_{n, U^1_j} \psi_1 -\int_{U^2_j} f \,  \hT_{n, U^2_j}\psi_2\right| \\
&\le 
72 \bar C_0 A \delta^{1-\gamma}  d_{\cW^s}(W^1, W^2)^\gamma \tri \Lp_n f \tri_- 
\left( 2^q 40 C_5  \delta^{\gamma} 
+ c C_5 \Lambda^{-n \gamma}  
+  2^q C_5 \Lambda^{-n} \delta \right)
\end{split}
\end{equation}
Then using this estimate in \eqref{eq:prepare-c}, and recalling \eqref{eq:unstable split} and \eqref{eq:V} yields
\begin{equation}
\label{eq:switch test}
\begin{split}
& \left|\frac{\int_{W^1} \cL_nf \, \psi_1}{\fint_{W^1}\psi_1} -\frac{\int_{W^2} \cL_nf \, \psi_2}{\fint_{W^2}\psi_2}\right|\leq 
 \Big\{2^{3-1/q} 3C_s^q  +C_4 L  \\
& \qquad  + 72 \bar C_0   \left( 2^q 40 C_5  \delta^{\gamma}  + c C_5 \Lambda^{-n \gamma}  
+  2^q C_5 \Lambda^{-n} \delta \right)    \Big\}  A\delta^{1-\gamma}  d_{\cW^s}(W^1, W^2)^\gamma \tri \Lp_n f \tri_- 
\end{split}
\end{equation}
which yields the wanted estimate, provided
\begin{equation}
\label{eq:c cond}
2^{3-1/q} C_s^q  +C_4 L 
 + 72 \bar C_0   \left( 2^q 40 C_5  \delta^{\gamma} + c C_5 \Lambda^{-n \gamma}  
+  2^q C_5 \Lambda^{-n} \delta  \right)  <  c. 
\end{equation}


\subsubsection{Proof of Lemma~\ref{lem:compare}}

\begin{proof}
(a) This is \cite[Lemma 3.3]{demzhang13}.

\smallskip
\noindent
(b) Recall that $U^k_j$ is defined as the graph of a function $G_{U^k_j}(r) = (r, \vf_{U^k_j}(r))$, for $r \in I^k_j$, $k = 1,2$.
Due to the vertical matching, we have $I^1_j = I^2_j$.

Now for $x \in U^1_j$, let $r \in I^1_j$ be such that $G_{U^1_j}(r) = x$.  Set
$\bx = G_{U^2_j}(r)$ and note that $x$ and $\bx$ lie on the same vertical line in $M$ since
$U^1_j$ and $U^2_j$ are matched.  Thus by \eqref{eq:u dist},
\begin{equation}
\label{eq:M0}
\frac{J_{U^1_j}T_n(x)}{\tJ_{U^2_j}T_n(x)} = \frac{J_{U^1_j}T_n(x)}{J_{U^2_j}T_n(\bx)}
\le e^{C_d (d(T_nx, T_n\bx)^{1/3}+ \phi(x, \bx))} 
\le e^{C_d M_0 d_{\cW^s}(W^1, W^2)^{1/3}},
\end{equation}
where $M_0$ is a constant depending only on the maximum and minimum slopes in $C^s$ and
$C^u$.  

Next, for $x \in U^1_j$ consider
\[
\frac{\psi_1 \circ T_n(x)}{\tpsi_2(x)}  \frac{\| G'_{U^1_j} \| \circ G_{U^1_j}^{-1}(x)}{\| G'_{U^2_j} \| \circ G_{U^1_j}^{-1}(x)}. 
\]
Let $T_n(x)=(r,G_{W^1}(r))$ and $T_n(\bar x)=(\bar r,G_{W^2}(\bar r))$, then
\[
|r-\bar r|\leq M_0 d_{\cW^s}(W^1,W^2) \, .
\]
If $r\in I_{W^2}$, then since $d_*(\psi_1, \psi_2)=0$,
\[
\frac{\psi_1 \circ G_{W^1}(r)}{\psi_2 \circ G_{W^2}(\bar r)} = \frac{\psi_1 \circ G_{W^1}(r)}{\psi_2 \circ G_{W^2}(r)} \frac{\psi_2 \circ G_{W^2}(r)}{\psi_2 \circ G_{W^2}(\bar r)} \leq 
\frac{ \| G_{W^2}'(r) \|}{\| G_{W^1}'(r) \| } e^{a d(G_{W_1}(r), G_{W^2}(\bar r))^\alpha}.
\]
Next, since $\| G_{W^1}' - G_{W^2}' \| = |\vf'_{W^1} - \vf'_{W^2}|$ and $\| G_{W^k}' \| \ge 1$, we have
\[
\frac{ \| G_{W^2}'(r) \|}{\| G_{W^1}'(r) \| } \le e^{ \| G_{W^1}' - G_{W^2}' \| } \le e^{d_{\cW^s}(W^1, W^2)} \, .
\]
Similarly, $\frac{\| G'_{U^1_j} \| \circ G_{U^1_j}^{-1}(x)}{\| G'_{U^2_j} \| \circ G_{U^1_j}^{-1}(x)} \le e^{d_{\cW^s}(U^1_j, U^2_j)}$.  
Hence, using part (a) of the lemma and assuming

\begin{equation}\label{eq:c-cond2}
C_5 n_0 \Lambda^{-n_0} \delta^{1-\alpha}  \le 1,
\end{equation}

yields
\[
\frac{\psi_1 \circ T_n(x)}{\tpsi_2(x)} \frac{\| G'_{U^1_j} \| \circ G_{U^1_j}^{-1}(x)}{\| G'_{U^2_j} \| \circ G_{U^1_j}^{-1}(x)}
\le e^{(aM_0^\alpha+2) d_{\cW^s}(W^1, W^2)^\alpha} \, . 
\]
 The same estimate holds if $\bar r \in I_{W^1}$. Otherwise it must be that
\[
|I_{W^1}\cap I_{W^2}|\leq M_0 d_{\cW^s}(W^1,W^2)
\]
but then, since $|I_{W^1}\Delta I_{W^2}|\leq d_{\cW^s}(W^1,W^2)$ we would have $|W^2|\leq (1+M_0) d_{\cW^s}(W^1,W^2)$,
which violates \eqref{eq:W2-long} together with the assumption, 
provided
\begin{equation}\label{eq:c-cond3}
c > 4(1+M_0)^q.
\end{equation}
The estimates with the opposite sign follow similarly. Putting together these
estimates yields part (b) of the lemma with $C_5 = M_0C_d  \delta^{1/3 - \alpha} + aM_0^\alpha + 2$.

\smallskip
\noindent
(c) As noted in \eqref{eq:deltapsi}, by (b) it immediately follows that
\[
\left| \psi^\Delta_{i,j}(x)\right|\leq\left| \hT_{n, U^1_j} \psi_1(x)-\tT_{n, U^2_j} \psi_2(x)\right|\leq 
2 C_5 d_{\cW^s}(W^1, W^2)^\alpha \psi_j^-(x).
\]
Next, for $x, y \in U^1_j$, let $\bx = G_{U^2_j} \circ G_{U^1_j}^{-1}(x)$, 
$\by = G_{U^2_j} \circ G_{U^1_j}^{-1}(y)$, and note these are well-defined due to the 
vertical matching between $U^1_j$ and $U^2_j$.  
Let $r = G_{U^1_j}^{-1}(x)$ and $s = G_{U^1_j}^{-1}(y)$. Recalling \eqref{eq:2deriv}, we have
\[
\frac{\| G'_{U^1_j}(r) \|}{\| G'_{U^1_j}(s) \|} \le e^{\| G'_{U^1_j}(r) - G'_{U^1_j}(s) \|} \le e^{B_* |r-s|} \le e^{B_*  d(x,y)} \, ,
\]
and similarly for $\| G'_{U^2_j} \|$.
Using this estimate together with the proof of Lemma~\ref{lem:test contract}(a),
\begin{equation}
\label{eq:incone2}
\begin{split}
\frac{\tT_{n, U^2_j} \psi_2(x)}{\tT_{n, U^2_j} \psi_2(y)} & = \frac{\hT_{n, U^2_j} \psi_2(\bx)}{\hT_{n, U^2_j} \psi_2(\by)} 
\frac{ \| G'_{U^2_j} (r) \|}{\| G'_{U^1_j}(r)\| } \frac{\| G'_{U^1_j}(s) \|}{\| G'_{U^2_j}(s) \|} \\
& \le e^{(a C_1^{-1} \Lambda^{-\alpha n} + C_d (2\delta)^{1/3-\alpha})  d(\bx,\by)^\alpha
+ 2B_* d(x,y)}\leq e^{ad(x,y)^\alpha},
\end{split}
\end{equation}
since $d(\bx, \by) \le M_0 d(x,y)$ and provided 
\begin{equation}\label{eq:cond-a}
(aC_1^{-1} \Lambda^{-\alpha n_0} + C_d (2\delta)^{1/3 - \alpha})M_0^\alpha
+ B_*  (2\delta)^{1-\alpha} < a.
\end{equation}

To abbreviate what follows, let us denote $g_1 = \hT_{n, U^1_j} \psi_1$ and $g_2 = \tT_{n, U^2_j} \psi_2$.
Then, given $x,y\in U^1_j$, we have $\psi_j^-(x)=g_{k(x)}$, $\psi_j^-(y)=g_{k(y)}$. If $k(x)=k(y)$, then, 
by Lemma~\ref{lem:test contract}(a) and \eqref{eq:incone2},
\[
\frac{\psi_j^-(x)}{\psi_j^-(y)}=\frac{g_{k(y)}(x)}{g_{k(y)}(y)}\leq e^{a d(x,y)^\alpha} \, .
\]
If $k(x)\neq k(y)$, then without loss of generality, we can take $k(x)=1$ and $k(y)=2$.  By definition, $g_1(x) \le g_2(x)$
and $g_2(y) \le g_1(y)$.  Hence,
\[
e^{-ad(x,y)^\alpha} \le \frac{g_1(x)}{g_1(y)} \le \frac{\psi_j^-(x)}{\psi_j^-(y)}=\frac{g_1(x)}{g_2(y)} \le \frac{g_2(x)}{g_2(y)} \leq e^{a d(x, y)^\alpha} \, .
\]
It follows that $\psi_j^-\in\cD_{a,\alpha}(U^1_j)$, and by \eqref{eq:incone2}, $\tT_{n, U^2_j} \psi_2 \in \cD_{a, \alpha}(U^1_j)$.

Then, for each $1>B\geq 2 C_5 d_{\cW^s}(W^1, W^2)^\alpha$ and $x,y\in U^1_j$, 
\[
\frac{\psi^\Delta_{i,j}(x)+B\psi_j^-(x)}{\psi^\Delta_{i,j}(y)+B\psi_j^-(y)}\leq \frac{(B+2 C_5 d_{\cW^s}(W^1, W^2)^\alpha )\psi_j^-(x)}{(B-2 C_5 d_{\cW^s}(W^1, W^2)^\alpha )\psi_j^-(y)}\leq e^{a d(x,y)^\alpha+ 4 B^{-1} C_5 d_{\cW^s}(W^1, W^2)^\alpha}\leq e^{a d(x,y)^\beta}
\]
provided $8 B^{-1}C_5 d_{\cW^s}(W^1, W^2)^\alpha\leq  a d(x,y)^\beta$ and
\begin{equation}\label{eq:cond-c3}
(2\delta)^{\alpha-\beta}\leq \frac 12.
\end{equation}

It remains to consider the case $8 B^{-1}C_5 d_{\cW^s}(W^1, W^2)^\alpha\geq  a d(x,y)^\beta$. 
Again we must split into two cases.  If $k(x)=k(y)=k$, then, setting $\{\ell\}=\{1,2\}\setminus \{k\}$,
\begin{equation} \label{eq:aergh}
\begin{split}
\frac{\psi^\Delta_{\ell,j}(x)+B\psi_j^-(x)}{\psi^\Delta_{\ell,j}(y)+B\psi_j^-(y)}&\leq \frac{g_\ell(x)+(B-1)g_k(x)}{g_\ell(y)+(B-1)g_k(y)}
\leq \frac{e^{ a d(x,y)^\alpha} g_\ell(y)+e^{- a d(x,y)^\alpha} (B-1)g_k(y)}{g_\ell(y)+(B-1)g_k(y)}\\
&\leq  e^{ a d(x,y)^\alpha}\left[1+\frac{2 a d(x,y)^\alpha}{B}\right]\leq e^{a [d(x,y)^{\alpha-\beta}(1+2B^{-1})] d(x,y)^\beta}
 \leq e^{\frac a 2 d(x,y)^\beta}
\end{split}
\end{equation}
provided that 
\[
d(x,y)^{\alpha-\beta}(1+2B^{-1})\leq 4 B^{-\frac\alpha\beta}\left[ 8  C_5 d_{\cW^s}(W^1, W^2)^\alpha a^{-1}\right]^{\frac{\alpha-\beta}\beta}\leq \frac 12\,.
\]
That is, 
\[
B\geq 8 \left[ C_5a^{-1}\right]^{\frac{\alpha-\beta}\alpha}d_{\cW^s}(W^1, W^2)^{\alpha-\beta}.
\]
The second case is $k=k(x)\neq k(y)=\ell$. In this case, there must exist $\bar x \in [x,y]$ such that
$\psi_j^-(\bx) = g_1(\bx) = g_2(\bx)$.  Then,
\[
\begin{split}
\frac{\psi^\Delta_{\ell,j}(x)+B\psi_j^-(x)}{\psi^\Delta_{\ell,j}(y)+B\psi_j^-(y)}
& =  \frac{g_\ell(x)+(B-1)g_k(x)}{Bg_\ell(\bar x)} \frac{g_\ell(\bar x)+(B-1)g_k(\bar x)}{g_\ell(\bx)+(B-1)g_k(\bx)}
\leq e^{ a  d(x,y)^\beta} 
\end{split}
\]
by the estimate \eqref{eq:aergh}. A similar estimate holds for $\psi_{k,j}^\Delta$.
It follows that we can choose
\begin{equation}\label{eq:setB}
B= 8 \left[ C_5a^{-1}\right]^{\frac{\alpha-\beta}\alpha}d_{\cW^s}(W^1, W^2)^{\alpha-\beta}
\end{equation}
and have $\psi^\Delta_{i,j}+B\psi_j^-\in\cD_{a,\beta}(U^1_j)$.
\end{proof}


\subsection{Conditions on parameters }
\label{sec:conditions}

In this section, we collect the conditions imposed on the cone parameters during the proof
of Proposition~\ref{prop:almost}.  Recall the conditions on the exponents
stated before the definition of $\cC_{c,A,L}(\delta)$:
$\alpha \in (0, 1/3]$, $q \in (0,1/2)$, $\beta < \alpha$ and $\gamma \le \min \{ \alpha - \beta, q \}$.

From \eqref{eq:delta_0 condition} and Lemma~\ref{lem:first L} we require,
\[
e^{a (2\delta)^\beta} < e^{2a \delta_0^\beta} \le 2 \quad \mbox{and} \quad
4A \bar C_0 \delta \delta_0^{-1} \le 1/4 \, . 
\]
From the proof of Lemma~\ref{lem:first L} and Lemma~\ref{lem:test contract}, we require
the following conditions on $n_0$,
\[
AC_0 \theta_1^{n_0}  \le 1/16 \quad \mbox{and} \quad C_1^{-1} \Lambda^{-\beta n_0} < 1 \, .
\]
From Lemma~\ref{lem:test contract}, Corollary~\ref{cor:contract} and the proof of Lemma~\ref{lem:compare}, we require
\[
a > aC_1^{-1} \Lambda^{-\beta n_0} + C_d \delta_0^{1/3 - \beta} \quad \mbox{and} \quad
a > (aC_1^{-1} \Lambda^{-\alpha n_0} + C_d (2\delta)^{1/3 - \alpha})M_0^\alpha
 + B_* (2\delta)^{1-\alpha} 
\]
(recall that we have chosen $n_0 \ge n_1$ after Corollary~\ref{cor:contract}).

From the bound on \eqref{eq:cone 3}, we require in \eqref{eq:AL},
\[
A > 4 L \, . 
\]
For the contraction of $c$, we require (see \eqref{eq:q-gamma1}, the proof of Lemma \ref{lem:compare} and \eqref{eq:c cond})
\[
\begin{split}
& c > \max \left\{  16 C_s^q  ,  4(1+M_0)^q \right\} \;;\quad C_5 \Lambda^{-n_0}  \delta^{1-\alpha}  \le 1  \; ; 
\quad (2\delta)^{\alpha - \beta} \le \tfrac 12  \; ; \\
&  2^{3-1/q} 3C_s^q  +C_4 L 
 + 72\bar C_0   \left( 2^q 40 C_5  \delta^{\gamma}  + c C_5 \Lambda^{-n _0\gamma}  
+  2^q C_5 \Lambda^{-n_0} \delta \right)  <  c. 
\end{split} 
\]
Finally, in anticipation of \eqref{eq:delta cond}, we require,
\begin{equation}
\label{eq:cA}
cA > 2 C_s \, .
\end{equation}
These are all the conditions we shall place on the parameters for the cone, except for $\delta$, which
we will take as small as required for the mixing arguments of Section~\ref{sec:L contract}.  Indeed, note that if the above conditions are satisfied for some $\delta=\delta_*$, then they are satisfied also for all $\delta\in (0,\delta_*)$.


\section{Contraction of $L$ and Finite Diameter}
\label{sec:L contract}

Proposition~\ref{prop:almost} proves that the parameters $c$ and $A$ of the cone $\cC_{c,A,L}(\delta)$
contract simply as a consequence of the uniform properties {\bf (H1)}-{\bf (H5)} for any sequence of maps
$(T_{\iota_j})_j \subset \cF(\tau_*, \cK_*, E_*)$.  In this section, however, we will restrict our sequence of
maps to be drawn from a sufficiently small neighborhood of a single map $T_0 \in \cF(\tau_*, \cK_*, E_*)$
in order to use the uniform mixing properties maps $T$ close to $T_0$ to prove that the parameter $L$ also 
contracts under the sequential dynamics.
This is done
in two steps.  First, in Section~\ref{sec:scale}, we use a length scale $\delta_0 \ge \sqrt{\delta}$ and compare averages
on the two length scales, $\delta$ and $\delta_0$,
culminating in Proposition~\ref{prop:alternative}.  This step does not yet require us to restrict our class of maps.
Second, in Section~\ref{sec:mix}, restricting our sequential system to a neighborhood of a fixed map $T_0$,
we obtain a bound on averages
in the length scale $\delta_0$ as expressed in Lemma~\ref{lem:match}.  This  leads to the strict contraction of $L$ established in 
Theorem~\ref{thm:cone contract},
which proves Theorem~\ref{thm:main}(a).  We prove Theorem~\ref{thm:main}(b)
in Section~\ref{sec:diam}, showing that the cone $\cC_{\chi c, \chi A, \chi L}(\delta)$  has finite diameter
in the cone $\cC_{c,A,L}(\delta)$  (Proposition~\ref{prop:diameter}).


\subsection{Comparing averages on different length scales}
\label{sec:scale}

Recall the length scale $\delta_0\in (0,1/2)$ from \eqref{eq:theta_1} and that $\delta < \delta_0/2$.  Also, recall that $\cW^s(\delta_0/2)$ denotes those curves in $\cW^s$ with length between
$\delta_0/2$ and $\delta_0$. We choose $\delta \le \delta_0^2$ and define
\[
\tri f \tri_+^0 = \sup_{\substack{W \in \cW^s(\delta_0/2) \\ \psi \in \cD_{a, \beta}(W) }}
\frac{\int_W f \, \psi \, dm_W}{\int_W \psi \, dm_W} ,
\qquad \qquad
\tri f \tri_-^0 = \inf_{\substack{W \in \cW^s(\delta_0/2) \\ \psi \in \cD_{a, \beta}(W) }}
\frac{\int_W f \, \psi \, dm_W}{\int_W \psi \, dm_W} .
\]
By subdividing curves of with length in $[\delta_0/2,\delta_0]$ into curves with length in $[\delta, 2\delta]$, we immediately deduce the relations,
\begin{equation}
\label{eq:diff scale}
\tri f \tri_- \le \tri f \tri_-^0 \le \tri f \tri_+^0 \le \tri f \tri_+ \, .
\end{equation}

\begin{lemma}
\label{lem:0 scale}
Assume $e^{a \delta_0^\beta} \le 2$, from \eqref{eq:delta_0 condition}, 
and $A\delta \le \delta_0/4$, from Lemma~\ref{lem:first L}.\\
For all $n \in\bN$, $\{ \iota_j \}_{j=1}^n \subset \cI(\tau_*, \cK_*, E_*)$ and $f\in \cC_{\chi c, \chi A, \chi L}(\delta)$ we have,\footnote{The second inequality in \eqref{eq:upper 0} follows from \eqref{eq:theta1 def}.}
\begin{eqnarray}
\tri \Lp_n f \tri_+^0 & \le&  \tri f \tri_+^0 +  3 C_0 \sum_{i=1}^n \theta_1^i \tri f \tri_+  
\le \tri f \tri_+^0 +  \frac 14 \tri f \tri_+ \, , \label{eq:upper 0}  \\
\tri \Lp_n f \tri_-^0 & \ge & \frac 34 \tri f \tri_-^0 \, . \label{eq:lower 0}
\end{eqnarray}
\end{lemma}

\begin{proof}
We prove \eqref{eq:upper 0} by induction on $n$.  It holds trivially for $n = 0$.  We assume the inequality
holds for $0 \le k \le n-1$ and prove the statement for $n$.

Let $W \in \cW^s(\delta_0/2)$.  Define $\hL_1(W)$ to be those elements of $\cG_1(W)$
having length at least $\delta_0/2$.  For $k > 1$, let $\hL_k(W)$ denote those curves of length
at least $\delta_0/2$ in $\cG_k(W)$ whose images are not already contained in an element 
of $\hL_i(W)$ for any $i = 1, \ldots, k-1$.  For $V_j \in \hL_k(W)$, let 
$P_k(j)$ be the collection of indices $i$ such that $W_i \in \cG_n(W)$ satisfies 
$T_{n-k}W_i \subset V_j$.  
Denote by $\cI^0_n(W)$ those indices $i$ for which $T_{n-k}W_i$ is never contained in an
element of $\cG_k(W)$ of length at least $\delta_0/2$, $1 \le k \le n$, and 
$\delta \le |W_i| < \delta_0/2$.  
Let $\cI_n(W)$ denote the remainder of the indices $i$ for curves in $\cG_n(W)$, i.e.
those curves $W_i$ of length shorter than $\delta$ and for which $T_{n-k}W_i$ is not contained in
an element of $\cG_k(W)$ of length at least $\delta_0/2$.
By construction, each $W_i \in \cG_n(W)$
belongs to precisely one $P_k(j)$ or $\cI^0_n(W)$ or $\cI_n(W)$.

Now, for $\psi \in \cD_{a, \beta}(W)$, recalling \eqref{eq:n k notation}, we have,
\[
\sum_{i \in P_k(j)} 
\int_{W_i} f \, \psi \circ T_n \, J_{W_i}T_n
= \int_{V_j} \Lp_{n-k} f \, \psi \circ T_{n, n-k} \, J_{V_j}T_{n, n-k }.
\]
Using this equality, we estimate,
\[
\begin{split}
\int_W \Lp_n f \, \psi & = \sum_{k=1}^n \sum_{V_j \in \hL_k(W)} \int_{V_j} \Lp_{n-k} f \, 
 \psi \circ T_{n,n-k} \, J_{V_j}T_{n,n-k}
\; \; + \; \; \sum_{i \in \cI^0_n(W)}  \int_{W_i} f \, \psi \circ T_n \, J_{W_i}T_n \\
& \qquad + \sum_{i \in \cI_n(W)}  \int_{W_i} f \, \psi \circ T_n \, J_{W_i}T_n \\
& \le \sum_{k=1}^n \sum_{V_j \in \hL_k(W)} \tri \Lp_{n-k} f \tri_+^0 \int_{V_j}  
\psi \circ T_{n,n-k} \, J_{V_j}T_{n,n-k}
\; \; + \; \; \sum_{i \in \cI^0_n(W)}  \tri f \tri_+ \int_{W_i} \psi \circ T_n \, J_{W_i}T_n \\
& \qquad + \sum_{i \in \cI_n(W)} A \delta^{1-q} |W_i|^q \tri f \tri_- |\psi|_{C^0(W)} 
|J_{W_i}T_n|_{C^0(W_i)} \\
& \le \sum_{k=1}^n \sum_{V_j \in \hL_k(W)} \Big(\tri f \tri_+^0 + 3 \sum_{i=1}^{n-k} C_0 \theta_1^i \tri f \tri_+ \Big) \int_{T_{n,n-k} V_j } \psi  \\
& \qquad +  \sum_{i \in \cI^0_n(W)} \tri f \tri_+ \frac{\delta_0}{2} |\psi|_{C^0(W)} |J_{W_i}T_n |_{C^0(W_i)}  + A \frac{\delta}{\delta_0} \delta_0 |\psi|_{C^0(W)} \tri f \tri_+ C_0 \theta_1^n \\
& \le \int_W \psi \; \Big(\tri f \tri_+^0 + 3 \sum_{i=1}^{n-1} C_0 \theta_1^i \tri f \tri_+ \Big) 
+ \Big(1+2A \frac{\delta}{\delta_0}\Big) e^{a\delta_0^\beta} \int_W \psi \; \tri f \tri_+ C_0 \theta_1^n ,
\end{split}
\] 
where for the second inequality we have used the inductive hypothesis, and for the 
second and third we have used Lemmas ~\ref{lem:full growth}-(a) and \ref{lem:avg}.
This proves the required inequality if $\delta_0$ is small enough that $e^{a\delta_0^\beta} \le 2$
and $\delta$ is small enough that $A \delta \le \delta_0 /4$, both of which we have assumed.

We prove \eqref{eq:lower 0} similarly, although now the inductive hypothesis is
$\tri \Lp_k f \tri_-^0 \ge (1 - 3\sum_{i=1}^{k} C_0 \theta_1^i)$ for each $k = 0, \ldots, n-1$.  
We begin with
the same decompostion of $\cG_n(W)$, although we simply drop the terms in $\cI_n^0(W)$
since they are all positive (see Remark \ref{rem:A-L}).
\[
\begin{split}
\int_W \Lp_n f \, \psi & = \sum_{k=1}^n \sum_{V_j \in \hL_k(W)} \int_{V^j} \Lp_{n-k} f \, 
\psi \circ T_{n, n-k} \, J_{V_j}T_{n,n-k}
\; \; + \; \; \sum_{i \in \cI^0_n(W)}  \int_{W_i} f \, \psi \circ T_n \, J_{W_i}T_n \\
& \qquad + \sum_{i \in \cI_n(W)}  \int_{W_i} f \, \psi \circ T_n \, J_{W_i}T_n \\
& \ge \sum_{k=1}^n \sum_{V_j \in \hL_k(W)} \tri \Lp_{n-k} f \tri_-^0 \int_{V^j}  
\psi \circ T_{n,n-k} \, J_{V_j}T_{n,n-k}\\
& \qquad
 - \sum_{i \in \cI_n(W)} A \delta^{1-q} |W_i|^q \tri f \tri_- |\psi|_{C^0(W)} 
|J_{W_i}T_n|_{C^0(W_i)} \\
& \ge \sum_{k=1}^n \sum_{V_j \in \hL_k(W)}  \int_{T_{n,n-k} V_j } \psi \;  \tri f \tri_-^0 \Big(1- 3 \sum_{i=1}^{n-k} C_0 \theta_0^i \Big) 
 - A \frac{\delta}{\delta_0} \delta_0 |\psi|_{C^0(W)} \tri f \tri_- C_0 \theta_1^n \\
& \ge \int_W \psi \; \tri f \tri_-^0 \Big(1 - 3 \sum_{i=1}^{n-1} C_0 \theta_1^i \Big) 
- 2A \frac{\delta}{\delta_0} e^{a\delta_0^\beta} \int_W \psi \; \tri f \tri_-^0 C_0 \theta_1^n \\
& \quad - \tri f \tri_-^0 \Big(1 - 3 \sum_{i=1}^{n-1} C_0 \theta_1^i  \Big) \sum_{i \in \cI_n(W) \cup \cI_n^0(W)} |W_i| |\psi|_{C^0(W)}
 |J_{W_i}T_n|_{C^0(W_i)} \\
& \ge \int_W \psi \; \tri f \tri_-^0 \Big( 1 - 3 \sum_{i=1}^{n-1} C_0 \theta_1^i 
- 2A \frac{\delta}{\delta_0} e^{a\delta_0^\beta}  C_0 \theta_1^n
- e^{a\delta_0^\beta} C_0 \theta_1^n \Big)  \, ,
\end{split}
\] 
where again we have used Lemmas~\ref{lem:full growth}(a) and \ref{lem:avg} as well as the bound $ \tri f \tri_- \le \tri f \tri_-^0$.
This proves the inductive claim, and from this, \eqref{eq:lower 0} follows from \eqref{eq:theta_1}.
\end{proof}
To continue it is useful to set 
\begin{equation}\label{eq:N-}
N(\delta)^- =  \frac{\log (8C_0(L \delta_0 \delta^{-1} + 2A))}{|\log \theta_1|}.
\end{equation}

Next, we have a partial converse of Lemma~\ref{lem:0 scale}.

\begin{lemma}
\label{lem:second scale}
For all $n \ge N(\delta)^- $ and $\{ \iota_j \}_{j=1}^n \subset \cI(\tau_*, \cK_*, E_*)$, 
we have
\[
\begin{split}
\tri \Lp_n f \tri_+ & \le \max_{k = 0, \ldots n-1} \tri \Lp_k f \tri_+^0 + \frac 18 \tri f \tri_- \\
\tri \Lp_n f \tri_- & \ge \frac 34 \min_{k = 0, \ldots n-1} \tri \Lp_k f \tri_-^0 - \frac 18 \tri f \tri_-
\end{split}
\]
\end{lemma}

\begin{proof}
The proof follows along the lines of the proof of Lemma~\ref{lem:0 scale}, using the same decomposition into $\hat{L}_k(W)$,
$\cI_n^0(W)$ and $\cI_n(W)$, except that now we begin with $W \in \cW^s(\delta)$ and $\psi \in \cD_{a, \beta}(W)$.
We have,
\[
\begin{split}
\int_W \Lp_n f \, \psi & 
\le \sum_{k=1}^n \sum_{V_j \in \hL_k(W)} \tri \Lp_{n-k} f \tri_+^0 \int_{V^j} 
\psi \circ T_{n,n-k} \, J_{V_j}T_{n,n-k}
\; \; + \; \; \sum_{i \in \cI^0_n(W)}  \tri f \tri_+ \int_{W_i} \psi \circ T_n \, J_{W_i}T_n \\
& \qquad + \sum_{i \in \cI_n(W)} A \delta^{1-q} |W_i|^q \tri f \tri_- |\psi|_{C^0(W)} 
|J_{W_i}T_n|_{C^0(W_i)} \\
& \le  \int_{W} \psi \max_{k=0, \ldots n-1} \tri \Lp_k f \tri_+^0  
+  \tri f \tri_+ C_0 \theta_1^n \frac{\delta_0}{ \delta} \int_W \psi  + 2 A  \tri f \tri_- C_0 \theta_1^n \int_W \psi \\
& \le \int_W \psi \; \Big( \max_{k=0, \ldots n-1} \tri \Lp_k f \tri_+^0 + \tri f \tri_- C_0 \theta_1^n (L \delta_0 \delta^{-1} + 2A) \Big) \, ,
\end{split}
\] 
which proves the first inequality, given our assumed bound on $n$.
Note that the ratio $\delta_0/\delta$ appears in the second term since
$|W_i| \le \delta_0/2$, while $|W| \ge \delta$. 

The second inequality follows similarly, again along the lines of Lemma~\ref{lem:0 scale}.
\[
\begin{split}
\int_W & \Lp_n f \, \psi 
\; \ge \;\sum_{k=1}^n \sum_{V_j \in \hL_k(W)} \tri \Lp_{n-k} f \tri_-^0 \int_{V^j} 
\psi \circ T_{n,n-k} \, J_{V_j}T_{n,n-k} \\
 & \qquad \qquad \;
 - \sum_{i \in \cI_n(W)} A \delta^{1-q} |W_i|^q \tri f \tri_- |\psi|_{C^0(W)} 
|J_{W_i}T_n|_{C^0(W_i)} \\
& \ge  \min_{k =0, \ldots n-1} \tri \Lp_k f \tri_-^0  \left( \int_W \psi - \sum_{i \in \cI_n(W) \cup \cI_n^0(W)} |W_i| |\psi|_{C^0(W)}  |J_{W_i}T_n|_{C^0(W_i)} \right) - 2 A   \int_W \psi \; \tri f \tri_- C_0 \theta_1^n \\
& \ge \int_W \psi \; \left( \min_{k = 0, \ldots n-1} \tri \Lp_k f \tri_-^0 ( 1 - \delta_0 \delta^{-1} C_0 \theta_1^n ) 
- 2A  C_0 \theta_1^n \tri f \tri_- \right) \, ,
\end{split}
\] 
and our bound on $n$ suffices to complete the proof of the lemma.
\end{proof}

\begin{prop}
\label{prop:alternative}
For all $n \ge N(\delta)^-$ and $\{ \iota_j \}_{j=1}^n \subset \cI(\tau_*, \cK_*, E_*)$,  either,
\[
\frac{\tri \Lp_n f \tri_+}{\tri \Lp_n f \tri_-} \le \frac{8}{9} \frac{\tri f \tri_+}{\tri f \tri_-} \, ,
\]
or 
\[
\tri \Lp_n f \tri_+ \le 8 \tri f \tri_+^0 \quad \mbox{and} \quad \tri \Lp_n f \tri_- \ge \frac{9}{20} \tri f \tri_-^0 \, .
\]
\end{prop}
\begin{proof}
Since $n \ge N(\delta)^- \ge n_0$, we may apply both Lemmas~\ref{lem:first L} and \ref{lem:second scale}.  Now, by Lemma \ref{lem:second scale},
\[
\tri \Lp_n f \tri_- \ge \frac 34 \min_{k = 0, \ldots n-1} \tri \Lp_k f \tri_-^0 - \frac 18 \tri f \tri_-
\ge \frac{9}{16} \tri f \tri_-^0 - \frac 14 \tri \Lp_n f \tri_-\, ,
\]
applying Lemma~\ref{lem:0 scale} to the first term and Lemma~\ref{lem:first L} to the second.  This yields immediately,
$\tri \Lp_n f \tri_- \ge \frac{9}{20} \tri f \tri_-^0$, which is the final inequality in the statement of the lemma.

Now consider the following alternatives.  If $\tri \Lp_n f \tri_+ \le \frac 25 \tri f \tri_+$, then 
\[
\frac{\tri \Lp_n f \tri_+}{\tri \Lp_n f \tri_-} \le \frac{\frac 25 \tri f \tri_+}{\frac{9}{20} \tri f \tri_-^0}
\le \frac{8}{9} \frac{\tri f \tri_+}{\tri f \tri_-} \,
\]
proving the first alternative.  On the other hand, if $\tri \Lp_n f \tri_+ \ge \frac 25 \tri f \tri_+$, then
using Lemmas~\ref{lem:second scale}, \ref{lem:0 scale} and \ref{lem:first L},
\[
\begin{split}
\tri \Lp_n f \tri_+ & \le \max_{k =0, \ldots n-1} \tri \Lp_k f \tri_+^0 + \frac 18 \tri f \tri_-
\le \tri f \tri_+^0 + \frac 14 \tri f \tri_+ + \frac 14 \tri \Lp_n f \tri_- \\
& \le \tri f \tri_+^0 + \frac 78 \tri \Lp_n f \tri_+ \, ,
\end{split}
\]
which yields the second alternative.
\end{proof}


\subsection{Mixing implies contraction of $L$}
\label{sec:mix}

The importance of Proposition~\ref{prop:alternative} is that either $L$ contracts within  $N(\delta)^-$ iterates
or we can compare ratios of integrals on the length scale $\delta_0$ (which is fixed independently of $\delta$).
In the latter case we will use the uniform mixing property of maps $T \in \cF(\tau_*, \cK_*, E_*)$ in order to compare the value
of $\int_W \cL_n f \psi$ for different $W$ of length approximately $\delta_0$.  
To this end, we will define a Cantor set $R_*$ comprised of local stable and unstable manifolds of a certain length in order to make our comparison when curves cross this set.

We begin by recalling the open neighborhoods in $\cF(\tau_*, \cK_*, E_*)$ defined by \eqref{eq:close d 1}.\\
Let $T \in \cF(\tau_*, \cK_*, E_*)$, and for $0 < \kappa < \frac 12 \min \{ \tau_*, \cK_* \}$, define
\begin{equation}
\label{eq:close d}
\cF(T, \kappa) = \{ \tT \in \cF(\tau_*, \cK_*, E_*) : \bd(Q(\tT), Q(T)) < \kappa \} \, .
\end{equation}
Recall the index set corresponding to $\cF(T, \kappa)$ is $\cI(T, \kappa) \subset \cI(\tau_*, \cK_*, E_*)$.
Thus $\iota \in \cI(T, \kappa)$ if and only if $T_\iota \in \cF(T, \kappa)$.  

\begin{lemma}
\label{lem:compact F}
For any $\kappa  \in \big(0, \frac 12 \{ \tau_*, \cK_* \} \big)$, 
the set $\cF(\tau_*, \cK_*, E_*)$ can be covered by finitely many sets 
$\cF(T, \kappa)$, $T \in \cF(\tau_*, \cK_*, E_*)$.
\end{lemma}
\begin{proof}
Each $T \in \cF(\tau_*, \cK_*, E_*)$ is associated with a billiard table $Q \in \cQ(\tau_*, \cK_*, E_*)$.
Such billiard tables have exactly $K$ boundary curves with $C^3$ norm uniformly bounded by $E_*$.
Since the torus is compact and the distance $ \bd (Q, \widetilde{Q})$ defined in 
Section~\ref{sec:bill family} measures distance only in the $C^2$ norm, the set
$\cQ(\tau_*, \cK_*, E_*)$ is compact in the distance $\bd$.  Thus for each $\kappa>0$, there exists
$N_\kappa \in \mathbb{N}$ and a set $\{ Q_{\iota_j} \}_{j =1}^{N_\kappa} \subset \cQ(\tau_*, \cK_*, E_*)$ such 
that\footnote{Recall from Section~\ref{sec:bill family} that by Proposition~\ref{prop:close maps}, 
$\cQ(Q_{\iota_j}, E_*; \kappa) = \{ Q \in \cQ(\frac 12 \tau_*, \frac 12 \cK_*, E_*) : \bd(Q, Q_{\iota_j}) < \kappa \}$.}
$\cup_j \cQ(Q_{\iota_j}, E_*; \kappa) \supset \cQ(\tau_*, \cK_*, E_*)$.
Since $\cF(Q_{\iota_j}, E_*; \kappa) \cap \cF(\tau_*, \cK_*, E_*) = \cF(T_{\iota_j}, \kappa)$, this yields the required covering. 
\end{proof}

\begin{remark}
The primary reason we restrict to $\tT \in \cF(T, \kappa)$ is to conclude Lemma~\ref{lem:proper cross}(b) for a fixed
time $n_*$ and rectangle $R_*$.  This will enable us to make a type of `matching' argument
for our sequential system, the main comparison being established in Lemma~\ref{lem:match}.

The reader familiar with the subject will notice that the matching described here requires weaker properties
than the usual arguments
used in coupling.  After stable curves are forced to cross a fixed rectangle by Lemma~\ref{lem:proper cross},
the `matched' pieces are not Cantor sets, but rather full curves.   The cone technique thus enables us to bypass
the use of real stable/unstable manifolds used in classical coupling arguments for billiards (see \cite[Section~7]{chernov book}),
and even the modified coupling developed for sequential systems which only couples for a finite time along
approximate invariant manifolds, as in
\cite{young zhang}, both of which require a more delicate use of the structure of invariant manifolds,
in particular control of the gaps in the Cantor sets used 
for coupling.
\end{remark}

For a fixed $T \in \cF(\tau_*, \cK_*, E_*)$,  we construct an approximate rectangle $D$ in $M$, contained in a single 
homogeneity strip, whose boundaries are comprised of two local
stable and two local unstable manifolds for $T$ as follows.  
Choose $\bar \delta_0>0$ and $x \in M$ such that dist$(T^{-n}x, \cS_1^{\bH}) \ge \bar \delta_0 \Lambda^{- |n|}$
for all $n \in \mathbb{Z}$.  This implies that the homogenous local  stable and unstable manifolds\footnote{ Although the stable/unstable directions in $M$ vary, they always belong to the global stable/unstable cones defined in {\bf (H1)}
and so are uniformly transverse.} of $x$, $W^s_{\bH}(x)$ and $W^u_{\bH}(x)$, have length at least  $\bar \delta_0$  
on either side of $x$.
By the Sinai Theorem applied to homogeneous unstable manifolds (see, for example, \cite[Theorem~5.70]{chernov book}), 
we may choose $\delta_0 < \bar \delta_0$ such that more than 0.99 of the measure of points 
in $W^u_{\bH}(x) \cap B_{2.1\delta_0}(x)$  have homogeneous local stable manifolds having length at least
$2.1\delta_0$  on both sides of $W^u_{\bH}(x)$, and analogously for the points in 
$W^s_{\bH}(x) \cap B_{2.1 \delta_0}(x)$.  Since these subsets of $W^{s/u}_{\bH}(x)$ are closed, there exist two 
extreme points on each manifold whose unstable/stable manifolds define a solid rectangle, which we will denote 
by $D'_{2\delta_0}$.  By choice of $\delta_0$, the stable and unstable manifolds comprising $\partial D'_{2\delta_0}$
have length at least $4 \delta_0$.
There must exist a rectangle $D_{2\delta_0}$ fully crossing $D'_{2\delta_0}$ in the stable direction 
and with boundary comprising two stable and two unstable manifolds, such that
the unstable diameter of $D_{2\delta_0}$ is between $\delta_0^4$ and $2\delta_0^4$,\footnote{ The choice of $\delta_0^4$ will be needed in Lemma \ref{lem:close if cross}.} and the
set of local homogeneous stable and unstable manifolds fully crossing $D_{2\delta_0}$ comprise at least
$9/10$ of the measure of $D_{2\delta_0}$ with respect to $\musrb$; otherwise, at most $9/10$ of the measure
of $W^u_{\bH}(x) \cap B_{2.1\delta_0}(x)$ would have long stable manifolds on either side of 
$W^u_{\bH}(x)$, contradicting our choice of $\delta_0$.
Similarly, define $D_{\delta_0} \subset D_{2\delta_0}$ to have precisely the same stable boundaries, but
stable diameter between $1.8 \delta_0$ and $2 \delta_0$ rather than $4 \delta_0$, still centered at $W^u_\bH(x)$. See figure \ref{fig:nonsenepuopiu} for a pictorial illustration of the above construction.

\begin{figure}
\begin{centering}
\begin{tikzpicture}[scale=0.6]
\node at (4,4) {$\bullet$};
\node at (3.7,3.5) {$x$};
\node at (10,3.7) {$W^s_{\bH}(x)$};
\node at (5,9.5) {$W^u_{\bH}(x)$};
\node at (-1,6.5) {$D_{2\delta_0}$};
\draw[->] (-0.3,6.4) -- (.6,6.1);
\node at (5,6.5) {$D_{\delta_0}$};
\node at (1,2) {$D'_{2\delta_0}$};
\draw[->] (9,6.4) -- (8.5,6);
\node at (10.5,6.8) {{\tiny stable curves}};
\node at (11.5,6.5)  {{\tiny properly  crossing} $\scriptscriptstyle R_*^{2\delta_0}$};
\draw (0,0)--(8,0);
\draw (0,0)--(0,8);
\draw(0,8)--(8,8);
\draw(8,8)--(8,0);
\draw[dashed] (0,5.5)--(8,5.5);
\draw[dashed] (0,6)--(8,6);
\fill[gray!40!white] (0,5.5) rectangle (8,6);
\fill[gray] (2,5.5) rectangle (6,6);
\draw[thick](-0.5,4)--(9,4);
\draw[thick](4,-1)--(4,9.5);
\draw[thin] (-.6,5.6)--(8.3,5.6);
\draw[thin] (-.5,5.65)--(8.6,5.65);
\draw[thin] (-.4,5.66)--(8.3,5.66);
\draw[thin] (-.7,5.665)--(8.9,5.665);
\draw[thin] (-.35,5.675)--(8.1,5.675);
\draw[thin] (-.8,5.680)--(8.5,5.68);
\draw[thin] (-.1,5.665)--(8.3,5.665);
\draw[thin] (-.3,5.765)--(8.3,5.765);
\draw[thin] (-.5,5.77)--(8.3,5.77);
\draw[thin] (-.4,5.75)--(8.1,5.75);
\draw[thin] (-.6,5.85)--(8.8,5.85);
\end{tikzpicture}
\caption{The boxes $D'_{2\delta_0}$ and $D_{2\delta_0}$. }\label{fig:nonsenepuopiu}
\end{centering}
\end{figure}
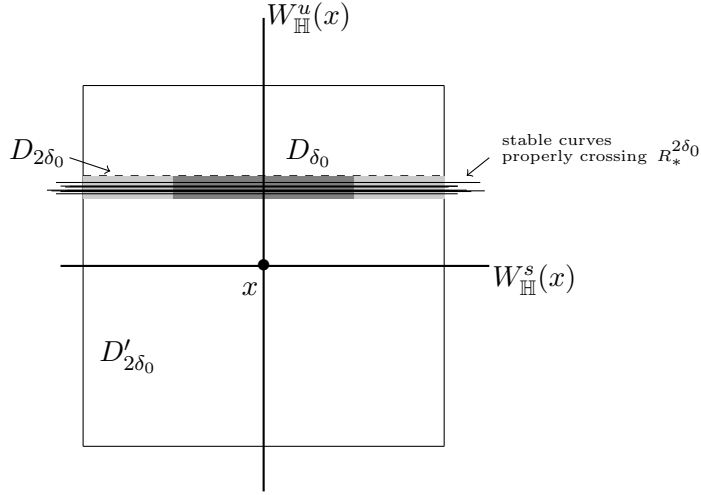
Let $\mathfrak{S}^{s/u}(D_{\delta_0})$ denote the maximal set of stable/unstable manifolds that fully cross $D_{\delta_0}$, including its boundary curves.  Define
$R_*^{\delta_0} = \mathfrak{S}^s(D_{\delta_0}) \cap \mathfrak{S}^u(D_{\delta_0})$ to be the Cantor rectangle
defined by the intersection of those maximal families.
Define $R_*^{2\delta_0}$ analogously with respect to $\mathfrak{S}^{s/u}(D_{2\delta_0})$.

By construction, $\musrb(R_*^{\delta_0}) > (0.9)^2 \musrb(D_{\delta_0}) \approx \delta_0^5$.
Below, we denote $D_{\delta_0}$ by $D(R_*^{\delta_0})$ since it is the minimal solid rectangle that defines $R_*^{\delta_0}$.  

We say that a stable curve $W \in \cW^s$ {\em properly crosses} a Cantor rectangle $R$ (in the stable direction) if $W$ intersects the interior of 
the solid rectangle $D(R)$, but does not terminate in
$D(R)$, and does not intersect the two stable manifolds contained in $\partial D(R)$.

\begin{lemma}
\label{lem:proper cross}
For $T \in \cF(\tau_*, \cK_*, E_*)$, let $R_*^{\delta_0} = R_*^{\delta_0}(T)$ be the Cantor rectangle constructed above.
\begin{itemize}
  \item[a)]  There exists $n_* \in \mathbb{N}$, depending only on $\delta_0$ and $\cF(\tau_*, \cK_*, E_*)$,  such that for all $T \in \cF(\tau_*, \cK_*, E_*)$ and all $W \in \cW^s$ with\footnote{Recall that
$\bar C_0$ is from Lemma~\ref{lem:full growth}.} $|W| \ge \delta_0/(6 \bar C_0)$, and all $n \ge n_*$,
$T^{-n}W$ contains a connected, homogeneous component that properly crosses $R_*^{\delta_0}(T)$.
  \item[b)]  There exists $\kappa>0$ such that for all $T \in \cF(\tau_*, \cK_*, E_*)$ and all $\{ \iota_j \}_{j=1}^{n_*} \subset \cI(T, \kappa)$,
  $T_{n_*}^{-1}W$ contains a connected, homogeneous component that properly crosses $R_*^{\delta_0}(T)$.  
\end{itemize}
\end{lemma}

\begin{proof}
First we fix $T \in \cF(\tau_*, \cK_*, E_*)$ and prove items (a) and (b) of the lemma for this $T$, i.e. we demonstrate
that such an $n_*$ and $\kappa$ exist depending on $T$.  Then we show how Lemma~\ref{lem:compact F}
implies that $n_*$ and $\kappa$ can be chosen uniformly for $T \in \cF(\tau_*, \cK_*, E_*)$.

\smallskip
\noindent
a)  Fix $T \in \cF(\tau_*, \cK_*, E_*)$.  
By \cite[Lemma~7.87]{chernov book}, there exist finitely many Cantor rectangles\footnote{ These Cantor
rectangles $R_i$ are maximal in the sense that they are the intersection of the maximal families 
of local invariant manifolds $\mathfrak{S}^{s/u}(D(R_i))$ that fully cross the solid rectangle $D(R_i)$.}
$\cR(\delta_0) =  \{ R_1, \ldots, R_k \}$, with
$\musrb(R_i) > 0$ for each $i$, such that any stable curve $W \in \cW^s$ with $|W| \ge \delta_0/(6\bar C_0)$ properly crosses at least one of them.
Let $\ve_{\cR}$ be the minimum length of an unstable manifold in $R_i$, for any $R_i \in \cR(\delta_0)$.

Consider the solid rectangle $\bar D(R_*^{2\delta_0}) \subset D(R_*^{2\delta_0})$ which crosses 
$D(R_*^{2\delta_0})$ fully in the stable direction, but
comprises the approximate middle $2/3$ of $D(R_*^{2\delta_0})$ in the unstable direction,  with approximately $1/3$ of the unstable diameter of $D(R_*^{2\delta_0})$ on each side of $\bar D(R_*^{2\delta_0})$. 
Similarly, let $\widetilde D(R_*^{2\delta_0}) \subset \bar D(R_*^{2\delta_0})$ denote the approximate
middle $1/3$ of $D(R_*^{2\delta_0})$ in the unstable direction.
Let $\bar R_*^{2\delta_0} := R_*^{2\delta_0} \cap \bar D(R_*^{2\delta_0})$ and 
let $\widetilde R_*^{2\delta_0} := R_*^{2\delta_0} \cap \widetilde D(R_*^{2\delta_0})$.
Note that $\musrb(\bar R_*^{2\delta_0}) > \musrb( \widetilde R_*^{2\delta_0}) > 0$ 
since $\musrb(R_*^{2\delta_0}) > (0.9)^2 \musrb(D(R_*^{2\delta_0}))$ by construction.

Now given $W \in \cW^s$ with $|W| \ge \delta_0/(6\bar C_0)$, let 
$R_i \in \cR(\delta_0)$ denote the Cantor rectangle which $W$ crosses properly.
By the mixing property of $T$, there exists $n_i^* > 0$ such that for all $n \ge n_i^*$, $T^n(\widetilde R_*^{2\delta_0} ) \cap R_i \neq \emptyset$.
We may increase $n_i^*$ if necessary so
that $C_1 \Lambda^{n_i^*} \delta_0^4/12 \ge \ve_{\cR}$.  We claim that $T^n(\bar R_*^{2\delta_0})$ properly crosses 
$R_i$ in the unstable direction for all $n \ge n_i^*$.
If not, then the unstable manifolds comprising $\bar R_*^{2\delta_0}$ must be cut by a singularity curve in $\cS_1^{\bH}$ before time $n_i^*$ (since otherwise they would be longer than  $2 \ve_{\cR}$ by choice of $n_i^*$), and the images of those unstable manifolds must terminate on
the unstable manifolds in $R_i$.  But this implies that some unstable manifolds in $R_i$ will be cut under $T^{-n}$, 
a contradiction.

Since $T^n(\bar R_*^{2\delta_0})$ properly crosses $R_i$ in the unstable direction, 
it follows that $T^n(D(\bar R_*^{2\delta_0}))$ contains a solid rectangle $D_*$ that fully crosses $D(R_i)$ in the unstable direction 
(here we use the fact that the stable manifolds of $\bar R_*^{2\delta_0}$ cannot be cut under $T^n$, as well as that
the singularity curves of $T^n$ can only terminate on other elements of $\cS_n^{\bH}$ \cite[Proposition~4.47]{chernov book}).
Define $V = W \cap D_*$ and note that $V$ fully crosses $D_*$ in the stable direction.
In particular, $V$ lies between two stable manifolds in $R_i$ and thus between two stable
manifolds in $T^n(\bar R_*^{2\delta_0})$.
 Thus  $T^{-n}V$ properly crosses $\bar R_*^{2\delta_0}$, and also $R_*^{2\delta_0}$, in the stable direction.  Since 
 $R_*^{\delta_0}$ has the
same stable boundaries as $R_*^{2\delta_0}$, but half the stable diameter, then $T^{-n}V$ also
properly crosses $R_*^{\delta_0}$, as required.
Since $\cR(\delta_0)$ is finite, setting $n_* = \max_{1 \le i \le k} \{ n_i^* \} < \infty$ completes the proof of part (a)
with $n_* = n_*(T)$ depending on $T$. 

\smallskip
\noindent
(b) In the proof of part (a), for $T \in \cF(\tau_*, \cK_*, E_*)$ we constructed a rectangle $\bar\cR_*^{2\delta_0}$ 
and a time $n_*$ so that for
any $W \in \cW^s$ and $n \ge n_*$, there exists $V \subset W$ such that $T^{-n}$ is smooth on $V$ and $T^{-n}V$ properly crosses $\bar\cR_*^{2\delta_0}$.   Now
for $\{ \iota_j \}_{j=1}^{n_*} \in \cI(T, \kappa)$, 
Proposition~\ref{prop:close maps}(b) guarantees
that $T_{n_*}^{-1}V$ is close to $T^{-n_*}V$ for $\kappa$ sufficiently small, except possibly when iterates land
in a neighborhood $N_{C\kappa^{1/2}}(\cS_{-1}^T \cup \cS_{-1}^{T_{\iota_j}})$.  
But in this case, Proposition~\ref{prop:close maps}(a) implies that for $T_\iota \in \cF(T, \kappa)$,
the singularity sets $\cS_{-1}^T$ and $\cS_{-1}^{T_\iota}$ either differ by at most $C\kappa^{1/2}$ or new components
are formed in a $C\kappa^{1/2}$ neighborhood of $\cS_0$. 
By construction, since $\bar\cR_*^{2\delta_0}$ has 2/3 the unstable diameter and
twice the stable diameter as $\cR_*^{\delta_0}$, then there exists $\kappa$, depending only on
$\delta_0$ and $n_*$, such that $T_{n_*}^{-1}V$ properly crosses $\cR_*^{\delta_0}$, as required.

\smallskip
Finally, we show how $n_*$ and $\kappa$ can be chosen uniformly in $\cF(\tau_*, \cK_*, E_*)$.  For
each $T \in \cF(\tau_*, \cK_*, E_*)$, parts (a) and (b) yield $n_*(T)$ and $\kappa(T)$ with the stated properties.
Then the set of open neighborhoods $\{ \cQ(Q(T), E_*; \kappa(T)/2) \}_{T \in \cF(\tau_*, \cK_*, E_*)}$ forms an open cover of $\cQ(\tau_*, \cK_*, E_*)$,
where $Q(T)$ is the billiard table associated with $T$.
By compactness (see the proof of Lemma~\ref{lem:compact F}) there exists  a finite subcover
$\{ \cQ(Q(T_{\iota_j}), E_*; \kappa(T_{\iota_j})/2 ) \}_{j=1}^{N_\ve}$.   
For any $T \in \cF(T_{\iota_j}, \kappa(T_{\iota_j})/2)$, we have
$\cF(T, \kappa(T_{\iota_j})/2) \subset \cF(T_{\iota_j}, \kappa(T_{\iota_j}))$.
Thus $n_*(T_{\iota_j})$ and $\frac 12 \kappa(T_{\iota_j})$ have the desired properties for this $T$. 
Setting $n_* = \max_j n_*(T_{\iota_j})$ proves part (a) and $\kappa = \frac 12 \min_j \kappa(T_{\iota_j})$
proves part (b) of the lemma.
\end{proof}

From this point forward, we fix $T_0 \in \cF(\tau_*, \cK_*, E_*)$ and let $R_* = R_*^{\delta_0}(T_0)$
as constructed above.  We will consider sequences $\{ \iota_j \}_j \subset \cI(T_0, \kappa)$, where
$\kappa$ is from Lemma~\ref{lem:proper cross}(b), i.e. we will draw from maps $T \in \cF(T_0, \kappa)$.

\begin{lemma}
\label{lem:close if cross}
Let $W^1, W^2 \in \cW^s$, $n \ge 0$ and $\{ \iota_j \}_{j=1}^n \subset \cI(T_0, \kappa)$.
Suppose $U_1 \in \cG_n(W^1)$ and
$U_2 \in \cG_n(W^2)$ properly cross $R_*$ and define $\bar U_i = U_i \cap D(R_*)$, $i=1,2$.
Then there exists $C_7 >0$, depending only on the maximum slope and maximum
curvature $\bar B$ of curves in $\cW^s$, such that
$d_{\cW^s}(\bar U_1, \bar U_2) \le C_7 \delta_0^2$.
\end{lemma}

\begin{proof}
Define a foliation of vertical line segments covering $D(R_*)$.  Due to the uniform transversality
of the stable cone with the vertical direction, it is clear that the length of the segments connecting
$\bar U_1$ and $\bar U_2$ have length at most $C_3 \delta_0^4$, where $C_3>0$ depends only on the
maximum slope in $C^s(x)$.  Moreover, the unmatched parts of $\bar U_1$ and $\bar U_2$ 
near the boundary of $D(R_*)$ also have length at most $C_3 \delta_0^4$. See Figure \ref{fig:chepalle2} for an illustration.
\begin{figure}
\begin{centering}
\begin{tikzpicture}[scale=1.2]
\node at (1.2,2.9) {\color{blue}$\bar U_1$};
\draw[->] (1.4,2.7)--(2.5, 1.9);
\draw[->] (7, 3.5)--(6,2.2);
\node at (7.3,3.8) {$D(R_*)$};
\node at (-1.8,1) {\color{red}$U_2$};
\draw[->] (-1.5,1)--(-.5, 0.6);
\node at (10.1,0.2) {\tiny an unmatched part of $\scriptscriptstyle {\color{red} \bar U_2}$};
\draw[->] (8.8,0.2)--(8.01,.95);
\draw[thick] (0,0)--(8.3,0);
\draw[thick] (0,0)--(0.2,2.2);
\draw[thick](0.2,2.2)--(7.9,2);
\draw[thick](7.9,2)--(8.3,0);
\draw[draw=red] (-1,.5)--(8.5,1);
\draw[draw=blue] (0.17,1.9)--(7.98,1.6);
\draw[dashed] (0.19,1.9)--(.19,0);
\draw[dashed] (7.97,1.6)--(7.97,0);
\end{tikzpicture}
\caption{ crossing $D(R_*)$. }\label{fig:chepalle2}
\end{centering}
\end{figure}
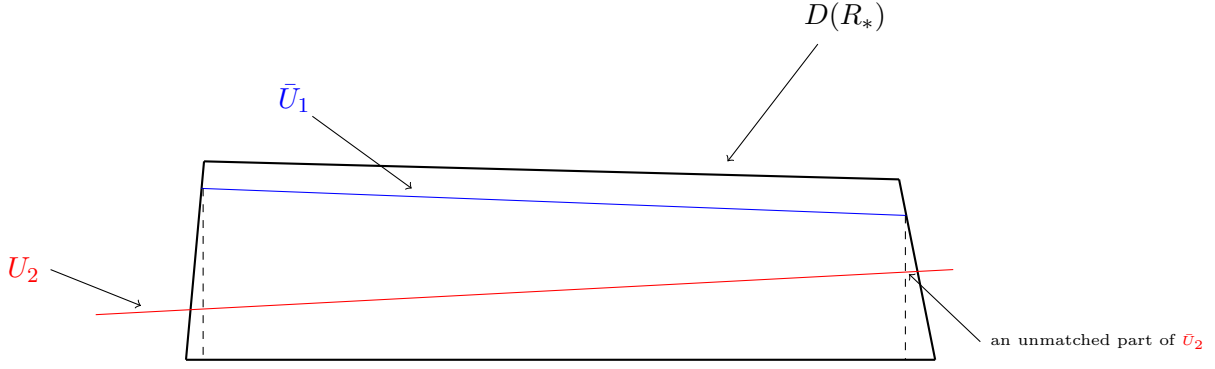

Recalling the definition of $d_{\cW^s}(\cdot, \cdot)$,
it remains to estimate the $C^1$ distance between the graphs of $\bar U_1$ and $\bar U_2$.
Denote by $\vf_1(r)$ and $\vf_2(r)$ the functions defining $\bar U_1$ and $\bar U_2$ on a common
interval $I = I_{\bar U_1} \cap I_{\bar U_2}$.  Let $\vf_i' = \frac{d \vf_i}{dr}$. 
For $x \in \bar U_1$ over $I$, let $\bar x  \in \bar U_2$
denote the point on the same vertical line segment as $x$.

Suppose there exists $x \in \bar U_1$ over $I$ such that $|\vf_1'(r(x)) - \vf_2'(r(\bar x))| > C \delta_0^2$
for some $C>0$, where $r(x)$ denotes the $r$-coordinate of $x = (r,\vf)$.  
Since the curvature of each $U_i$ is bounded by $\bar B$ by definition, we have
$|\vf_i''| \le \bar B (1+(K_{\max} + \tau_{\min}^{-1})^2)^{3/2} =: \bar C_7$.  

Now consider an interval $J \subset I$ of radius $\delta_0^2$ centered at $r(x)$.
Then $|\vf_1'(r) - \vf_1'(r(x))|  \le \bar C_7 |r-r(x)|$ for all $r \in J$, and similarly for $\vf_2'$.
Thus,
\[
|\vf_1'(r) - \vf_2'(r)| \ge C \delta_0^2 - 2 \bar C_7 \delta_0^2 = (C - 2\bar C_7) \delta_0^2 \; \mbox{ for all $r \in J$}.
\]
This in turn implies that there exists $r \in J$ such that $|\vf_1(r) - \vf_2(r) | \ge (C - 2\bar C_7) \delta_0^4$,
which is a contradiction if $C - 2\bar C_7 > C_3$.  This proves the lemma with
$C_7 = 2 \bar C_7 + C_3$.
\end{proof}

Recall that by Lemma~\ref{lem:H metric}, for $W\in \cW^s$ the cone $\cD_{a, \alpha}(W)$ has finite diameter in $\cD_{a, \beta}(W)$ for
$\alpha > \beta$, so that 
\begin{equation}
\label{eq:D_0}
\rho_{W, a, \beta}(g_1, g_2) \le D_0 \quad \mbox{ for all $g_1, g_2 \in \cD_{a, \alpha}(W)$}
\end{equation} 
for some constant $D_0 > 0$ depending only on $a, \alpha$ and $\beta$.  Without loss of generality,
we take $D_0 \ge  1$.

\begin{lemma}
\label{lem:match}
Suppose $W^1, W^2 \in \cW^s$ with $|W^1|, |W^2| \in [\delta_0/2, \delta_0]$ and $d_{\cW^s}(W^1, W^2) \le C_7\delta_0^2$.  Assume
$\psi_\ell \in \cD_{a,\alpha}(W^\ell)$ with $\int_{W^1} \psi_1 = \int_{W^2} \psi_2 =1$.  

Recall that $\delta \le \delta_0^2$ and let $\kappa>0$ be from Lemma~\ref{lem:proper cross}. Let $C>0$ be such that if $n \ge C \log(\delta_0/\delta)$ then $C_5  \Lambda^{-n} \le \delta/\delta_0^2$, where $C_5$ is from Lemma~\ref{lem:compare}.  
For all $n$ such that $n \ge C \log (\delta_0/\delta) \ge 2n_0$ and all $\{ \iota_j \}_{j=1}^n \subset \cI(T_0, \kappa)$,
we have
\[
\frac{\int_{W^1} \Lp_n f \, \psi_1}{\int_{W^2} \Lp_n f \, \psi_2} \le 2 
\]
for all $f\in \cC_{c, A, L}(\delta)$, provided
\[
\left[\frac{2\bar C_0C_3C_7(3L A \delta^{1-q}\delta_0^{2q}
+ 3 L\delta_0^2)}{1 - \Lambda^{-q}}  + 2\bar C_0 A \delta^{1-q} (2\delta^q + c\delta^{\gamma+q} + D_0 \delta^q + 2 \delta_0^q)  \right] 6e^{2a\delta_0^\alpha} \leq \delta_0.
\]
\end{lemma}

\begin{remark}
\label{rem:delta_0 ineq}
Since $\delta \le \delta_0^2$, the condition of Lemma~\ref{lem:match} will be satisfied if
\begin{equation}
\label{eq:delta_0 ineq}
\left[\frac{2\bar C_0C_3C_7 (3L A \delta_0+ 3L\delta_0)}{1 - \Lambda^{-q}}  +  2\bar C_0 A 
\delta_0^{1-q} (2 \delta_0^q + c\delta_0^{2\gamma +q} + D_0 \delta_0^q +2)  \right] 6e^{2a\delta_0^\alpha} \leq 1.
\end{equation}
This will determine our choice of $\delta_0$.
\end{remark}

\begin{proof}
We will change variables to integrate on $T_n^{-1}W^\ell$, $\ell = 1, 2$.  As in Section~\ref{subsec:contract c}, we split
$\cG_n(W^\ell)$ into matched pieces $\{ U^\ell_j \}_j$ and unmatched pieces $\{ V^\ell_j \}_j$.  Corresponding
matched pieces $U^1_j$ and $U^2_j$ are defined as graphs $G_{U^\ell_j}$
over the same $r$-interval $I_j$ and are connected by a foliation of vertical line segments.  Following
\eqref{eq:unstable split}, we write,
\[
\int_{W^{\ell} }\Lp_n f \, \psi_{\ell} = \sum_j \int_{U^{\ell}_j} f \, \hT_{n,U^{\ell}_j} \psi_{\ell} + \sum_j \int_{V^{\ell}_j} f \, \hT_{n,V^{\ell}_j} \psi_{\ell},
\]
where $\hT_{n,U^\ell_j} \psi_\ell := \psi_\ell \circ T_n \, J_{U^\ell_j}T_n$, and 
similarly for $\hT_{n,V^\ell_j} \psi_\ell$, 
$\ell = 1, 2$.

We perform the estimate over unmatched pieces first, following the same argument as in Section~\ref{subsec:contract c}
to conclude that $|T_{n-i-1}V^1_j| \le C_3 \Lambda^{-i} d_{\cW^s}(W^1, W^2) \le C_3 C_7 \Lambda^{-i} \delta_0^2$,
for any curve $V^1_j$ created at time $i$, $0 \le i \le n-1$. 

Recalling the sets $P(i)$ from Section~\ref{subsec:contract c} of unmatched pieces created at time $i$, we split the 
estimate into curves $P(i; S)$ if $|T_{n-i-1}V^1_j| < \delta$ and curves $P(i; L)$ if $|T_{n-i-1}V^1_j| \ge \delta$.  

The estimate over short unmatched pieces is given by (recalling the notation from \eqref{eq:n k notation}),
\begin{equation}
\label{eq:unmatched short}
\begin{split}
\sum_{i=0}^{n-1} \sum_{j \in P(i; S)} & \left| \int_{V^1_j} f \, \hT_{n,V^1_j} \psi_1 \right|
= \sum_{i=0}^{n-1} \sum_{j \in P(i; S)} \left| \int_{T_{n-i-1}V^1_j} \Lp_{n-i-1}f \cdot 
\psi_1 \circ T_{n, n-i-1} \, J_{T_{n-i-1}V^1_j}T_{n, n-i-1}
 \right| \\ 
& \le \sum_{i=0}^{n-1} \sum_{j \in P(i; S)} A \delta^{1-q} C_3^q  \Lambda^{-iq} d_{\cW^s}(W^1, W^2)^q \tri \Lp_{n-i-1} f \tri_- |\psi_1|_{C^0}
|J_{T_{n-i-1}V^1_j}T_{n, n-i-1 }|_{C^0} \\
& \le \frac{\bar C_0 A}{1 - \Lambda^{-q}} C_3^q C_7^q \delta_0^{2q}  3L \tri \Lp_n f \tri_-  \delta^{1-q} |\psi_1|_{C^0} \, ,
\end{split}
\end{equation}
where we have used Lemma~\ref{lem:full growth}-(b), $|W^1| \in [\delta_0/2, \delta_0]$, and 
Remark~\ref{rem:improve} to estimate the sum over the Jacobians, as well as \eqref{eq:tri} to estimate $\tri \Lp_{n-i-1} f \tri_- \le 3L \tri \Lp_n f \tri_-$.

For the estimate over long pieces, we subdivide them into curves of length between $\delta$ and $2\delta$  and estimate them by $\tri \Lp_{n-i-1} f \tri_+$, then we recombine them to obtain,
\begin{equation}
\label{eq:unmatched long}
\begin{split}
\sum_{i=0}^{n-1} \sum_{j \in P(i; L)} &\left| \int_{V^1_j} f \, \hT_{n,V^1_j} \psi_1 \right|
= \sum_{i=0}^{n-1} \sum_{j \in P(i; L)} \left| \int_{T_{n-i-1}V^1_j} \Lp_{n-i-1}f \cdot 
\psi_1 \circ T_{n, n-i-1} \, J_{T_{n-i-1}V^1_j}T_{n, n-i-1} 
\right| \\
& \le \sum_{i=0}^{n-1} \sum_{j \in P(i; L)} \tri \Lp_{n-i-1} f \tri_+ \int_{T_{n-i-1}V^1_j} 
\psi_1 \circ T_{n, n-i-1} \, J_{T_{n-i-1}V^1_j}T_{n, n-i-1} \\
& \le 3 L \tri \Lp_n f \tri_- \sum_{i=0}^{n-1} \sum_{j \in P(i; L)} |T_{n-i-1}V^1_j|  |\psi_1|_{C^0} |J_{T_{n-i-1}V^1_j}T_{n, n-i-1}|_{C^0} \\
& \le \frac{C_3 C_7\bar C_0}{1 - \Lambda^{-1}} \delta_0^2 3 L  \tri \Lp_n f \tri_-|\psi_1|_{C^0} \, ,
\end{split}
\end{equation}
where, in third line we used \eqref{eq:tri}, 
and in the fourth line, since $|W^1| \ge \delta_0/2$, we used Remark~\ref{rem:improve} to drop the second term in Lemma~\ref{lem:full growth}(b).

Next, we estimate the integrals over the matched pieces $U^1_j$.  We argue as in
Section~\ref{subsec:contract c}, but our estimates here are somewhat simpler since we do not need to show that parameters
contract.

We first treat the matched short pieces with $|U^1_j| < \delta$ much as we did the unmatched ones.  
By Lemma~\ref{lem:compare},
$d_{\cW^s}(U^1_j, U^2_j) \le C_5  \Lambda^{-n} d_{\cW^s}(W^1, W^2) \le \delta$,
since we have chosen $n \ge C \log (\delta_0/\delta)$.   Thus if
$|U^1_j| < \delta$ then $|U^2_j| < 2\delta$,  and the analogous fact holds for short
curves $|U^2_j| < \delta$.
With this perspective, we call $U^\ell_j$ short if either $|U^1_j| < \delta$ or $|U^2_j| < \delta$.  
On short pieces, we apply \eqref{eq:cone 3}
\begin{equation}
\label{eq:match short}
\sum_{j \; \mbox{\scriptsize short}} \left| \int_{U^1_j} f \, \hT_{n,U^1_j} \psi_1 \right|
 \le \sum_{j \; \mbox{\scriptsize short}} 2 A \delta \tri f \tri_- |\psi_1|_{C^0} |J_{U^1_j} T_n|_{C^0} 
 \le 4 A \delta \tri \Lp_n f \tri_- \bar C_0 |\psi_1|_{C^0} \, ,
\end{equation}
where we have again used Lemmas~\ref{lem:full growth}(b) and  \ref{lem:first L}  for the second inequality. Remark that the same argument holds for $U^2_j$ with test function $\psi_2$.

Finally, to estimate the integrals over matched curves with $|U^1_j|, |U^2_j| \ge \delta$ we 
follow equation \eqref{eq:c-decomposition}, recalling \eqref{eq:switch}, although we no longer have Lemma~\ref{lem:compare}(c) at our disposal,
\begin{equation}
\label{eq:match decomp}
\begin{split}
&\left|\int_{U^1_j} f \, \hT_{n,U^1_j} \psi_1 -\int_{U^2_j} f \,  \hT_{n,U^2_j} \psi_2\right| 
\leq \left|\int_{U^1_j} f \, \hT_{n,U^1_j} \psi_1 -\int_{U^1_j} f \, \tT_{n,U^2_j} \psi_2\right|  \\
&+\left|\frac{\int_{U^1_j} f \, \tT_{n,U^2_j} \psi_2}{\fint_{U^1_j}  \tT_{n,U^2_j} \psi_2} -\frac{\int_{U^2_j} f \,  \hT_{n,U^2_j} \psi_2}{\fint_{U^2_j}  \hT_{n,U^2_j} \psi_2}\right|\fint_{U^2_j}  \hT_{n,U^2_j} \psi_2+\left|\frac{\int_{U^1_j} f \, \tT_{n,U^2_j} \psi_2}{\fint_{U^1_j} \tT_{n,U^2_j} \psi_2}\right|\,\left| \frac{|U^2_j|-|U^1_j|}{|U^1_j|}\right|\fint_{U^2_j}  \hT_{n,U^2_j} \psi_2\\
&\leq  \left|\int_{U^1_j} f \, \hT_{n,U^1_j} \psi_1 -\int_{U^1_j} f \, \tT_{n,U^2_j} \psi_2\right| 
+d_{\cW^s}(U^1_j, U^2_j)^\gamma \delta^{1-\gamma}   c A  \tri f \tri_-  |J_{U^2_j}T_n|_{C^0} |\psi_2|_{C^0} \\
&\quad +A\delta d_{\cW^s}(U^1_j, U^2_j)   \tri f \tri_- |J_{U^2_j}T_n|_{C^0} |\psi_2|_{C^0} \, ,
\end{split}
\end{equation}
where we have used \eqref{eq:length-ratio} to estimate $\left| \frac{|U^2_j|-|U^1_j|}{|U^1_j|}\right|$.   

To estimate the first term on the right side above, we use \eqref{eq:cone 3} and Lemma~\ref{lem:b property},
\[
\begin{split}
\left|\int_{U^1_j} f \, \hT_{n,U^1_j} \psi_1 -\int_{U^1_j} f \, \tT_{n,U^2_j} \psi_2\right| 
& \le \left|\frac{\int_{U^1_j} f \, \hT_{n,U^1_j} \psi_1}{\fint_{U^1_j} \hT_{n,U^1_j} \psi_1} - \frac{\int_{U^1_j} f \, \tT_{n,U^2_j} \psi_2}{\fint_{U^1_j} \tT_{n,U^2_j} \psi_2} \right|  \fint_{U^1_j} \hT_{n,U^1_j} \psi_1 \\
& \qquad + \frac{\int_{U^1_j} f \, \tT_{n,U^2_j} \psi_2}{\fint_{U^1_j} \tT_{n,U^2_j} \psi_2}
\left| \fint_{U^1_j} \hT_{n,U^1_j} \psi_1 - \fint_{U^1_j} \tT_{n,U^2_j} \psi_2 \right| \\
& \le 2 \delta L \rho(\hT_{n,U^1_j} \psi_1, \tT_{n,U^2_j} \psi_2) \tri f \tri_- |J_{U^1_j}T_n|_{C^0} |\psi_1|_{C_0} \\
& \qquad + A  \delta^{1-q} \delta_0^q  \tri f \tri_- \big( |J_{U^1_j}T_n|_{C^0} |\psi_1|_{C_0} + 2|J_{U^2_j}T_n|_{C^0} |\psi_2|_{C_0} \big) \, ,
\end{split}
\]
where we have used $|U^1_j| \le \delta_0$ in the last line. 
We may apply \eqref{eq:D_0} since $\hT_{n,U^1_j} \psi_1, \tT_{n,U^2_j} \psi_2 \in \cD_{a, \alpha}(U^1_j)$ by 
Lemma~\ref{lem:compare}.  Now putting the above estimate together with \eqref{eq:match decomp},
recalling $d_{\cW^s}(U^1_j, U^2_j) \le  \delta$,
and using Lemma~\ref{lem:full growth}-(b) and Remark~\ref{rem:improve} as well as  Lemma~\ref{lem:first L}, we sum over $j$ to obtain,
\begin{equation}
\label{eq:match long}
\begin{split}
\sum_{j \; \mbox{\scriptsize long}} & \left|\int_{U^1_j} f \, \hT_{n,U^1_j} \psi_1 -\int_{U^2_j} f \,  \hT_{n,U^2_j} \psi_2\right| \\
& \qquad \le 2 A \delta^{1-q}   \tri \Lp_n f \tri_- \bar C_0 \left( c \delta^{\gamma + q} + \delta^{1+q}  + \frac{2LD_0  \delta^q}{A} 
+2 \delta_0^q  \right) (|\psi_1|_{C^0} + |\psi_2|_{C^0}) .
\end{split}
\end{equation}

Collecting \eqref{eq:unmatched short}, \eqref{eq:unmatched long}, \eqref{eq:match short}  and \eqref{eq:match long}, and recalling $D_0 \ge  1$ and $A > 4 L$, yields
\[
\begin{split}
\int_{W^1} \Lp_n f \, \psi_1 & \leq \;\frac{\bar C_0C_3C_7 (3L A \delta^{1-q}\delta_0^{2q}+3L\delta_0^2)}{1 - \Lambda^{-q}} \tri \Lp_n f \tri_-   |\psi_1|_{C^0} 
+ 4 \bar C_0 A \delta \tri \Lp_n f \tri_-  |\psi_1|_{C^0} \\
& \qquad
 + \sum_j \int_{U^2_j} f \,  \hT_{U^2_j}\psi_2 + 2 A \delta^{1-q} \tri \Lp_n f \tri_- \bar C_0 ( c \delta^{\gamma+ q}  +  D_0 \delta^q + 2 \delta_0^q  ) (|\psi_1|_{C^0} + |\psi_2|_{C^0})  \\
& \leq \; \bigg\{1+\Big[\frac{2 \bar C_0C_3C_7 (3L A \delta^{1-q}\delta_0^{2q}+3L\delta_0^2) }{1 - \Lambda^{-q}} \\
&  \qquad + 2 \bar C_0 A  \delta^{1-q}  (2 \delta^q + c\delta^{\gamma+q} + D_0 \delta^q +2 \delta_0^q )  \Big]\frac{|\psi_1|_{C^0} + |\psi_2|_{C^0}}{\int_{W^2}\psi_2}\bigg\} \int_{W^2} \Lp_n f \, \psi_2 \, .
\end{split}
\]
Now since $\int_{W^i} \psi_i =  1$, we have $e^{- a \delta_0^\alpha} \le |W^i| \psi_i \le e^{a \delta_0^\alpha}$. Thus since $|W^i| \ge \delta_0/3$,
\[
\frac{|\psi_1|_{C^0} + |\psi_2|_{C^0}}{\int_{W^2}\psi_2}
\le \frac{6}{\delta_0} e^{2a \delta_0^\alpha} \, , 
\] 
which proves the lemma.
\end{proof}

Our strategy will be the following.  For $W^1$, $W^2 \in \cW^s(\delta_0/2)$, $n$ sufficiently large and 
$\{ \iota_j \}_{j=1}^n \subset \cI(T_0, \kappa)$,  we wish to compare
$\int_{W^1} \Lp_n f \, \psi_1$ with $\int_{W^2} \Lp_n f \, \psi_2$, where we normalize
$\int_{W^1} \psi_1 = \int_{W^2} \psi_2 = 1$.  By Lemmas~\ref{lem:proper cross} and \ref{lem:close if cross}, 
we find $U^\ell_i \in \cG_{n_*}(W^\ell)$, $\ell = 1, 2$, such that $U^\ell_i$ properly crosses $R_*$, and 
$d_{\cW^s}(\bar U^1_i, \bar U^2_i) \le C_7 \delta_0^2$, where $\bar U^\ell_i = U^\ell_i \cap D(R_*)$.

Next, for each $i$, we wish to compare $\int_{\bar U^1_i} \Lp_{n-n_*} f \, \hT_{n_*, U^1_i} \psi_1$
with $\int_{\bar U^2_i} \Lp_{n-n_*} f \, \hT_{n_*, U^2_i} \psi_2 $,
where, to abbreviate notation, 
$\hT_{n_*, U^\ell_i} \psi_\ell = \psi_\ell \circ T_{n, n-n_*} J_{U^\ell_i}T_{n, n-n_*}$.
However, the weights $\int_{\bar U^\ell_i} \hT_{n_*, U^\ell_i} \psi_\ell$ may be very different for $\ell = 1,2$ since 
the stable Jacobians along the respective orbits before time $n_*$ may not be comparable.  To remedy this,
we adopt the following strategy for matching integrals on curves.  

For each curve $U^\ell_i \in \cG_{n_*}(W)$ which properly crosses $R_*$, we redefine $\bar U^\ell_i$
to denote the middle third of $U^\ell_i \cap D(R_*)$ (and so having length at least $2\delta_0/3$). Let $M^\ell$ denote the index set of such $i$.

Let $p^{(\ell)}_i = \int_{\bar U^\ell_i} \hT_{n_*, U^\ell_i} \psi_\ell$, and let 
$m_\ell = \sum_{i \in M^\ell} p^{(\ell)}_i$.  Without loss of generality, assume $m_2 \ge m_1$.

We will match the integrals $\sum_{i \in M^1} \int_{\bar U^1_i} \Lp_{n-n_*}f \, \hT_{n_*, U^1_i} \psi_1$ with
 $\sum_{j \in M^2} \frac{m_1}{m_2} \int_{\bar U^2_j} \Lp_{n-n_*}f \, \hT_{n_*, U^2_j} \psi_2$.
The remainder of the integrals
$\sum_{j \in M^2} \frac{m_2 - m_1}{m_2} \int_{\bar U^2_j} \Lp_{n-n_*}f \, \hT_{n_*, U^2_j} \psi_2$ as well as any unmatched
pieces (including the outer two-thirds of each $U^\ell_i$) we continue to iterate until such time as they can be matched
as the middle third of a curve that properly crosses $R_*$.

Set $\hT_{n_*, U^2_j} \tpsi_2 = \frac{m_1}{m_2} \hT_{n_*, U^2_j} \psi_2$, and 
consider the following decomposition of the integrals we want to match,
\[
\sum_{\substack{i \in M^1 \\ j \in M^2}} \int_{\bar U^1_i} \Lp_{n-n_*}f \, \hT_{n_*, U^1_i} \psi_1 \, \frac{p_j^{(2)}}{m_2}
\quad \mbox{and} \quad
\sum_{\substack{i \in M^1 \\ j \in M^2}} \int_{\bar U^2_j} \Lp_{n-n_*}f \, \hT_{n_*, U^2_j} \tpsi_2 \, \frac{p_i^{(1)}}{m_1}
\]
For each pair $i,j$ in the first sum, the test function has integral weight $\frac{p^{(1)}_i p^{(2)}_j}{m_2}$, and the same is true for the
corresponding pair in the second sum.  Thus these integrals are paired precisely according to the assumptions of 
Lemma~\ref{lem:match}.  It follows that if $n - n_* \ge C\log(\delta_0/\delta)$, then 
\begin{equation}
\label{eq:match compare}
\begin{split}
\sum_{i \in M^1} \int_{\bar U^1_i} \Lp_{n-n_*}f \, & \hT_{n_*, U^1_i} \psi_1
 = \sum_{\substack{i \in M^1 \\ j \in M^2}} \int_{\bar U^1_i} \Lp_{n-n_*}f \, \hT_{n_*, U^1_i}\psi_1 \frac{p_j^{(2)}}{m_2} \\
& \le 2 \sum_{\substack{i \in M^1 \\ j \in M^2}} \int_{\bar U^2_j} \Lp_{n-n_*}f \, \hT_{n_*, U^2_j} \tpsi_2 \, \frac{p_i^{(1)}}{m_1} 
=  2  \sum_{j \in M^2} \int_{\bar U^2_j} \Lp_{n-n_*}f \, \hT_{n_*,U^2_j} \tpsi_2  \, .
\end{split}
\end{equation}

We want to repeat the above construction until most of the mass has been compared. To this end we 
set up an inductive scheme.
Consider the family of curves $W^\ell_i \in \cG_{n_*}(W^\ell)$ that have not been matched.  Each 
carries a test function $\psi_{\ell, i} := \hT_{n_*, W^\ell_i} \tpsi_\ell$,
where to keep our notation uniform, we set $\tilde \psi_1 = \psi_1$. 
  Renormalizing by a factor $\mathfrak{r}_{\ell,1} < 1$,
we have $\sum_i \int_{W^\ell_i} \psi_{\ell,i} = 1$.

\begin{defin}
\label{def:admit}
Given a countable collection of curves and test functions, $\mathfrak{F} = \{ W_i, \psi_i \}_i$,  with $W_i \in \cW^s$, $|W_i| \le \delta_0$, 
$\psi_i \in \cD_{a, \alpha}(W_i)$ and $\sum_i \int_{W_i} \psi_i = 1$, we call
$\mathfrak{F}$ an {\em admissible family} if 
\begin{equation}
\label{eq:C* def}
\sum_i  \fint_{W_i} \psi_i \le C_* \, , \quad \mbox{where $C_* := 3\bar C_0 \delta_0^{-1}$.} 
\end{equation}
\end{defin}

Notice that any stable curve $W \in \cW^s(\delta_0/2)$ together with test function $\psi \in \cD_{a,\alpha}(W)$ normalized so that
$\int_W \psi = 1$ forms an admissible family since $|W| \ge \delta_0/2$.
The content of the following lemma is that an admissible family can be iterated and remain admissible; moreover,
a family with larger average integral in \eqref{eq:C* def} can be made admissible under iteration.

\begin{lemma}
\label{lem:propogate}
Let $\{ W_i, \psi_i \}_i$ be a countable collection of curves $W_i \in \cW^s$, $|W_i| \le \delta_0$, with functions 
$\psi_i \in \cD_{a,\alpha}(W_i)$,
normalized so that $\sum_i p_i = 1$, where $p_i = \int_{W_i} \psi_i$.  
Suppose that $\sum_i |W_i|^{-1} p_i = C_\sharp$.

Choose $n_\sharp \in \mathbb{N}$ so that $C_0 \theta_1^{n_\sharp} \frac{C_\sharp}{C_*} \le 1/6$.  
Then for all $n \ge n_\sharp$, and all $\{ \iota_k \}_{k=1}^n \subset \cI(T_0, \kappa)$, 
the dynamically iterated family $\{ V^i_j \in \cG_n(W_i), \hT_{n,V^i_j} \psi_i \}_{i,j}$ 
is admissible.
\end{lemma}

\begin{proof}
Setting $p^{(i)}_j = \int_{V^i_j} \hT_{n,V^i_j} \psi_i = \int_{V^i_j} \psi_i \circ T_n J_{V^i_j}T_n$, it is immediate that
$\sum_{i,j} p^{(i)}_j = 1$.

Now fix $W_i$ and consider $V^i_j \in \cG_n(W_i)$.  Then using Lemmas~\ref{lem:full growth} and \ref{lem:avg} we estimate,
\[
\begin{split}
\sum_j |V^i_j|^{-1} p^{(i)}_j & = \sum_j \fint_{V^i_j} \psi_i \circ T_n \, J_{V^i_j}T_n \le \sum_j |\psi_i|_{C^0(W_i)} |J_{V^i_j}T_n|_{C^0(V^i_j)} \\
& \le |\psi_i|_{C^0} ( \bar C_0 \delta_0^{-1} |W_i| + C_0 \theta_1^n )
\le \bar C_0 \delta_0^{-1} e^{a \delta_0^\alpha} p_i + C_0 \theta_1^n e^{a \delta_0^\alpha} |W_i|^{-1} p_i  \, .
\end{split}
\]
Using that $e^{a \delta_0^\alpha} \le 2$, we sum over $i$ and use the assumption on the family $\{ W_i, \psi_i \}_i$ to obtain,
\begin{equation}
\label{eq:C* contract}
\sum_{i,j} \sum_j |V^i_j|^{-1} p^{(i)}_j
\le \sum_i \big( 2 \bar C_0 \delta_0^{-1}  p_i + 2 C_0 \theta_1^n  |W_i|^{-1} p_i \big)
\le 2 \bar C_0 \delta_0^{-1} + 2 C_0 \theta_1^n C_\sharp \, .
\end{equation}
Thus if $n \ge n_\sharp$, the above expression is bounded by $C_*$, as required.
\end{proof}

\begin{theorem}
\label{thm:cone contract}
Let $L \ge 60$.  Suppose $a, c, A$ and $L$ satisfy the conditions of Section~\ref{sec:conditions}, and that in addition, $\delta \le \delta_0^2$ satisfy \eqref{eq:delta_0 ineq} and \eqref{eq:A-cond-s4}.
Then there exists $\chi < 1$, independent of the cone parameters,
\footnote{ Indeed, using Proposition~\ref{prop:almost} and choosing $L \ge 60$, we can always choose
$\chi = \frac 89$, although this will affect the choice of $N_\cF$.}
 and $k_*\in\bN$ such that if $n \in \mathbb{N}$ satisfies $n \ge  N_\cF :=  N(\delta)^- + k_*n_*$,\footnote{ Recall that $n_*$ is defined in Lemma \ref{lem:proper cross} while $N(\delta)^-$ is defined in equation \eqref{eq:N-}.} with $k_*$ depending only on $\delta_0, L, n_*$ (see equation \eqref{eq:k*choice})
and $\{ \iota_j \}_{j=1}^n \subset \cI(T_0, \kappa)$, where $\kappa >0$ is from Lemma~\ref{lem:proper cross}(b),
then $\Lp_n \cC_{ c, A, L}(\delta) \subset \cC_{  \chi c, \chi A, \chi L}(\delta) $.
\end{theorem}

\begin{proof}
As before, we take $f \in \cC_{c,A,L}(\delta)$,
$W^1$, $W^2 \in \cW^s(\delta_0/2)$ and test functions $\psi_\ell \in \cD_{a,\beta}(W^\ell)$ such that
$\int_{W^1} \psi_1 = \int_{W^2} \psi_2 = 1$. In order to iterate the matching argument described above,
we need upper and lower bounds on the amount of mass matched via the 
process described by \eqref{eq:match compare}.
\medskip

{\em Upper Bound on Matching.} 
By definition of $\bar U^\ell_i$, for each curve $U^\ell_i$ that properly crosses $R_*$ at time $n_*$, 
at least 2/3 of the length of that curve remains not matched.  Thus if $p_i = \int_{U^\ell_i} \hT_{n,U^\ell_i} \tpsi_\ell$, then at least
$(1- e^{a \delta_0^\alpha}/3) p_i$ remains unmatched.  Using $e^{a \delta_0^\alpha} \le 2$, we conclude that at least
$(1/3) p_i$ of the mass remains unmatched.
Thus if  $\mathfrak{r}$ denotes the total mass remaining after matching at time $n_*$, we have $\mathfrak{r} \ge 1/3$.  
Renormalizing the family by $\mathfrak{r}$,
we have $\sum_i |W_i|^{-1} \frac{p_i}{\mathfrak{r}} \le 3 C_*$. 

By the proof of Lemma~\ref{lem:propogate} with $C_\sharp = 3C_*$, we see that choosing 
$n_\sharp$ such that
$6C_0 \theta_1^{n_\sharp} \le 1/3$, then the bound in \eqref{eq:C* contract} is less than $C_*$, and the family recovers its regularity
in the sense of Lemma~\ref{lem:propogate} after this number of iterates.

\medskip
{\em Lower Bound on Matching.}  By definition of admissible family, for each $\ve > 0$,
$\sum_{|W_i| < \ve} p_i \le C_* \ve$.
So choosing $\ve = \delta_0 / (6\bar C_0)$, we have that
\[
\sum_{|W_i| \ge \delta_0/(6 \bar C_0)} p_i \ge \frac 12 \, .
\]
On each $W_i$ with $|W_i| \ge \delta_0/(6 \bar C_0)$, we have at least one $U^i_j \in \cG_{n_*}(W_i)$ that properly
crosses $R_*$ by Lemma~\ref{lem:proper cross}.  Then denoting by $\bar U^i_j$ the matched part (middle third)
of $U^i_j$ and setting 
\[
 \ve_{n_*}= \frac{C^{n_*} \delta_0^{(5/3)^{n_*}}}{12\,\delta_0}
 \]
 we have
\[
\begin{split}
\int_{\bar U^i_j} \hT_{n_*,U^i_j} \tpsi_i & = \int_{\bar U^i_j} \tpsi_i \circ T_{n, n-n_*} \, J_{U^i_j}T_{n, n-n_*}
\ge \tfrac{\delta_0}{3} \inf \tpsi_i \inf {J_{U^i_j}T_{n, n-n_*} } \\
& \ge \tfrac{1}{3} e^{-a \delta_0^\alpha} p_i e^{-C_d \delta_0^{1/3}} \frac{|T_{n, n-n*}U^i_j|}{|U^i_j|}
\ge  \ve_{n_*} p_i \, ,
\end{split}
\]
where we have used the fact that if $W \in \cW^s$ and $T_{\iota_j}^{-1}W$ is a homogeneous stable curve, then
$|T_{\iota_j}^{-1}W| \le C^{-1} |W|^{3/5}$ for some constant $C>0$ by {\bf (H1)} (see, for example  \cite[eq. (6.9)]{demzhang14}).

Thus a lower bound on the amount of mass coupled at time $n_*$ is $\frac{\ve_{n_*}}{2} > 0$.

\medskip
We are finally ready to put these elements together.  For $k_* \in \mathbb{N}$ and
$k = 1, \ldots k_*$, let $M^\ell(k)$ denote the index set of curves in
$\cG_{kn_*}(W^\ell)$ which are matched at time $kn_*$. By choosing $\delta_0$ small, we can ensure
that $n_\sharp \le n_*$, where $n_\sharp$ from Lemma~\ref{lem:propogate} corresponds to $C_\sharp = 3 C_*$.
Thus the family of remaining curves is always admissible at time $kn_*$.  Let $M^\ell(\sim)$ denote the index set of curves
that are not matched by time $k_*n_*$.  We estimate using \eqref{eq:match compare} at each time $n=k n_*$, 

\begin{equation}
\label{eq:gen match}
\begin{split}
\int_{W^1} \Lp_nf \, \psi_1 & = \sum_{k=1}^{k_*} \sum_{i \in M^1(k)} \int_{\bar U^1_i} \Lp_{n-kn_*} f \, \hT_{kn_*, U^1_i} \tpsi_1
+ \sum_{i \in M^1(\sim)} \int_{V^1_i} \Lp_{n-k_*n_*} f \, \hT_{k_*n_*, V^1_i} \tpsi_1 \\
& \le \sum_{k=1}^{k_*} \sum_{i \in M^2(k)} 2 \int_{\bar U^2_i} \Lp_{n-kn_*} f \, \hT_{kn_*, U^2_i} \tpsi_2 
+ \sum_{i \in M^1(\sim)} \int_{V^1_i} \Lp_{n-k_*n_*} f \, \hT_{k_*n_*, V^1_i} \tpsi_1
\end{split}
\end{equation}
We estimate the sum over unmatched pieices $M^\ell(\sim)$ by splitting the estimate in curves longer than $\delta$, $M^\ell(\sim; Lo)$,
and curves shorter than $\delta$, $M^\ell(\sim; Sh)$.
\[
\begin{split}
\sum_{i \in M^\ell(\sim)} & \int_{V^\ell_i} \Lp_{n-k_*n_*} f \, \hT_{k_*n_*, V^\ell_i} \tpsi_\ell 
= \sum_{i \in M^\ell(\sim; Lo)} \int_{V^\ell_i} \Lp_{n-k_*n_*} f \, \hT_{k_*n_*, V^\ell_i} \tpsi_\ell
+ \sum_{i \in M^\ell(\sim; Sh)} \int_{V^\ell_i} \Lp_{n-k_*n_*} f \, \hT_{k_*n_*, V^\ell_i} \tpsi_\ell \\
& \le  \sum_{i \in M^\ell(\sim; Lo)} \tri \Lp_{n-k_*n_*} f \tri_+  \int_{V^\ell_i} \hT_{k_*n_*, V^\ell_i} \tpsi_\ell
+ \sum_{i \in M^\ell(\sim; Sh)} A \tri \Lp_{n-k_*n_*} f \tri_- \delta |\psi_\ell|_{C^0} |J_{V^\ell_i}T_{n, n-k_*n_*}|_{C^0} \\
& \le (1-\tfrac{\ve_{n_*}}{2})^{k_*} 3L \tri \Lp_n f \tri_- +2  A  \tri \Lp_n f \tri_- \delta |\psi_\ell|_{C^0} \bar C_0 \, .
\end{split}
\]
where we have used \eqref{eq:tri} and the fact that $k_*n_* \ge n_0$.  For the sum over long pieces, we used that the total mass of unmatched
pieces decays exponentially in $k$, while for the sum over short pieces,
we used
Lemma~\ref{lem:full growth} and Remark~\ref{rem:improve} to sum over the Jacobians since $|W^1| \ge \delta_0/2$.
Finally, since $|\psi_1|_{C^0} \le e^{a \delta_0^\alpha} \fint_{W^1} \psi_1 \le \frac{4}{\delta_0}$, we conclude,
\[
\begin{split}
\sum_{i \in M^1(\sim)}  \left| \int_{V^1_i} \Lp_{n-k_*n_*} f \, \hT_{k_*n_*, V^1_i} \tpsi_1 \right|
& \le \left( 3L (1- \tfrac{\ve_{n_*}}{2})^{k_*}  + 8A \bar C_0 \tfrac{\delta}{\delta_0}  \right) \tri \Lp_n f \tri_- \\
& \le \left( 3L (1- \tfrac{\ve_{n_*}}{2})^{k_*}  + 8A \bar C_0 \tfrac{\delta}{\delta_0}  \right) \int_{W^2} \Lp_n f \, \psi_2 \, ,
\end{split}
\]
using the fact that $\int_{W^2} \psi_2 = 1$.  A similar estimate holds for the sum over curves in $M^2(\sim)$.
Finally, we put together this estimate with \eqref{eq:gen match} to obtain,
\[
\begin{split}
\int_{W^1} \Lp_nf \, \psi_1 
& \le \sum_{k=1}^{k_*} \sum_{i \in M^2(k)} 2 \int_{\bar U^2_i} \Lp_{n-kn_*} f \, \hT_{kn_*, U^2_i} \tpsi_2
+ \sum_{i \in M^1(\sim)} \int_{V^1_i} \Lp_{n-k_*n_*} f \, \hT_{k_*n_*, V^1_i} \tpsi_1 \\
& \le 2 \int_{W^2} \Lp^n f \, \psi_2 
+ 2  \sum_{j \in M^2(\sim)} \left| \int_{V^2_j} \Lp_{n-k_*n_*} f \, \hT_{k_*n_*, V^2_j} \tpsi_2 \right| \\
& \qquad + \sum_{i \in M^1(\sim)} \left| \int_{V^1_i} \Lp_{n-k_*n_*} f \, \hT_{k_*n_*, V^1_i} \tpsi_1 \right| \\
& \le \int_{W^2} \Lp_n f \, \psi_2 \left( 2 + 3 \big(3L (1- \tfrac{\ve_{n_*}}{2})^{k_*}  + 8A \bar C_0 \tfrac{\delta}{\delta_0}\big) \right)  \, .
\end{split}
\]
We choose $k_*$ such that 
\begin{equation}\label{eq:k*choice}
3L(1 - \frac{\ve_{n_*}}{2})^{k_*} < \frac 16.
\end{equation}
Note that this choice of $k_*$ depends only on $\delta_0$ via $\ve_{n_*}$,
and not on $\delta$.  Next, choose $\delta>0$ sufficiently small that 
\begin{equation}\label{eq:A-cond-s4}
8A \bar C_0 \delta/\delta_0 < \frac 16.
\end{equation}
These choices imply that 
\begin{equation}
\label{eq:balance}
\int_{W^1} \Lp_n f \, \psi_1 \le 3 \int_{W^2} \Lp_n f \, \psi_2 \, .
\end{equation}

Finally, we prove that $L$ must contract by at least $\frac 89$. This is implied directly by the first alternative of Proposition~\ref{prop:alternative}. So suppose instead that the second alternative holds.
Since \eqref{eq:balance} holds for all $W^1, W^2 \in \cW^s(\delta_0/2)$ and test functions $\psi_1, \psi_2$ with 
$\int_{W^1} \psi_1 = \int_{W^2} \psi_2 = 1$, we conclude
that, for $k \ge k_*$ and $m\geq N(\delta)^-$,
\[
\frac{\tri \Lp_{kn_*+m} f \tri_+}{\tri \Lp_{kn_*+m} f \tri_- }  \le \frac{160}{9} \frac{\tri \Lp_{kn_*} f \tri^0_+}{\tri \Lp_{kn_*} f \tri^0_- }
\le  \frac{160}{3} \le \frac{8}{9} L \, ,
\]
if we choose $L \ge 60$.  
\end{proof}


\subsection{Finite diameter}
\label{sec:diam}

In this section we prove the following proposition, which completes the proof of Theorem~\ref{thm:main}.

\begin{prop}
\label{prop:diameter}
For any  $\chi \in\left(\max \{ \frac 12,  \frac{1}{L}, \frac 1{\sqrt{A-1}} \} , 1\right)$,   the cone $\cC_{\chi c, \chi A, \chi L}(\delta)$ has  diameter less than
$\Delta:= \log \left( \frac{(1+\chi)^2}{(1-\chi)^2} \chi L \right)<\infty$ in $\cC_{c,A,L}(\delta)$,
assuming $\delta>0$ is sufficiently small to satisfy \eqref{eq:delta cond}.
\end{prop}

\begin{proof}
For brevity, we will denote $\cC = \cC_{c,A,L}(\delta)$ and $\cC_\chi = \cC_{ \chi c, \chi A, \chi L}(\delta)$.
For $f \in \cC_\chi$, we will show that $\rho(f, 1) < \infty$, where $\rho$ denotes distance in the cone $\cC$.
Fix $f \in \cC_\chi$ throughout.

According to \eqref{eq:H def} if we find $\lambda>0$ such that $f-\lambda\succeq 0$, then $\bar \alpha(1,f)\geq \lambda$. 

Notice that $\tri f- \lambda \tri_\pm = \tri f \tri_\pm - \lambda$.  Hence  $f - \lambda$  satisfies \eqref{eq:cone 2} if 
\[
\tri f\tri_+-\lambda \le L (\tri f\tri_- - \lambda)\quad \Longleftarrow \quad \lambda \le \frac{L(1-\chi)}{L-1} \tri f\tri_-=: \bar \alpha_1 \, ,
\]
where we have used that $f \in \cC_\chi$.

Similarly,  $f - \lambda$  satisfies \eqref{eq:cone 3} if, for all $W \in \cW^s_-(\delta)$ and $\psi\in  \cD_{a,\beta}(W)$,
\[
|W|^{-q}\frac{\left|\int_W f\psi-\lambda\int_W \psi \right|}{\fint_W\psi}\leq A\delta^{1-q}(\tri f\tri_--\lambda)
 \quad \Longleftarrow \quad \lambda\leq \frac{(1-\chi) A \tri f\tri_-}{A+1}=: \bar \alpha_2 \, .
\]
Next, notice that for any $\lambda \ge 0$, $W^1, W^2 \in \cW^s_-(\delta)$ and 
$\psi_\ell \in \cD_{a,\alpha}(W^\ell)$,  
\begin{equation}
\label{eq:c cancel}
\begin{split}
\left| \frac{\int_{W^1}(f - \lambda) \psi_1}{\fint_{W^1} \psi_1} \right. & \left. - \frac{\int_{W^2}(f-\lambda)\psi_2}{\fint_{W^2} \psi_2} \right|
= \left| \frac{\int_{W^1} f \, \psi_1}{\fint_{W^1} \psi_1} - \frac{\int_{W^2} f \, \psi_2}{\fint_{W^2} \psi_2}
- \lambda(|W^1|- |W^2|) \right| \\
& \le \chi^2 d_{\cW^s}(W^1, W^2)^\gamma \delta^{1-\gamma}  c A \tri f \tri_- 
+ \lambda (\delta + C_s) d_{\cW^s}(W^1, W^2) \, ,
\end{split}
\end{equation}
where we have used \eqref{eq:W-difference}, so that $f - \lambda$ satisfies \eqref{eq:cone 5} if 
\[
\chi^2 d_{\cW^s}(W^1, W^2)^\gamma \delta^{1-\gamma}   c A \tri f \tri_- 
+ \lambda (\delta + C_s) \delta^{ 1- \gamma} d_{\cW^s}(W^1, W^2)^\gamma
\le d_{\cW^s}(W^1, W^2)^\gamma {\delta^{1-\gamma} } c A (\tri f \tri_- - \lambda) \, .
\]
This occurs whenever
\[
\lambda \le \frac{cA \tri f \tri_- (1-\chi^2)}{\delta+C_s+ cA}
\quad \Longleftarrow \quad
\lambda \le (1-\chi) \tri f \tri_- =: \bar \alpha_ 3 \, ,
\]
provided that $\delta$ is chosen sufficiently small that
\begin{equation}
\label{eq:delta cond}
\delta + C_s  \le \chi c A \, ,
\end{equation}
which is possible since $cA > 2C_s$ by \eqref{eq:cA}  and $\chi > 1/2$. 

Note that $\bar \alpha_2  \le \bar \alpha_3 \le \bar \alpha_1$, so that 
$\bar\alpha_2 = \min_i \{ \bar \alpha_i \}$.  Thus if $\lambda \le \bar \alpha_2$, then 
$f - \lambda \in \cC$, i.e.
$\bar \alpha(1, f) \ge \bar \alpha_2$.

Next, we proceed to estimate $\bar \beta(1,f)$ for $f \in \cC_\chi$.  If we find $\mu > 0$ such that $\mu - f \in \cC$, 
this will imply that $\bar \beta(1,f) \le \mu$.   Remarking that
$\tri \mu - f \tri_{\pm} = \mu - \tri f \tri_{\mp}$, we have that $\mu - f$ satisfies \eqref{eq:cone 2} 
 if
\[
\mu  \ge \frac{L \tri f \tri_+ - \tri f \tri_-}{L-1} \quad \Longleftarrow \quad
\mu \ge \frac{L}{L-1} \tri f \tri_+ =: \bar \beta_1 \, ,
\]
while $\mu - f$ satisfies \eqref{eq:cone 3} if for all $W \in \cW^s_-(\delta)$, $\psi \in \cD_{a,\beta}(W)$,
\[
|W|^{-q} \frac{| \mu \int_W \psi - \int_W f \, \psi |}{\fint_W \psi} \le A \delta^{1-q}(\mu - \tri f \tri_+)
\quad \Longleftarrow \quad 
\mu \ge \frac{(1 + \chi) A}{A-{2^{1-q} }} \tri f \tri_+ =: \bar \beta_2 \, .
\]
Finally, recalling \eqref{eq:c cancel} and again \eqref{eq:W-difference}, we have that
$\mu -f$ satisfies \eqref{eq:cone 5} whenever
\[
\chi^2 d_{\cW^s}(W^1, W^2)^\gamma \delta^{1-\gamma}  c A \tri f \tri_- 
+ \mu (\delta + C_s) \delta^{ 1-\gamma} d_{\cW^s}(W^1, W^2)^\gamma
\le d_{\cW^s}(W^1, W^2)^\gamma  \delta^{1-\gamma}  c A (\mu - \tri f \tri_+) \, .
\]
This is implied by,
\[
\mu \ge \frac{cA (1 + \chi^2)}{cA - (\delta + C_s) } \tri f \tri_+
\quad \Longleftarrow \quad
\mu \ge \frac{1 + \chi^2}{1 - \chi} \tri f \tri_+ =: \bar \beta_3 \, ,
\]
where again we have assumed \eqref{eq:delta cond}.

Defining $\bar \beta = \max_i \{ \bar \beta_i \}$, it follows that if $\mu \ge \bar \beta$, then
$\mu - f \in \cC$.  
Thus $\bar \beta \ge \bar \beta(1, f)$.
Since  $\chi > 1/L$ and $\chi^2 > 1/(A-1)$, it holds that
$\bar \beta_3 \ge \bar \beta_2 \ge \bar \beta_1$.  Thus $\bar \beta = \bar \beta_3$.
Our assumption also implies $\chi > 1/A$, so that 
$\bar \alpha_2 \ge \frac{1-\chi}{1+\chi} \tri f \tri_-$.

Finally, recalling \eqref{eq:H def}, we have
\[
\rho(1,f) = \log \left( \frac{\bar \beta(1,f)}{\bar \alpha(1,f)} \right)
\le \log \left( \frac{\bar \beta_3}{\bar \alpha_2} \right)
\le \log \left( \frac{\frac{1+\chi^2}{1-\chi}}{\frac{1-\chi}{1+\chi}} \frac{\tri f \tri_+}{\tri f \tri_-} \right)
\le \log \left( \frac{(1+\chi)^2}{(1-\chi)^2} \chi L \right) \, ,
\]
for all $f \in \cC_\chi$, completing the proof of the proposition.
\end{proof}

\begin{remark} Note that, setting $\chi_*=\max \{ \frac 12,  \frac{1}{L}, \frac 1{\sqrt{A-1}} \} $, for $\chi\leq \chi_*$  Proposition \ref{prop:diameter} implies only that the diameter of $\cC_{\chi c, \chi A, \chi L}(\delta)\subset \cC_{\chi_* c, \chi_* A, \chi_* L}(\delta)$, in $\cC_{c,A,L}(\delta)$, is bounded by
$ \log \left( \frac{(1+\chi_*)^2}{(1-\chi_*)^2} \chi_* L \right)$. If needed, a more accurate formula can be easily obtained, but it would be more cumbersome.
\end{remark}


\section{Loss of Memory and Convergence to Equilibrium }
\label{sec:exp-mix}

In this section we show how Theorem~\ref{thm:main} (i.e. Theorem~\ref{thm:cone contract} and Proposition~\ref{prop:diameter} ) imply
the loss of memory and convergence to equilibrium stated in Theorems~\ref{thm:memory state} and \ref{thm:equi state}. 
For a single map, the loss of memory is simply decay of correlations and the results are comparable to the ones obtained in \cite{demzhang11} since they apply to a similar (very) large class of observables (and possibly even distributions). Our loss of memory result is new for our class of billiards, although see  Remark~\ref{rmk:stenlund} and \cite{young zhang}
for loss of memory in a related billiards model.
Before proving the main results of this section (Theorem~\ref{thm:memory} and Corollary \ref{cor:extend} prove
Theorem~\ref{thm:memory state} while Theorem~\ref{thm:equi} and Corollary~\ref{cor:extend} prove
Theorem~\ref{thm:equi state}), we 
establish a key lemma that integration with respect to $\musrb$ against suitable
test functions respects the ordering in our cone.
Recall the vector space of functions $\cA$ defined in Section~\ref{sec:cone_dist}.

The parameters $a, q, \alpha, \beta, \gamma, c, A, L,\delta_0$ are fixed as to satisfy the relations described in Section \ref{sec:conditions}, hence Theorem~\ref{thm:main} holds true.
With Proposition~\ref{prop:almost} in mind, we prove our next lemma with respect to the slightly larger cone
$\cC_{c,A,3L}(\delta) \supset \cC_{c,A,L}(\delta)$.

\begin{lemma}
\label{lem:order}
Let $\delta>0$ be small enough that 
$2 C_\ell C_h (1+A)  ( \delta^{4/3} + \delta^{1/3 + \beta} a \ell_{\max} ) < 1$,
where $C_\ell, C_h>0$ are from \eqref{eq:dominate} and $\ell_{\max}$ is the maximum
diameter of the connected components of $M$.

Suppose $\psi \in C^1(M)$ satisfies $2(2\delta)^{1-\beta} |\psi'|_{C^0(M)} \le a \min_M \psi$. 
If $f, g \in \cA$ with $g-f \in \cC_{c,A,3L}(\delta)$,  then 
$\int f \, \psi \, d\musrb \le \int g \, \psi \, d\musrb$. 
\end{lemma}

\begin{proof}
Let $\psi_{\min} = \min_M \psi$.  The assumption on $\psi$ implies that $\psi \in \cD_{\frac{a}{2},\beta}(W)$ for each 
$W \in \cW^s_-(\delta)$ since,
\[
\left| \log \frac{\psi(x)}{\psi(y)} \right| \le 
\frac{1}{\psi_{\min}} |\psi(x) - \psi(y)| \le \frac{|\psi'|_{C^0(M)}}{\psi_{\min}} d(x,y) 
\le \frac{|\psi'|_{C^0(M)}}{\psi_{\min}} (2\delta)^{1-\beta} d(x,y)^\beta  \, .
\]

Suppose $f, g \in \cA$ satisfy  $g-f \in \cC_{c,A,3L}(\delta)$. 
Then according to
\eqref{eq:tri def} and \eqref{eq:cone 3}, for all $\psi \in \cD_{a,\beta}(W)$,
\begin{eqnarray}
\tri g - f \tri_- \! \! \! \! \! & \! \! \int_W \psi \; \le \; \int_W (g-f) \psi \, dm_W \; \le \; \tri g-f \tri_+ \int_W \psi
& \quad \forall\, W \in \cW^s(\delta)
\label{eq:long bound} \\
& \left| \int_W (g-f) \psi \, dm_W \right| \; \le \; \tri g-f \tri_- A \delta^{1-q}|W|^q \fint_W \psi
& \quad  \forall\, W \in \cW^s_-(\delta).
\label{eq:short bound}
\end{eqnarray}

Next, we disintegrate $\musrb$ according to a smooth foliation of stable curves as follows.
Since the stable cones defined in {\bf (H1)} are globally constant and uniform in the family
$\cF(\tau_*, \cK_*, E_*)$, we fix a direction in the stable cone
and consider stable curves in the form of line segments with this slope.  Let $k_\delta \ge k_0$
denote the minimal index $k$ of a homogeneity strip $\bH_k$ such that the stable line segments
in $\bH_k$ have length less than $\delta$.  Due to the fact that the minimum
slope in the stable cone is $\cK_{\min} > 0$, we have 
\begin{equation}
\label{eq:k delta}
k_\delta = C_h \delta^{-1/3},
\end{equation}
for some constant $C_h >0$ independent of $\delta$.

Now for $k < k_\delta$, 
we decompose $\bH_k$ into horizontal bands $B_{k,i}$ such that every maximal line segment of the
chosen slope in $B_{k,i}$ has equal length between $\delta$ and $2\delta$.  We do the
same on $\bH_0 := M \setminus (\cup_{k \ge k_0} \bH_k)$.
On each $B_{k,i}$, define a foliation of such parallel line segments $\{ W_\xi \}_{\xi \in \Xi_{k,i}} \subset \cW^s(\delta)$.
Using the smoothness of this foliation, we disintegrate $\musrb$ into conditional 
measures $\cos \vf(x) dm_{W_\xi}$ on $W_\xi$
and a factor measure $\hat\mu$ on the index set $\Xi_{k,i}$.  Note that our conditional
measures are not normalized - we include this factor in $\hat\mu$.  
Finally, on each homogeneity strip $\bH_k$, $k \ge k_\delta$, we carry out a similar
decomposition, but using homogeneous parallel line segments of maximal length in $\bH_k$ 
(which are necessarily shorter than length $\delta$).  
We use the notation $\{ W_\xi \}_{\xi \in \Xi_{k,1}} \subset \cW^s_-(\delta)$ for the foliations in these homogeneity strips
since there is only one band in each of these $\bH_k$.
Note that for all $k$ and $i$, we have $\hat\mu(\Xi_{k,i}) \le C_\ell$, 
for some constant $C_\ell$ depending only on the chosen slope and spacing of homogeneity strips.

Also, it follows as in \eqref{eq:distortion}, that for $x, y \in W \in \cW^s_-(\delta)$,
\[
\log \frac{\cos \vf(x)}{\cos \vf(y)} \le C_d (2\delta)^{1/3-\beta} d(x,y)^\beta \, ,
\]
so that $\cos \vf \in \cD_{\frac{a}{2}, \beta}(W)$ by the assumption of 
Lemma~\ref{lem:test contract}.
Thus $\psi \cos \vf \in \cD_{a, \beta}(W)$ for all $W \in \cW^s_-(\delta)$.

Using this fact and our disintegration of $\musrb$, we estimate the integral 
applying \eqref{eq:long bound} on $\Xi_{k,i}$ for $k< k_\delta$ and 
\eqref{eq:short bound} on $\Xi_{k,1}$ for $k \ge k_{\delta}$,
\begin{equation}
\label{eq:dominate}
\begin{split}
\int_M (g-f) & \psi \, d\musrb  = \sum_{i, k<k_\delta} \int_{\Xi_{k,i}} \int_{W_\xi} (g-f) \psi \cos \vf \, dm_{W_\xi} d\hat\mu(\xi) 
+ \sum_{k \ge k_\delta} \int_{\Xi_{k,1}} \int_{W_\xi} (g-f) \psi  \cos \vf \, dm_{W_\xi} d\hat\mu(\xi) \\
& \ge \tri g-f \tri_- \left(
\sum_{i, k< k_\delta} \int_{\Xi_{k,i}} \int_{W_\xi} \psi  \cos \vf \, dm_{W_\xi} d\hat\mu(\xi)
- A \delta \sum_{k \ge k_\delta} \int_{\Xi_{k,1}} \fint_{W_\xi} \psi \cos \vf  \, dm_{W_\xi} d\hat\mu(\xi) \right) \\
& \ge \tri g-f \tri_- \left( \psi_{\min} \musrb(M \setminus (\cup_{k \ge k_\delta} \bH_k)) - A \delta C_\ell 
|\psi|_{C^0} \sum_{k \ge k_\delta} k^{-2} \right) \\
& \ge \tri g-f \tri_- \left( \psi_{\min}(1 - 2 C_\ell C_h \delta^{4/3}) - |\psi|_{C^0} A C_\ell C_h 2 \delta^{4/3} \right) \, ,
\end{split}
\end{equation}
where we have estimated $\sum_{k \ge k_\delta} k^{-2} \le 2 k_\delta^{-1} \le 2 C_h \delta^{1/3}$ and
$\musrb(\cup_{k \ge k_\delta} \bH_k) \le 2 C_\ell C_h \delta^{4/3}$.

Now $|\psi|_{C^0} \le \psi_{\min} + \ell_{\max} |\psi'|_{C^0}$, where $\ell_{\max}$ 
is the maximum diameter of the connected components of $M$.  Then by the
assumption on $\psi$, we have
\[
\begin{split}
2 C_\ell C_h (1+A) \delta^{4/3} |\psi|_{C^0} & \le 2 C_\ell C_h (1+A) \delta^{4/3} \psi_{\min} (1 + \ell_{\max} \tfrac a2 (2\delta)^{\beta-1} ) \\
& \le \psi_{\min} 2 C_\ell C_h (1+A) ( \delta^{4/3} + a \ell_{\max} \delta^{1/3 + \beta})
 \le \psi_{\min} \, ,
\end{split}
\]
where for the last inequality we have used the assumption on $\delta$ in the statement of the
lemma.  We conclude that the lower bound in \eqref{eq:dominate} cannot be less than 0.
\end{proof}

\begin{remark}
\label{rem:after order}
Since Remark~\ref{rem:A-L} applies equally well to $\cC_{c,A,3L}(\delta)$, 
Lemma \ref{lem:order} implies there exists $\bar C \ge 1$ such that $\int_M f \, d\musrb \ge \bar C^{-1}  \tri f \tri_- > 0$ for all $f \in \cC_{c,A,L}(\delta)$. 

Using instead the upper bound in \eqref{eq:long bound} and following the estimate of \eqref{eq:dominate} yields,
\[
0 < \int_M f \psi \, d\musrb \le \tri f \tri_+ C  |\psi|_{C^0} \, ,
\]
for all $f \in \cC_{c,A,L}(\delta)$ and $\psi$ as in the statement of Lemma~\ref{lem:order}.  
Since any $\psi \in C^1(M)$ can be made to satisfy the condition of Lemma~\ref{lem:order}
by adding a constant (see the definition of $C_\psi$ in \eqref{eq:C psi} below), the estimate
can be extended to all $\psi \in C^1(M)$ to obtain,
\[
\int_M f \psi \, d\musrb \le \tri f \tri_+ C | \psi |_{C^1} \, .
\]
\end{remark}

Loss of memory and convergence to equilibrium, including equidistribution, readily follow from the contraction in the
projective metric $\rho_{\cC}(\cdot, \cdot)$ of the cone $\cC_{c, A, L}(\delta)$. Set $\musrb(f) = \int_M f \, d\musrb$.

Recall $N_\cF := N(\delta)^- + k_*n_*$ from Theorem~\ref{thm:cone contract} and the definition of
an $N_{\cF}$-admissible sequence from Section~\ref{sec:main}:
A sequence $( \iota_j )_j$, $\iota_j \in \cI(\tau_*, \cK_*, E_*)$, is $N_{\cF}$-admissible if 
there exist sequences $(T_k)_{k \ge 1} \subset \cF(\tau_*, \cK_*, E_*)$ and $(N_k)_{k \ge 1}$ 
with $N_k \ge N_\cF$, such that $T_{\iota_j} \in \cF(T_k, \kappa)$
for all $k \ge 1$ and $j \in [1 + \sum_{i=1}^{k-1} N_i, \sum_{i=1}^k N_i]$.

That is, an admissible sequence remains
in a $\kappa$ neighborhood of $T_k$ for $N_k \ge N_\cF$ iterates at a time, but may undergo large
changes between such blocks.

Our first theorem concerns loss of memory for functions in our cone, both with respect to $\musrb$ and with 
respect to the iteration of 
individual stable curves. It does not use property {\bf (H5)}, although it does use that
$\musrb$ is a conformal measure for $\cL_n$, i.e. $\musrb(\cL_n f) = \musrb(f)$.

\begin{theorem}
\label{thm:memory}
Let $\delta>0$ satisfy the assumption of Lemma~\ref{lem:order}.  
There exists $C>0$ and $\vartheta<1$ such that for all admissible sequences $(\iota_j)_j$, all $n \ge 0$, and all
$f, g \in \cC_{c,A,L}(\delta)$ with $\int_M f \, d\musrb = \int_M g \, d\musrb$:

\begin{itemize}
  \item[a)]  For all all $W \in \cW^s(\delta)$ and all $\psi \in C^1(W)$, we have 
\[
\left| \fint_{W} \Lp_n f  \,  \psi  \, dm_{W} - \fint_{W} \Lp_n g \,  \psi \, dm_{W} \right|
\le C \vartheta^n  \, |\psi|_{C^1} \min \{ \tri f \tri_+,  \tri g \tri_+ \} \, ;
\]
\item[b)] For all $\psi  \in C^1(M)$,
\begin{equation}
\label{eq:M conv}
\left| \int_M \cL_n f \, \psi \, d\musrb - \int_M \cL_n g \, \psi \, d\musrb \right| 
\le C \vartheta^n |\psi|_{C^1(M)} \min \{ \tri f \tri_+ , \tri g \tri_+ \} \, .
\end{equation}
\end{itemize}
\end{theorem}

\begin{proof}
(a)
Recall the definition of $\tri \cdot \tri_+ $ for elements of  $\cA$ from Definition \ref{def:cA} and \eqref{eq:tri def}, 
\[
\tri f \tri_+ = \sup_{\stackrel{\scriptstyle W \in \cW^s(\delta)}{\psi \in \cD_{a,\beta}(W)}} \frac{\left|\int_W f \psi \, dm_W\right|}{\int_W \psi \, dm_W} ,
\]
and note that by \eqref{eq:posi}, $\tri \cdot \tri_+ $ is an order-preserving semi-norm in $\cA$.\footnote{ A semi-norm $\|\cdot\|$ is order preserving if $-g\preceq f\preceq g$ implies $\|f\|\leq \|g\|$.} One can check directly that $\cA$ is an integrally closed vector lattice.
Also $ \musrb(f) :=  \int_M f \, d\musrb$ is homogeneous and order preserving in $\cC_{c,A,3L}(\delta)$ by Lemma \ref{lem:order} applied to $\psi \equiv 1$. 

We would like to apply Theorem~\ref{thm:cone contract} to each block of $N_\cF$ iterates in the admissible sequence; 
however, at time $n$, the sequence may have completed fewer than $N_\cF$ iterates in its current block
so it may be that $\cL_n f, \cL_n g \notin \cC_{c,A,L}(\delta)$. But since $n \ge N_\cF > n_0$, 
it follows from Proposition~\ref{prop:almost}
that $\cL_n f, \cL_n g \in \cC_{c,A,3L}(\delta)$.
Then denoting by $\rho_{\cC'}$ the metric in the larger cone $\cC_{c,A,3L}(\delta)$,
\cite[Lemma 2.2]{LSV98} implies that,
for all $f,g \in \cC_{c,A,L}(\delta)$ with $\musrb(f) = \musrb(g)$,
\footnote{\cite[Lemma~2.2]{LSV98} is stated for order preserving norms but its proof holds verbatim for order preserving 
semi-norms, see \cite[Lemma D.4]{dkl21}.}
\begin{equation}
\label{eq:adapted}
\tri \cL_n f-\cL_n g\tri_+\leq \left(e^{\rho_{\cC'}(\cL_n f,\cL_n g)}-1\right)\min  \{\tri \cL_n f \tri_+, \tri \cL_n g \tri_+\}. 
\end{equation}
Using the definition of admissible sequence, we may peel off the most recent $j$ iterates, where $j < n_0 + N_\cF$, 
such that $\cL_{n-j} f ,\cL_{n-j} g \in \cC_{c,A,L}(\delta)$ and $n-j$ is chosen so that we have undergone at
least $N_\cF$ iterates in the current block at time $n-j$. Then applying Theorem~\ref{thm:cone contract} to each block of 
$N_k$ iterates, 
and using Proposition~\ref{prop:diameter} and \cite[Theorem~1.1 and Remark~1.2]{liv95}, 
for all $n \ge N_\cF$,
\[
\rho_{\cC'}(\cL_n f, \cL_n g) \le \rho_{\cC}(\cL_{n-j}f, \cL_{n-j}g) \le \vartheta^{n-j-k} \rho_{\cC}(\cL_k f, \cL_k g) \, ,
\]
where $\vartheta=\left[\tanh(\Delta/4)\right]^{ 1/(2 N_\cF)}$,
and $k \in [N_\cF, 2N_\cF -1]$ is the least integer $\ge N_\cF$ so that 
$\cL_{n-j-k}$ ends in a contracting block.

Finally, we use the fact that $\cL_k f, \cL_k g \in \cC_{\chi c, \chi A, \chi L}(\delta)$ 
together with Proposition~\ref{prop:diameter} to conclude
$\rho_{\cC}(\cL_k f, \cL_k g) \le \Delta$.
Combining these estimates with Lemma~\ref{lem:first L} yields, 
\begin{equation}
\label{eq:contract}
\tri \cL_n f-\cL_n g\tri_+\leq   C \vartheta^n\min \{\tri f \tri_+, \tri g \tri_+\},
\end{equation}
where $C = \frac 32 \Delta e^\Delta \vartheta^{-3 N_\cF - n_0}$. 
This proves (a) for any $W \in \cW^s(\delta)$
and $\psi \in \cD_{a,\beta}(W)$. 

To extend this estimate
to more general  $\psi \in C^1(W)$, define $\tpsi = \psi + C_\psi$, where
\begin{equation}
\label{eq:C psi}
C_\psi = |\psi_{\min}| + \tfrac{2}{a} |\psi'|_{C^0}(2\delta)^{1-\beta} \, .
\end{equation}
Then $\tpsi' = \psi'$ and $\min_W \tpsi \ge \frac{2}{a} |\tpsi'|_{C^0} (2\delta)^{1-\beta}$,
so that $\tpsi \in \cD_{\frac{a}{2}, \beta}(W)$ by the proof of Lemma~\ref{lem:order}. 
Then since also $C_\psi \in \cD_{a, \beta}(W)$, the required estimate
follows by writing $\psi = \tpsi - C_\psi$ and using the triangle inequality.

\medskip
\noindent
(b) Following the same strategy as above, given $\psi \in C^1(M)$ satisfying
the assumption of Lemma~\ref{lem:order}, we define a pseudo-norm for $f \in \cA$ by
\begin{equation}
\label{eq:psi seminorm}
\| f \|_\psi = \left| \int_M f \, \psi \, d\musrb \right| \, .
\end{equation}
By Lemma~\ref{lem:order}, $\| \cdot \|_\psi$ is an order-preserving semi-norm, and as in \eqref{eq:adapted}, invoking
again \cite[Lemma~2.2]{LSV98}, Theorem~\ref{thm:cone contract}, Proposition~\ref{prop:diameter}
and \cite[Theorem~1.1]{liv95}, we have for $f, g \in \cC_{c,A,L}(\delta)$ with
$\musrb(f) = \musrb(g)$ and $n \ge N_\cF$,
\[
\| \cL_n f - \cL_n g \|_\psi \le C \vartheta^n \min \{ \| \cL_n f \|_\psi , \| \cL_n g \|_\psi \} \le C \vartheta^n |\psi|_{C^0} \min \{
\tri f \tri_+ , \tri g \tri_+ \} \, ,
\]
where we applied Remark~\ref{rem:after order}.   This 
proves (b) for $\psi$ satisfying the assumption of 
Lemma~\ref{lem:order}.  We extend to more general $\psi \in C^1(M)$ by defining
$\tpsi = \psi + C_\psi$, where $C_\psi$ is given by \eqref{eq:C psi}, and arguing as in the
proof of part (a).
\end{proof}

Since our maps all preserve $\musrb$, the loss of memory also implies equidistribution of measures supported on stable curves and convergence to equilibrium, both of which are summarized in the following theorem.

\begin{theorem}
\label{thm:equi}
Let $\delta>0$ satisfy the assumption of Lemma~\ref{lem:order}.  
There exists $C>0$ such that for all $n \ge 0$ and admissible sequences $(\iota_j)_j \subset \cI(\tau_*, \cK_*, E_*)$, $\vartheta$ as in Theorem \ref{thm:memory}, and all $f, g \in \cC_{c,A,L}(\delta)$, with $\musrb(f) = \musrb(g)$:

\begin{itemize}
  \item[a)]  For all $W_1, W_2 \in \cW^s(\delta)$  and all $\psi_i \in C^1(W_i)$ with $\fint_{W_1} \psi_1 = \fint_{W_2} \psi_2$, we have
\[
\left| \fint_{W_1}  \Lp_n f  \,  \psi_1  \, dm_{W_1} - \fint_{W_2} \Lp_n g \,  \psi_2 \, dm_{W_2} \right|
\le C \vartheta^n  \, (|\psi_1|_{C^1} + |\psi_2|_{C^1} ) \musrb(f) \, ;
\]
in particular, for all $W \in \cW^s(\delta)$ and $\psi\in C^1(W)$,
\begin{equation}
\label{eq:avg conv}
\left| \fint_W \Lp_n f \, \psi \, dm_W - \musrb(f) \fint_W \psi \, dm_W \right| 
\le C \vartheta^n \,  |\psi|_{C^1} \musrb(f)   \, ;
\end{equation}
  \item[b)]  for all $\psi \in C^1(M)$,
\[
\left| \int_M f \, \psi \circ T_n \, d\musrb - \int_M f \, d\musrb \int_M \psi \, d\musrb \right|
\le C \vartheta^n |\psi|_{C^1(M)} \musrb(f) \, .
\]
\end{itemize}
\end{theorem}

\begin{proof}
a) Since $\cL_n 1 = 1$ and $\tri \musrb(f) \tri_+ = \musrb(f)$, applying \eqref{eq:contract}  with $g = \musrb(f)$  implies,
\begin{equation}
\label{eq:equi}
\begin{split}
\left| \fint_W \Lp_n f \, \psi \, dm_W - \musrb(f) \fint_W \psi \right| &=
\fint_W \psi  \left| \frac{\int_W \Lp_n f \, \psi \, dm_W}{\int_W\psi} -  \frac{\int_W \Lp_n (\musrb(f)) \, \psi}{\int_W\psi}  \right|\\
&\le C \vartheta^n \, |\psi|_{C^0}  \musrb(f)\, ,
\end{split}
\end{equation}
which proves \eqref{eq:avg conv}  
for $\psi \in \cD_{a, \beta}(W)$.  We extend this estimate
to more general  $\psi \in C^1(W)$ by defining $\tpsi = \psi + C_\psi$ as in \eqref{eq:C psi}
and arguing as in the proof of Theorem~\ref{thm:memory}(a).
Finally, the first inequality of part (a) follows from an application of the triangle inequality.

\medskip
\noindent
b) Since $\musrb$ is conformal with respect to $\cL_{T}$ for each $T \in \cF(\tau_*, \cK_*, E_*)$,
and using that $\cL_n 1 = 1$, we have
\[
 \int_M f \, \psi \circ T_n \, d\musrb - \int_M f \, d\musrb \int_M \psi \, d\musrb
 = \int_M  \cL_n (f - \musrb(f))   \, \psi \, d\musrb \, .
\] 
Thus applying \eqref{eq:M conv} to $g = \musrb(f)$ proves part (b)
since $\tri \musrb(f) \tri_+ = \musrb(f)$.

\end{proof}

We may extend Theorems~\ref{thm:memory} and \ref{thm:equi} to piecewise
H\"older continuous functions, as long as the discontinuities are transverse to the stable cone.
Recall the definition of a regular partition $\pa$ from Definition~\ref{def:part}
and the set $C^t(\pa)$ of functions which are $t$-H\"older continuous on each element of $\pa$,
i.e. which satisfy
\[
|f|_{C^t(\pa)} = \sup_{P \in \pa} |f|_{C^t(P)} < \infty \, .
\]

\begin{cor}
\label{cor:extend}
Let $\pa$ be a regular partition of $M$ and let $t \ge \gamma$.  Then the convergence in Theorems~\ref{thm:memory} and \ref{thm:equi} extend to all 
$f, g \in C^t(\pa)$, with $\max\{|f|_{C^t(\pa)}, |g|_{C^t(\pa)}\}$ in place of $\min \{ \tri f \tri_+, \tri g \tri_+ \}$
on the right hand side in Theorem~\ref{thm:memory} and in place of $\musrb(f)$ on the
right hand side in Theorem~\ref{thm:equi}.
\end{cor}

The proof of this corollary relies on the following lemma.

\begin{lemma}
\label{lem:dominate}
If $\pa$ is a regular partition of $M$ and $f \in C^t(\pa)$ with $t \ge \gamma$, 
then $\lambda + f \in \cC_{c, A, L}(\delta)$ for any 
\[
\lambda \geq \max \left\{  \frac{L+1}{L-1} |f|_\infty , \frac{A+2^{1-q}}{A- 2^{1-q}} |f|_\infty, 
\frac{cA + (2 \delta^t + 8 K C_\pa C_s + 6 C_s) }{cA - 2C_s} |f|_{C^t(\pa)}\right\} \, .
\]
\end{lemma}

\begin{proof}[Proof of Corollary~\ref{cor:extend}]
Let $f, g \in C^t(\pa)$ with $\musrb(f) = \musrb(g)$ and let $\psi \in C^1(M)$.  
Let $\lambda_f, \lambda_g$ be the constants
from Lemma~\ref{lem:dominate} corresponding to $f$ and $g$, respectively, and set 
$\lambda = \max \{ \lambda_f, \lambda_g \}$.
Then $f + \lambda, g + \lambda \in \cC_{c,A,L}(\delta)$ and $\musrb(f+\lambda) = \musrb(g + \lambda)$, so that by Theorem~\ref{thm:memory}(b), for all $n \ge 0$,
\[
\begin{split}
\left| \int_M \cL_n (f-g) \, \psi \, d\musrb \right| & = \left| \int_M \cL_n (f+\lambda-(g+\lambda)) \, \psi \, d\musrb \right| \\
& \le C' \vartheta^n |\psi|_{C^1(M)} \max \{ |f|_{C^t(\pa)}, |g|_{C^t(\pa)} \} \, ,
\end{split}
\]
since $\tri f+ \lambda\tri_+ \le \lambda + |f|_\infty$, and by Lemma~\ref{lem:dominate},
$\lambda_f \ge C'' |f|_{C^t(\pa)}$, with analogous estimates for $g$.  This proves the analog of
part (b) Theorem~\ref{thm:memory} and the proof of part (a) follows similarly, replacing the integral over $M$
by the integral over $W \in \cW^s$.

The extension of Theorem~\ref{thm:equi} to $f, g \in C^t(\pa)$ follows analogously, 
replacing $f$ and $g$ in \eqref{eq:equi} with $f+\lambda$ and $g+\lambda$, respectively to prove the analogue of \eqref{eq:avg conv}, and then using the
triangle inequality to deduce the first inequality of part (a).  Finally, part (b) follows immediately once
$f$ is replaced by $f+\lambda$ since $\int_M \psi \circ T_n \, d\musrb = \int_M \psi \, d\musrb$ due to
{\bf (H5)}.
\end{proof}

\begin{proof}[Proof of Lemma~\ref{lem:dominate}]
We must show that $\lambda + f$ satisfies conditions \eqref{eq:cone 2} - \eqref{eq:cone 5} in the
definition of $\cC_{c,A,L}(\delta)$.  Since 
\begin{equation}
\label{eq:tri up/down}
\tri \lambda + f \tri_+ \le \lambda + | f |_\infty, \qquad \mbox{and} \qquad
\tri \lambda + f \tri_- \ge \lambda - |f|_\infty \, ,
\end{equation} 
to guarantee \eqref{eq:cone 2}, we need
\[
\frac{\lambda + |f|_\infty}{\lambda - |f|_\infty} \le L 
\qquad \impliedby \qquad
\lambda \ge |f|_\infty \frac{L+1}{L-1} \, .
\]
Next, to guarantee \eqref{eq:cone 3}, for $W \in \cW^s_-(\delta)$, $\psi \in \cD_{a, \beta}(W)$, 
we need,
\[
\begin{split}
|W|^{-q} \frac{\int_W (\lambda +f)\psi}{\fint_W \psi} \le A \delta^{1-q} (\lambda - |f|_\infty)
& 
 \quad \impliedby \quad
|W|^{1-q}(\lambda + |f|_\infty) \le A \delta^{1-q}(\lambda - |f|_\infty)
\\
& \quad \impliedby \quad
\lambda \ge |f|_\infty \frac{A + 2^{1-q}}{A - 2^{1-q}} \, .
\end{split}
\]
Lastly, we need to show that \eqref{eq:cone 5} is satisfied.  For this, we prove the claim:
\begin{equation}
\label{eq:C1 bound}
\left| \frac{\int_{W^1} f \psi_1}{\fint_{W^1} \psi_1 } - \frac{\int_{W^2} f \psi_2}{\fint_{W^2} \psi_2 } \right|
\le  ( 2\delta^t + 8 K C_\pa C_s + 6 C_s ) \delta^{1-\gamma} d_{\cW^s}(W^1, W^2)^\gamma  |f|_{C^t(\pa)}  \, ,
\end{equation}
for $W^1, W^2, \psi_1, \psi_2$ as in \eqref{eq:cone 5}.  
As in Section~\ref{subsec:contract c}, we partition $W^k$ into matched pieces $U^k_j$ and
unmatched pieces $V^k_i$ such that $U^1_j$, $U^2_j$ belong to the same element $P \in \pa$
and are defined over the same $r$-interval $I_j$ for each $j$.   By assumption on
$\pa$, $\# \{ U^k_j \}_j \le K$, $\# \{ V^k_j \}_{k,j} \le 2K$, and 
$|V^k_j| \le C_s C_{\pa} d_{\cW^s}(W^1, W^2)$.

Recalling the notation from Section~\ref{sec:distances}, we express the matched pieces as graphs
over their common $r$-interval,
$U^k_j = \{ G_{U^k_j}(r) = (r, \vf_{U^k_j}(r)) : r \in I_j \}$, for $k=1,2$.

As in Section~\ref{subsec:contract c}, we assume without loss of generality that $|W_2| \ge |W_1|$
and $\fint_{W_1} \psi_1 = 1$.  Also, we may assume 
$|W_2| \ge 2 C_s \delta^{1-\gamma} d_{\cW^s}(W_1, W_2)^\gamma$; 
otherwise, \eqref{eq:C1 bound} is trivially
bounded by $2 |W_2| |f|_\infty \le 4 C_s \delta^{1-\gamma} d_{\cW^s}(W_1, W_2)^\gamma |f|_\infty$.

Next,
\begin{equation}
\label{eq:C1 split}
\begin{split}
\left| \frac{\int_{W^1} f \psi_1}{\fint_{W^1} \psi_1 } - \frac{\int_{W^2} f \psi_2}{\fint_{W^2} \psi_2 } \right|
& \le \left| \int_{W^1} f \psi_1 - \int_{W^2} f \psi_2 \right|
+ \int_{W^2} |f| \psi_2 \left| 1 - \frac{1}{\fint_{W^2} f \psi_2}  \right| \\
& \le \sum_j \left| \int_{U^1_j} f \psi_1 - \int_{U^2_j} f \psi_2 \right| + \sum_{k,i} \left| \int_{V^k_i} f \psi_k \right| 
+ |f|_\infty \left| \int_{W^2} \psi_2 - |W^2| \right| 
\end{split}
\end{equation}
To estimate the first term on the right hand side, recalling \eqref{eq:psi-distance} and $d_*(\psi_1, \psi_2) = 0$, we have
for $r \in I_j$,
\[
\begin{split}
| & (f\psi_1) \circ G_{U^1_j}(r)  \| G'_{U^1_j}(r) \| - (f\psi_2) \circ G_{U^2_j}(r) \| G'_{U^2_j}(r) \| | \\
& = \psi_1 \circ G_{U^1_j}(r) \| G'_{U^1_j}(r) \| | f \circ G_{U^1_j}(r) - f \circ G_{U^2_j}(r) | 
\le \psi_1 \circ G_{U^1_j}(r) \| G'_{U^1_j}(r) \| H^t_P(f) d_{\cW^s}(W^1, W^2)^t \, ,
\end{split}
\]
where $H^t_P(f)$ denotes the H\"older constant of $f$ on $P \in \pa$.
Integrating over $I_j$ yields,
\begin{equation}
\label{eq:C1 first}
 \left| \sum_j \int_{U^1_j} f \psi_1 - \int_{U^2_j} f \psi_2 \right|
 \le \sum_j H^t_\pa(f) d_{\cW^s}(W^1, W^2)^t \int_{U^1_j} \psi_1
 \le |W^1| H^t_\pa(f) d_{\cW^s}(W^1, W^2)^t \, . 
\end{equation}
For the second term on the right side of \eqref{eq:C1 split},
$|V^k_i| \le C_s C_\pa d_{\cW^s}(W^1, W^2)$ plus \eqref{eq:psi2} implies
\begin{equation}
\label{eq:C1 second}
\sum_{k,i} \left| \int_{V^k_i} f \psi_k \right| \le 2K
 |f|_\infty  2 e^{2a(2\delta)^\alpha} C_s C_\pa d_{\cW^s}(W_1, W_2) \, ,
\end{equation}
while for the third term, \eqref{eq:new-difference} implies
\[
|f|_\infty \left| \int_{W^2} \psi_2 - |W^2| \right| \le |f|_\infty 6 C_s d_{\cW^s}(W^1, W^2) \, .
\]
Collecting this estimate together with \eqref{eq:C1 first} and \eqref{eq:C1 second} 
in \eqref{eq:C1 split}, and recalling \eqref{eq:delta_0 condition}, we obtain
\[
\left| \frac{\int_{W^1} f \psi_1}{\fint_{W^1} \psi_1 } - \frac{\int_{W^2} f \psi_2}{\fint_{W^2} \psi_2 } \right|
\le  2\delta H^t_\pa(f) d_{\cW^s}(W^1, W^2)^t + |f|_\infty (8K C_\pa  + 6) C_s d_{\cW^s}(W^1, W^2) 
 \, ,
\]
proving the bound in \eqref{eq:C1 bound} since $d_{\cW^s}(W_1, W_2) \le \delta$ and $t \ge \gamma$.

With the claim proved, we proceed to verify \eqref{eq:cone 5}.  Using \eqref{eq:W-difference} we estimate,
\[
\begin{split}
\left| \frac{\int_{W^1} ( f+ \lambda) \psi_1}{\fint_{W^1} \psi_1 } \right. & - \left. \frac{\int_{W^2} (f + \lambda) \psi_2}{\fint_{W^2} \psi_2 } \right| 
\le \left| \frac{\int_{W^1} f \psi_1}{\fint_{W^1} \psi_1 } - \frac{\int_{W^2} f \psi_2}{\fint_{W^2} \psi_2 } \right| + \lambda | |W^1| - |W^2| | \\
& \le (2 \delta^t + 8 K C_\pa C_s + 6 C_s) \delta^{1-\gamma} d_{\cW^s}(W_1, W_2)^\gamma |f|_{C^t(\pa)} + \lambda 2 C_s d_{\cW^s}(W^1, W^2) \, .
\end{split}
\]
Thus \eqref{eq:cone 5} will be verified if
\[
\begin{split}
(2 \delta^t + 8 K C_\pa C_s + 6 C_s) \delta^{1-\gamma} & d_{\cW^s}(W_1, W_2)^\gamma |f|_{C^t(\pa)} 
 + \lambda 2 C_s d_{\cW^s}(W^1, W^2) \\
& \le cA \delta^{1-\gamma} d_{\cW^s}(W_1, W_2)^\gamma (\lambda - |f|_\infty) \, ,
\end{split}
\]
which is implied by the final condition on $\lambda$ in the statement of the lemma since
$d_{\cW^s}(W_1, W_2) \le \delta$ and $cA > 2C_s$ by \eqref{eq:cA}.
\end{proof}


\section{Applications}\label{sec:appl}

Suppose that we have a billiard table $Q = \mathbb{T}^2 \setminus \cup_i B_i$
and that the particle can escape from the table by entering certain
sets $\bG \subset Q$, which we call {\em gates} or {\em holes}, 
but only at times $k N$ for some $N\in\bN$. 
One could easily consider also the case of $\bG \subset Q\times S^1$ (i.e. some velocity directions are forbidden,
as studied in \cite{dem bill}), but we prefer to keep things simple. 
In the literature, one often takes $N=1$, i.e. the particle can escape at each iterate of the map, but then the
holes are required to be very small, see for example \cite{DWY, dem inf, dem bill}. 
By contrast, in this paper we will be interested in relatively large holes and so we will replace the 
assumption of smallness with an assumption of occasional escape through possibly large sets.  This
will facilitate the application of this method to two situations we have in mind:  chaotic scattering (Section~\ref{sec:scattering})
and a random Lorentz gas (Section~\ref{sec:lorentz}).

We begin with the same setup as in Section~\ref{sec:bill family}, fixing $K$ numbers
$\ell_1, \ldots, \ell_K >0$ and identifying them as the arclengths of scatterers belonging to
$\cQ(\tau_*, \cK_*, E_*)$ for some fixed choice of $\tau_*, \cK_*, E_* \in \mathbb{R}^+$.  
As in Section~\ref{sec:uni hyp}, we fix an index set $\cI(\tau_*, \cK_*, E_*)$, identifying $\iota \in \cI(\tau_*, \cK_*, E_*)$ with a map $T_\iota \in \cF(\tau_*, \cK_*, E_*)$ induced by the table $Q_\iota \in \cQ(\tau_*, \cK_*, E_*)$.

A hole $\bG_\iota \subset Q_\iota$ induces a hole $H_\iota \subset M$ in the phase space of the
billiard map $T_\iota$.  We formulate here two abstract conditions on the set $H_\iota$, and then 
provide examples of concrete, physically relevant situations which induce holes
satisfying our conditions in Section~\ref{sequential}.  
\begin{itemize}
  \item[(O1)]  (Complexity)  There exists $P_0 >0$ such that any stable curve of
  length at most $\delta$ can be cut into at most $P_0$ pieces by $\partial H_\iota$, where $\delta$
  is the length scale of the cone $\cC_{c, A, L}(\delta)$.
  \item[(O2)]  (Uniform transversality)  There exists $C_t > 0$ such that, for any stable curve
  $W \in \cW^s$ and $\ve>0$, $m_W(N_\ve(\partial H_\iota)) \le C_t \ve$, where $N_\ve(A)$ is the $\ve$-neighborhood of $A$ in $M$.
\end{itemize}

\begin{remark}
Assumption (O2) can be weakened to, e.g., $m_W(N_\ve(\partial H_\iota)) \le C_t \ve^{1/2}$,
but this would then require $d_{\cW^s}(W^1, W^2) \le \delta^2$ in our definition of cone
condition \eqref{eq:cone 5}.  Similar modifications are made to weaken the transversality assumption in
the Banach space setting, see for example \cite{demzhang14, dem bill}.
\end{remark}

For $H \subset M$ satisfying (O1) and (O2), we let $\diam^s(H)$ denote the maximal length of a stable curve in $H$, which
we call the {\em stable diameter}.

As in Section~\ref{sec:mix},
we fix $T_0 \in \cF(\tau_*, \cK_*, \tau_*)$ and consider sequences $\{ \iota_j \}_j  \subset \cI(T_0, \kappa)$,
where $\kappa >0$ is from Lemma~\ref{lem:proper cross}(b).  Recalling \eqref{eq:close d},
this means we will initially consider sequential open systems comprised of maps $T \in \cF(\tau_*, \cK_*, E_*)$
with $\bd(Q(T), Q(T_0)) < \kappa$.
We will then extend this to $n_\star$-admissible sequences for appropriate $n_\star$ depending on $H$.

Denote by $\ind_{A}$ the characteristic function of the set $A$.  
The relevant transfer operator for the open
sequential system (opening once every $n_\star$ iterates) is given by 
$\cL_{H, n_\star}=\cL_{n_\star} \ind_{H^c}$, where $H^c$ denotes the
complement of $H$ in $M$, and $\cL_{n_\star} = \cL_{T_{\iota_{n_\star}}} \cdots \cL_{T_{\iota_1}}$ is the usual transfer 
operator for the $n_\star$-step sequential dynamics.
The main objective is to control the action of the multiplication operator $\ind_{H^c}$
on the cone $\cC_{c,A,L}(\delta)$.  
\begin{remark}\label{conefixed}
From now on we will consider parameters $c,A,L$ fixed so that all the conditions of Section \ref{sec:conditions} apply for all $\delta$ smaller that some fixed $\delta_*$.
\end{remark}

\subsection{ Relatively small holes}
\label{sec:small}

First we consider holes $H$ whose stable diameter is short compared to the length scale $\delta$. 

\begin{lemma}\label{lem:case-a} 
If $H \subset M$ satisfies (O1) and (O2), and if  $\diam^s(H)  \le \delta \left[ \frac{1}{4P_0A } \right]^{1/q}$, 
then 
\[
\ind_{H^c}[\cC_{c,A,L}(\delta)]\subset  \cC_{c',A',L'}(\delta),
\]
where 
\[
\begin{split}
&L'=2 P_0^{1-q} e^{a(2\delta)^\beta} A  \,  , \quad
A'= 2 P_0^{1-q} e^{a(2\delta)^\beta} A \, , \\
&c'=   P_0^q e^{a (2\delta)^\alpha} +
2 \left( 2^q \delta+ \tfrac 34  c \right) +  4  (P_0 + 2) P_0^{q-1} C_t^q   \, .
\end{split}
\]
\end{lemma}
\begin{proof}
Letting $f \in \cC_{c, A, L}(\delta)$, we must control the cone conditions one by one. We begin with \eqref{eq:cone 2}.
Given $W\in\cW^s(\delta)$, let $\cG_0$ denote the collection of connected curves in $W\setminus H$. 
Then applying \eqref{eq:cone 3} to each $W' \in \cG_0$, 
for $\psi \in \cD_{a, \beta}(W)$, we estimate
\begin{equation}\label{eq:appli1}
\begin{split}
\int_W(\ind_{H^c} f )\psi \, dm_W&= \sum_{W'\in\cG_0}\int_{W'}f \psi \, dm_{W'}\\
&\leq \sum_{W'\in\cG_0}\fint_{W'}\psi dm_{W'} |W'|^{q}A\delta^{1-q} \tri f\tri_-\\
&\leq \sum_{W'\in\cG_0}|W'|^{q} e^{a(2\delta)^\beta}\fint_{W}\psi dm_{W} A\delta^{1-q} \tri f\tri_-\\
&\leq  P_0^{1-q} e^{a(2\delta)^\beta} A  \tri f\tri_- \int_{W}\psi \, dm_{W},
\end{split}
\end{equation}
where, in the last line, we have used the H\"older inequality to estimate the sum on $W'$, recalling that, by (O1), the sum has, at most, $P_0$ elements.
On the other hand, if the collection of disjoint curves $\{W_i\}$ is such that $\cup_i W_i=W\cap H$,
\[
\begin{split}
\int_W(\ind_{H^c} f )\psi \, dm_W&= \int_{W} f \psi \, dm_{W} -\int_W(\ind_{H} f )\psi \, dm_W\\
&\geq \tri f\tri_-\int_{W} \psi \, dm_{W}-\sum_i |W_i|^qA\delta^{1-q}\tri f\tri_-\fint_{W_i} \psi \, dm_{W_i}\\
&\geq \left\{1- e^{a(2\delta)^\beta} A P_0 \delta^{-q} \diam^s(H)^q  \right\} \tri f\tri_- \int_{W}\psi \, dm_{W} \, .
\end{split}
\]
Hence, for $\diam^s(H)$ small enough,
\begin{equation}
\label{eq:lower H}
\tri \ind_{H^c}f\tri_-\geq \frac 12 \tri f\tri_-.
\end{equation}
Accordingly, taking the supremum over $W,\psi$ in \eqref{eq:appli1},
\[
 \tri \ind_{H^c} f\tri_+\leq 2  P_0^{1-q} e^{a(2\delta)^\beta} A \tri \ind_{H^c} f\tri_-=:L' \tri \ind_{H^c} f\tri_- 
\]
Next, to verify \eqref{eq:cone 3}, if $W\in\cW^s_-(\delta)$, then estimating
as in \eqref{eq:appli1},
\begin{equation}\label{eq:3H}
\begin{split}
\int_W(\ind_{H^c} f )\psi \, dm_W&=\sum_{W'\in\cG_0}\int_{W'}f \psi \, dm_{W'}\\
&\leq\sum_{W'\in\cG_0}e^{a(2\delta)^\beta} |W'|^{q}A\delta^{1-q} \tri f\tri_- \fint_{W}\psi \, dm_{W}\\
&\leq P_0^{1-q}|W|^q e^{a(2\delta)^\beta} A\delta^{1-q} \tri f\tri_- \fint_{W}\psi \, dm_{W}\\
&\leq 2 P_0^{1-q}|W|^q e^{a(2\delta)^\beta} A\delta^{1-q} \tri\ind_{H^c} f\tri_- \fint_{W}\psi \, dm_{W}\\
&=: A'|W|^q \delta^{1-q} \tri\ind_{H^c} f\tri_- \fint_{W}\psi \, dm_W \, ,
\end{split}
\end{equation}
where we have used \eqref{eq:lower H} for the third inequality.

We are left with the last cone condition, \eqref{eq:cone 5}. We take $W^1, W^2 \in \cW^s_-(\delta)$ with $d_{\cW^s}(W^1, W^2) \le \delta $,
and $\psi_i \in \cD_{a,\alpha}(W_i)$ with $d_*(\psi_1, \psi_2) = 0$.

As in Section~\ref{subsec:contract c}, we may assume without loss of generality that $|W^2|\geq |W^1|$ and $\fint_{W^1} \psi_1 = 1$. First of all note that, by condition \eqref{eq:cone 3} and our estimate above, 
\[
\frac{\int_{W^k} \ind_{H^c} f\psi_k}{\fint_{W^k}\psi_k}\leq A' |W^k|^q\delta^{1-q}\tri \ind_{H^c} f\tri_-
\leq \frac 12 d_{\cW^s}(W^1,W^2)^\gamma \delta^{1-\gamma}  c A' \tri \ind_{H^c} f\tri_-,
\]
for $k=1,2$,
provided $ |W^2|^q \leq \delta^{q- \gamma} \frac c2 d_{\cW^s}(W^1,W^2)^{\gamma}$.  
Accordingly, it suffices to consider the case $|W^2|^q \geq \delta^{q-\gamma} \frac c2 d_{\cW^s}(W^1,W^2)^{\gamma}$. 

It follows from \eqref{eq:W-difference} that  
$|W^1|^q \geq \frac 12 \delta^{q-\gamma} \frac c2 d_{\cW^s}(W^1,W^2)^\gamma$, recalling that $d_{\cW^s}(W^1, W^2) \le \delta$ and \eqref{eq:q-gamma1}.
By (O2), we may decompose $W^k \cap H^c$ into at most $P_0$ `matched' pieces 
$W_j^k$ such that  $d_{\cW^s}(W_j^1,W_j^2)\leq d_{\cW^s}(W^1,W^2)$ and $I_{W^1_j} = I_{W^2_j}$,
and  at most\footnote{According to (O1), $W^k$ is divided into at most $P_0$ pieces, and 
$W^k \cap H^c$ comprises at most $\frac{P_0}{2}+1$ of them.  Each such piece can give rise
to at most 2 unmatched pieces.} 
$P_0 + 2$ `unmatched' pieces $\wcW^k_i$, which satisfy, 
\[
|\wcW^k_i| \leq C_t d_{\cW^s}(W^1,W^2).
\]
Then, using condition \eqref{eq:cone 3} and noticing that $d_*(\psi_1|_{W^1_j},\psi_2|_{W^2_j})=0$, 
\begin{equation}
\label{eq:indicator c}
\begin{split}
&\left|\frac{\int_{W^1} \ind_{H^c} f \psi_1}{\fint_{W^1}\psi_1}  - \frac{\int_{W^2} \ind_{H^c} f \psi_2}{\fint_{W^2}\psi_2} \right|\leq
\sum_j\left|\frac{\int_{W^1_j} f \psi_1}{\fint_{W^1}\psi_1}  - \frac{\int_{W^2_j} f \psi_2}{\fint_{W^2}\psi_2} \right|+
\sum_{i,k} |\wcW^k_i|^q \delta^{1-q} A \tri f \tri_- e^{a(2\delta)^\alpha} \\
&\leq \; \sum_j \frac{\fint_{W^1_j}  \psi_1}{\fint_{W^1}\psi_1}d_{\cW^s}(W^1, W^2)^\gamma \delta^{1-\gamma}  c A \tri f \tri_-  
+\sum_j \left|\frac{\int_{W^2_j} f \psi_2}{\fint_{W^2}\psi_2}\left[ 1- \frac{\fint_{W^1_j}  \psi_1\fint_{W^2}  \psi_2}{\fint_{W_j^2}  \psi_2\fint_{W^1}\psi_1}\right]\right|\\
&\phantom{\leq}
\quad + 8 (P_0+2) C_t^q d_{\cW^s}(W^1,W^2)^{\gamma} \delta^{1-\gamma} A \tri \ind_{H^c} f\tri_-,
\end{split}
\end{equation}
using \eqref{eq:lower H}.  Next, since $I_{W^1_j} = I_{W^2_j}$, recalling
Remark~\ref{rem:change-int} and \eqref{eq:W-difference} we have 
$\int_{W^1_j} \psi_1 = \int_{W^2_j} \psi_2$ and\footnote{Since $I_{W^1_j} = I_{W^2_j}$,
the term on the right side of \eqref{eq:W-difference} proportional to $C_s$ is absent in this case.} $||W^1_j| - |W^2_j|| \le |W^1_j| d_{\cW^s}(W^1, W^2)$.
Then applying \eqref{eq:cone 3} and recalling $\fint_{W_1} \psi_1 = 1$,
\begin{equation}
\label{eq:indicator split}
\begin{split}
& \left|\frac{\int_{W^2_j} f \psi_2}{\fint_{W^2}\psi_2} \left[ 1- \frac{\fint_{W^1_j}  \psi_1\fint_{W^2}  \psi_2}{\fint_{W_j^2}  \psi_2\fint_{W^1}\psi_1}\right]\right|
\le A \tri f \tri_- \frac{\fint_{W^2_j} \psi_2}{\fint_{W^2} \psi_2} |W^2_j|^q \delta^{1-q}
\left| 1 - \frac{|W^2_j|}{|W^1_j|} \fint_{W^2} \psi_2 \right| \\
& \; \; \le A \tri f \tri_- e^{a (2\delta)^\alpha} \left( |W^2_j|^q \delta^{1-q} \left|1- \frac{|W^2_j|}{|W^1_j|} \right|
+ \frac{|W^2_j|^q}{|W^2|^q} \left(\frac{\delta}{|W^2|} \right)^{1-q} \left| |W^2| - \int_{W^2} \psi_2 \right| \frac{|W^2_j|}{|W^1_j|} \right) \\
& \; \; \le A \tri f \tri_- 2 \left( |W^2_j|^q \delta^{1-q} d_{\cW^s}(W^1, W^2)
+ 2 \frac{|W^2_j|^q}{|W^2|^q} \left(\frac{\delta}{|W^2|} \right)^{1-q} \left| |W^2| - \int_{W^2} \psi_2 \right|  \right).
\end{split}
\end{equation}
Next, recalling $|W^2| \ge  \delta^{1 - \frac{\gamma}{q} } [c/2]^{\frac 1q} d_{\cW^s}(W^1, W^2)^{\frac{\gamma}{q}}$ 
and using \eqref{eq:new-difference} yields,
\[
\begin{split}
 \left(\frac{\delta}{|W^2|} \right)^{1-q} \left| |W^2| - \int_{W^2} \psi_2 \right|  
& \le 6 C_s [2/c]^{\frac{1}{q}(1-q)} { \delta^{\frac{\gamma}{q} - \gamma} } d_{\cW^s}(W^1, W^2)^{1+\gamma - \frac{\gamma}{q}} \\
& \le 4^{-\frac{1}{q} 6 c \delta^{1-\gamma} } d_{\cW^s}(W^1, W^2)^\gamma \, ,
\end{split}
\]
where we have again used \eqref{eq:q-gamma1} and $d_{\cW^s}(W^1, W^2) \le \delta$.
Using this estimate and the fact that $q \le 1/2$ in \eqref{eq:indicator split} and summing over $j$ yields,
\[
\begin{split}
\sum_j \left|\frac{\int_{W^2_j} f \psi_2}{\fint_{W^2}\psi_2} \left[ 1- \frac{\fint_{W^1_j}  \psi_1\fint_{W^2}  \psi_2}{\fint_{W_j^2}  \psi_2\fint_{W^1}\psi_1}\right]\right|
& \le 2 A  \delta^{1-\gamma}  \tri f \tri_- d_{\cW^s}(W^1, W^2)^\gamma \sum_j \delta^{1-q}
|W^2_j|^q +  \frac 34  c \frac{|W^2_j|^q}{|W^2|^q} \\
& \le 2 A \delta^{1-\gamma} \tri f \tri_- d_{\cW^s}(W^1, W^2)^\gamma P_0^{1-q} \left( 2^q \delta +  \tfrac 34  c \right) \, .
\end{split}
\]
Finally, using this estimate in \eqref{eq:indicator c} concludes the proof of the lemma,
\[
\begin{split}
\left|\frac{\int_{W^1} \ind_{H^c} f \psi_1}{\fint_{W^1}\psi_1}  - \frac{\int_{W^2} \ind_{H^c} f \psi_2}{\fint_{W^2}\psi_2} \right|
& \leq
d_{\cW^s}(W^1, W^2)^\gamma  \delta^{1-\gamma}  A 2 P_0^{1-q}  \tri \ind_{H^c} f \tri_- 
\left( P_0^q e^{a (2\delta)^\alpha} \; + \right.  \\
& \quad \left. + \; 2 \left( 2^q \delta + \tfrac 34  c \right) +  4   (P_0 + 2) P_0^{q-1} C_t^q \right) \, ,
\end{split}
\]
where we have again used \eqref{eq:lower H}.
\end{proof}

Remark that, by Theorem~\ref{thm:cone contract}, we know that there exists 
$N_{\cF}\in\bN$, $N_{\cF}\leq k_*n_*+C_{\star}|\ln\delta|$ where $n_*$, defined in Lemma \ref{lem:proper cross}, and $k_*$ from Theorem~\ref{thm:cone contract}, are uniform in $\cF(\tau_*, \cK_*, E_*)$, while $C_\star$ depends only on $c,A,L$, such that 
$\Lp_{N_{\cF}} \cC_{c,A,L}(\delta) \subset \cC_{\chi c, \chi A, \chi L}(\delta)$, for all $\{ \iota_j \}_{j=1}^{N_\cF} \subset \cI(T_0, \kappa)$. 

To state the next result we need to make explicit the coice of the cone parameters. Let $c', A', L'$ be given by Lemma~\ref{lem:case-a}.   Choose (minimal)
$c'' \ge c'$, $A'' \ge A'$ and $L'' \ge L'$ satisfying the conditions in Section~\ref{sec:conditions}, and 
$\delta'' \le \delta$ satisfying \eqref{eq:delta_0 ineq} and \eqref{eq:A-cond-s4} with respect to $A''$ and $L''$.

Define $N_{\cF}' = k_*n_* + C'_\star |\ln \delta''|$, where $C'_\star$, $k_*$ and $n_*$
are from Theorem~\ref{thm:cone contract} applied to the cone $\cC_{c'',A'',L''}(\delta'')$.  
Since, as remarked in Proposition~\ref{prop:almost}
and Theorem~\ref{thm:cone contract}, $\chi$ is independent of the cone parameters,
we have $\cL_n \cC_{c'',A'',L''}(\delta'') \subset \cC_{\chi c'', \chi A'', \chi L''}(\delta'')$ for all $n \ge N_{\cF}'$.

Recall that $\kappa>0$ from Lemma~\ref{lem:proper cross} depends only on the family
$\cF(\tau_*, \cK_*, E_*)$.

\begin{prop}\label{prop:case-a} 
Let $n_\star = N_{\cF}'$.  There exists $J \in \mathbb{N}$, depending only on $c,A,L, P_0,C_t$, 
such that if assumptions (O1) and (O2) are satisfied 
and $\diam^s(H) \le {\delta''} \left[ \frac{1}{4P_0A} \right]^{1/q}$, then there exists $\chi' \in (0,1)$ such that 
for all $n\geq J n_\star$, and all $n_\star$-admissible sequences $(\iota_j )_{j \ge 1}$,
$[\cL_n\ind_{H^c}]\cC_{c,A,L}(\delta'')\subset  \cC_{\chi' c,\chi' A,\chi' L}(\delta'')$.
\end{prop}
\begin{proof}
For $n = m N_{\cF}'$,
we may apply both Lemma~\ref{lem:case-a} and Theorem~\ref{thm:cone contract} to obtain,
\[
[\Lp_n \ind_{H^c}] \cC_{c,A,L}(\delta'') \subset \Lp_{mN_{\cF}'} \cC_{c'',A'',L''}(\delta'')
\le \cC_{\chi^m c'', \chi^m A'', \chi^m L''}(\delta'') \, ,
\]
for as long as $\chi^m c'' > c$, $\chi^m A'' > A$ or $\chi^m L'' > L$.  Letting $m_1$ denote
the greatest $m$ such that $\chi^m c'' > c$ or $\chi^m A'' > A$  or $\chi^m L'' > L$, and setting
$J = m_1+1$ produces the required contraction.
\end{proof}

\begin{remark} 
Taking $\kappa=0$ we can also consider the case of a single map, $T_{\iota_j} = T_0$ for each $j$.
Then once we know the transfer operator for the open system acts as a strict contraction on the cone, it is straightforward to
recover the usual full set of results for open systems with exponential escape, including a unique escape
rate and limiting conditional invariant measure for all elements of the cone.  See Theorem~\ref{thm:open}
for an example. 
\end{remark}


\subsection{ Large holes}
\label{sec:large}

The preceding pertains to relatively small holes. For many applications, large holes must be considered. 
To do so requires either a much closer look at the combinatorics of the trajectories or requiring the holes to open at even longer time intervals than what was needed before. 
We will pursue the second, much easier, option with the intent to show that large holes are not 
out of reach. To work with large holes it is convenient to strengthen hypothesis (O1):
\begin{itemize}
 \item[(O$1'$)]  (Complexity)  There exists $P_0 >0$ such that any stable curve of
  length at most $\delta_0$ can be cut into at most $P_0$ pieces by $\partial H$.
 \end{itemize}
 The main difference between small and large holes is that, according to Lemma~\ref{lem:case-a},
 for holes with sufficiently small stable diameter, multiplication by $\ind_{H^c}$ maps $\cC_{c,A,L}(\delta)$
 into a cone with larger parameters; by contrast, for large holes, multiplying by the indicator function may produce functions that do not belong to any cone and we must use mixing to recover this property, as detailed in Lemma~\ref{lem:large hole}.
To avoid trivialities, we only consider holes with $\musrb(H) < 1$.
 
When iterating $T^{-1}_nW$ for $W \in \cW^s$, we will need to distinguish between elements of 
$\G_n(W)$ which intersect $H$ and those that do not.  Recall that $\G_n(W)$ subdivides
long homogeneous connected components of $T^{-1}_nW$ into curves of length between
$\delta_0$ and $\delta_0/3$.  We let $\G_n^H(W)$ denote the connected components
of $W_i \cap H^c$, for $W_i \in \G_n(W)$, where $H^c = M \setminus H$.  Following the notation of 
Section~\ref{sec:prop proof}, let $Lo_n^H(W; \delta)$ denote those elements of $\G_n^H(W)$
having length at least $\delta$ and let $Sh_n^H(W; \delta)$ denote those elements having length
at most $\delta$. 
 
Without the small hole condition, hypotheses (O$1'$) and (O2) are insufficient to prove Lemma~\ref{lem:case-a};
however, one can recover the results of Proposition \ref{prop:case-a} and its consequences
provided one is willing to wait for a longer time. 
To prove the following result, we recall again Definition~\ref{def:admissible} of 
admissible sequence.
We call a sequence $( \iota_j )_{j \ge 1}$, $\iota_j \in \cI(\tau_*, \cK_*, E_*)$, {\em $N$-admissible} if
there exist sequences $(T_k)_{k \ge 1} \subset \cF(\tau_*, \cK_*, E_*)$ and $(N_k)_{k \ge 1}$ 
with $N_k \ge N$, such that $T_{\iota_j} \in \cF(T_k, \kappa)$
for all $k \ge 1$ and $j \in [1 + \sum_{i=1}^{k-1} N_i, \sum_{i=1}^k N_i]$.

\begin{lemma}\label{lem:H3}
If (O1$'$) and (O2) are satisfied, then for each $\delta>0$ small enough (depending on $\musrb(H)$) there exists $n_\delta\in\bN$, $n_\delta\leq C\ln\delta^{-1}$ for some constant $C>0$, such that for all 
$n_\delta$-admissible sequences $(\iota_j)_j$,
all $W\in\cW^s(\delta)$ and $n\geq n_\delta$,
 \[
 \sum_{W'\in Lo^H_n(W; \, \delta)} |W|^{-1}\int_{W'} J_{W'}T_n {\; \geq\;} \frac 12(1 - \musrb(H)) \, .
 \]
 \end{lemma}
\begin{proof}
Arguing exactly as in  Lemma \ref{lem:case-a} it follows that if (O$1'$) and (O2) are satisfied, then there exists $c'\geq c,A'\geq A, L'\geq L$ such that $\ind_{H^c}+  1  \in \cC_{c',A',L'}(\delta)$
and we may choose $c', A', L'$ and $\delta>0$ such that the conditions of
Theorem~\ref{thm:cone contract} are satisfied. 
Setting $n_\delta := N_{\cF}'$ from Theorem~\ref{thm:cone contract} for these cone parameters, we apply
equation~\eqref{eq:avg conv} of Theorem~\ref{thm:equi} to this larger cone,
\[
\left|\fint_W \cL_n (\ind_{H^c}) - (1-\musrb(H)  ) \right|=\left|\fint_W \cL_n (\ind_{H^c}+1) -  2
+  \musrb(H)\right|\leq C_H\vartheta^n \, .
\]
On the other hand, recalling Lemma \ref{lem:full growth},
\[
\begin{split}
\left| \fint_W \cL_n (\ind_{H^c})-\sum_{W'\in Lo^H_n(W; \, \delta)} |W|^{-1}\int_{W'} J_{W'}T_n \right|&\leq \sum_{W'\in Sh^H_n(W; \, \delta)} |W|^{-1}\int_{W'} J_{W'}T_n\\
&\leq P_0(\bar  C_0\delta_0^{-1}\delta+C_0\theta_1^n),
\end{split}
\] 
which implies the lemma.
\end{proof}

We are now able to state the analogue of Proposition \ref{prop:case-a} without the small hole condition.  Note, however, that now $n_\star$ has a worse dependence on $\delta$ that we refrain from making explicit.
We recall from Remark~\ref{conefixed} that we have fixed the parameters $c, A, L$ of the cone, but
we may choose $\delta$ smaller as needed.

\begin{prop}\label{prop:case-b} 
Under assumptions (O1$'$) and (O2), for each $\delta>0$ small enough there exist
$\chi' \in (0,1)$ and $J, n_\star\in\bN$ depending on (O$1'$), (O2), $\musrb(H)$ and $\delta$, such that, for all 
$n_\star$-admissible sequences $(\iota_j)_j$ and for all $n\geq J n_\star$,
$[\cL_n\ind_{H^c}]\cC_{c,A,L}(\delta)\subset  \cC_{\chi' c,\chi' A,\chi' L}(\delta)$.
\end{prop}

Before proving Proposition \ref{prop:case-b}, we state an auxiliary lemma, similar to Lemma~\ref{lem:case-a}.

\begin{lemma}
\label{lem:large hole}
If $H$ satisfies (O$1'$) and (O2),
there exists\footnote{ Since we have fixed the cone constants $c,A,L$, the number $\bar n_\delta$
depends on the constants appearing in (O$1'$) and (O2) as well as $\musrb(H)$ and the choice
of $\delta$, from Lemma~\ref{lem:H3}.} 
$\bar n_\delta \ge n_\delta$ such that  for $n \ge \bar n_\delta$ and all $n_\delta$-admissible sequences $(\iota_j)_j$, we have
$[\Lp_n\ind_{H^c}] \cC_{c,A,L}(\delta) \subset \cC_{c', A', L'}(\delta)$, where
\[
c' = cP_0, \quad A' = A \frac{6}{1-\musrb(H)}, \quad \mbox{and} \quad L' = L \frac{9}{1-\musrb(H)} \, .
\]
\end{lemma}

\begin{proof}[Proof of Proposition~\ref{prop:case-b}]
As in Section~\ref{sec:small}, we may choose minimal $c'' \ge c'$, $A'' \ge A'$ and $L'' \ge L'$
and $\delta >0$ sufficiently small to satisfy the hypotheses of Theorem~\ref{thm:cone contract}.
Then letting $n_\star = \max \{ N_{\cF}' , \bar n_\delta \}$, with\footnote{$N_{\cF}'$ is number from
Theorem~\ref{thm:cone contract} applied to the cone with larger constants $c'', A'', L''$.} 
$ N_{\cF}' = C'_\star |\ln \delta|  + k_*n_*$ as before,
we may apply both Lemma~\ref{lem:large hole} and Theorem~\ref{thm:cone contract} to obtain,
\[
[\Lp_n \ind_{H^c}] \cC_{c,A,L}(\delta) \subset \Lp_{m n_\star} \cC_{c'',A'',L''}(\delta)
\le \cC_{\chi^m c'', \chi^m A'', \chi^m L''}(\delta) \, ,
\]
for as long as $\chi^m c'' > c$ or $\chi^m A'' > A$ or $\chi^m L'' > L$.  Letting $m_1$ denote
the greatest  $m$ such that $\chi^m c''  > c$, $\chi^m A''  > A$ or $\chi^m L''  >  L$, and setting
$J = m_1+1$ produces the required contraction.
\end{proof}

\begin{proof}[Proof of Lemma~\ref{lem:large hole}]
Let $n\geq n_\delta$ (from Lemma \ref{lem:H3}) and $f \in \cC_{c, A, L}(\delta)$.  For each $W\in  \cW^s(\delta)$ and $\psi  \in \cD_{a,\beta}(W)$, we have
\begin{equation}
\label{eq:H split}
\int_W\psi \, \cL_n ( \ind_{H^c} f )=\sum_{W_i\in Lo^H_{n}(W; \delta)}\int_{W_i}  \hT_{n,i} \psi \, f  + \sum_{W_i \in Sh^H_{n}(W; \delta)}\int_{W_i }  \hT_{n,i} \psi \, f ,
\end{equation}
where we are using the notation of  Section~\ref{sec:test} for the test functions.  
Since any element of $\G_n(W)$ may produce up to $P_0$ elements of  
$Sh^H_{n}(W; \delta)$ according to assumption (O$1'$), we estimate
\[
\begin{split}
\int_W \psi \, \cL_n ( \ind_{H^c} f ) & \le  \sum_{W_i \in Lo_n^H(W;  \delta)} \tri f \tri_+  \int_{T_nW_i} \psi + 
A P_0 \tri f \tri_- e^{a(2\delta)^\beta} \int_W \psi \; (\bar C_0 \delta \delta_0^{-1} + C_0 \theta_1^n) \\
& \le \tri f \tri_+ \int_W \psi \; \left( 1 + A P_0e^{a(2\delta)^\beta} (\bar C_0 \delta \delta_0^{-1} 
+ C_0 \theta_1^n) \right) \, ,
\end{split}
\]
where we have used $|W| \ge \delta$ and cone condition \eqref{eq:cone 3},  as well as
Lemma~\ref{lem:full growth}(b) to sum over elements of $Sh^H_n(W; \delta)$.
 
Analogously, using Lemma \ref{lem:H3},
\[
\begin{split}
\int_W \psi \, \cL_n ( \ind_{H^c} f ) & \ge  \sum_{W_i \in Lo_n^H(W; \delta)} \tri f \tri_-  \int_{T_nW_i} \psi -
A P_0 \tri f \tri_- e^{a(2\delta)^\beta} \int_W  \psi  \; (\bar C_0 \delta \delta_0^{-1}+ C_0 \theta_1^n) \\
& \ge \tri f \tri_- \int_W \psi \; \left( \frac {e^{-a(2\delta)^\beta}}2(1-\musrb(H))- A P_0e^{a(2\delta)^\beta} (\bar C_0 \delta \delta_0^{-1} + C_0 \theta_1^n ) \right) \, .
\end{split}
\]
Let $n_2$ be such that $2AP_0C_0\theta_1^{n_2}\leq \frac {1}{24} (1- \musrb(H))$, then for $n\geq n_2$ and $\delta$ small enough we have
\begin{equation}
\label{eq:lower big H}
\tri \Lp_n(\ind_{H^c} f) \tri_- \ge \tri f \tri_-  \;  \frac{1}{6}(1-\musrb(H)).
\end{equation}
Accordingly, for $n \ge \max\{n_2, n_\delta \} =: \bar n_\delta$ and $\delta$ small enough, we obtain
\begin{equation}
\label{eq:L bound}
\frac{\tri \cL_n ( \ind_{H^c} f ) \tri_+ }{\tri \cL_n ( \ind_{H^c} f ) \tri_-}
\le \frac{ \frac 32 \tri f \tri_+}{\tri f \tri_- (\frac 16(1-\musrb(H))} \le \frac{9L}{1-\musrb(H)} =: L' \, .
\end{equation}

The contraction of $A$ follows step-by-step from our estimates in 
Section~\ref{sec:contraction-A}.  Taking $W \in \cW^s_-(\delta)$ and grouping terms as in \eqref{eq:H split} we treat both
long and short pieces precisely as in Section~\ref{sec:contraction-A} with the additional
observation that each element of $\G_n(W)$ produces at most $P_0$ elements
of $Sh_n^H(W; \delta)$ by assumption (O$1'$).  Thus \eqref{eq:A} becomes,
\begin{equation}
\label{eq:A H}
\begin{split}
& \frac{|\int_W \psi \, \Lp_n (\ind_{H^c}f) |}{\fint_W \psi} 
 \le A \delta^{1-q} |W|^q \tri f \tri_-
\left( 2 L A^{-1} + P_0 e^{a(2\delta)^\beta} ( \bar C_0 \delta_0^{-1} |W| + C_0 \theta_1^n )^{1-q}  \right) 
\\
& \quad \le A \delta^{1-q}|W|^q \tri \cL_n(\ind_{H^c}f) \tri_- \frac{6}{1-\musrb(H)} =: A' \delta^{1-q}|W|^q \tri \cL_n(\ind_{H^c}f) \tri_-\, ,
\end{split}
\end{equation}
where we have applied \eqref{eq:lower big H} and assumed $n \ge \max\{n_2, n_\delta \}$.

Finally, we show how the parameter $c$ contracts from cone condition \eqref{eq:cone 5}.
Following Section~\ref{subsec:contract c}, we take $W^1, W^2 \in \cW^s_-(\delta)$ with
$d_{\cW^s}(W^1, W^2) \le  \delta$, and $\psi_k \in \cD_{a,\alpha}(W^k)$ with 
$d_*(\psi_1, \psi_2) = 0$.  As before, we assume without loss of generality 
that $|W^2|\geq |W^1|$ and $\fint_{W^1} \psi_1 = 1$. 

We begin by recording that, by \eqref{eq:A H}, 
\[
\frac{\int_{W^k} \psi_k \, \Lp_n(\ind_{H^c} f)}{\fint_{W^k}\psi_k}\leq A' |W^k|^q\delta^{1-q}\tri \Lp_n (\ind_{H^c} f) \tri_-\leq \frac 12 d_{\cW^s}(W^1,W^2)^\gamma 
\delta^{1-\gamma}  c A' \tri \ind_{H^c} f\tri_-,
\]
for $k=1,2$, provided $|W^2|^q \leq \delta^{q-\gamma} \frac c2 d_{\cW^s}(W^1,W^2)^{\gamma}$.
Accordingly, it suffices to consider the case $|W^2|^q \geq \delta^{\gamma - q} \frac c2
d_{\cW^s}(W^1,W^2)^{\gamma}$. 

It follows from \eqref{eq:W-difference} that  
$|W^1|^q \geq \frac 12 \delta^{q-\gamma} \frac c2 d_{\cW^s}(W^1,W^2)^\gamma$, recalling that $d_{\cW^s}(W^1, W^2) \le  \delta$ and \eqref{eq:q-gamma1}.

Next, following \eqref{eq:prepare-c}, we decompose elements of $\G_n^H(W^k)$ into
matched and unmatched pieces, as in \eqref{eq:unstable split}.  We estimate the
unmatched pieces precisely as in \eqref{eq:V}, noting that by (O$1'$) and the transversality
condition $(O2)$, each previously unmatched element of
$\G_n(W^k)$ may be subdivided into at most $P_0$ additional unmatched pieces $V^k_j$, 
while each matched element may produce up to $P_0$ additional unmatched pieces
each having length at most,
\[
|V^k_j| \le C_t C_5 \Lambda^{-n} d_{\cW^s}(W^1, W^2) \, ,
\]
by Lemma~\ref{lem:compare}(a).  Thus,
\begin{equation}
\label{eq:V H}
\sum_{j,k} \left| \int_{V^k_j} f \, \hT_{n,V^k_j}\psi_k \right| \le \frac{9  P_0}{1-\musrb(H)} C_4 A L 
\delta^{1-\gamma}  d_{\cW^s}(W^1, W^2)^\gamma \tri \Lp_n(\ind_{H^c} f) \tri_- \, ,
\end{equation}
where we have used \eqref{eq:lower big H} in \eqref{eq:tri} to estimate
\begin{equation}
\label{eq:convert}
\tri \Lp_n f \tri_- \le \tri \Lp_n f \tri_+ \le \tfrac 32 \tri f \tri_+ \le \tfrac 32 L \tri f \tri_- \le \tfrac{9 L}{1-\musrb(H)} \tri \Lp_n(\ind_{H^c} f) \tri_- \, . 
\end{equation}

The estimate on matched pieces proceeds precisely as in \eqref{eq:c-decomposition}, and with an
additional factor of $P_0$ in \eqref{eq:summingU2}, we arrive at \eqref{eq:use-in-4},
again applying \eqref{eq:lower big H},
\[
\begin{split}
& \sum_j \left|\int_{U^1_j} f \, \hT_{n,U^1_j} \psi_1 -\int_{U^2_j} f \,  \hT_{n,U^2_j}\psi_2\right| \\
&\le 
\tfrac{6P_0}{1-\musrb(H)} 24 \bar C_0 C_s A \delta^{1-\gamma} d_{\cW^s}(W^1, W^2)^\gamma \tri \Lp_n(\ind_{H^c} f) \tri_- \left( 2^q 40 C_5 \delta^{q-\gamma} 
+ c C_5 \Lambda^{-n \gamma}  
+  2^q C_5  \Lambda^{-n}  \delta \right) .
\end{split}
\]
Combining this estimate together with \eqref{eq:V H} in \eqref{eq:prepare-c} (with $A'$ in place of
$A$ in \eqref{eq:prepare-c}), and recalling \eqref{eq:unstable split}, yields by \eqref{eq:switch test},
\[
\left|\frac{\int_{W^1} \cL_nf \, \psi_1}{\fint_{W^1}\psi_1} -\frac{\int_{W^2} \cL_nf \, \psi_2}{\fint_{W^2}\psi_2}\right|
\le \frac{6 P_0 }{1-\mu(H)} c A  \delta^{1-\gamma}  d_{\cW^s}(W^1, W^2)^\gamma \tri \Lp_n(\ind_{H^c} f) \tri_- \, ,
\]
where we have applied \eqref{eq:c cond} to simplify the expression.  Setting $c' = P_0 c$ and
recalling the definition of $A'$ from \eqref{eq:A H} completes the proof of the lemma.
\end{proof}


\subsection{ Loss of memory for sequential open billiards}
\label{sequential}
We conclude the section by illustrating several physically relevant models to which our results apply. 
Admittedly, we cannot treat the most general cases, yet we believe the following shows convincingly that the 
techniques developed here can be the basis of a general theory.

Dispersing billiards with small holes have been studied in \cite{DWY, dem inf, dem bill},
and results obtained regarding the existence and uniqueness of limiting distributions
in the form of SRB-like conditionally invariant measures, and singular invariant measures
supported on the survivor set.  In the present context, we are interested in generalizing these
results to the non-stationary setting.
Analogous results for sequences of expanding maps with holes have been proved in
\cite{ott1, ott2}.

For concreteness, we give two example of physical holes that satisfy our hypotheses, 
following
\cite{DWY, dem bill}.

\smallskip
\noindent
{\em Holes of Type I.}  Let $\bG \subset \partial Q$ be an arc in the boundary of one of the
scatterers.  Trajectories of the billiard flow are absorbed when they collide with
$\bG$.  This induces a hole $H$ in the phase space $M$ of the billiard map of the
form $(a,b) \times [-\pi/2, \pi/2]$.  Note that $\partial H$ consists of two vertical lines, so that
$H$ satisfies assumption (O2) since the vertical direction is uniformly transverse to the
stable cone, as well as assumptions (O1) and (O1$'$) with $P_0 = 3$.

\smallskip
\noindent
{\em Holes of Type II.}  Let $\bG \subset Q$ be an open convex set bounded away from
$\partial Q$ and having a $C^3$ boundary.  Such a hole induces a hole $H$ in $M$ 
via its `forward shadow.'

We define $H$ to be the set of $(r,\vf) \in M$ whose backward trajectory under the billiard flow enters
$\bG$ before it collides with $\partial Q$.  Thus points in $M$ which are about to enter
$\bG$ before their next collision under the forward billiard flow are considered still in the open
system, while those points in $M$ which would have passed through $\bG$ on the way
to their current collision are considered to have been absorbed by the hole.

With this definition, the geometry of $H$ is simple to state:  if we view $\bG$ as an additional
scatterer in $Q$, then $H$ is simply the image of $\bG$ under the billiard map.  Thus
$H$ will have connected components on each scatterer that has a line of sight to $\bG$,
and $\partial H$ will comprise curves of the form $\cS_0 \cup T(\cS_0)$, which
are positively sloped curves, all uniformly transverse to the stable cone.  Thus holes 
of Type II satisfy (O2) as well as (O1) and (O1$'$) with $P_0=3$.  (See the discussion
in \cite[Section~2.2]{dem bill}.)

\medskip
Still other holes are presented in \cite{dem bill} such as side pockets, or holes that depend
on both position and angle, which satisfy (O1), (O1$'$) and (O2), but for the sake of brevity, we do not
repeat those definitions here.

As noted, both holes of Type I and Type II satisfy (O1) and (O1$'$) with $P_0=3$.  Moreover, holes
of Type I satisfy (O2) with $C_t$ depending only on the maximum slope of curves in the
stable cone, which is uniform in the family $\cF(\tau_*, \cK_*, E_*)$ according to {\bf (H1)}:  
this (negative) slope is bounded below by $-\cK_{\max} - \frac{1}{\tau_{\min}}$, so choosing
$C_t \ge \cK_* + \tau_*^{-1}$ suffices.  Since $\partial H$ for holes of Type II have positive
slope, the same choice of $C_t$ will suffice for such holes to satisfy (O2).

Fix $\cF(\tau_*, \cK_*, E_*)$ and define $\cH(P_0, C_t)$ to be the collection of holes 
$H \subset M$ with $\musrb(H) \le 1/2$ and satisfying (O1) or (O$1'$) and (O2) with the given constants $P_0$ and $C_t$. 
We define a non-stationary open billiard by 
fixing a sequence of holes $H_k \in \cH(P_0, C_t)$, $k \in \mathbb{Z}^+$, satisfying either 
(O1) and (O2) or (O$1'$) and (O2).  In the first case, let $n_\star$ be from 
Proposition~\ref{prop:case-a}, while in the second, let $n_\star$ be 
from Proposition~\ref{prop:case-b}.\footnote{ Requiring $\musrb(H) \le 1/2$ enables a uniform choice of $n_\star$ for all $H \in \cH(P_0, C_t)$.}
Next, choose an $n_\star$-admissible sequence $(\iota_j )_j$, $\iota_j \in \cI(\tau_*, \cK_*, E_*)$. 

Recall \eqref{eq:n k notation}:  For $u, v \in \mathbb{N}$, $v > u$,
let $T_{v,u} = T_{\iota_v} \circ \cdots \circ T_{\iota_{u+1}}$.
For each $k \ge 1$, the open system relative to $H_k$ is defined by 
$\rT_k : (T_{k n_\star, (k-1)n_\star})^{-1}(M \setminus H_k) \to M \setminus H_k$, where
\[
\rT_k(x) = T_{\iota_{kn_\star }} \circ \cdots \circ T_{\iota_{(k-1) n_\star}} (x) \, \mbox{ for } \, x \in (T_{kn_\star, (k-1)n_\star})^{-1}(M \setminus H_k) \, .
\]
To concatenate these open maps into a sequential system, define
\[
 \rT_{j,i}(x) = \rT_j \circ \cdots \circ \rT_i(x) \, \mbox{ for } \, x \in \cap_{l=1}^j \rT_i^{-1} \circ \cdots \circ \rT_l^{-1}(M \setminus H_l) \, ,
\]
thus we allow escaping once every $n_\star$ iterates along the admissible sequence.
The transfer operator for the sequential open system is defined by 
\begin{equation}\label{eq:sequential}
\rL_{j,i} f = \cL_{T_{\iota_{(j+1)n_\star}} \circ \cdots T_{\iota_{jn_\star}} } \ind_{H_j^c} \cdots \cL_{T_{\iota_{(i+1)n_\star}} \circ \cdots T_{\iota_{in_\star}} }
\ind_{H_i^c} f \, .
\end{equation}
We will be interested in the evolution of probability densities under the sequential system,
given by $\frac{\rL_{n,k} f}{ \int_M \rL_{n,k} f \, d\musrb }$. 
Note that if $f\in \cC_{c,A,L}(\delta)$ then $\int_M \rL_{n,k} f \, d\musrb > 0$ for each $n$ (thus the normalization is well defined).
When $f \ge 0$, this normalization coincides with the $L^1(\musrb)$ norm; however, we use the integral
rather than the $L^1$ norm as the normalization since the integral is order preserving with respect to our cone, while the $L^1$ norm is not.
We conclude the section with a result regarding exponential loss of memory for the sequence of open billiards.
\begin{theorem}
\label{thm:sequential}
Fix $\tau_*, \cK_*>0$ and $E_* <\infty$, and let $a,c, A, L, \delta$ and $\delta_0$ satisfy the 
conditions of Theorem~\ref{thm:cone contract} and Lemma~\ref{lem:order}.  Let $P_0, C_t > 0$.
There exist $C>0$ and $\vartheta <1$ such that for all sequences $(H_i)_i \subset  \cH(P_0, C_t)$ satisfying either
(O1) and (O2) or (O$1'$) and (O2), all $n_\star$-admissible sequences $(\iota_j)_j \subset \cI(\tau_*, \cK_*, E_*)$,
for all $\psi \in C^1(M)$, all $f, g \in  \cC_{c,A,L}(\delta)$, all $n\ge1$ and all $1 \le k \le n$,
\[
\left| \int_M \frac{\rL_{n,k} f}{ \musrb( \rL_{n,k} f ) } \, \psi \, d\musrb  
-  \int_M \frac{\rL_{n,k} g}{ \musrb (\rL_{n,k} g )} \, \psi \, d\musrb \right|
\le C L \vartheta^{n-k} |\psi|_{C^1(M)}  \, .
\]
\end{theorem}

\begin{proof}
Remark that the constants appearing in Propositions \ref{prop:case-a} and \ref{prop:case-b} are uniform,
depending only on $\cF(\tau_*, \cK_*, E_*)$, $P_0$ and $C_t$.
Hence, if $f, g \in\cC_{c,A,L}(\delta)$, then for each $k\leq n\in\bN$, $\rL_{n,k} f$, $\rL_{n,k} g \in \cC_{c,A,L}(\delta)$. 
Since  $\int_M \frac{\rL_{n,k} f}{ \musrb(\rL_{n,k} f )  } d\musrb = \int_M \frac{\rL_{n,k} g}{\musrb( \rL_{n,k} g ) } d\musrb =1$, the theorem follows arguing exactly as in the proof of Theorem \ref{thm:memory}(b),  using again the order preserving
semi-norm $\| \cdot \|_\psi$, as well as the fact that by Remark~\ref{rem:after order},
\[
\frac{ \| \rL_{n,k} f \|_\psi }{ \musrb( \rL_{n,k} f) } \le \bar C |\psi|_{C^1} \frac{\tri  \rL_{n,k} f \tri_+}{\tri \rL_{n,k} f \tri_-} \le \bar C L |\psi|_{C^1} \, .
\]
When invoking \eqref{eq:adapted}, it holds that $\rho_{\cC}( \rL_{n,k} f/ \musrb( \rL_{n,k} f ), \rL_{n,k} g/ \musrb( \rL_{n,k} g) ) = \rho_{\cC} (\rL_{n,k} f, \rL_{n,k} g)$ due to the projective nature of the metric. 
\end{proof}

Note that, by changing variables, 
$\int_M \rL_{n,k}  f \, \psi \, d\musrb = \int_{\rM_{n,k}} f \, \psi \circ \rT_{n,k} \, d\musrb$, 
where $\rM_{n,k} = \cap_{i=k}^n \rT_k^{-1} \circ \cdots \circ \rT_i^{-1}(M \setminus H_i)$.
Thus the conclusion of the theorem is equivalent to the expression,
\[
\left| \frac{\int_{\rM_{n,k}} f \, \psi \circ \rT_{n,k} \, d\musrb}{\int_{\rM_{n,k}} f \, d\musrb}
-  \frac{\int_{\rM_{n,k}} g \, \psi \circ \rT_{n,k} \, d\musrb}{\int_{\rM_{n,k}} g \, d\musrb} \right|
\le C  L  \vartheta^{n-k} |\psi|_{C^1(M)}   \, .
\]

Next we show that sequential systems with holes allow us to begin investigating some physical problems that have attracted much attention: chaotic scattering and random Lorentz gasses.


\subsection{Chaotic scattering (boxed)}
\label{sec:scattering}
Consider a collection of strictly convex pairwise disjoint obstacles $\{B_i\}$ in $\R^2$ {\em for which the non-eclipsing condition may fail}.\footnote{ Remember that the {\em non-eclipsing condition} is the requirement that the convex hull of any two obstacles does not intersect any other obstacle.} Assume that there exists a closed rectangular box $R=[a,b]\times[c,d]$ such that if an obstacle does not intersect its boundary, then it is contained in the box. In addition, if an obstacle intersects the boundary of $R$, then it is symmetrical with respect to a reflection  across all the linear pieces of the boundary which the obstacle intersects (see Figure \ref{fig:chaotic} for a picture).  Finally, we will assume a finite horizon condition on the cover $\widetilde Q$ defined after Remark~\ref{rem:whyB}.

\begin{remark}
The restriction regarding symmetrical reflections on the configuration of obstacles is necessary only because we did not develop the theory in the case of billiards in a polygonal box (see Remark \ref{rem:whyB} and the following text to see why this is relevant). Such an extension is not particularly difficult and should eventually be done. Other extensions that should be within reach of our technology are more general types of holes and billiards with corner points. Here, however, we are interested in presenting the basic ideas; addressing all possible situations would make our message harder to understand.
\end{remark}

\begin{figure}
\begin{centering}
\begin{tikzpicture}[scale=0.6]
\fill[gray!60, rotate=45] (6,3) ellipse (1 and 1.8);
\fill[gray!60, rotate=-45] (1,7) ellipse (1 and 2);
\fill[gray!60] (8,3) circle (1);
\fill[gray!60] (3,8) ellipse (1 and .5);
\fill[gray!60] (5,0) ellipse (.7 and 2);
\fill[gray!60] (2,4) ellipse (.5 and .3);
\fill[gray!60] (3,1.5) ellipse (.3 and .5);
\fill[gray!60] (0,0) circle (1.7);
\draw[dashed] (0,0)--(8,0);
\draw [dashed](0,0)--(0,8);
\draw[dashed](0,8)--(8,8);
\draw[dashed](8,8)--(8,0);
\draw (-8,2)--(-8,6);
\draw[->](-8,2)--(-6,2);
\draw[->](-8,3)--(-6,3);
\draw[->](-8,4)--(-6,4);
\draw[->](-8,5)--(-6,5);
\draw[->](-8,6)--(-6,6);
\node at (-7,1) {{\tiny Incoming particle beam}};
\end{tikzpicture}
\caption{Obstacle configuration for which the non-eclipse condition fails and the box $R$ (dashed line). }\label{fig:chaotic}
\end{centering}
\end{figure}
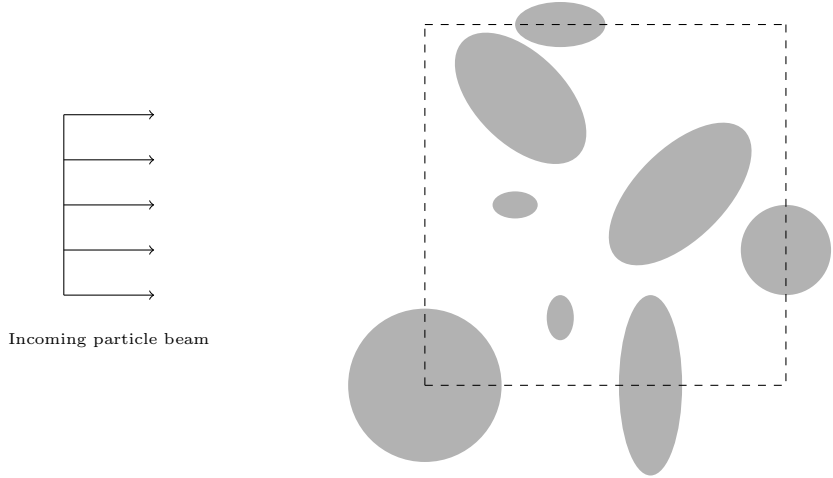

\begin{lemma}\label{lem:escape}
If a particle exits $R$ at time $t_0\in\R$, then, in the time interval $(t_0,\infty)$,  it will experience only a finite number of collisions and it will never enter $R$ again. 
\end{lemma}
\begin{proof}
Recall that $R=[a,b]\times[c,d]$.
Of course, the lemma is trivially true if, after exiting $R$, the particle has no collisions. Let us imagine that the particle, after exiting from the vertical side $(b,c)-(b,d)$, collides instead  with the obstacle $B_i$ at the point $p=(p_1,p_2)$. Note that $B_i$ must then intersect the same boundary, otherwise it would be situated to the left of the line $x=b$ and the particle could not collide since necessarily $p_1>b$. Our hypothesis that $B_i$ be symmetric with respect to reflection across $x=b$ implies that also $( 2b - p_1, p_2 )\in\partial B_i$. Thus, by the convexity of $B_i$, the horizontal segment joining $p$ and $(2b - p_1, p_2)$ is contained in $B_i$. This implies that, calling $\eta=(\eta_1,\eta_2)$ the normal to $\partial B_i$ in $p$, it must be $\eta_1\geq 0$. In addition, if $v=(v_1,v_2)$ denotes the particle's velocity just before collision, it must be that $v_1>0$ since the particle has crossed a vertical line to exit $R$. Finally, $\langle v,\eta\rangle\leq 0$, otherwise the particle would not collide with $B_i$. But since the velocity after collision is given by $v^+=v-\langle v,\eta\rangle\eta$, it follows $v^+_1=v_1-\langle v,\eta\rangle\eta_1\geq v_1$. That is, the particle cannot come back to the box $R$. Since all the obstacles are contained in a larger box $B_1$ and since there is a minimal distance between obstacles, the above also implies that the particle can have only finitely many collisions in the future. The other cases can be treated exactly in the same manner.
\end{proof}

\begin{remark}\label{rem:whyB}
We want to consider a scattering problem: the particles enter the box coming from far away and with random position and/or velocity, interact and, eventually, leave the box. The basic question is how long they stay in the box or, better, what is the probability that they stay in the box longer than some time $t$.
This is nothing other than an open billiard with holes. Unfortunately, the holes are large and our current theory allows us to deal with large holes only if enough hyperbolicity is present.  To extend the result to systems with small hyperbolicity is a very important (and hard) problem as one needs to understand the combinatorics of the trajectories for long times.

An alternative is to study the scattering problem under the non-eclipsing condition.
 Such an assumption avoids the technicalities associated with billiards
and results in Axiom A dynamics with a natural finite Markov partition for the collision map  
on the survivor set. This problem has been studied and strong results proving exponential escape as well as
exponential mixing on the survivor set have been obtained in both discrete \cite{Mo91, LM96, Mo07} as well as continuous \cite{St} time.  There are also recent results on the rigidity problem
for such open billiards \cite{DS20, DS22}.
Yet
the condition is artificial once there are more than 2 scatterers, hence the importance of developing an alternative approach.
\end{remark}
Given the above remark we modify the system in order to have the needed hyperbolicity. This is not completely satisfactory, yet it shows that our machinery can deal with large holes and illustrates exactly what further work is necessary to address the general case.

{\em Fixing $N$ sufficiently large, we suppose that when a particle enters the box, the boundaries of the box become reflecting and are transparent again only between the collisions $kN$ and $kN+1$, $k\in\bN$,  counting only collisions with the convex obstacles.} 

More precisely, consider the billiard in $R$ with elastic reflection at $\partial R$. We call such a billiard table $Q$. Let $M= \big( \cup_i\partial B_i \cap R \big) \times [-\frac \pi 2,\frac \pi 2]$ be the Poincar\'e section,\footnote{ Recall that  $\vf \in [-\frac \pi 2,\frac \pi 2]$ is the angle made
by the post-collision velocity vector and the outward pointing normal to the boundary.} and consider the Poincar\'e map $T:M\to M$ describing the dynamics from one collision with a convex body to the next.  Unfortunately, this is not a type of billiard that fits our assumptions since the table has corner points.
Yet, when the particle collides with  $\partial R$ we can reflect the box and imagine that the particle continues in a straight line. Note that, by our hypothesis, the image of the obstacles that intersect the boundary are the obstacles themselves; this is the reason why we restrict the obstacle configuration. We can then reflect the box three times, say across its right and top sides and then once more to make a full rectangle with twice the width and height of $R$,  and identify the opposite sides of this larger rectangle. In this way we obtain a torus $\bT^2$ containing pairwise disjoint convex obstacles. Such a torus is covered by four copies of $R$, let us call them $\{R_i\}_{i=1}^4$. We call such a billiard 
$\widetilde Q$, and we consider the Poincar\'e map $\widetilde T$ which maps from one collision with a convex body to the next,  and denote its phase space by $\widetilde M = \cup_{i=1}^4 \tM_i$.  

{\em Our final assumption on the obstacle configuration is that $\widetilde Q$  is a Sinai billiard with finite horizon. }\\
Hence $\widetilde T : \widetilde M \circlearrowleft$ falls within the scope of our theory. By construction there is a map $\pi:\widetilde M \to M$ which sends the motion on the torus to the motion in the box. Indeed, if $\tilde x\in\widetilde M$ and $x=\pi(\tilde x)$, then $T^n(x)=\pi(\widetilde T^n(\tilde x))$, for all $n\in\bN$.

We then consider the maps $\tilde S=\widetilde T^N$  and $S=T^N$, again $\pi(\widetilde S(\tilde x))=S(\pi(\tilde x))$. 
Define also the projections $\tilde \pi_1 : \widetilde M \to \widetilde Q$ and $\pi_1 : M \to Q$, which map
a point in the Poincar\'e section to its position on the billiard table.
For $\tilde x \in \widetilde M$, let us call $\widetilde O(\tilde x)$ the straight trajectory in $\bT^2$ between $\tilde \pi_1(\tilde x)$ and $\tilde \pi_1(\widetilde T(\tilde x))$, and setting $x = \pi(\tilde x)$, $O(x)$ the trajectory between $\pi_1(x)$ and $\pi_1(T((x))$. Note that the latter trajectory can consist of several straight segments joined at the boundary of $R$, where a reflection takes place. By construction, if $\widetilde O(\tilde x)$ intersects $m$ of the sets $\partial R_i$, then the trajectory  
$O(x)$ experiences $m$ reflections with $\partial R$. Accordingly, we introduce, in our billiard system $(\widetilde M,\widetilde S)$, the following holes : $\widetilde H=\widetilde T \{\tilde x\in\widetilde M\;:\; \widetilde O(\tilde x) \cap(\cup_i \partial R_i)\neq \emptyset\}$ and set $H=\pi(\widetilde H)$. 

The above makes precise the previous informal statement: the system $(M,S)$ with hole $H$, describes the dynamics of the billiard $(M,T)$ in which the  particle can exit $R$ only at the times $kN$, $k\in\bZ$.  
The transfer operator associated with the open system $(M, S; H)$
is $\ind_{H^c} \cL_S \ind_{H^c}$, yet since $(\ind_{H^c} \cL_S \ind_{H^c})^n = \ind_{H^c} (\cL_S \ind_{H^c})^n$, it is
equivalent to study the asymptotic properties of $\rL_S :=\cL_S \ind_{H^c}$.

For a function $f : M \to \mathbb{C}$, we define its lift $\tf : \tM \to \mathbb{C}$ by $\tf = f \circ \pi$.  The
pointwise identity then follows,
\begin{equation}
\label{eq:ptwise}
\rL_{\tS} \tf := \cL_{\tS} ( \ind_{\tH^c} \tf) = \cL_{\tS} ((\ind_{H^c} f) \circ \pi) = (\rL_S f) \circ \pi \, .
\end{equation}
While $\widetilde H$ is not exactly a hole of Type II, its boundary nevertheless comprises increasing curves since it is a forward image under the flow of a wave front with zero curvature (a segment of $\partial R_i$).   
Hence condition (O$1'$) of Section~\ref{sec:large}
holds with $P_0 =3$
and condition (O2) holds with $C_t$ depending only on the uniform angle between
the stable cone and the vertical and horizontal directions in $\widetilde M$. 
Thus Proposition~\ref{prop:case-b} applies to $\rL_{\tS}$ with $n_\star$ depending on 
$C_t$ and $P_0 = 3$.
In fact, our next result shows that also $\rL_S$ contracts $\cC_{c, A, L}(\delta)$ on $M$.

\begin{prop}
\label{prop:project}
Let $n_\star \in \mathbb{N}$ be from Proposition~\ref{prop:case-b} corresponding to $P_0=3$ and $C_t>0$.  
Then for each small enough $\delta>0$, there exist $c, A, L >0$, $\chi \in (0,1)$ such that
choosing $N \ge n_\star$, $\rL_S(\cC_{c,A, L}(\delta)) \subset \cC_{\chi c, \chi A, \chi L}(\delta)$, where
$S = T^N$.
\end{prop}

\begin{proof}
As already noted above, Proposition~\ref{prop:case-b} implies the existence of $\delta, c, A, L$ and $\chi$ such that
$\rL_{\tS}(\tC_{c,A, L}(\delta)) \subset \tC_{\chi c, \chi A, \chi L}(\delta)$ if we choose $N \ge n_\star$.  Note that
the constant $C_t$ is the same on $\tM$ and $M$.  In fact the same choice of parameters for the
cone works for $\rL_S$.

For any stable curve $W$, $\pi^{-1}W = \cup_{i=1}^4 \tW_i$ where each $\tW_i$ is a stable curve satisfying
$\pi(\tW_i) = W$.  Since $\pi$ is invertible on each $\tM_i$, we may define the restriction 
$\pi_i = \pi|_{\tM_i}$ such that $\pi_i^{-1}(W) = \tW_i$.  Conversely, the projection of any stable curve $\tW$ 
in $\tM$ is also a stable curve in $M$.

Since each $\pi_i$ is an isometry, and recalling \eqref{eq:ptwise},  for any stable curve $W \subset M$, each $f \in \cC_{c,A,L}(\delta)$, and all $n \ge 0$, 
\[
\int_{\tW_i} \psi \circ \pi \, \rL_{\tS}^n \tf \, dm_{\tW} = \int_W \psi \, \rL_S^n f \, dm_W, \quad \forall \; \psi \in C^0(\tW),
\]
where $\tf = f \circ \pi$.  Moreover, if $\psi \in \cD_{a, \beta}(W)$, then $\psi \circ \pi \in \cD_{a, \beta}(\tW_i)$, 
for each $i = 1, \ldots, 4$.  This implies in particular that 
$\tri \rL_S^n f \tri_{\pm} = \tri \rL_{\tS}^n \tf \tri_{\pm}$
for all $n \ge 0$, and that $f \in \cC_{c,A,L}(\delta)$ if and only if $\tf = f \circ \pi \in \tC_{c,A,L}(\delta)$.
Consequently, $\rL_S f \in \cC_{\chi c, \chi A, \chi L}(\delta)$ if and only if $\rL_{\tS} \tf \in \tC_{\chi c, \chi A, \chi L}(\delta)$,
which proves the proposition.
\end{proof}

In contrast to the sequential systems studied in Section~\ref{sequential}, the open billiard in this section corresponds to
a fixed billiard map $T$ (and its lift $\widetilde T$).  Thus we can expect the (normalized) iterates of $\rL_{S}$ to converge
to a type of equilibrium for the open system.  Such an equilibrium is termed a limiting or physical conditionally invariant measure in the
literature, and often corresponds to a maximal eigenvalue for $\rL_{S}$ on a suitable function space.  
Unfortunately, conditionally
invariant measures for open ergodic invertible systems are necessarily singular with respect to the invariant measure and 
so will not be contained in our cone $\cC_{c, A, L}(\delta)$, which is a set of functions.  
However, we will show that for our open billiard,
the limiting conditional invariant measure is contained in the completion of $\cC_{c,A,L}(\delta)$ with respect
to the following norm.

\begin{defin}
Let $\bV = \textrm{span} \big( \cC_{c,A,L}(\delta) \big)$.  For all $f \in \bV$ we define 
\[
\| f \|_\star = \inf \{ \lambda \ge 0 : - \lambda \preceq f \preceq \lambda \} \, .
\]
\end{defin}

\begin{lemma}
\label{lem:star}
The function $\|\cdot\|_\star$ has the following properties:
\begin{itemize}
  \item[a)] The function $\| \cdot \|_\star$ is an order-preserving norm, that is: $- g \preceq f \preceq g$ implies 
$\| f \|_\star \le \| g \|_\star$.
  \item[b)]  There exists $C>0$ such that for all $f \in \cC_{c,A,L}(\delta)$ and $\psi \in C^1(M)$,
  \[
  \left| \int_{M} f \, \psi \, d\musrb \right|  \le C \tri f \tri_+ | \psi |_{C^1(M)} \le C \| f \|_{\star} | \psi |_{C^1(M)} \, .
  \]
\end{itemize}
\end{lemma}

\begin{proof}
In this proof, for brevity we write $\cC$ in place of $\cC_{c,A,L}(\delta)$.

\smallskip
\noindent
a)  First, we show that $\| f \|_\star < \infty$ for any $f \in \bV$, i.e. for any $f \in \bV$ we can find $\lambda >0$
such that $\lambda + f, \lambda - f \in \cC$.
By the proof of Proposition~\ref{prop:diameter}, we claim\footnote{ This claim implies that the cone is Archimedean.} 
that for any $f \in \cC$, we can
find $\mu >0$ such that $\mu - f$ belongs to $\cC$.  This follows since the second part of the proof with $\chi = 1$ yields that $\mu - f$ satisfies \eqref{eq:cone 2} if $\mu \ge \frac{L}{L-1} \tri f \tri_+$, it satisfies \eqref{eq:cone 3} if $\mu \ge \frac{2A}{A- 2^{1-q}} \tri f \tri_+$, and
it satisfies \eqref{eq:cone 5} if $\mu \ge \frac{2cA}{cA - \delta - C_s} \tri f \tri_+$.  Taking $\mu$ large enough to satisfy these 3 conditions proves the claim.

Next, consider $f = \alpha g + \beta h$ with $g, h \in \cC$ and $\alpha, \beta \in \mathbb{R}$.  If $\alpha, \beta >0$, then since $\cC$ is closed under addition, the above claim yields 
$\mu>0$ such that $\mu -  f$ and $\mu + f$ are in $\cC$ and thus $\| f \|_\star \le \mu$.  
It remains to consider the case $\alpha < 0$, $\beta >0$ since the remaining cases are similar.
Let $\mu_g >0$ satisfy $\mu_g - g$ belongs to $\cC$.  Set $A = \mu_g |\alpha|$.  Then
$A + f = |\alpha| (\mu_g - g) + \beta h$ is the sum of elements in $\cC$ and thus is in $\cC$.
Similarly, let $\mu_h >0$ satisfy $\mu_h - h$ belongs to $\cC$ and set $B = \mu_h \beta$.
Then $B - f = |\alpha| g + \beta (\mu_h - h)$ is again in $\cC$.  Thus
$\| f \|_\star \le \max \{ A, B \}$.

Next, if $\| f \|_\star = 0$, then there exists a sequence $\lambda_n \to 0$ such that $- \lambda_n \preceq f \preceq \lambda_n$, and so $\lambda_n + f, \lambda_n - f \in \cC$ for each $n$.  Since $\cC$ is closed (see footnote~\ref{foo:close}), this yields
$f, -f \in \cC \cup \{ 0 \}$ and so $f = 0$ since $\cC \cap - \cC = \emptyset$ by construction.

Since $f \preceq g$ is equivalent to $\nu f \preceq \nu g$ for $\nu\in\bR_+$, it follows immediately that $\|\nu f\|_\star=\nu\| f \|_\star$.

To prove the triangle inequality, let $f, g\in\bV$.   For each $\ve>0$, there exists 
$a,b$, $a\leq \ve+\| f\|_\star$, $b\leq \ve+\|g\|_\star$ such that 
$-a \preceq f\preceq a$ and $-b \preceq g\preceq b$.  Then
\[
-(\| f\|_\star+\| g\|_\star+2\ve) \preceq -(a+b)\preceq f+g\preceq a+b\leq \| f\|_\star+\| g\|_\star+2\ve \, ,
\]
implies the triangle inequality by the arbitrariness of $\ve$. We have thus proven that $\|\cdot\|_\star$ is a norm.

Next, suppose that $-g\preceq f\preceq g$ and let $b$ be as above. Then
\[
-\| g\|_\star - \ve \preceq -b \preceq -g\preceq f\preceq g\preceq b \preceq \| g\|_\star + \ve
\]
which implies $\| f\|_\star\leq \| g\|_\star$, again by the arbitrariness of $\ve$. Hence, the norm is order preserving.

\smallskip
\noindent
b) The first inequality is contained in Remark~\ref{rem:after order}.  For the second inequality, we will prove
that
\begin{equation}
\label{eq:major}
\tri f \tri_+ \le \| f \|_\star \qquad \mbox{for all $f \in \cC$}.
\end{equation}
To see this, note that if $-\lambda \preceq f \preceq \lambda$, then $\tri \lambda - f \tri_- \ge 0$ by
Remark~\ref{rem:A-L}.  Thus for any $\tW \in \widetilde \cW^s$ and $\psi \in \cD_{a,\beta}(\tW)$,
\[
0 \le \frac{\int_{W} (\lambda - f) \, \psi}{\int_{W} \psi } \implies \frac{\int_{W} f \, \psi}{\int_{W} \psi} \le \lambda \, ,
\]
and taking suprema over $W$ and $\psi$ yields $\tri f \tri_+ \le \lambda$, which implies \eqref{eq:major}.
\end{proof}

Define $\bV_\star$ to be the completion of $\bV$ in the $\| \cdot \|_\star$ norm.  
$\bV_\star$ is a Banach space.
Let $\cC_\star$ be the closure of $\cC_{c, A, L}(\delta)$ in $\bV$.

We remark that by Lemma~\ref{lem:star}(b), $\cC_\star$ embeds naturally into $(C^1(M))'$, where $(C^1(M))'$ is the closure of $C^0(M)$
with respect to the norm
$\| f \|_{-1} = \sup_{| \psi |_{C^1}\leq1} \int_{M} f \psi \, d\musrb$. 
We shall show that the conditionally invariant measure for the open system $(M, T; H)$ belongs to $\cC_\star$.

\begin{theorem} 
\label{thm:open}
Let $(M, S; H)$ be as defined above, where $S = T^N$.  If $N \ge n_\star$, where $n_\star$ is from Proposition~\ref{prop:case-b}, then: 
\begin{itemize}
  \item[a)] $\displaystyle h := \lim_{n \to \infty} \frac{\rL_S^n 1}{\musrb( \rL_S^n 1 ) }$ is an element of $\cC_\star$.
  Moreover, $h$ is a nonnegative probability measure satisfying $\rL_S h = \nu h$ for some $\nu \in (0,1)$
  such that 
  $$\log \nu = \lim_{n \to \infty} \frac 1n \log \musrb(\cap_{i=0}^n S^{-i}(M \setminus H)) \, ,$$ 
  i.e. $- \log \nu$ is the escape rate of the open system.
  \item[b)]   There exists $C>0$ and $\vartheta \in (0,1)$ such that for all $f \in \cC_{c,A,L}(\delta)$ and $n \ge 0$,
  \[
\left\| \frac{\rL_S^n f}{ \musrb(\rL_S^n f)}  - h \right\|_\star \le C \vartheta^n\, .
\]  
  In addition, there exists a linear functional $\ell : \cC_{c, A, L}(\delta) \to \mathbb{R}$
  such that for all $f \in \cC_{c,A,L}(\delta)$, $\ell(f) > 0$ and 
  \[
\| \nu^{-n}\rL^n_{S} f-  \ell( f) h \|_\star \leq C \vartheta^n \ell( f ) \| h \|_\star.
\]
The constant $C$ depends on $\cC_{c,A,L}(\delta)$, but not on $f$.
\end{itemize}
\end{theorem}

\begin{remark}
(a)  The conclusions of Theorem~\ref{thm:open} apply equally well to the open system $(\tM, \tS; \tH)$.

\smallskip
\noindent
(b) By Lemma~\ref{lem:star}(b), the convergence in the $\| \cdot \|_\star$ norm given by Theorem~\ref{thm:open}(b) 
implies convergence when integrated against smooth functions $\psi \in C^1(M)$. As usual, by standard approximation arguments, the same holds for H\"older functions.

\smallskip
\noindent
(c) Also by Lemma~\ref{lem:star}(b), the above convergence in $\| \cdot \|_\star$ implies leafwise convergence as well.  First note that
for $W \in \cW^s(\delta)$, each $f \in \cC_{c,A,L}(\delta)$ induces a leafwise distribution on $W$ defined by
$f_W(\psi) = \int_W f \, \psi \, dm_W$, for $\psi \in \cD_{a,\beta}(W)$.  This extends by density to $f \in \cC_\star$.
Since $h \in \cC_\star$ by Theorem~\ref{thm:open}(a), let $h_W$ denote the leafwise measure induced by $h$ on $W \in \cW^s(\delta)$.
Then by Lemma~\ref{lem:star}(b) and Theorem~\ref{thm:open}(b), there exists $C>0$ such that for all $n \ge 0$,
\[
\left| \frac{\int_W \rL_S^n f \, \psi \, dm_W}{\musrb(\rL_S^n f)} - h_W(\psi) \right| \le C \delta^{-1} \vartheta^n \, ,
\quad \forall f \in \cC_{c,A,L}(\delta) , \forall \psi \in C^\beta(W) \, ,
\]
and also,
\[
\left| \nu^{-n} \int_W \rL_S^n f \, \psi \, dm_W - \ell(f) h_W(\psi) \right| \le C \delta^{-1} \vartheta^n \ell(f) \, .
\]
In particular, the escape rate with respect to $f dm_W$ on each $W \in \cW^s(\delta)$ equals the escape rate with respect to
$\musrb$. 
\end{remark}
\begin{proof}[Proof of Theorem~\ref{thm:open}]
We argue as in the proof of Theorem~\ref{thm:memory}.   Recalling that $\| \cdot \|_\star$ is an order-preserving
norm, we can apply \cite[Lemma 2.2]{LSV98},  taking the homogeneous function $\rho$ to also be $\| \cdot \|_\star$
and obtain that, as in \eqref{eq:adapted}, for all $f,g\in\cC_{c,A,L}(\delta)$,
\begin{equation}
\label{eq:cauchy}
\left\| \frac{\rL^n_{S} f}{\| \rL^n_{S} f \|_\star}-\frac{\rL^n_{S} g}{\| \rL^n_{S} g\|_\star} \right\|_\star \leq C\vartheta^n \, ,
\end{equation}
since $\left\| \frac{\rL^n_{S} f}{ \| \rL^n_{S} f \|_\star }  \right\|_* = 1$ and similarly for $g$.    
This implies that $\left( \frac{\rL^n_{S} f}{\| \rL^n_{S} f \|_\star} \right)_{n \ge 0}$ is a Cauchy sequence in the $\| \cdot\|_\star$ norm, 
and in addition, the limit is independent of $f$.
Hence, defining $h_0 = \lim_{n \to \infty} \frac{\rL_{S}^n 1}{\| \rL_{S}^n 1\|_\star}$, we have
$h_0 \in \cC_\star$ with $\| h_0 \|_\star = 1$  such that\footnote{ Note that $\rL_{S}$ extends naturally to $(\cC^1(M))'$
and therefore to $\cC_\star$.} for all $\psi \in C^1(M)$,
\[
\int_{M} \rL_{S} h_0 \psi =\lim_{n\to\infty}\frac1{\| \rL^n_{S} 1\|_\star}\int_{\tM} \rL^{n+1}_{S} 1\psi=\lim_{n\to\infty}\frac{\| \rL^{n+1}_{S} 1\|_\star}{\| \rL^n_{S}1\|_\star}\int_{M} h_0 \psi
=\| \rL_{S} h_0\|_\star \int_{M} h_0\psi=:\nu \int_{M} h_0\psi \, ,
\]
where all integrals are taken with respect to $\musrb$. Thus, $\rL_{S} h_0=\nu h_0$.
Moreover, the definition of $h_0$ implies that,
\begin{equation}
\label{eq:h0}
| h_0(\psi) | \le |\psi|_{C^0} \lim_{n \to \infty} \frac{ \musrb( \rL_{S}^n 1 )}{\| \rL_{S}^n 1\|_\star} = | \psi|_{C^0} h_0(1)
\, , \quad \forall \, \psi \in C^1(M) \, ,
\end{equation}
thus $h_0$ is a measure.  In addition, by the positivity of $\rL_{S}$, $h_0$ is a nonnegative measure and
since $\| h_0 \|_\star = 1$, it must be that $h_0(1) \neq 0$.  Thus we may renormalize and define
\[
h := \frac{1}{h_0(1)} h_0 \, .
\]
Then $\frac{\ind_{H^c} h}{h(H^c)}$ represents the limiting conditionally invariant probability measure for the open 
system $(M, S; H)$.  However, we will work with $h$ rather than its restriction to $H^c$ because $h$ contains information 
about entry into $H$, which we will exploit in Proposition~\ref{prop:entry} below.

Due to the equality in \eqref{eq:h0}, $h$ has the alternative characterization,
\[
h = \lim_{n \to \infty} \frac{\rL_{S}^n 1}{\musrb( \rL_{S}^n 1)} = \lim_{n \to \infty} \frac{\rL_{S}^n 1}{\musrb( \rM^n)} \, , 
\]
as required for item (a) of the theorem, where $\rM^n = \cap_{i=0}^n S^{-i} (M \setminus H)$ and
convergence is in the $\| \cdot \|_\star$ norm.

Remark that \eqref{eq:cauchy} implies $\frac{\rL_{S}^n f}{\| \rL_{S}^n f \|_\star}$ converges to $h_0$ at the
exponential rate $\vartheta^n$.  Integrating this relation and using Lemma~\ref{lem:star}(b), we conclude that
in addition the normalization ratio $\frac{\musrb( \rL_{S}^n f)}{\| \rL_{S}^n f \|_\star }$ converges to
$h_0(1)$ at the same exponential rate.  Putting these two estimates together and using the
triangle inequality yields for all $n \ge 0$,
\[
\left\| \frac{\rL_{S}^n f}{\musrb(\rL_{S}^n f)} - h \right\|_\star \le C \vartheta^n h_0(1)^{-1} \, , \quad \forall \, f \in \cC_{c,A,L}(\delta) \, ,
\]
proving the first inequality of item (b).

Next, for each, $f \in \cC_{c,A,L}(\delta)$ let
\begin{equation}\label{eq:conve_l}
\ell(f)=\limsup_{n\to\infty} \nu^{-n}  \musrb( \rL^n_{S} f ) \, . 
\end{equation}
Note that $\ell$ is bounded, homogeneous of degree one and order preserving. 
By Lemma~\ref{lem:star}(b), $\ell$ can be extended to $\cC_\star$.
Since $\ell(h)=1$,
$\nu^{-n} \rL_{S}^n h = h$ and $\ell(\nu^{-n}\rL^n_{S} f)=\ell(f)$ we can apply, again, 
\cite[Lemma 2.2]{LSV98}
as in \eqref{eq:adapted} to $f$ and $\ell(f)h$ and obtain
\begin{equation}
\label{eq:ell0-prop}
\| \nu^{-n}\rL^n_{S} f-  h \ell( f) \|_\star = \nu^{-n} \| \rL^n_{S} f - \ell(f) \rL^n_{S} h \|_\star 
\leq C \vartheta^n \ell( f ) \| h \|_\star \, ,
\end{equation}
proving the second inequality of item (b) of the theorem. Note that \eqref{eq:ell0-prop} also implies
(integrating and applying Lemma~\ref{lem:star}(b) ) that the limsup in \eqref{eq:conve_l} is, in fact, a limit, and hence $\ell$ is linear.  Remark that $\ell$ is also nonnegative for $f \in \cC_{c,A,L}(\delta)$ by Remark~\ref{rem:after order}. 

By definition, if $f \in \cC_{c,A,L}(\delta)$ and $\lambda> \| f \|_\star$ then $\lambda + f, \lambda - f \in \cC_{c,A,L}(\delta)$,
so that using the linearity and nonnegativity of $\ell$ yields,
\begin{equation}
\label{eq:dom ell}
- \lambda \ell(1) \le \ell(f) \le \lambda \ell(1) \, , \quad \forall \; f \in \cC_{c,A,L}(\delta), \; \lambda > \| f \|_\star \, .
\end{equation}
Thus either $\ell(f) = 0$ for all $f \in \cC_{c,A,L}(\delta)$ or $\ell(f) \neq 0$ for all $f \in \cC_{c,AL}(\delta)$.  But
if the first alternative holds, then by the continuity of $\ell$ with respect to the $\| \cdot \|_\star$ norm 
(Lemma~\ref{lem:star}(b)), $\ell$ is identically 0 on $\cC_\star$, which is a contradiction since $\ell(h) =1$.
Thus $\ell(f) > 0$ for all $f \in \cC_{c,A,L}(\delta)$.

Finally, applying \eqref{eq:ell0-prop} to $f \equiv 1$ integrated with respect to $\musrb$ and using again Lemma~\ref{lem:star}(b), we obtain
\[
| \nu^{-n} \musrb(\rM^n)  - \ell(1) | \le C \vartheta^n \ell(1) \| h \|_\star \, ,
\]
which in turn implies that $\log \nu = \lim_{n \to \infty} \frac 1n \log \musrb(\rM^n)$  since $\ell(1) \neq 0$, as required for the
remaining item of part (a) of the theorem.  Note that $\nu \neq 0$ by Remark~\ref{rem:after order} and \eqref{eq:lower big H}, while $\nu \neq 1$ by monotonicity since the escape rate for this class of billiards is known
to be exponential for arbitrarily small holes \cite{DWY, dem bill}.
\end{proof}

We can use Theorem~\ref{thm:open} to obtain exit statistics from the open billiard in the plane.
As an example, for $\theta \in [0, 2\pi)$ let us define $H_\theta$ to be the set of $x \in H$ such that
the first intersection of  $O(T^{-1} x)$ with $\partial R$ has velocity 
making an angle of $\theta$ with the positive horizontal axis.  
Note that $H_\theta$ is a finite union of  increasing
curves since it is the image of a wave front with zero curvature moving with parallel velocities.  
The fact that $H_\theta$
comprises increasing  curves is not altered by the fact that the flow in $R$ may reflect off of $\partial R$ 
several times before arriving at a scatterer because such collisions are neutral; also, since the corners 
of $R$ are right angles, the flow remains continuous at these corner points.

Suppose the incoming particles at time zero are distributed according to a probability 
measure $f d\musrb$ with density $f \in\cC_{c,A,L}(\delta)$. 
The probability that a particle leaves the box at time $nN$ with a direction in the interval $\Theta=[\theta_1, \theta_2]$, 
call it  $\bP_f(x_n\in[\theta_1, \theta_2])$, can be expressed as 
\begin{equation}
\label{eq:P def}
\bP_f(x_n\in[\theta_1, \theta_2]) = \int_M \ind_{H_\Theta} \rL_S^n f \, d\musrb \, ,
\end{equation}
 where $H_\Theta := \cup_{\theta \in \Theta} H_\theta$.  Although the boundary of $H_\Theta$ comprises
increasing curves as already mentioned, the restriction on the angle may prevent $\partial H_\Theta$ from enjoying the property of continuation of singularities common to billiards.  See Figure~\ref{fig:hole} (see also
\cite[Sect.~8.2.2]{dem bill} for other examples of holes without the continuation of singularities property).

Similarly, for $p \in \partial R$, define $H_p$ to be the set of $x \in H$ such that the last intersection of $O(T^{-1}x)$
with $\partial R$ is $p$.  Then for an interval $P \subset \partial R$, we define $H_P = \cup_{p \in P} H_p$,
and $\int_M \ind_{H_P} \rL_S^n f$ denotes the probability that a particle leaves the box at time $nN$ through
the boundary interval $P$.

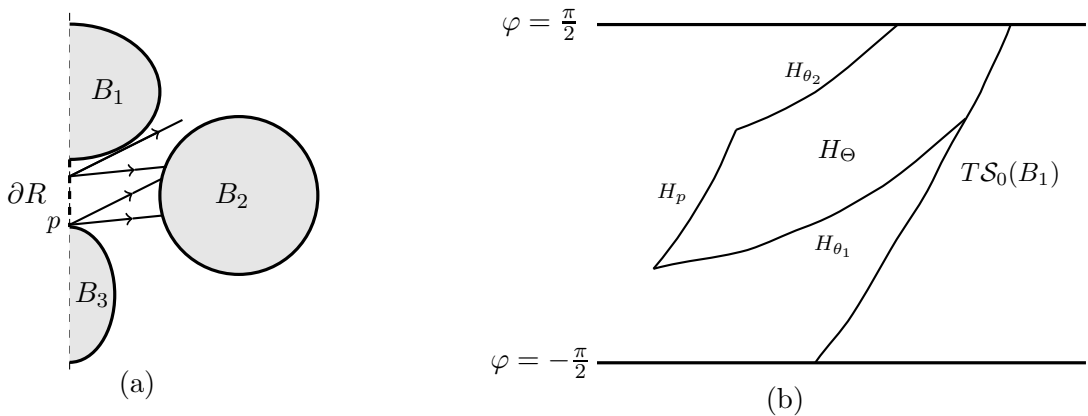
\begin{figure}[ht]
\begin{centering}
\begin{minipage}{.48 \linewidth} 
\hspace{1.5cm}	
\begin{tikzpicture}[scale=0.30]
\draw [dashed](0,3.5)--(0,-12.5);
\draw[very thick, dashed] (0,-3)--(0,-6);
\fill[gray!20!white] (0,3)--(0,-3) arc (-90:90:4 and 3)--cycle;
\draw[very thick] (0,-3) arc (-90:90:4 and 3);
\fill[gray!20!white] (0,-6)--(0,-12) arc (-90:90:2 and 3)--cycle;
\draw[very thick] (0,-12) arc (-90:90:2 and 3);
\fill[gray!20!white] (7.5,-4.6) circle (3.5);
\draw[very thick] (7.5,-4.6) circle (3.5);
\node at (1.7,0) {$B_1$};
\node at (7.2,-4.6) {$B_2$};
\node at (1,-9) {$B_3$};
\node at (-1.9, -4.5) {$\partial R$};
\node at (-0.7, -5.9) {\small $p$};
\node at (3, -13) {(a)};

\draw[thick, ->] (0,-5.9) -- (2.8, -5.62); 
\draw[thick] (2.8, -5.62) -- (4.1, -5.49);
\draw[thick, ->] (0,-5.9) -- (2.8, -4.5); 
\draw[thick] (2.8, -4.5) -- (4.1, -3.85);

\draw[thick] (0, -3.75) -- (5, -1.25); 
\draw[thick, ->] (0, -3.75) -- (4, -1.75);
\draw[thick, ->] (0, -3.75) -- (2.9, -3.46); 
\draw[thick] (2.9, -3.46) -- (4.3, -3.32);

\end{tikzpicture}

\end{minipage}
\begin{minipage}{.48 \linewidth}
\hspace{1.5cm}

\begin{tikzpicture}[scale = 0.5]

\draw[very thick] (2,11) -- (15,11);
\draw[very thick] (2,2) -- (15,2);

\draw[thick] (13,11) to[out=250, in=67.5] (12.2,9.2) to[out=242.5, in=65] (11,7) to [out=245, in=60] (9.8,5)
	to [out=240, in=57.5] (8.8,3.3) to [out=237.5, in=55] (7.8,2);
\draw[thick] (10,11) to[out=220, in=35] (7.8,9.2) to[out=210, in=20] (5.7,8.2);
\draw[thick] (11.8,8.5) to[out=220, in=35] (9.6,6.7) to[out=210, in=20] (7.5,5.6) to [out=200, in=18] (6,5)
	to [out=195, in=15] (3.5,4.5);

\draw[thick] (5.7,8.2) to[out=245, in=60] (5,6.7) to[out=240, in=50] (3.5,4.5);

\node at (0.5,2){$\vf = - \frac{\pi}{2}$};
\node at (0.5,11){$\vf = \frac{\pi}{2}$};
\node at (8.4,7.65){\small $H_{\Theta}$};
\node at (13, 7){\small $T\cS_0(B_1)$};
\node at (7.5,9.7){\scriptsize $H_{\theta_2}$};
\node at (8.3,5.1){\scriptsize $H_{\theta_1}$};
\node at (4,6.5){\scriptsize $H_p$};
\node at (7,1) {(b)};

\end{tikzpicture}
\end{minipage}
\end{centering}

\caption{a) Sample rays with $\theta = \theta_1$ and $\theta=\theta_2$
striking the scatterer $B_2$. The point $p$ is the topmost point of $\partial B_3$.
b) Component of $H_\Theta$ on the scatterer $B_2$.  In this configuration, $H_{\theta_1}$
intersects
the singularity curve $T\cS_0$ coming from $B_1$ while $H_{\theta_2}$ reaches $\cS_0$ directly; however,
the left boundary of $H_{\Theta}$ is an arc of $H_p$ and the continuation of singularities properties fails for a hole
of this type since $\theta_1 > 0$. }
\label{fig:hole}
\end{figure}

\begin{prop}
\label{prop:entry}
For any intervals of the form $\Theta = [\theta_1, \theta_2]$, or $P = [p_1, p_2]$, any $f \in C^1(M)$ with
$f \ge 0$ and
$\int f \, d\musrb = 1$, and all $n \ge 0$, we have\footnote{ If instead $f \in \cC_{c,A,L}(\delta)$, $f \ge 0$ and
$\int f \, d\musrb =1$, 
then $\| f \|_{C^1}$ can be dropped from the right hand side.}  
\[
\begin{split}
&\bP_{f} (x_n\in \Theta) =\nu^n h (\ind_{H_\Theta})  \ell(f)+ \|f\|_{C^1} \cO\big(\nu^n\vartheta^{\frac{q}{q+1}n}\big) \, ,
\quad \mbox{and} \quad\\
&\bP_{f} (x_n\in P) =\nu^n h (\ind_{H_P})  \ell(f)+  \|f\|_{C^1} \cO\big(\nu^n\vartheta^{\frac{q}{q+1}n}\big) \, .
\end{split}
\]
\end{prop}

\begin{remark}
If $f \in \cC_{c,A,L}(\delta)$, then $\ell(f) > 0$ by Theorem~\ref{thm:open}(b), and Proposition~\ref{prop:entry} provides a precise asymptotic 
for the escape of particles through $H_{\Theta}$ and $H_P$.  For more general $f \in C^1(M)$, it may be
that $\ell(f) =0$, in which case Proposition~\ref{prop:entry} merely gives an upper bound on the exit statistic compared
to the rate of escape given by $\nu$.
\end{remark}

\begin{proof}
We prove the statement for $\ind_{\Theta}$.  The statement for $\ind_P$ is similar.  

To start with we assume $f\in \cC_{c,A,L}(\delta)$, and  $f \ge 0$ with $\int f \, d\musrb = 1$.
As already mentioned,
$\partial H_\Theta$ comprises finitely many increasing curves in $M$ and so $H_\Theta$ satisfies
(O$1'$) and (O2) with $P_0=3$ and $C_t$ depending only on the uniform angle between the 
stable cone and $\partial H_\Theta$, which is strictly positive due to {\bf (H1)}. 
Since $\ind_{H_\Theta}$
is not in $C^1(M)$, we cannot apply Lemma~\ref{lem:star}(b) directly;  we will use a mollification to bypass this
problem.
 
Let $\rho : \bR^2 \to \bR^2$ be a nonnegative, $C^\infty$ function supported in the unit disk with $\int \rho = 1$, and define
$\rho_\ve(\cdot) = \ve^{-2} \rho( \cdot \, \ve^{-1} )$.  For $\ve > 0$, define the mollification,
\[
\psi_\ve(x) = \int \ind_{H_\Theta}(y) \rho_\ve(x-y) \, dy \qquad x \in M \, .
\]
We have $| \psi_\ve|_{C_0} \le 1$ and $| \psi_\ve' |_{C^0} \le C\ve^{-1}$.  Note that $\psi_\ve = \ind_{H_\Theta}$
outside an $\ve$-neighborhood of $\partial H_\theta$ (including $\cS_0$).  Letting $\tpsi_\ve$ denote a $C^1$ function with $|\tpsi_\ve|_{C^0} \le 1$,  which is 1 on $N_\ve(\partial H_\Theta)$ and 0 on $M \setminus N_{2 \ve}(\partial H_\Theta)$,
we have $| \ind_{H_\Theta} - \psi_\ve | \le \tpsi_\ve$.
Due to (O2), for any
$W \in \cW^s$ such that $W \cap N_\ve(\partial H_{\Theta}) \neq \emptyset$, 
using first the fact that $f \ge 0$ and then applying cone condition \eqref{eq:cone 3},
\begin{equation}
\label{eq:mollify}
\begin{split}
\int_W | \ind_{H_\Theta} - \psi_\ve | \, \rL_S^n f \, dm_W 
& \le \int_W \tpsi_\ve \, \rL_S^n f \, dm_W \le \int_{W \cap N_{2\ve}(\partial H_{\Theta})} \rL_S^n f \, dm_W \\
& \le 2^{1+q} A \delta^{1-q} C_t^q \ve^q \tri \rL_S^n f \tri_- \, ,
\end{split}
\end{equation}
where we have used the fact that $W \cap N_{2\ve}(\partial H_{\Theta})$ has at most 2 connected components
of length $2C_t \ve$.  Then integrating over $M$ and disintegrating $\musrb$ as in the proof of Lemma~\ref{lem:order},
we obtain,
\begin{equation}
\label{eq:epsilon}
\int_M | \ind_{H_\Theta} - \psi_\ve | \, \frac{\rL_S^n f}{\musrb(\rL_S^n f)} \, d\musrb 
\le \int_M  \tpsi_\ve \, \frac{\rL_S^n f}{\musrb(\rL_S^n f)} \, d\musrb  \le C \ve^q \frac{ \tri \rL_S^n f \tri_-}{\musrb(\rL_S^n f)} \, .
\end{equation}
By Remark~\ref{rem:after order}, $\musrb(\rL_S^n f) \ge \bar C^{-1} \tri \rL_S^n f \tri_-$, so the bound is uniform in $n$.
Since $\tpsi_\ve \in C^1(M)$ the  bound carries over to $h(\tpsi_\ve)$, and
since $h$ is a nonnegative measure, to $h(\ind_{H_\Theta} - \psi_\ve)$.  Thus for each $n \ge 0$ and $\ve>0$,
\begin{equation}
\label{eq:cone asy}
\begin{split}
\int \ind_{H_\Theta} \, \rL_S^n f \, d\musrb 
& = \int  ( \ind_{H_\Theta} - \psi_\ve ) \, \rL_S^n f \, d\musrb
+ \left( \int \psi_\ve \, \rL_S^n f \, d\musrb - \nu^n \ell(f) h(\psi_\ve) \right) \\
& \qquad + \nu^n \ell(f) h(\psi_\ve - \ind_{H_\Theta}) +
\nu^n \ell(f) h(\ind_{H_\theta}) \\
& = \cO\big(\ve^q \nu^n \ell(f)\big) + \cO\big(|\psi_\ve|_{C^1} \nu^n \vartheta^n \ell(f) \big) + \nu^n \ell(f) h(\ind_{H_\Theta}) \, ,
\end{split}
\end{equation}
where we have applied \eqref{eq:epsilon} to the first and third terms and 
Theorem~\ref{thm:open}(b) and Lemma~\ref{lem:star}(b) to the second term.  
Since $|\psi_\ve|_{C^1} \le \ve^{-1}$, choosing $\ve = \vartheta^{n/(q+1)}$
yields the required estimate for $f\in \cC_{c,A,L}(\delta)$. 

To conclude, note that by Lemma~\ref{lem:dominate}, there exists $C_\flat>0$ such that, if $f\in C^1(M)$, then, 
for each $\lambda\geq C_\flat \|f\|_{\cC^1}$, $\lambda+f\in\cC_{c,A,L}(\delta)$. 
Hence, by the linearity of the integral, $\ell(f)$ as defined in \eqref{eq:conve_l} can be extended to $f \in C^1$
by $\ell(f) = \ell(\lambda + f) - \ell(\lambda)$, and the limsup is in fact a limit since since the limit exists for $\lambda + f, \lambda \in \cC_{c,A,L}(\delta)$ (see \eqref{eq:ell0-prop} and following). 

Now take $f \in C^1$ with $\int f \, d\musrb =1$ and $\lambda \ge C_\flat \| f \|_{C^1}$ as above.   
Then, necessarily $\lambda + f \ge 0$, and so recalling \eqref{eq:P def}, we have
\[
\begin{split}
\bP_{\frac{\lambda+f}{1+\lambda}}(x_n\in \Theta)&=\int_M \ind_{H_\Theta} \rL_S^n { \left( \tfrac{\lambda+f}{1+\lambda} \right) }
=\frac{\lambda}{1+\lambda}\int_M \ind_{H_\Theta} \rL_S^n1+\frac{1}{1+\lambda}\int_M \ind_{H_\Theta} \rL_S^n f \\
&= \frac{\lambda}{1+\lambda}\bP_{1}(x_n\in \Theta)+\frac{1}{1+\lambda}\bP_{f}(x_n\in \Theta) .
\end{split}
\]
Hence by \eqref{eq:cone asy},
\[
\begin{split}
\bP_{f}(x_n\in \Theta)&=(1+\lambda)\bP_{\frac{\lambda+f}{1+\lambda}}(x_n\in \Theta)-\lambda \bP_{1}(x_n\in \Theta)\\
&=\nu^n h (\ind_{H_{\Theta}}) \big(\lambda \ell(1)+\ell(f) \big)-\nu^n h (\ind_{H_{ \Theta}}) \lambda \ell(1)+ \lambda \cO\big(\nu^n\vartheta^{\frac{q}{q+1}n}\big)\\
&=\nu^n h (\ind_{H_{ \Theta}}) \ell(f)+\|f\|_{\cC^1}\cO\big(\nu^n\vartheta^{\frac{q}{q+1}n}\big).
\end{split}
\]
\end{proof}


\subsection{Random Lorentz gas (lazy gates)} 
\label{sec:lorentz}
Consider a Lorentz gas as described in \cite[Section 2]{AL}. That is, we have a lattice of cells of size one with 
circular obstacles of fixed  radius $r$ at their corners and a random obstacle $B(z)$ of fixed 
radius $\rho$ and center in a set $\cO$ at their interior.\footnote{ The assumption that all obstacles are circular is not essential and can be relaxed by requiring that the obstacles at the corners are symmetric with respect to reflections as described in Section~\ref{sec:scattering}. }
The central obstacle is small enough not to intersect with the other obstacles but large enough to prevent trajectories from crossing the cell without colliding with an obstacle.
We call the openings between different cells {\em gates}, see Figure 5b, and require that no trajectory can cross two
gates without making at least one collision with the obstacles.  Thus we fix $r$ and $\rho$ satisfying\footnote{ Finite horizon requires $r \ge \frac{1}{1+\sqrt{2}}$, yet our added condition that a particle cannot cross diagonally from, say, $\hat R_1$ to $\hat R_2$ without making a collision requires further that $r \ge \frac 13$.}
the following conditions:
\begin{equation}
\label{eq:r rho}
 \tfrac{1}{3} \le r < \tfrac 12 \, , \quad \mbox{and} \quad 1-2r < \rho < \tfrac{\sqrt{2}}{2} - r \, .
\end{equation}
With $r$ and $\rho$ fixed, the set of 
possible configurations of the central obstacle are described by $\omega \in \Omega=\cO^{\bZ^2}$. 
In order to ensure that particles cannot cross directly from $\hat R_1$ to $\hat R_3$ or from $\hat R_2$ to $\hat R_4$
without colliding with an obstacle, and to ensure a minimum distance between scatterers, we fix $\ve_* > 0$ and require the center $c = (c_1, c_2)$ of the random obstacle 
$B_\omega$, $\omega \in \Omega$, 
(the central obstacle $C_5$ in Figure~5b) to satisfy,
\begin{equation}
\label{eq:c}
1 - (r+\rho - \ve_*)  \le c_1, c_2 \le r+\rho - \ve_* \, .
\end{equation}
Note that \eqref{eq:r rho} and \eqref{eq:c} imply that all possible positions of the central scatterer $B_\omega$ result
in a billiard table with $\tau_{\min} \ge \tau_* := \min \{ \ve_*, 1-2r \} > 0$.  

On $\Omega$ the space of translations $\xi_z$, $z\in\bZ^2$, acts naturally as $[\xi _z(\omega)]_x=\omega_{z+x}$, see Figure 5a. We assume that the obstacle configurations are described by a measure $\bP_e$ which is ergodic with respect to the translations.

\begin{figure}[ht]\
\begin{minipage}{.48 \linewidth} 
\hspace{1.5cm}	
\begin{tikzpicture}[scale=0.40]
\fill[gray!20!white] (0,0) circle (1.5);
\draw (0,0) circle (1.5);
\fill[gray!20!white] (4,0) circle (1.5);%
\draw (4,0) circle (1.5);
\fill[gray!20!white] (8,0) circle (1.5);%
\draw (8,0) circle (1.5);
\fill[gray!20!white] (0,-4) circle (1.5);
\draw (0,-4) circle (1.5);
\fill[gray!20!white] (4,-4) circle (1.5);%
\draw (4,-4) circle (1.5);
\fill[gray!20!white] (8,-4) circle (1.5);%
\draw (8,-4) circle (1.5);
\fill[gray!20!white] (0,-8) circle (1.5);
\draw (0,-8) circle (1.5);
\fill[gray!20!white] (4,-8) circle (1.5);%
\draw (4,-8) circle (1.5);
\fill[gray!20!white] (8,-8) circle (1.5);%
\draw (8,-8) circle (1.5);
\fill[gray!20!white] (2.2,-2) circle (1);
\draw (2.2,-2) circle (1);
\node at (2.2,-2) {{\tiny$ B_\omega(a)$}};
\fill[gray!20!white] (6,-1.8) circle (1);
\draw (6,-1.8) circle (1);
\node at (6,-1.8) {{\tiny$ B_\omega(b)$}};
\fill[gray!20!white] (2,-6.3) circle (1);
\draw (2,-6.3) circle (1);
\node at (2,-6.3) {{\tiny$ B_\omega(0)$}};
\fill[gray!20!white] (6,-5.7) circle (1);
\draw (6,-5.7) circle (1);
\node at (6.1,-5.7) {{\tiny$ B_\omega(c)$}};
\node at  (5,-11) {$a=(1,0); b=(1,1); c=(1,0)$};
\end{tikzpicture}
\end{minipage}
\begin{minipage}{.48 \linewidth}
\hspace{1.3cm}
\begin{tikzpicture}[scale=0.60]
\draw[dashed] (0,0)--(8,0);
\draw [dashed](0,0)--(0,-8);
\draw[dashed](0,-8)--(8,-8);
\draw[dashed](8,-8)--(8,0);
\draw[very thick, dashed] (3,0)--(5,0);
\draw[very thick, dashed] (3,-8)--(5,-8);
\draw[very thick, dashed] (0,-3)--(0,-5);
\draw[very thick, dashed] (8,-5)--(8,-3);
\fill[gray!20!white] (0,0)--(0,-3.5) arc (-90:0:3.5)--cycle;
\draw[very thick] (0,-3.5) arc (-90:0:3.5);
\fill[gray!20!white] (0,-8)--(3.5,-8) arc (0:90:3.5)--cycle;
\draw[very thick] (3.5,-8) arc (0:90:3.5);
\fill[gray!20!white] (8,-8)--(4.5,-8) arc (180:90:3.5)--cycle;
\draw[very thick] (4.5,-8) arc (180:90:3.5);
\fill[gray!20!white] (8,0)--(4.5,0) arc (180:270:3.5)--cycle;
\draw[very thick] (4.5,0) arc (180:270:3.5);
\fill[gray!20!white] (4.3,-4.4) circle (1.5);
\draw[very thick] (4.3,-4.4) circle (1.5);
\node at (1.7,-1.2) {$C_2$};
\node at (6.3,-1.2) {$C_1$};
\node at (1.7,-6.2) {$C_3$};
\node at (6.4,-6.2) {$C_4$};
\node at (4,-4) {$C_5$};
\node at (8.7,-4) {$\hat R_1$};
\node at (-.8,-4) {$\hat R_3$};
\node at (4,-.7) {$\hat R_2$};
\node at (4,-7.3) {$\hat R_4$};
\node at (.3,-1.5) {$r$};
\node at (5,-4.7) {$\rho$};
\draw (4.3,-4.4)--(5.8,-4.4); 
\end{tikzpicture}
\end{minipage}
\\ \vskip.1cm
{ Fig 5a: {\it Configuration of random obstacles $B_\omega(z)$}\hskip 1.2cm Fig 5b: {\it Poincar\'e section $C_i$ and gates $\hat R_i$\hspace{.5cm}}}
\end{figure}

Exactly as in the section \ref{sec:scattering}, we assume that the gates are reflecting and become transparent only after $N$ collisions with the obstacles. Thus when the particle enters a cell it will stay in that cell for at least $N$ collisions 
with the obstacles, hence the {\em lazy} adjective. 

As described in section \ref{sec:scattering}, when the particle reflects against a gate one can reflect the table  three times and see the flow (for the times at which the gates are closed) as a flow in a finite horizon Sinai billiard on the two torus. 
Note that the Poincar\'e section $M=\cup_{i=1}^5 C_i\times[-\frac\pi 2,\frac \pi 2]$ in each cell is exactly the same for each $\omega$ and $z$
since the arclength of the boundary is always the same, 
while the Poincar\'e map $T_z$ changes depending on the position of the central obstacle, see Figure 5b. 
Let us call $\cF(\tau_*)$ the collection of the different resulting billiard maps
corresponding to tables that maintain a minimum distance $\tau_*>0$ between obstacles, as required by
\eqref{eq:r rho} and \eqref{eq:c}. 
  (Note that the
parameters $\cK_*$ and $E_*$ of Section~\ref{sequential} are fixed in this class once $r$ and $\rho$ are fixed.)
The only difference with Section~\ref{sec:scattering}, as far as the dynamics in a cell is concerned, 
consists in the fact that we have to be more specific about which cell the particle enters, 
as now exiting from one cell means entering into another.

Recalling the notation of Section~\ref{sec:scattering}, if we call $R(z)$ the cell at the position $z\in\bZ^2$, then the 
gates $\hat R_i$ are subsets of $\partial R(z)$.  We denote  by $\widetilde R(z)$
the lifted cell (viewed as a subset of $\mathbb{T}^2$) after reflecting $R(z)$ three times, and
by $(\tM, \tT_z)$ the corresponding billiard map.  As before, the projection $\pi: \tM \to M$ satisfies
$\pi \circ \tT = T \circ \pi$. 
Then the hole $\widetilde H(z)$ can be written as $\widetilde H(z)=\cup_{i=1}^4 \widetilde H_i(z)$, where 
$\pi(\tH_i(z))=: H_i(z)$ are the points $x \in M$ such that $O(T^{-1}x) \cap \partial R(z)  \in \hat R_i$.\footnote{  The hole depends on the trajectory of $x$, which is different in different cells and hence depends on $z$,
while the gates $\hat R_i$ are independent of $z$.}  Due to our assumption \eqref{eq:r rho}, this point of
intersection is unique for each $x$ since consecutive collisions with $\partial R$ cannot occur.  
Then $H(z) = \pi(\tH(z)) = \cup_{i=1}^4 H_i(z)$.

As discussed in Section~\ref{sec:scattering}, the holes are neither of Type I nor of Type II, 
yet they satisfy (O1$'$) and (O2) with $P_0 = 3$ and $C_t$ depending only on the uniform angle 
between the stable cone for the induced billiard map and the horizontal and vertical directions.

Yet for our dynamics, when a particle changes cell at the $N$th collision, it is because after $N-1$ collisions,
that particle is in $G_i(z) := T_z^{-1}H_i(z)$, and in fact it will never reach $H_i(z)$.  Unfortunately, the geometry
of $G(z) := \cup_{i=1}^4 G_i(z)$ is not convenient for our machinery  since
$\partial G(z)$ may contain stable curves,
yet we will be able reconcile this difficulty after
defining the dynamics precisely as follows.

The phase space is $\bZ^2\times M$. For $x \in M$, denote by $p(x)$ the position of $x$ in $R(z)$ and by
$\theta(x)$ the angle of its velocity with respect to the positive horizontal axis in $R(z)$. 
We define 
\[
w(z, x )=\begin{cases} 0=:w_0 &\text{ if } x  \not \in G(z)\\
e_1=: w_1&\text{ if } x \in G_1(z)\\
e_2=:w_2&\text{ if } x \in G_2(z)\\
-e_1=:w_3&\text{ if } x \in G_3(z)\\
-e_2=:w_4&\text{ if } x \in G_4(z).
\end{cases}
\]
Also we set $\mathfrak{W}=\{w_0,\dots, w_4\}$.
If $ x  \in G_i(z)$, then we call $\bar q(x)=(q,\theta) \in \hat R_i\times  [0, 2\pi)$ the point 
$\bar q$ such that $q=O(x)\cap \hat R_i$ and $\theta = \theta(x)$, i.e. without reflection at $\hat R_i$.
We then consider $\bar q$ as a point in the cell $z+w(z,x) = z + w_i$ and call $T_{z,i}(x)$ the post-collisional velocity at the next collision with an obstacle under the flow starting at $\bar q$.  
Note that in the cell $R(z+w_i)$, $\bar q \in \hat R_{\bar i}$, where $\bar i = i + 2 \mbox{ (mod 4*)}$.\footnote{ By (mod 4*) we mean cyclic addition on 1, 2, 3, 4 rather than 0, 1, 2, 3.}
Thus if $\Phi_t^z$ denotes the flow in $R(z)$, then with this notation, $G_i(z)$ is the projection on $M$ of
$\hat R_i$ under the inverse flow $\Phi_{-t}^z$ while $H_{\bar i}(z+w(z,x))$ is the projection on $M$
of $\hat R_{\bar i}$ under the forward flow $\Phi_t^{z+w_i}$.  Thus,
\begin{equation}
\label{eq:correct}
H_{\bar i}(z + w_i) = T_{z,i} G_i(z) \implies \ind_{G_i(z)} \circ T_{z,i}^{-1} = \ind_{H_{\bar i}(z+w_i)} \, ,
\end{equation}
which is a relation we shall use to control the action of the relevant transfer operators below. 

Differing slightly from the previous section, here it is convenient to set $S_z=T_z^{N-1}$, and define
\[
F(z,x)=\begin{cases} (z, S_z\circ T_z(x))=:(z, \widehat S_z(x)) &\text{ if } x\not \in G(z)\\
(z+w(z,x), S_{z+w(z,x)}\circ T_{z,i}(p))=:(z+w(z,x), \widehat S_z(x))&\text{ if } x\in G_i(z).
\end{cases}
\]
We set $(z_n,x_n)=F^n(z,x)$ and we call $n$ the {\em macroscopic time}, which corresponds to $Nn$ collisions
with the obstacles.
The above corresponds to a dynamics in which when the particle enters a cell, it is 
trapped in the cell for $N$ collisions with the obstacles; then the gates open and until the next collision the particle can change cells, after 
which it is trapped again for $N$ collisions and so on.

We want to compute the probability that a particle visits the sets  $G_{k_0}(z_0),\cdots G_{k_{n-1}}(z_{n-1})$,  in this order, where we have set $G_0(z)=M\setminus \cup_{i=1}^4 G_i(z)$. Similarly, we define $H_0(z) = M \setminus \cup_{i=1}^4 H_i(z)$.  This itinerary corresponds to a particle that at time $i$ changes its position in the lattice by 
$w_{k_i}$. 
Following the notation of \cite{AL}, we call $\bP_\omega$ the probability distribution in the path space 
$\mathfrak{W}^\bN$ conditioned on the central obstacles being in the positions specified  
by $\omega\in\Omega$. 
Hence, if the particle starts from the cell $z_0 = (0,0)$ with $x$  distributed according to a 
probability measure $f d\musrb$ with density $f \in \cC_{c,A,L}(\delta)$, then we have\footnote{ Since $z_0 = (0,0)$, it is equivalent to
specify $z_1, \ldots z_n$ or $w_{k_0}, \ldots w_{k_{n-1}}$ since $w_{k_j}$ can be recovered as $w_{k_j} = z_{j+1} - z_j$.} $z_n= \sum_{k=0}^{n-1} w_{k_i}$ and, for each obstacle distribution $\omega\in \Omega$,
\begin{equation}\label{eq:path_prob}
\begin{split}
\bP_\omega(z_0, z_1,\dots, z_n)  =&\int_M  f(x) \ind_{G_{k_0}(z_0)}(x)\ind_{G_{k_1}(z_1)}(\widehat S_{0}(x))\cdots \\
&\phantom{\int_M}
\cdots \ind_{G_{k_{n-1}}(z_{n-1})}(\widehat S_{z_{n-2}}\circ\cdots\circ \widehat S_{0}(x)) \, d\musrb(x)  \\
=&\int_M\rL_{G_{k_{n-1}}(z_{n-1})}\cdots \rL_{G_{k_0}(z_0)} f \, d\musrb 
\end{split}
\end{equation}
where $\rL_{G_{k_j}(z_j)} := \cL_{T_{z_{j+1}}}^{N-1}\cL_{T_{z_j,k_j}}\ind_{G_{k_j}(z_j)}$, and we have set $T_{z,0}:=T_z$. See \cite{AL} for more details. We will prove below that if  $N$ is sufficiently large, then Theorem \ref{thm:sequential} applies to each 
operator $\rL_{G_{k}}$.  This suffices to obtain an exponential  loss of memory property (the analogue of the result obtained for piecewise expanding maps in \cite[Theorem 6.1]{AL}), that is property {\bf Exp} in \cite[Section 4.1]{AL}. 
This is the content of the following theorem.

\begin{theorem}
\label{thm:mixing}
There exist $C_*>0$, $\vartheta\in (0,1)$ and $N \in \mathbb{N}$ such that for $\bP$-a.e. $\omega \in \Omega$,
if $x$ is distributed according to $f \in \cC_{c,A,L}(\delta)$, with $f \ge 0$ and $\int_M f \, d\musrb= 1$,
$z_0=(0,0)$, and the gates open only once every $N$ collisions, 
then for all $n > m \ge 0$ and 
all $w \in \mathfrak{W}^\bN$,
\begin{equation}\label{eq:gibbs}
\left|  \bP_{\omega}(w_{ k_n }\; | \; w_{ k_0} \ldots w_{ k_{n-1}}) - \bP_{\xi_{z_{m}} \omega} ( w_{k_n} \; | \; w_{k_m} \ldots w_{k_{n-1}}) \right| \leq C_* \vartheta^{n-m}.
\end{equation}
\end{theorem}

\begin{proof}
Note that for $m \ge 0$, $\xi_{z_m}\omega$ sends the cell at $z_m$ to (0,0).  Thus according to
equation \eqref{eq:path_prob}, for $x$ distributed according to $f \in \cC_{c,A,L}(\delta)$ with $z_0 = (0,0)$,
we have
\[
\bP_{\xi_{z_m}\omega} (w_{k_m}, \ldots w_{k_n}) = \int_M \rL_{G_{k_n}(z_n)} \cdots \rL_{G_{k_m}(z_m)}f\, d\musrb\, .
\]
As remarked earlier, the sets $G_i(z)$ do not satisfy assumption (O2) so that Proposition~\ref{prop:case-b} does not
apply directly.  Yet, it follows from \eqref{eq:correct} that for $g \in \cC_{c,A,L}(\delta)$,
\[
\rL_{G_{k_j}(z_j)} g =  \cL_{T_{z_{j+1}}}^{N-1}\cL_{T_{z_j,k_j}}(\ind_{G_{k_j}(z_j)} g)  
= \cL_{T_{z_{j+1}}}^{N-1} \big( \ind_{H_{\bar k_j}(z_{j+1})} \cL_{T_{z_j,k_j}}g \big) \, ,
\]
where, as before, $\bar k_j = k_j + 2$ (mod 4*).  Then, just as in the proof of Proposition~\ref{prop:case-b},
it may be the case that $\cL_{T_{z_j, k_j}}g$ is not in $\cC_{c,A,L}(\delta)$.  Yet,
it is immediate from our estimates in Section~\ref{sec:cone} that $\cL_{T_{z_j, k_j}}g \in \cC_{c',A', 3L}(\delta)$
for any billiard map $T_{z_j, k_j} \in \cF(\tau_*)$ for some constants $c', A'$ depending only on $\cF(\tau_*)$.  
As in the proof of Proposition~\ref{prop:case-b}, we may choose 
constants $c'' \ge c'$, $A'' \ge A'$ and $L'' \ge 3L$ and $\delta >0$ sufficiently small to satisfy
the hypotheses of Theorem~\ref{thm:cone contract}.
Then since the sets $H_i(z)$ do satisfy (O$1')$ and (O2) with $P_0 = 3$ and $C_t$ depending only on the
angle between the stable cone and the vertical and horizontal directions, which has a uniform minimum in the family $\cF(\tau_*)$, 
there exists $\chi<1$ and $N$ sufficiently large as in Proposition~\ref{prop:case-b} so that\footnote{ Here in fact our operators are of the form $\cL^n\ind_H$ while in Proposition~\ref{prop:case-b} they have the form
$\cL_n \ind_{H^c}$ for some set $H$.  Yet, this is immaterial since the boundaries of $H$ and $H^c$ in $M$ are the
same so that (O$1'$) and (O2), and in particular Lemma~\ref{lem:H3}, apply equally well to both sets.}
$\big[ \cL_{T_{z_{j+1}}}^{N-1} \ind_{H_{\bar k_j}(z_{j+1})} \big] \cC_{c',A', 3L}(\delta) \subset \cC_{\chi c, \chi A, \chi L}(\delta)$, and both $\chi$ and $N$ are independent of $z_{j+1}$ and $k_j$.  This implies in particular that
\[
\rL_{G_i(z)} \cC_{c,A,L}(\delta) \subset \cC_{\chi c, \chi A, \chi L}(\delta) \qquad \mbox{for each $i$ and all $z \in \bZ^2$.}
\]
Now the assumption that the gates only open every $N$ collisions implies that for every 
$\omega \in \Omega$, every path is the result of
an $N$-admissible sequence.

As in the proof of Theorem~\ref{thm:memory}, using the fact that $\musrb( \cdot )$ is homogeneous and
order preserving on $\cC_{c,A,L}(\delta)$ and that $\musrb(\bar \cL_m f) = \musrb(f) =1$, where
$\bar \cL_m f = \frac{\rL_{G_{k_m-1}(z_{m-1})} \cdots \rL_{G_{k_0}(z_0)} f}{\int_M \rL_{G_{k_m-1}(z_{m-1})} \cdots \rL_{G_{k_0}(z_0)} f} \in \cC_{c,A,L}(\delta)$, we estimate as in \eqref{eq:adapted} and \eqref{eq:contract},
\begin{equation}
\label{eq:integral contract}
\begin{split}
\int_M  & \rL_{G_{k_{n-1}}(z_{n-1})} \cdots \rL_{G_{k_m}(z_m)} (f - \bar\cL_m f) \, d\musrb \\
& \le C \vartheta^{n-m} \min \left\{ \int_M \rL_{G_{k_{n-1}}(z_{n-1})} \cdots \rL_{G_{k_m}(z_m)} f,
\int_M  \rL_{G_{k_{n-1}}(z_{n-1})} \cdots \rL_{G_{k_m}(z_m)} \bar\cL_m f \right\} \, ,
\end{split}
\end{equation}
for some $\vartheta<1$ depending on the diameter of $\cC_{\chi c, \chi A, \chi L}(\delta)$ in $\cC_{c,A,L}(\delta)$. 

Finally, the left hand side of \eqref{eq:gibbs} reads
\[
\begin{split}
&\left| \frac{\int_M\rL_{G_{k_{n-1}}(z_{n-1})}\cdots \rL_{G_{k_0}(z_0)} f}{\int_M\rL_{G_{k_{n-2}}(z_{n-2})}\cdots \rL_{G_{k_0}(z_0)}f}
-\frac{\int_M\rL_{G_{k_{n-1}}(z_{n-1})}\cdots \rL_{G_{k_m}(z_m)} f}{\int_M\rL_{G_{k_{n-2}}(z_{k_{n-2}})}\cdots \rL_{G_{k_m}(z_m)}f }\right| \\
& \le \left| \frac{\int_M\rL_{G_{k_{n-1}}(z_{n-1})}\cdots \rL_{G_{k_m}(z_m)} \bar \cL_m f -  \int_M\rL_{G_{k_{n-1}}(z_{n-1})}\cdots \rL_{G_{k_m}(z_m)} f   }{\int_M\rL_{G_{k_{n-2}}(z_{n-2})}\cdots \rL_{G_{k_m}(z_m)} \bar\cL_m f} \right| \\
& \qquad + \left| \frac{\int_M\rL_{G_{k_{n-1}}(z_{n-1})}\cdots \rL_{G_{k_m}(z_m)} f}{\int_M\rL_{G_{k_{n-2}}(z_{k_{n-2}})}\cdots \rL_{G_{k_m}(z_m)} \bar \cL_m f }     - \frac{\int_M\rL_{G_{k_{n-1}}(z_{n-1})}\cdots \rL_{G_{k_m}(z_m)} f}{\int_M\rL_{G_{k_{n-2}}(z_{k_{n-2}})}\cdots \rL_{G_{k_m}(z_m)}f }\right| \\
& \le C \vartheta^{n-m} + C \vartheta^{n-m-1} \, ,
\end{split}
\]
where we have applied \eqref{eq:integral contract} twice and used the fact that
$\frac{\int_M\rL_{G_{k_{n-1}}(z_{n-1})}\cdots \rL_{G_{k_m}(z_m)} g}{\int_M\rL_{G_{k_{n-2}}(z_{k_{n-2}})}\cdots \rL_{G_{k_m}(z_m)}g } \le 1$ for any $g \in \cC_{c,A,L}(\delta)$.  
\end{proof}

In particular, Theorem \ref{thm:mixing}, together with\footnote{ Remark that \cite[Theorem~6.4]{AL} requires
$\musrb(G_i(z))$ to be the same for each $i$ and $z$, independently of $\omega$.  This is precisely the case
here since $G_i(z)$ is defined as the projection of $\hat R_i$ under the inverse flow $\Phi^z_{-t}$, and 
Leb$(\hat R_i \times [0, 2\pi))$ in the phase space of the flow is independent of $i$, while 
$\musrb$ is the projection onto $M$ of Lebesgue measure, which is invariant under the flow.}
  \cite[Theorem~6.4]{AL},
implies that $\lim_{n\to \infty} \frac 1nz_n =0$ for $\bP_e$ almost all $\omega$, that is, the walker has, $\bP_e$-almost-surely, no drift.   See \cite[Section~6]{AL} for details.\footnote{The arguments in \cite[Section~6]{AL} are developed for expanding maps, but the relevant parts apply verbatim to the present context.}
This latter fact could be deduced also from the ergodicity result in \cite[Theorem 5.4]{Lenci06}; however, Theorem  \ref{thm:mixing} is much stronger (indeed, by \cite[Theorem 6.4]{AL}, it implies \cite[Theorem 5.4]{Lenci06}) since it proves some form of memory loss that is certainly not implied by ergodicity alone. It is therefore sensible to expect that more information on the random walk will follow from Theorem  \ref{thm:mixing}, although this will require further work.

We conclude with a corollary of Theorem~\ref{thm:mixing} which implies the same exponential loss of
memory for particles distributed according to two different initial distributions.  For $f \in \cC_{c,A,L}(\delta)$,
let $\bP_{\omega, f}( \cdot)$ denote the probability in the path space 
$\mathfrak{W}^\bN$ conditioned on the central obstacles being in position $\omega\in\Omega$ and with
$x$ initially distributed according to $f d\musrb$. 

\begin{cor}
There exist $C>0$ and $\vartheta\in (0,1)$ such that for all $f, g \in \cC_{c,A,L}(\delta)$ with $\int_M f = \int_M g =1$
and $\bP$-a.e. $\omega \in \Omega$, if
$z_0=(0,0)$, then for all $n \ge 0$ and 
all $w \in \mathfrak{W}^\bN$,
\[
\left|  \bP_{\omega, f}(w_{ k_n} \; | \; w_{ k_0} \ldots w_{ k_{n-1}}) - \bP_{\omega, g} ( w_{k_n} \; | \; w_{k_0} \ldots w_{k_{n-1}}) \right| \leq C \vartheta^n.
\]
\end{cor}

\begin{proof}
The proof is the same as that of Theorem~\ref{thm:mixing} since \eqref{eq:integral contract} holds as well with
$\bar \cL_m f$ replaced by $g$.
\end{proof}

\section*{Declarations}

\subsection*{Competing Interests}

This work was supported by the PRIN Grant ``Regular and stochastic behavior in dynamical systems" (PRIN 2017S35EHN). C. Liverani acknowledges the MIUR Excellence Department Project awarded to the Department of Mathematics, University of Rome Tor Vergata, CUP E83C18000100006.  M. Demers was partially supported by NSF grants DMS 1800321 and DMS 2055070.

C. Liverani is a member of the Editorial Board of Communications in Mathematical Physics.

\subsection*{Data Sharing}

Data sharing is not applicable to this article as no datasets were generated or analyzed during the current study.

\newpage


%
%
%


\small

\begin{thebibliography}{9999}
 \bibitem[AL]{AL} R.~Aimino, C.~Liverani \emph{Deterministic walks in random environment}. Annals of Probability, Volume {\bf 48}, Number 5 (2020), 2212-2257.
 
\bibitem[AFGV]{atnip}  J.~Atnip, G.~Froyland, C.~Gonz\'alez-Tokman and S.~Vaienti, \emph{Thermodynamic formalism for 
random weighted covering systems}, to appear in Comm. Math. Phys.

\bibitem[B1]{Ba1} V.~Baladi, \emph{Anisotropic Sobolev spaces and
    dynamical transfer operators: $\cC^\infty$ foliations},	 Algebraic and Topological Dynamics, Sergiy
    Kolyada, Yuri Manin and Tom Ward, eds.	Contemporary Mathematics, Amer. Math. Society, (2005)
    123-136. 
 
 \bibitem[B2]{Babook} V. Baladi, {\em Dynamical zeta functions and dynamical determinants for hyperbolic maps. A functional approach}. Results in Mathematics and Related Areas. 3rd Series. A Series of Modern Surveys in Mathematics, {\bf 68}. Springer, Cham, 2018.
 
   \bibitem[BD1]{max}  V.~Baladi and M.F.~Demers, \emph{On the measure of maximal entropy
  for finite horizon Sinai billiard maps,}  J. Amer. Math. Soc. {\bf 33}  (2020) 381--449.
  
  \bibitem[BD2]{thermo} V.~Baladi and M.F.~Demers, \emph{Thermodynamic formalism for dispersing billiards,} preprint 2020.
   
  \bibitem[BDL]{BDL} V.~Baladi, M.F.~Demers and C.~Liverani {\em Exponential decay of correlations for finite horizon Sinai billiard flows}. Invent. Math. {\bf 211} (2018), no. 1, 39--177.
  
  \bibitem[BT]{BaT} V. Baladi and M. Tsujii, {\em Anisotropic H\"older
    and Sobolev spaces for hyperbolic diffeomorphisms}, Ann. Inst. Fourier.
    {\bf 57} (2007), 127-154.
      
 \bibitem[BDKL]{DS20} P.~Balint, J.~De~Simoi, V.~Kaloshin and M.~Leguil, \emph{Marked length spectrum, homoclinic orbits and the geometry of open dispersing billiards,} Comm. Math. Phys.
 {\bf 374} (2020), 1531--1575.
      
      
 \bibitem[Bir]{Bir}  G.~Birkhoff, \emph{ Extensions of Jentzsch's theorem}, Trans. Amer. Math. Soc. {\bf 85} (1957), 219--227.
 
 \bibitem[BKL]{BKL} M. Blank; G. Keller; C. Liverani, {\em Ruelle-Perron-Frobenius spectrum for Anosov maps}. Nonlinearity {\bf 15} (2002), no. 6, 1905--1973. 

\bibitem[BSC]{bsc} L.~Bunimovich, Ya.~G.~Sinai, and N.~Chernov, \emph{Markov partitions
  for two-dimensional hyperbolic billiards}, Russian Math. Surveys {\bf 45} (1990), 105-152.
 
\bibitem[C1]{Ch06} N.~Chernov, \emph{Advanced statistical properties of dispersing billiards}, J. Stat. Phys. {\bf 122} (2006), 1061--1094.
 
 \bibitem[C2]{Ch08} N.~Chernov, {\em Sinai billiards under small external forces II},
   Ann. Henri Poincar\'e, {\bf{9}} (2008), 91--107.

 \bibitem[CD]{CD} N.~Chernov and D.~Dolgopyat,
{\em Brownian Brownian Motion -- I}, Memoirs of American Mathematical Society, \textbf{198}: 927 (2009).  
   
  \bibitem[CM]{chernov book} N.~Chernov and R.~Markarian, \emph{Chaotic Billiards}, 
  Mathematical Surveys and Monographs {\bf 127}, Amer. Math. Soc. (2006), 316 pp.
   
   \bibitem[CZ]{chernov zhang}  N~.Chernov and H.-K.~Zhang, \emph{On statistical properties of hyperbolic systems with
   singularities}, J. Stat. Phys. {\bf 136} (2009), 615--642.
   
   \bibitem[D1]{dem inf} M.F.~Demers, \emph{Escape rates and physical measures for the infinite horizon Lorentz gas with holes}, Dynamical Systems: An International Journal 
    {\bf 28}:3 (2013), 393--422.   
  \bibitem[D2]{dem bill}  M.F.~Demers, \emph{Dispersing billiards with small holes}, in Ergodic theory, open dynamics and coherent structures, Springer Proceedings in Mathematics {\bf 70} 
  (2014), 137--170.   
   
   
   
   \bibitem[DL]{DL08} M.F~Demers; C. Liverani, {\em Stability of statistical properties in two-dimensional piecewise hyperbolic maps}. Trans. Amer. Math. Soc.{\bf 360} (2008), no. 9, 4777--4814.
   
  \bibitem[DWY]{DWY} M.F.~Demers, P.~Wright and L.-S.~Young, \emph{Escape rates and physically relevant measures for billiards with small holes}, Comm. Math. Phys. {\bf 294}:2 (2010), 353-388.   
  
  \bibitem[DZ1]{demzhang11} M.F.~Demers and H.-K.~Zhang, \emph{Spectral analysis of
  the transfer operator for the Lorentz Gas}, J. Modern Dnyam. {\bf 5}:4 (2011), 665-709.
  \bibitem[DZ2]{demzhang13} M.F.~Demers and H.-K.~Zhang, \emph{A functional analytic
  approach to perturbations of the Lorentz gas}, Comm. Math. Phys. {\bf 324}:3 (2013), 767--830.
     \bibitem[DZ3]{demzhang14} M.F.~Demers and H.-K.~Zhang, \emph{Spectral analysis of hyperbolic systems with
  singularities}, Nonlinearity {\bf 27} (2014), 379--433.

\bibitem[DKL1]{dkl21} M. F. Demers, N. Kiamari, C.  Liverani, \emph{ Transfer operators in hyperbolic dynamics. An introduction.} 33 Colloquio Brasilero de Matematica. Brazilian Mathematics Colloquiums series, Editora do IMPA. pp.252 (2021). ISBN 978-65-89124-26-9.

\bibitem[DeL1]{DS16}   J. De Simoi and C. Liverani, {\em Statistical properties of mostly contracting fast-slow partially hyperbolic systems}. Invent. Math. {\bf 206} (2016), no. 1, 147--227.
  
\bibitem[DeL2]{DS18}   J. De Simoi and C. Liverani, {\em  Limit theorems for fast-slow partially hyperbolic systems}. Invent. Math. {\bf 213} (2018), no. 3, 811--1016.

\bibitem[DLPV]{DLPV}   J. De Simoi, C. Liverani and C. Poquet; D. Volk, {\em Fast-slow partially hyperbolic systems versus Freidlin-Wentzell random systems}. J. Stat. Phys. {\bf 166} (2017), no. 3-4, 650--679.       

\bibitem[DKL2]{DS22}  J.~De~Simoi, V.~Kaloshin and M.~Leguil, \emph{Marked length 
spectral determination of analytic chaotic billiards with axial symmetries,}
to appear in Inventiones Math. 

\bibitem[DS]{DS} N. Dobbs and M. Stenlund, {\em Quasistatic dynamical systems}. Ergodic Theory Dynam. Systems {\bf 37} (2017), no. 8, 2556--2596.
\bibitem[Do04a]{Do04a}
D.~Dolgopyat.
\newblock Limit theorems for partially hyperbolic systems.
\newblock {\em Trans. Amer. Math. Soc.}, 356(4):1637--1689 (electronic), 2004.

\bibitem[Do04b]{Do04b}
D.~Dolgopyat.
\newblock On differentiability of {SRB} states for partially hyperbolic
  systems.
\newblock {\em Invent. Math.}, 155(2):389--449, 2004.

\bibitem[Do05]{Do05}
D.~Dolgopyat.
\newblock Averaging and invariant measures.
\newblock {\em Mosc. Math. J.}, 5(3):537--576, 742, 2005.
  
   
 \bibitem[DFGV1]{DFGV1} D. Dragi\v{c}evi\'c; G. Froyland; C. González-Tokman; S. Vaienti {\em A spectral approach for quenched limit theorems for random expanding systems}. Comm. Math. Phys. {\bf 360}:3  (2018), 1121--1187.
  
 \bibitem[DFGV2]{DFGV2} D. Dragi\v{c}evi\'c; G. Froyland; C. González-Tokman; S. Vaienti {\em A spectral approach for quenched limit theorems for random hyperbolic dynamical systems}. Trans. Amer. Math. Soc. {\bf 373}:1  (2020), 629--664.
 
  
  
  \bibitem[Fr86]{Fr86} D.Fried, {\em The zeta functions of Ruelle and Selberg}. I. Ann. Sci. École Norm. Sup. (4) {\bf 19} (1986), no. 4, 491--517.
  
  \bibitem[GO]{ott2} B.~Geiger and W.~Ott, \emph{Nonstationary open dynamical systems},
  arXiv:1808.05315v4 (July 2019).

\bibitem[GL]{gouezel liverani} S. Gou\"ezel and C. Liverani, \emph{Banach spaces
adapted to Anosov systems}, Ergodic Th. Dynam. Sys. {\bf 26}:1 (2006), 189-217.

\bibitem[GL2]{gouezel liverani2} S. Gou\"ezel and C. Liverani, \emph{Compact locally maximal hyperbolic sets for smooth maps: fine statistical properties}. J. Differential Geom. {\bf 79} (2008), no. 3, 433--477. 
\bibitem[He]{Hen} H.Hennion, \emph{Sur un theoreme spectral et son application aux noyaux Lipchitziens,} Proceedings of the American Mathematical Society, {\bf 118} (1993), 627--634.

\bibitem[KL99]{KL99} G. Keller; C. Liverani, {\em Stability of the spectrum for transfer operators}. Ann. Scuola Norm. Sup. Pisa Cl. Sci. (4) {\bf 28} (1999), no. 1, 141--152. 

\bibitem[Ki99]{Kitaev} A.Yu. Kitaev, {\em Fredholm determinants for hyperbolic
diffeomorphisms of finite smoothness}, Nonlinearity, {\bf 12} (1999), 141--179.

\bibitem[Kry]{Kry} N.S. Krylov, {\em Works on the foundations of statistical physics}.
Translated from the Russian by A. B. Migdal, Ya. G. Sinai and Yu. L. Zeeman. With a preface by A. S. Wightman. With a biography of Krylov by V. A. Fock. With an introductory article "The views of N. S. Krylov on the foundations of statistical physics'' by Migdal and Fok. With a supplementary article "Development of Krylov's ideas'' by Sinai. Princeton Series in Physics. Princeton University Press, Princeton, N.J., 1979.

 \bibitem[LY]{LY73}
    A.~Lasota and J.~Yorke.
    {\em On the existence of invariant measures for piecewise monotonic transformations}.
    Trans. Amer. Math. Soc. {\bf 186} (1973), 481--488 (1974).

\bibitem[Le06]{Lenci06}M.~Lenci.
\newblock Typicality of recurrence for {L}orentz gases.
\newblock {\em Ergodic Theory Dynam. Systems}, 26(3):799--820, 2006.
  
  \bibitem[L95a]{liv95} C.~Liverani, \emph{Decay of correlations}, Ann. Math. {\bf 142} (1995), 239-301.
  \bibitem[L95b]{liv95b}   C.~Liverani, \emph{Decay of correlations for piecewise expanding maps}. J. Statist. Phys. {\bf 78} (1995), no. 3-4, 1111--1129.
  \bibitem[LSV]{LSV98} C.~Liverani, B.~Saussol and S.~Vaienti.
\emph{Conformal measure and decay of correlation for covering weighted
  systems}, Ergodic Theory Dynam. Systems {\bf 18}:6 (1998), 1399--1420.
  \bibitem[LM]{LM}  C.~Liverani andV.~Maume-Deschamps, \emph{Lasota-Yorke maps with holes: conditionally invariant probability measures and invariant probability measures on the survivor set}. Ann. Inst. H. Poincar\'e Probab. Statist. {\bf 39} (2003), no. 3, 385--412. 
  
\bibitem[LoM]{LM96} A.~Lopes and R.~Markarian, \emph{Open billiards: invariant and conditionally invariant probabilities on Cantor sets}.
SIAM J. Appl. Math.{\bf  56} (1996), no. 2, 651--680. 

 \bibitem[MO]{ott1} A.~Mohapatra and W.~Ott, \emph{Memory loss for nonequilibrium open dynamical systems},
Discrete and Contin. Dynam. Sys. (Series A) {\bf 34}:9 (2014), 3747--3759.

\bibitem[Mo91]{Mo91} T.~Morita,
{\em The symbolic representation of billiards without boundary condition}.
Trans. Amer. Math. Soc. {\bf 325} (1991), no. 2, 819--828. 


\bibitem[Mo07]{Mo07}  T. Morita, \emph{Meromorphic extensions of a class of zeta functions for two-dimensional billiards without eclipse}. Tohoku Math. J. (2) {\bf 59} (2007), no. 2, 167--202.


\bibitem[RS]{RS75}
    D.~Ruelle \& D.~Sullivan.
    {\em Currents, flows and diffeomorphisms.}
    Topology {\bf 14} (1975), no. 4, 319--327.
    
 \bibitem[Ru76]{Ruelle76}    D. Ruelle. {\em Zeta-functions for expanding maps and Anosov flows}. Invent. Math., {\bf 34}, 231--242, 1976.
 
\bibitem[Ru96]{Rugh} H.H. Rugh, {\em Generalized Fredholm determinants and
Selberg zeta functions for Axiom A dynamical systems}, Ergod.
Th.\& Dynam. Sys., {\bf 16} (1996), 805-819.

\bibitem[S]{Sinai}  Ya.G.~Sinai, \emph{Dynamical systems with elastic reflections.  Ergodic properties of dispersing billiards,} Russ. Math. Surveys {\bf 25} (1970), 137--189.

\bibitem[St]{St} L.~Stoyanov, \emph{Spectrum of the Ruelle operator and exponential decay of 
correlations for open billiard flows,} Amer. J. Math {\bf 123}:4 (2001), 715--759.

\bibitem[SYZ]{young zhang}  M.~Stenlund, L.-S.~Young and H.-K.~Zhang,
\emph{Dispersing billiards with moving scatterers}, Comm. Math. Phys. {\bf 332}:3 (2013), 909--955.

\bibitem[Y98]{You98} L-S Young, \emph{Statistical properties of dynamical systems with some hyperbolicity}. Ann. of Math. (2) {\bf 147} (1998), no. 3, 585--650.

\bibitem[Y99]{You99} L-S Young, \emph{Recurrence times and rates of mixing}. Israel J. Math. {\bf 110} (1999), 153--188.   

\bibitem[Z]{zhang}  H.-K. Zhang, \emph{Current in periodic Lorentz gases with twists}, Comm. Math. Phys. {\bf 306} (2011), 747--776.

\end{thebibliography}
\end{document}